\documentclass[oneside,british]{amsart}
\usepackage[T1]{fontenc}
\usepackage[latin9]{inputenc}
\usepackage{array}
\usepackage{longtable}
\usepackage{amstext}
\usepackage{amsthm}
\usepackage{amssymb}

\makeatletter

\providecommand{\tabularnewline}{\\}

\numberwithin{equation}{section}
\numberwithin{figure}{section}
\theoremstyle{plain}
\newtheorem{thm}{\protect\theoremname}[section]
\theoremstyle{definition}
\newtheorem{defn}[thm]{\protect\definitionname}
\theoremstyle{remark}
\newtheorem{rem}[thm]{\protect\remarkname}
\theoremstyle{plain}
\newtheorem{cor}[thm]{\protect\corollaryname}
\theoremstyle{remark}
\newtheorem*{rem*}{\protect\remarkname}
\theoremstyle{definition}
\newtheorem*{defn*}{\protect\definitionname}
\theoremstyle{plain}
\newtheorem{lem}[thm]{\protect\lemmaname}
\theoremstyle{plain}
\newtheorem{prop}[thm]{\protect\propositionname}
\theoremstyle{plain}
\newtheorem*{thm*}{\protect\theoremname}

\usepackage[pdftex,pdfpagelabels,bookmarks,hyperindex,hyperfigures]{hyperref}

\usepackage{thmtools}

\makeatother

\usepackage{babel}
\providecommand{\corollaryname}{Corollary}
\providecommand{\definitionname}{Definition}
\providecommand{\lemmaname}{Lemma}
\providecommand{\propositionname}{Proposition}
\providecommand{\remarkname}{Remark}
\providecommand{\theoremname}{Theorem}

\begin{document}
\title{On self-affine measures associated to strongly irreducible and proximal
systems}
\author{\noindent Ariel Rapaport}
\thanks{This research was supported by the Israel Science Foundation (grant
No. 619/22). The author is a Horev Fellow at the Technion -- Israel Institute of
Technology.}
\begin{abstract}
Let $\mu$ be a self-affine measure on $\mathbb{R}^{d}$ associated
to an affine IFS $\Phi$ and a positive probability vector $p$. Suppose
that the maps in $\Phi$ do not have a common fixed point, and that
standard irreducibility and proximality assumptions are satisfied
by their linear parts. We show that $\dim\mu$ is equal to the Lyapunov
dimension $\dim_{L}(\Phi,p)$ whenever $d=3$ and $\Phi$ satisfies
the strong separation condition (or the milder strong open set condition).
This follows from a general criteria ensuring $\dim\mu=\min\{d,\dim_{L}(\Phi,p)\}$,
from which earlier results in the planar case also follow. Additionally,
we prove that $\dim\mu=d$ whenever $\Phi$ is Diophantine (which
holds e.g. when $\Phi$ is defined by algebraic parameters) and the
entropy of the random walk generated by $\Phi$ and $p$ is at least
$(\chi_{1}-\chi_{d})\frac{(d-1)(d-2)}{2}-\sum_{k=1}^{d}\chi_{k}$,
where $0>\chi_{1}\ge...\ge\chi_{d}$ are the Lyapunov exponents. We
also obtain results regarding the dimension of orthogonal projections
of $\mu$.
\end{abstract}

\maketitle

\section{\label{sec:Introduction}Introduction}

\subsection{\label{subsec:Background}Background}

Fix $d\ge1$ and let $\Phi=\{\varphi_{i}(x)=A_{i}x+a_{i}\}_{i\in\Lambda}$
be a finite collection of invertible affine contractions of $\mathbb{R}^{d}$.
That is, $a_{i}\in\mathbb{R}^{d}$, $A_{i}\in\mathrm{GL}(d,\mathbb{R})$
and $\Vert A_{i}\Vert_{op}<1$ for each $i\in\Lambda$, where $\Vert\cdot\Vert_{op}$
is the operator norm. Such a collection is called an affine iterated
function system (IFS). It is well known (see \cite{Hut}) that there
exists a unique nonempty compact $K_{\Phi}\subset\mathbb{R}^{d}$
with $K_{\Phi}=\cup_{i\in\Lambda}\varphi_{i}(K_{\Phi})$. It is called
the attractor or self-affine set corresponding to $\Phi$.

Certain natural measures are associated to $\Phi$. For a topological
space $X$ denote the collection of compactly supported Borel probability
measures on $X$ by $\mathcal{M}(X)$. Given a probability vector
$p=(p_{i})_{i\in\Lambda}$ there exists a unique $\mu\in\mathcal{M}(\mathbb{R}^{d})$
which satisfies the relation $\mu=\sum_{i\in\Lambda}p_{i}\cdot\varphi_{i}\mu$
(again, see \cite{Hut}), where $\varphi_{i}\mu:=\mu\circ\varphi_{i}^{-1}$
is the pushforward of $\mu$ via $\varphi_{i}$. The measure $\mu$,
which is supported on $K_{\Phi}$, is called the self-affine measure
corresponding to $\Phi$ and $p$. From now on we shall always assume
that $p_{i}>0$ for each $i\in\Lambda$, in which case $\mathrm{supp}(\mu)=K_{\Phi}$.

Many mathematical problems surround self-affine sets and measures,
but perhaps the most natural one is to determine their dimension.
It has been studied by many authors, and its computation is one of
the major open problems in fractal geometry. In what follows $\dim_{H}E$
stands for the Hausdorff dimension of a set $E$.

In the 1980's Falconer \cite{falconer1988hausdorff} introduced a
natural upper bound for $\dim_{H}K_{\Phi}$, which is called the affinity
dimension. It is denoted by $\dim_{A}\Phi$ and depends only on the
linear parts $\{A_{i}\}_{i\in\Lambda}$. Falconer has shown that when
$\Vert A_{i}\Vert_{op}<1/2$ for $i\in\Lambda$, and under a natural
randomization of the translations $\{a_{i}\}_{i\in\Lambda}$, the
equality
\begin{equation}
\dim_{H}K_{\Phi}=\min\{d,\dim_{A}\Phi\}\label{eq:dim K =00003D dim_A}
\end{equation}
holds almost surely. (In fact Falconer proved this with $1/3$ as
the upper bound on the norms; it was subsequently shown by Solomyak
\cite{So} that $1/2$ suffices.)

We turn to discuss the dimension of self-affine measures, which is
the main topic of this paper. First we provide the necessary background
regarding dimension of general measures. Given $\theta\in\mathcal{M}(\mathbb{R}^{d})$
we say that $\theta$ is exact dimensional if there exists a nonnegative number
$\dim\theta$, called the dimension of $\theta$, so that
\[
\underset{\epsilon\downarrow0}{\lim}\:\frac{\log\theta(B(x,\epsilon))}{\log\epsilon}=\dim\theta\:\text{ for }\theta\text{-a.e. }x,
\]
where $B(x,\epsilon)$ is the closed ball with centre $x$ and radius
$\epsilon$. The number $\dim\theta$ is equal to the value given
to $\theta$ by other commonly used notions of dimension (see \cite{FLR}).
In particular, $\dim\theta$ is equal to the Hausdorff and entropy
dimensions of $\theta$, which are defined as follows
\[
\dim_{H}\theta:=\inf\{\dim_{H}E\::\:E\subset\mathbb{R}^{d}\text{ is Borel with }\theta(E)>0\}
\]
and
\begin{equation}
\dim_{e}\theta:=\underset{n\rightarrow\infty}{\lim}\:\frac{1}{n}H(\theta,\mathcal{D}_{n}).\label{eq:ent dim}
\end{equation}
Here $H(\theta,\mathcal{D}_{n})$ denotes the entropy of $\theta$
with respect to the level-$n$ dyadic partition $\mathcal{D}_{n}$
of $\mathbb{R}^{d}$ (see Sections \ref{subsec:Dyadic-partitions}
and \ref{subsec:Entropy} below).

Self-affine measures are always exact dimensional. This important
fact has first been proven by Bárány and Käenmäki \cite{BK} under
additional assumptions. The general case was later obtained by Feng
\cite{Fe}.

Since $\mu$ is supported on $K_{\Phi}$, it follows that $\dim\mu\le\dim_{H}K_{\Phi}$.
For that reason, it is usually the case that in order to study the
dimension of a self-affine set, one needs to study the dimension of
the self-affine measures it supports.

There is a natural upper bound for $\dim\mu$ which is analogous to
the affinity dimension. In order to define it, let $H(p):=-\sum_{i\in\Lambda}p_{i}\log p_{i}$
be the entropy of $p$ and denote by $0>\chi_{1}\ge...\ge\chi_{d}$
the Lyapunov exponents corresponding to $\{A_{i}\}_{i\in\Lambda}$
and $p$ (see Section \ref{subsec:Lyapunov-exponents-and Fur meas}
below). Now set
\[
m=\max\{0\le j\le d\::\:0\le H(p)+\chi_{1}+...+\chi_{j}\},
\]
and define,
\[
\dim_{L}(\Phi,p):=\begin{cases}
m-\frac{H(p)+\chi_{1}+...+\chi_{m}}{\chi_{m+1}} & ,\text{ if }m<d\\
-d\frac{H(p)}{\chi_{1}+...+\chi_{d}} & ,\text{ if }m=d
\end{cases}.
\]
The number $\dim_{L}(\Phi,p)$ is called the Lyapunov dimension corresponding
to $\Phi$ and $p$.

It has been shown by Jordan, Pollicott and Simon \cite{JPS} that
$\dim_{L}(\Phi,p)$ is always an upper bound for $\dim\mu$. They
also showed that when $\Vert A_{i}\Vert_{op}<1/2$ for $i\in\Lambda$, and under a natural
randomization of the translations $\{a_{i}\}_{i\in\Lambda}$, the equality
\begin{equation}
\dim\mu=\min\{d,\dim_{L}(\Phi,p)\}\label{eq:dim mu =00003D dim_L}
\end{equation}
holds almost surely.

The aforementioned results show that (\ref{eq:dim K =00003D dim_A})
and (\ref{eq:dim mu =00003D dim_L}) hold typically, but do not provide
any explicit examples. It is of course desirable to find
explicit and verifiable conditions under which these equalities would
hold. We start with some separation conditions.

\begin{defn}
\label{def:OSC,  SOSC and SSC }It is said that $\Phi$ satisfies
the open set condition (OSC) if there exists a nonempty bounded open
set $U\subset\mathbb{R}^{d}$ so that $\cup_{i\in\Lambda}\varphi_{i}(U)\subset U$
with the union disjoint. If additionally $U\cap K_{\Phi}\ne\emptyset$,
we say that $\Phi$ satisfies the strong open set condition (SOSC).
When the sets $\{\varphi_{i}(K_{\Phi})\}_{i\in\Lambda}$ are disjoint,
we say that $\Phi$ satisfies the strong separation condition (SSC). It is clear that the SSC implies the SOSC.
\end{defn}

We say that $\Phi$ is $p$-conformal if $\chi_{1}=\chi_{d}$. It
is easy to show that $\dim\mu=\dim_{L}(\Phi,p)$
whenever $\Phi$ is $p$-conformal and the SSC holds.
When $\Phi$ consists of contracting similarities it is said to be a self-similar
IFS. In the self-similar
case $\Phi$ is always $p$-conformal, and we have $\dim\mu=\dim_{L}(\Phi,p)$ and $\dim_{H}K_{\Phi}=\dim_{A}\Phi$ under the OSC (see e.g. \cite{Edg-integral}).

When $d=1$, in which case $\Phi$ is always self-similar, Hochman
was able replace the OSC with a much milder Diophantine condition which
we now define. Denote by $\Lambda^{*}$ the set of finite words over
$\Lambda$. Given $i_{1}...i_{n}=u\in\Lambda^{*}$ we write $\varphi_{u}$
in place of $\varphi_{i_{1}}\circ...\circ\varphi_{i_{n}}$.
\begin{defn}
\label{def:Diophantine =000026 ESC}Let $\Vert\cdot\Vert$ be a norm
on the vector space of affine maps from $\mathbb{R}^{d}$ into itself.
We say that $\Phi$ is Diophantine if there exists $\epsilon>0$ so
that $\Vert\varphi_{u_{1}}-\varphi_{u_{2}}\Vert\ge\epsilon^{n}$ for
all $n\ge1$ and $u_{1},u_{2}\in\Lambda^{n}$ with $\varphi_{u_{1}}\ne\varphi_{u_{2}}$.
We say that $\Phi$ generates a free semigroup if $\varphi_{u_{1}}\ne\varphi_{u_{2}}$
for all distinct $u_{1},u_{2}\in\Lambda^{*}$. Finally, we say that
$\Phi$ is exponentially separated, or that it satisfies the exponential
separation condition (ESC), if it is Diophantine and generates a free
semigroup.
\end{defn}

It is not difficult to see that $\Phi$ is always Diophantine when
it is defined by algebraic parameters. That is, when the entries of
$\{A_{i}\}_{i\in\Lambda}$ and $\{a_{i}\}_{i\in\Lambda}$ are all
algebraic numbers.
\begin{thm}[Hochman, \cite{Ho1}]
\label{thm:Ho1}Suppose that $d=1$ and $\Phi$ satisfies the ESC.
Then (\ref{eq:dim K =00003D dim_A}) and (\ref{eq:dim mu =00003D dim_L})
are satisfied.
\end{thm}

We point out that this result is actually obtained in \cite{Ho1}
under a milder version of the ESC, in which the exponential separation
is only assumed to hold for infinitely many integers $n\ge1$. In
this paper we shall use the more restrictive definition given above.

In higher dimensions, exponentially separated self-similar sets and
measures were studied by Hochman in \cite{Ho}.

In the general self-affine setting, when $d\ge2$ it is possible
for (\ref{eq:dim K =00003D dim_A}) and (\ref{eq:dim mu =00003D dim_L})
to fail even when the SSC holds. It is easy to construct such examples
in which there exits a proper linear subspace of $\mathbb{R}^d$ which is invariant under
all of the linear parts $\{A_{i}\}_{i\in\Lambda}$. We would
like to avoid such situations, and the following definition is useful
in this context.
\begin{defn}
Let $V$ be a finite dimensional real vector space and let $\mathbf{S}$
be a subsemigroup of $\mathrm{GL}(V)$. We say that $\mathbf{S}$
is strongly irreducible if for every proper linear subspace $W$ of
$V$ the orbit $\{S(W)\::\:S\in\mathbf{S}\}$ is infinite.
\end{defn}

We would also like to avoid the simpler but quite different conformal
case. For this we make the following definition.
\begin{defn}
Let $V$ be a finite dimensional real vector space and let $\mathbf{S}$
be a subsemigroup of $\mathrm{GL}(V)$. We say that $\mathbf{S}$
is proximal if there exist $S_{1},S_{2},...\in\mathbf{S}$ and $\alpha_{1},\alpha_{2},...\in\mathbb{R}$
so that $\{\alpha_{n}S_{n}\}_{n\ge1}$ converges to an endomorphism
of $V$ of rank $1$.
\end{defn}

Denote by $\mathbf{S}_{\Phi}^{\mathrm{L}}\subset\mathrm{GL}(d,\mathbb{R})$
the semigroup generated by the linear parts $\{A_{i}\}_{i\in\Lambda}$.
In what follows we restrict to the case in which $\mathbf{S}_{\Phi}^{\mathrm{L}}$
is strongly irreducible and proximal. Under these assumptions it is
always the case that $\chi_{1}>\chi_{2}$ and there exists a
unique $\nu_{1}^{*}\in\mathcal{M}(\mathrm{P}(\mathbb{R}^{d}))$ which
is stationary with respect to $\theta^{*}:=\sum_{i\in\Lambda}p_{i}\delta_{A_{i}^{*}}$. We refer to \cite{BQ} and \cite{BL} for the proofs of these facts.
Here $\mathrm{P}(\mathbb{R}^{d})$ is the projective space of $\mathbb{R}^{d}$,
$\delta_{A_{i}^{*}}$ is the Dirac mass at $A_{i}^{*}$, and by saying
that $\nu_{1}^{*}$ is $\theta^{*}$-stationary we mean that $\nu_{1}^{*}=\sum_{i\in\Lambda}p_{i}\cdot A_{i}^{*}\nu_{1}^{*}$, where $A_{i}^{*}\nu_{1}^{*}$ is the pushforward of $\nu_{1}^{*}$ via the map $\overline{x}\rightarrow\overline{A_i^*x}$.
The measure $\nu_{1}^{*}$ is called the Furstenberg measure corresponding
to $\theta^{*}$.

We turn to discuss the planar case. First we have the following
result.
\begin{thm}[Bárány--Hochman--Rapaport, \cite{BHR}]
\label{thm:BHR}Suppose that $d=2$, $\mathbf{S}_{\Phi}^{\mathrm{L}}$
is strongly irreducible and proximal, and $\Phi$ satisfies the SOSC.
Then $\dim_{H}K_{\Phi}=\dim_{A}\Phi$ and $\dim\mu=\dim_{L}(\Phi,p)$.
\end{thm}

For a linear subspace $V\subset\mathbb{R}^{d}$ write $P_{V}$ for
its orthogonal projection. Theorem \ref{thm:BHR} is obtained in \cite{BHR}
by showing that
\begin{equation}
\dim P_{V}\mu=\min\{1,\dim\mu\}\text{ for }\nu_{1}^{*}\text{-a.e. }V\in\mathrm{P}(\mathbb{R}^{2}).\label{eq:BHR a.e.  V}
\end{equation}
This together with the Ledrappier-Young formula for self-affine measures
(see Section \ref{subsec:Ledrappier-Young-formula} below) and the
SOSC assumption yields the theorem.
\begin{rem}
\label{rem:all projs to lines}Understanding the projections of concrete
fractal sets and measures is an important strand of research in fractal
geometry (see e.g. the surveys \cite{FFJ_proj_survey,Sh_proj_survey}).
In particular, one would like to go beyond general results such as
Marstrand\textquoteright s Theorem, and determine the dimension of
every projection (or at least every projection outside some very small
set).

In this connection, it is actually shown in \cite{BHR} that $P_{V}\mu$
is exact dimensional with dimension $\min\{1,\allowbreak\dim\mu\}$
for all $V\in\mathrm{P}(\mathbb{R}^{d})$. This is achieved by combining
(\ref{eq:BHR a.e.  V}) with an argument of Falconer and Kempton \cite{FK},
which relies on the machinery developed by Hochman and Shmerkin \cite{HoSh}.
\end{rem}

In the next theorem we were able to replace the SOSC assumption with the much
milder ESC. For this studying the dimension of projections is not
enough, and the non-conformality plays a much more significant role.

We point out that when $d\ge2$ it is possible for the ESC to be
satisfied while the maps in $\Phi$ all have a common fixed point.
In this situation $K_{\Phi}$ is a singleton, and (\ref{eq:dim K =00003D dim_A})
and (\ref{eq:dim mu =00003D dim_L}) do not hold unless $\Phi$ consists
of a single map. Thus, when only the ESC is assumed we need to explicitly
rule out this possibility.
\begin{thm}[Hochman--Rapaport, \cite{HR}]
\label{thm:HR}Suppose that $d=2$, the maps in $\Phi$ do not have
a common fixed point, $\mathbf{S}_{\Phi}^{\mathrm{L}}$ is strongly
irreducible and proximal, and $\Phi$ satisfies the ESC. Then (\ref{eq:dim K =00003D dim_A})
and (\ref{eq:dim mu =00003D dim_L}) are satisfied.
\end{thm}

In \cite{BHR} and \cite{HR} we actually only proved (\ref{eq:dim mu =00003D dim_L}).
The fact that (\ref{eq:dim mu =00003D dim_L}) implies (\ref{eq:dim K =00003D dim_A})
in the strongly irreducible and proximal case in the plane follows
from the earlier work of Morris and Shmerkin \cite{MoSh}. They showed
that under the hypotheses of Theorems \ref{thm:BHR} and \ref{thm:HR},
one can construct a self-affine IFS satisfying the same hypotheses,
so that its attractor is contained in $K_{\Phi}$ and it generates
some self-affine measure whose Lyapunov dimension is arbitrarily close
to the affinity dimension of the original system.

When $d\ge3$ conditions for the equality $\dim\mu=\dim_{L}(\Phi,p)$
have been obtained in \cite{BK,Fe,Ra}. In all of these works it is
assumed that the SSC or SOSC are satisfied, and that various Furstenberg
measures have large enough dimension. The last assumption is particularly
restrictive, as it is usually very hard to verify. Up to this point
and to the best of the author's knowledge, when $d\ge3$ these are
the only explicit examples in the non-conformal strongly irreducible case for
which (\ref{eq:dim mu =00003D dim_L}) has been verified.

The main objective of the present work is to make progress with the
verification of (\ref{eq:dim mu =00003D dim_L}) in higher dimensions.
New results regarding the dimension of projections of $\mu$ are also
obtained. We combine ideas from \cite{BHR}, used there for studying
$1$-dimensional projections, with the non-conformal methods
developed in \cite{HR}. In order to deal with new difficulties arising
in higher dimensions, our arguments involve significant new ideas.
We state our results in Section \ref{subsec:Statement-of-results},
and provide an overview of the proofs in Section \ref{subsec:About-the-proof}.

\subsection{\label{subsec:Statement-of-results}Statement of results}

In Sections \ref{subsec:SOSC in R^3} and \ref{subsec:Full-dimensionality}
we present our two main theorems, which verify (\ref{eq:dim mu =00003D dim_L})
in new concrete situations. In Section \ref{subsec:A-general-criteria}
we present a general criteria ensuring (\ref{eq:dim mu =00003D dim_L}).

We continue to use the notations from the previous subsection. In
particular, $\Phi=\{\varphi_{i}(x)=A_{i}x+a_{i}\}_{i\in\Lambda}$ is an affine IFS on $\mathbb{R}^{d}$, $K_{\Phi}$ is the attractor of $\Phi$, $\mathbf{S}_{\Phi}^{\mathrm{L}}$ is the semigroup generated by the linear parts $\{A_{i}\}_{i\in\Lambda}$, $p=(p_{i})_{i\in\Lambda}$ is a positive
probability vector, and $\mu$ is the self-affine measure corresponding
to $\Phi$ and $p$.

\subsubsection{\label{subsec:SOSC in R^3}Systems with the SOSC in $\mathbb{R}^{3}$
and low dimensional projections}

Our first main theorem is the following.
\begin{thm}
\label{thm:main d=00003D3}Suppose that $d=3$, $\mathbf{S}_{\Phi}^{\mathrm{L}}$
is strongly irreducible and proximal, and $\Phi$ satisfies the SOSC.
Then $\dim\mu=\dim_{L}(\Phi,p)$.
\end{thm}

\begin{rem*}
In a forthcoming paper, Morris and Sert \cite{MS} plan to extend
the work of Morris and Shmerkin \cite{MoSh} to higher dimensions.
This together with Theorem \ref{thm:main d=00003D3} will imply the
equality $\dim_{H}K_{\Phi}=\dim_{A}\Phi$ whenever $d=3$, $\mathbf{S}_{\Phi}^{\mathrm{L}}$
is strongly irreducible and proximal, and $\Phi$ satisfies the SOSC.
\end{rem*}

Theorem \ref{thm:main d=00003D3} will follow from a different result
we now state. This result is valid for all $d\ge1$, and deals with projections
of $\mu$ onto subspaces of dimension at most $2$. This is in analogy
with the fact that Theorem \ref{thm:BHR} follows from (\ref{eq:BHR a.e.  V}),
which is a statement about projections onto lines.

In higher dimensions we shall need additional irreducibility and proximality
assumptions. For $1\le m\le d$ denote by $\wedge^{m}\mathbb{R}^{d}$
the vector space of alternating $m$-forms on the dual of $\mathbb{R}^{d}$
(see Section \ref{subsec:spaces of alt forms}). Given $A\in\mathrm{GL}(d,\mathbb{R})$
let $\wedge^{m}A\in\mathrm{GL}(\wedge^{m}\mathbb{R}^{d})$ be with
\[
\wedge^{m}A(x_{1}\wedge...\wedge x_{m})=(Ax_{1})\wedge...\wedge(Ax_{m})\text{ for }x_{1},...,x_{m}\in\mathbb{R}^{d},
\]
and write $\wedge^{m}\mathbf{S}_{\Phi}^{\mathrm{L}}:=\{\wedge^{m}S\::\:S\in\mathbf{S}_{\Phi}^{\mathrm{L}}\}$.
\begin{defn}
Given $1\le m<d$ we say that $\mathbf{S}_{\Phi}^{\mathrm{L}}$ is
$m$-strongly irreducible if $\wedge^{m}\mathbf{S}_{\Phi}^{\mathrm{L}}$
is strongly-irreducible as a subsemigroup of $\mathrm{GL}(\wedge^{m}\mathbb{R}^{d})$.
Similarly, we say that $\mathbf{S}_{\Phi}^{\mathrm{L}}$ is $m$-proximal
if $\wedge^{m}\mathbf{S}_{\Phi}^{\mathrm{L}}$ is proximal.
\end{defn}

In what follows we shall always make the following assumptions:
\begin{equation}
\text{The maps in \ensuremath{\Phi\ }do not have a common fixed point;}\label{eq:no com fix asump}
\end{equation}
\begin{equation}
\mathbf{S}_{\Phi}^{\mathrm{L}}\text{ is }m\text{-strongly irreducible and }m\text{-proximal for each }1\le m<d.\label{eq:m-ired and m-prox assump}
\end{equation}

Assumption (\ref{eq:no com fix asump}) is equivalent to $K_{\Phi}$
not being a singleton. From (\ref{eq:no com fix asump}) together
with the irreducibility of $\mathbf{S}_{\Phi}^{\mathrm{L}}$, it actually
follows that there does not exist a proper affine subspace $V$ of
$\mathbb{R}^{d}$ so that $K_{\Phi}\subset V$ (see Lemma \ref{lem:K not cont in V}).

Assumption (\ref{eq:m-ired and m-prox assump}) is standard in the
theory of products of random matrices. It is satisfied for instance
when $\mathbf{S}_{\Phi}^{\mathrm{L}}$ is Zariski dense in $\mathrm{GL}(d,\mathbb{R})$,
and implies many desirable properties. In particular, it implies that
$\chi_{1}>...>\chi_{d}$ (see \cite[Theorem IV.1.2]{BL}), where recall
that $\chi_{1},...,\chi_{d}$ are the Lyapunov exponents corresponding
to $\{A_{i}\}_{i\in\Lambda}$ and $p$.
\begin{rem}
\label{rem:m-con equiv to (d-m)-cond}Clearly $\mathbf{S}_{\Phi}^{\mathrm{L}}$
is $1$-strongly irreducible and $1$-proximal if and only if it is
strongly irreducible and proximal. Moreover, from \cite[page 87]{BL}
and \cite[Lemmata 6.15 and 6.23]{BQ}, it follows that $\mathbf{S}_{\Phi}^{\mathrm{L}}$
is $m$-strongly irreducible and $m$-proximal if and only if it is
$(d-m)$-strongly irreducible and $(d-m)$-proximal. This is why in
Theorem \ref{thm:main d=00003D3} we only assume strong irreducibility
and proximality of $\mathbf{S}_{\Phi}^{\mathrm{L}}$ instead of assuming
(\ref{eq:m-ired and m-prox assump}).
\end{rem}

Recall from (\ref{eq:ent dim}) that we denote the entropy dimension
of $\theta\in\mathcal{M}(\mathbb{R}^{d})$ by $\dim_{e}\theta$ (assuming
it exists). Also recall that for a linear subspace $V\subset\mathbb{R}^{d}$
we write $P_{V}$ for its orthogonal projection. We can now state
the aforementioned result about low dimensional projections of $\mu$.
\begin{thm}
\label{thm:proj onto 1 =000026 2 dim subspaces}Suppose that (\ref{eq:no com fix asump})
and (\ref{eq:m-ired and m-prox assump}) hold and that $\Phi$ satisfies
the ESC. Then for every linear subspace $V\subset\mathbb{R}^{d}$
with $\dim V\le2$,
\[
\dim_{e}P_{V}\mu=\min\left\{ \dim V,\dim_{L}(\Phi,p)\right\} .
\]
\end{thm}

\begin{rem*}
Given a bounded nonempty subset $E$ of $\mathbb{R}^{d}$, denote
its box counting dimension by $\dim_{B}E$ (assuming it exists). By
combining Theorem \ref{thm:proj onto 1 =000026 2 dim subspaces} with
the forthcoming work of Morris and Sert \cite{MS}, it should be possible
to show that for every linear subspace $V\subset\mathbb{R}^{d}$ with
$\dim V\le2$
\[
\dim_{B}P_{V}K_{\Phi}=\min\left\{ \dim V,\dim_{A}\Phi\right\} ,
\]
whenever (\ref{eq:no com fix asump}) and (\ref{eq:m-ired and m-prox assump})
hold and $\Phi$ satisfies the ESC.
\end{rem*}

Since $\mu$ is exact dimensional we have $\dim\mu=\dim_{e}\mu$.
Thus, the statements made in Theorems \ref{thm:Ho1}, \ref{thm:BHR}
and \ref{thm:HR} regarding $\dim\mu$ all follow immediately from
Theorem \ref{thm:proj onto 1 =000026 2 dim subspaces}. Theorem \ref{thm:main d=00003D3} will follow almost directly from
Theorem \ref{thm:proj onto 1 =000026 2 dim subspaces} and the Ledrappier-Young
formula. Theorem \ref{thm:proj onto 1 =000026 2 dim subspaces} in
turn will be deduced from a general criteria ensuring (\ref{eq:dim mu =00003D dim_L}),
which we present below.

Recall that $\dim_{L}(\Phi,p)$ is always an upper bound for $\dim\mu$,
and note that $\dim_{L}(\Phi,p)\le2$ whenever $H(p)\le-\chi_{1}-\chi_{2}$.
Moreover, $\dim\mu\ge\dim_{e}P_{V}\mu$ for every linear subspace
$V\subset\mathbb{R}^{d}$ for which $\dim_{e}P_{V}\mu$ exists. The
following corollary follows directly from Theorem \ref{thm:proj onto 1 =000026 2 dim subspaces}
and these observations.
\begin{cor}
\label{cor:dim mu =00003D dim_L when H(p)<=00003D2 exp}Suppose that
(\ref{eq:no com fix asump}) and (\ref{eq:m-ired and m-prox assump})
hold, $\Phi$ satisfies the ESC, $d\ge2$ and $H(p)\le-\chi_{1}-\chi_{2}$.
Then $\dim\mu=\dim_{L}(\Phi,p)$.
\end{cor}

\subsubsection{\label{subsec:Full-dimensionality}Full dimensionality of $\mu$
and its projections}

Our second main theorem provides conditions for $\mu$ and its projections
to be of full dimension. In order to state it we need to define the
entropy of the random walk generated by $p$ and $\Phi$.

Denote by $\mathrm{A}_{d,d}$ the affine group of $\mathbb{R}^{d}$,
and write $p_{\Phi}$ for the probability measure on $\mathrm{A}_{d,d}$
determined by $p$ and $\Phi$. That is, $p_{\Phi}:=\sum_{i\in\Lambda}p_{i}\delta_{\varphi_{i}}$.
For $n\ge1$ denote by $p_{\Phi}^{*n}$ the convolution of $p_{\Phi}$
with itself $n$ times.
\begin{defn}
\label{def:of rw ent} The random walk entropy of $p_{\Phi}$ is denoted by $h(p_{\Phi})$
and defined as follows
\begin{equation}
h(p_{\Phi}):=\underset{n\rightarrow\infty}{\lim}\:\frac{1}{n}H(p_{\Phi}^{*n}),\label{eq:def of rw_ent}
\end{equation}
where $H(p_{\Phi}^{*n})$ is the Shannon entropy of the discrete probability
measure $p_{\Phi}^{*n}$.
\end{defn}

The limit in (\ref{eq:def of rw_ent}) exists by sub-additivity. Note
that $h(p_{\Phi})=H(p)$ whenever $\Phi$ generates a free semigroup.

We can now state our second main result. For $1\le m\le d$ set
\[
\varrho_{m}:=(\chi_{1}-\chi_{m})\frac{(m-1)(m-2)}{2}-\sum_{k=1}^{m}\chi_{k},
\]
and denote by $\mathrm{Gr}_{m}(d)$ the Grassmannian manifold of $m$-dimensional
linear subspaces of $\mathbb{R}^{d}$. Recall from Definition \ref{def:Diophantine =000026 ESC}
the definition of a Diophantine system.
\begin{thm}
\label{thm:main all d}Suppose that (\ref{eq:no com fix asump}) and
(\ref{eq:m-ired and m-prox assump}) are satisfied and that $\Phi$
is Diophantine. Then for all $1\le m\le d$ with $h(p_{\Phi})\ge\varrho_{m}$,
\[
\dim_{e}P_{V}\mu=m\text{ for each }V\in\mathrm{Gr}_{m}(d).
\]
In particular, if $h(p_{\Phi})\ge\varrho_{d}$ then $\dim\mu=d$.
\end{thm}

Let us demonstrate how this theorem can be used in order to construct
self-affine measures of full dimension. As pointed out above, $\Phi$
is always Diophantine whenever it is defined by algebraic parameters.
Additionally, if $\zeta_{1},\zeta_{2}\in(0,1)$ are algebraic conjugates
and for $j=1,2$ we have an IFS $\Psi_{j}=\{x\rightarrow\zeta_{j}B_{i}x+b_{i}\}_{i\in\Lambda}$,
where $B_{i}\in\mathrm{GL}(d,\mathbb{Q})$ and $b_{i}\in\mathbb{Q}^{d}$
for each $i\in\Lambda$, then $\Psi_{1}$ generates a free semigroup
if and only if $\Psi_{2}$ generates a free semigroup. In particular,
if $\Psi_{1}$ satisfies the SSC then $\Psi_{2}$ generates a free
semigroup.

By using these observations it is easy to construct explicit examples
in which (\ref{eq:no com fix asump}) and (\ref{eq:m-ired and m-prox assump})
hold, $\Phi$ satisfies the ESC, and $H(p)\ge\varrho_{d}$. In this
situation Theorem \ref{thm:main all d} implies that $\dim\mu=d$.
When $d\ge3$, these seem to be the first explicit examples in the
strongly irreducible non-conformal case for which the equality $\dim\mu=d$
has been verified.

It is desirable to establish (\ref{eq:dim mu =00003D dim_L}) whenever
(\ref{eq:no com fix asump}) and (\ref{eq:m-ired and m-prox assump})
hold and $\Phi$ satisfies the ESC. In particular, under these conditions
one would like to show that $\dim\mu=d$ whenever $H(p)\ge-\sum_{k=1}^{d}\chi_{k}$.
Later on, after we state our general criteria, it will become clear why at present we are unable to obtain these stronger statements.
\begin{rem*}
Suppose that $d=3$, (\ref{eq:no com fix asump}) and (\ref{eq:m-ired and m-prox assump})
hold, and $\Phi$ satisfies the ESC. From Theorem \ref{thm:main all d}
it follows that $\dim\mu=3$ if $H(p)\ge-\chi_{2}-2\chi_{3}$. Moreover,
from Corollary \ref{cor:dim mu =00003D dim_L when H(p)<=00003D2 exp}
we obtain that $\dim\mu=\dim_{L}(\Phi,p)$ if $H(p)\le-\chi_{1}-\chi_{2}$.
Thus, in the present situation in $\mathbb{R}^{3}$ it remains to
verify (\ref{eq:dim mu =00003D dim_L}) when $-\chi_{1}-\chi_{2}<H(p)<-\chi_{2}-2\chi_{3}$.
\end{rem*}

Given Remark \ref{rem:all projs to lines}, it is natural to try to
use the machinery developed in \cite{HoSh} in order to replace $\dim_{e}P_{V}\mu$
with $\dim_{H}P_{V}\mu$ in Theorems \ref{thm:proj onto 1 =000026 2 dim subspaces}
and \ref{thm:main all d}. This should be relatively straightforward
when the subspace $V$ is of dimension $1$. On the other hand, due
to the non-conformality of $\mu$, the situation becomes much more
complicated when $\dim V>1$. In any case, we do not pursue this improvement
in the present work.

The following corollary for self-affine sets and their projections
follows directly from Theorem \ref{thm:main all d}, together with
standard relations between the dimensions of sets and measures.
\begin{cor}
\label{cor:dim of set and proj of set}Suppose that (\ref{eq:no com fix asump})
and (\ref{eq:m-ired and m-prox assump}) are satisfied and that $\Phi$
is Diophantine. Then for all $1\le m\le d$ with $h(p_{\Phi})\ge\varrho_{m}$,
\[
\dim_{B}P_{V}(K_{\Phi})=m\text{ for each }V\in\mathrm{Gr}_{m}(d).
\]
Moreover, if $h(p_{\Phi})\ge\varrho_{d}$ then $\dim_{H}K_{\Phi}=d$.
\end{cor}

Note that the assumptions of Corollary \ref{cor:dim of set and proj of set}
are independent of $p$. Thus, its conclusion holds whenever $h(p_{\Phi})\ge\varrho_{m}$
for some choice of strictly positive probability vector $p$.

\subsubsection{\label{subsec:A-general-criteria}A general criteria}

Our next result is a general criteria ensuring (\ref{eq:dim mu =00003D dim_L}).
In order to state it we need some preparations.

Write $\beta$ for the Bernoulli measure on $\Lambda^{\mathbb{N}}$
corresponding to $p$, and let $\Pi:\Lambda^{\mathbb{N}}\rightarrow K_{\Phi}$
be the coding map corresponding to $\Phi$. That is $\beta:=p^{\mathbb{N}}$,
and
\[
\Pi\omega:=\underset{n\rightarrow\infty}{\lim}\:\varphi_{\omega_{0}}\circ...\circ\varphi_{\omega_{n}}(0)\text{ for }\omega\in\Lambda^{\mathbb{N}}.
\]
Note that $\mu=\Pi\beta$.

For $0\le m\le d$ recall that $\mathrm{Gr}_{m}(d)$ denotes the Grassmannian
of $m$-dimensional linear subspaces of $\mathbb{R}^{d}$. As we note
in Section \ref{subsec:Lyapunov-exponents-and Fur meas} below, from
(\ref{eq:m-ired and m-prox assump}) it follows that there exists
a unique $\nu_{m}\in\mathcal{M}(\mathrm{Gr}_{m}(d))$ which is stationary
with respect to $\theta:=\sum_{i\in\Lambda}p_{i}\delta_{A_{i}}$.
That is, $\nu_{m}$ is the unique element in $\mathcal{M}(\mathrm{Gr}_{m}(d))$
which satisfies $\nu_{m}=\sum_{i\in\Lambda}p_{i}\cdot A_{i}\nu_{m}$,
where $A_{i}\nu_{m}$ is the pushforward of $\nu_{m}$ via the map
taking $V\in\mathrm{Gr}_{m}(d)$ to $A_{i}(V)$. Similarly, there
exists a unique $\nu_{m}^{*}\in\mathcal{M}(\mathrm{Gr}_{m}(d))$ which
is stationary with respect to $\theta^{*}:=\sum_{i\in\Lambda}p_{i}\delta_{A_{i}^{*}}$.
The measures $\nu_{m}$ and $\nu_{m}^{*}$ are called the Furstenberg
measures on $\mathrm{Gr}_{m}(d)$ corresponding to $\theta$ and $\theta^{*}$
respectively.

Additionally, as we point out in Section \ref{subsec:Coding and Furstenberg maps},
from (\ref{eq:m-ired and m-prox assump}) it also follows that for
each $0\le m\le d$ there exists a Borel map $L_{m}:\Lambda^{\mathbb{N}}\rightarrow\mathrm{Gr}_{m}(d)$
so that $L_{m}\beta=\nu_{m}$ and for $\beta$-a.e. $\omega\in\Lambda^{\mathbb{N}}$
the sequence
\[
\{A_{\omega_{0}}...A_{\omega_{n}}\nu_{m}\}_{n\ge1}\subset\mathcal{M}(\mathrm{Gr}_{m}(d))
\]
converges weakly to $\delta_{L_{m}(\omega)}$. The maps $L_{0},...,L_{d}$
are called the Furstenberg boundary maps.

Write $\mathcal{B}_{\mathbb{R}^{d}}$ for the Borel $\sigma$-algebra
of $\mathbb{R}^{d}$. Given a linear subspace $V$ of $\mathbb{R}^{d}$,
denote by $\{\beta_{\omega}^{V}\}_{\omega\in\Lambda^{\mathbb{N}}}$
the disintegration of $\beta$ with respect to $\Pi^{-1}(P_{V}^{-1}\mathcal{B}_{\mathbb{R}^{d}})$
(see Section \ref{subsec:Disintegrations}). From Lemma \ref{lem:nu=00007Bkappa<del=00007D<eps}
below we get that for each $1\le k\le m\le d$ and $V\in\mathrm{Gr}_{m}(d)$
\[
P_{V}L_{k}(\omega)\in\mathrm{Gr}_{k}(V)\text{ for }\beta\text{-a.e. }\omega,
\]
where $\mathrm{Gr}_{k}(V)$ denotes the Grassmannian of $k$-dimensional
linear subspaces of $V$. From this and since $\beta=\int\beta_{\omega}^{V}\:d\beta(\omega)$,
we obtain that for each $1\le k\le m\le d$ and $V\in\mathrm{Gr}_{m}(d)$
\[
P_{V}L_{k}\beta_{\omega}^{V}\in\mathcal{M}(\mathrm{Gr}_{k}(V))\text{ for }\beta\text{-a.e. }\omega.
\]

We shall also need the following consequence of the Ledrappier-Young
formula (see Section \ref{subsec:Ledrappier-Young-formula}). Under
our assumptions, it was first proven by Bárány and Käenmäki \cite[Proposition 5.9]{BK}.
The general case was later obtained by Feng \cite[Theorem 1.4]{Fe}.
\begin{thm}
\label{thm:cons of led-you intro}Suppose that (\ref{eq:m-ired and m-prox assump})
holds. Then there exist numbers $\Delta_{1},...,\Delta_{d}\in[0,1]$
so that for each $1\le m\le d$ and $\nu_{m}^{*}$-a.e. $V\in\mathrm{Gr}_{m}(d)$
the measure $P_{V}\mu$ is exact dimensional with $\dim P_{V}\mu=\Sigma_{1}^{m}$,
where $\Sigma_{1}^{m}:=\sum_{k=1}^{m}\Delta_{k}$.
\end{thm}

For convenience we also set $\Sigma_{1}^{0}:=0$.
\begin{rem}
\label{rem:lb all proj}By Lemma \ref{lem:lb on ent of psi mu} below,
it in fact holds that $\underline{\dim}_{e}P_{V}\mu\ge\Sigma_{1}^{m}$
for all $1\le m\le d$ and all $V\in\mathrm{Gr}_{m}(d)$. Here $\underline{\dim}_{e}$
denotes the lower entropy dimension, which is defined by replacing
$\lim$ with $\liminf$ in (\ref{eq:ent dim}).
\end{rem}

\begin{rem}
\label{rem:exists unique m}Note that $\dim\mu=\Sigma_{1}^{d}$, and
that when $\dim\mu<d$ there exists a unique $1\le m\le d$ so that
$\Sigma_{1}^{m-1}=m-1$ and $\Delta_{m}<1$.
\end{rem}

We can now state our general criteria.
\begin{thm}
\label{thm:gen criteria}Suppose that (\ref{eq:no com fix asump})
and (\ref{eq:m-ired and m-prox assump}) hold, $\Phi$ satisfies the
ESC, and $\dim\mu<d$. Let $1\le m\le d$ be such that $\Sigma_{1}^{m-1}=m-1$
and $\Delta_{m}<1$, and assume that
\begin{equation}
P_{V}L_{k}\beta_{\omega}^{V}=\delta_{P_{V}L_{k}(\omega)}\text{ for each }1\le k<m,\:\nu_{m}^{*}\text{-a.e. }V\text{ and }\beta\text{-a.e. }\omega.\label{eq:factor cond in gen criteria thm}
\end{equation}
Then $\dim\mu=\Sigma_{1}^{m}=\dim_{L}(\Phi,p)$.
\end{thm}

\begin{rem}
\label{rem:equiv form of cond in gen criteria}Condition (\ref{eq:factor cond in gen criteria thm})
is equivalent to the factorization of $P_{V}L_{k}$ via $P_{V}\Pi$
for each $1\le k<m$ and $\nu_{m}^{*}$-a.e. $V$. Indeed, given $1\le k<m\le d$
and $V\in\mathrm{Gr}_{m}(d)$, the assumption
\[
P_{V}L_{k}\beta_{\omega}^{V}=\delta_{P_{V}L_{k}(\omega)}\text{ for }\beta\text{-a.e. }\omega
\]
is easily seen to be equivalent to the existence of a Borel map $L_{k,V}:V\rightarrow\mathrm{Gr}_{k}(V)$
so that $P_{V}L_{k}(\omega)=L_{k,V}(P_{V}\Pi\omega)$ for $\beta$-a.e.
$\omega$.
\end{rem}

As noted after the statement of Theorem \ref{thm:main all d}, ideally
we would have liked to establish Theorem \ref{thm:gen criteria} without
assuming (\ref{eq:factor cond in gen criteria thm}). Unfortunately,
at this point we are unable to deduce this condition from the other
assumptions made in the theorem. On the other hand, we can prove the
following theorem which verifies (\ref{eq:factor cond in gen criteria thm})
for $k=m-1$. When $d=m=2$ this statement was obtained in \cite[Theorem 1.5]{HR}.
\begin{thm}
\label{thm:factor of proj of L_m-1}Suppose that $d\ge2$ and that
(\ref{eq:no com fix asump}) and (\ref{eq:m-ired and m-prox assump})
are satisfied. Let $2\le m\le d$ and $V\in\mathrm{Gr}_{m}(d)$ be
such that $\Sigma_{1}^{m-1}=m-1$ and $P_{V}\mu$ is exact dimensional
with $\dim P_{V}\mu<m$. Then $P_{V}L_{m-1}\beta_{\omega}^{V}=\delta_{P_{V}L_{m-1}(\omega)}$
for $\beta$-a.e. $\omega$.
\end{thm}

For $m=1$ condition (\ref{eq:factor cond in gen criteria thm}) holds trivially. When $m=2$ in Theorem \ref{thm:gen criteria}, the condition follows from the last statement and Theorem \ref{thm:cons of led-you intro}. These facts will enable us to deduce Theorem \ref{thm:proj onto 1 =000026 2 dim subspaces} from Theorem \ref{thm:gen criteria}.

\begin{rem*}
It seems that Theorem \ref{thm:gen criteria} can be used in order to make further progress with the verification of (\ref{eq:dim mu =00003D dim_L}). Suppose that $d=3$, (\ref{eq:m-ired and m-prox assump}) holds, and $\nu_2$ is not of full dimension (i.e. $\dim\nu_2<\dim\mathrm{Gr}_{2}(3)=2$). The author believes that under these assumptions, it should be possible to use the arguments in the proofs of Theorem \ref{thm:factor of proj of L_m-1} and \cite[Theorem 1.5]{HR} in order to establish the factorization of $L_1$ via $L_2$. That is, it should be possible to show that there exists a Borel map $\widetilde{L}_1:\mathrm{Gr}_{2}(3)\rightarrow\mathrm{Gr}_{1}(3)$ so that $L_1(\omega)=\widetilde{L}_1(L_2(\omega))$ for $\beta$-a.e. $\omega$. It is not difficult to see that this, together with Theorems \ref{thm:gen criteria} and \ref{thm:factor of proj of L_m-1}, will give (\ref{eq:dim mu =00003D dim_L}) whenever $d=3$, (\ref{eq:no com fix asump}) and (\ref{eq:m-ired and m-prox assump}) hold, $\Phi$ satisfies the
ESC, and $\dim\nu_2<2$.
\end{rem*}

\subsection{\label{subsec:About-the-proof}An outline of the proofs}

We next provide an outline of the proofs of the results stated above.
We do not give full details here, and some of our discussions are
imprecise. Everything will be repeated in a rigorous manner in later
parts of the paper. Our aim here is to guide the reader in understanding
the roles played by the various parts of the proofs. We also state
a number of intermediate results which we find to be interesting in
their own right. We continue to use the notations from the previous
subsections, and to assume that (\ref{eq:no com fix asump}) and (\ref{eq:m-ired and m-prox assump})
are satisfied.

\subsubsection{\label{subsec:Entropy-estimates}Entropy estimates}

Given a probability space $(X,\mathcal{F},\theta)$ and countable
partitions $\mathcal{D},\mathcal{E}\subset\mathcal{F}$, write $H(\theta,\mathcal{D})$
for the entropy of $\theta$ with respect to $\mathcal{D}$ and $H(\theta,\mathcal{D}\mid\mathcal{E})$
for the conditional entropy given $\mathcal{E}$ (see Section \ref{subsec:Entropy}).

For $1\le m\le d$ and $n\ge0$ denote by $\mathcal{D}_{n}^{m}$ the level-$n$
dyadic partition of $\mathbb{R}^{m}$, where we usually omit the superscript
$m$ when it is clear from the context. Given $\theta\in\mathcal{M}(\mathbb{R}^{m})$ and $D\in\mathcal{D}_{n}$ with $\theta(D)>0$, we call $\theta_{D}:=\frac{1}{\theta(D)}\theta|_{D}$
an $n$th level component of $\theta$. Given a word $u\in\Lambda^{*}$
we refer to $\varphi_{u}\mu$ as a cylinder measure of $\mu$.

Our first task is to bound from below the entropy of projections of
components of cylinders of $\mu$. In order to formulate this more
precisely, for $1\le m\le d$ and $n\ge1$ denote by $\Psi_{m}(n)$
the set of words $i_{0}...i_{k}\in\Lambda^{*}$ so that
\[
\alpha_{m}(A_{i_{0}}...A_{i_{k}})\le2^{-n}<\alpha_{m}(A_{i_{0}}...A_{i_{k-1}}),
\]
where $\alpha_{1}(A)\ge...\ge\alpha_{d}(A)>0$ are the singular values
of $A\in\mathrm{GL}(d,\mathbb{R})$. Given $1\le m\le d$, $\epsilon>0$,
sufficiently large $k\ge1$, and $V\in\mathrm{Gr}_{m}(d)$, in Proposition
\ref{prop:inductive main proj prop} we show that for most scales
$n\ge1$, most words $u\in\Psi_{m}(n)$, and most $D\in\mathcal{D}_{n}^{d}$
with respect to $\varphi_{u}\mu$, we have
\[
\frac{1}{k}H\left(P_{V}(\varphi_{u}\mu)_{D},\mathcal{D}_{n+k}\right)>\Sigma_{1}^{m}-\epsilon,
\]
where $\Sigma_{1}^{m}$ is defined in Theorem \ref{thm:cons of led-you intro}.

The proof of this proposition is carried out by induction on $m$,
and relies on the statement made in the Ledrappier-Young formula (Theorem
\ref{thm:LY formula for SA} below) regarding the dimension of projections
of slices of $\mu$. From the proposition we also derive lower bounds
on the entropy of projections of $\mu$ and its components.
\begin{rem}
\label{rem:stronger version}When $m=1$ it is possible to obtain
a stronger version of Proposition \ref{prop:inductive main proj prop},
in which no randomization of the scales or components is involved.
In the case $d=2$ and $m=1$, such a statement was obtained in \cite[Lemma 3.4]{HR}
(see also \cite[Proposition 3.16]{BHR} for an earlier version).
\end{rem}

\subsubsection{\label{subsec:Factorization-of}Factorization of $L_{m-1}$}

Equipped with the aforementioned entropy estimates, we turn to prove
Theorem \ref{thm:factor of proj of L_m-1}. The core idea of the proof
is taken from the proof of \cite[Theorem 1.5]{HR}, which deals with
the case $d=m=2$. In higher dimensions however we are faced with
new technical difficulties. Loosely speaking, these difficulties
come up mainly since our entropy lower bounds for projections involve more randomization compared to the bounds from \cite{BHR,HR} (see Remark \ref{rem:stronger version}). Nevertheless, we are able to extend the argument to the present situation.

For the rest of this outline, suppose that $\dim\mu<d$ and let $1\le m\le d$
be such that $\Sigma_{1}^{m-1}=m-1$ and $\Delta_{m}<1$ (see Remark
\ref{rem:exists unique m}). Otherwise, if $\dim\mu=d$ then all of
our results follow easily from the Ledrappier-Young formula and the
entropy lower bounds for projections of $\mu$.

We can use Theorem \ref{thm:factor of proj of L_m-1} in order to
establish the Factorization of $L_{m-1}$ via $\Pi$. In case $m<d$,
denote by $\iota_{d-m}$ the Plücker embedding of $\mathrm{Gr}_{d-m}(d)$
into $\mathrm{P}(\wedge^{d-m}\mathbb{R}^{d})$ (see Section \ref{subsec:spaces of alt forms}).
From the $(d-m)$-strong irreducibility of $\mathbf{S}_{\Phi}^{\mathrm{L}}$,
it follows that for every proper linear subspace $Z$ of $\wedge^{d-m}\mathbb{R}^{d}$
\[
\nu_{m}^{*}\left\{ V\in\mathrm{Gr}_{m}(d)\::\:\iota_{d-m}(V^{\perp})\subset Z\right\} =0.
\]
Moreover, when $m\ge2$ it follows from Theorem \ref{thm:cons of led-you intro}
that the assumptions of Theorem \ref{thm:factor of proj of L_m-1}
are satisfied for $\nu_{m}^{*}$-a.e. $V\in\mathrm{Gr}_{m}(d)$. The
following theorem will be deduced by using these facts.
\begin{thm}
\label{thm:L_m-1 factors via Pi}Suppose that (\ref{eq:no com fix asump})
and (\ref{eq:m-ired and m-prox assump}) hold and that $\dim\mu<d$.
Let $1\le m\le d$ be such that $\Sigma_{1}^{m-1}=m-1$ and $\Delta_{m}<1$.
Then there exists a Borel map $\widehat{L}_{m-1}:\mathbb{R}^{d}\rightarrow\mathrm{Gr}_{m-1}(d)$
so that $L_{m-1}(\omega)=\widehat{L}_{m-1}(\Pi\omega)$ for $\beta$-a.e.
$\omega$.
\end{thm}

\begin{rem*}
In what follows we write $L_{m-1}$ in place of $\widehat{L}_{m-1}$.
Which $L_{m-1}$ is intended will be clear from the context.
\end{rem*}
There is no reason to expect for the map $L_{m-1}:\mathbb{R}^{d}\rightarrow\mathrm{Gr}_{m-1}(d)$
to be continuous. On the other hand, since it is measurable it is
nearly continuous by Lusin\textquoteright s theorem. We use this fact
in order to get more refined lower bounds on the entropy of projections
of components of $\mu$. More concretely, given $1\le l<m$, sufficiently
large $k\ge1$, and $V\in\mathrm{Gr}_{m}(d)$, we show in Proposition
\ref{prop:L =000026 ent of comp} that for most scales $n\ge1$, most
$x\in\mathbb{R}^{d}$ with respect to $\mu$, and all $W\in\mathrm{Gr}_{l}(V)$
not containing a line too close to $(P_{V}L_{m-1}(x))^{\perp}\cap V$,
it holds that $\frac{1}{k}H(P_{W}P_{V}\mu_{\mathcal{D}_{n}(x)},\mathcal{D}_{n+k})$
is close to its maximal possible value (which is $l+O(1/k)$ since
$\dim W=l$). Here $\mathcal{D}_{n}(x)$ denotes the unique $D\in\mathcal{D}_{n}$
with $x\in D$.

\subsubsection{\label{subsec:An-entropy-increase}An entropy increase result}

For $1\le k\le l$ denote by $\mathrm{A}_{l,k}$ the space of surjective
affine maps from $\mathbb{R}^{l}$ to $\mathbb{R}^{k}$. With the
refined lower bounds of Proposition \ref{prop:L =000026 ent of comp}
in hand, we shall prove an entropy increase result for convolutions
of $\mu$ with measures in $\mathcal{M}(\mathrm{A}_{d,m})$. This
result is valid under relatively mild assumptions, and we find it
to be one of the main contributions of this paper. In order to state
it and describe the rest of our arguments, some additional notations
are needed.

Let $d_{\mathrm{A}_{d,m}}$ be an $\mathrm{A}_{m,m}$-invariant metric
on $\mathrm{A}_{d,m}$, which is induced by a Riemannian metric (see Section
\ref{subsec:An-invariant-metric}). That is,
\[
d_{\mathrm{A}_{d,m}}(\varphi\circ\psi_{1},\varphi\circ\psi_{2})=d_{\mathrm{A}_{d,m}}(\psi_{1},\psi_{2})\text{ for all }\psi_{1},\psi_{2}\in\mathrm{A}_{d,m}\text{ and }\varphi\in\mathrm{A}_{m,m}.
\]
The topology induced by $d_{\mathrm{A}_{d,m}}$ is equal to the subspace
topology induced by the norm topology of $\mathrm{A}_{d,m}^{\mathrm{vec}}$,
where $\mathrm{A}_{d,m}^{\mathrm{vec}}$ is the vector space of affine
maps from $\mathbb{R}^{d}$ into $\mathbb{R}^{m}$. For $\theta\in\mathcal{M}(\mathrm{A}_{d,m}^{\mathrm{vec}})$
and $\sigma\in\mathcal{M}(\mathbb{R}^{d})$ we denote by $\theta.\sigma\in\mathcal{M}(\mathbb{R}^{m})$
the pushforward of $\theta\times\sigma$ via the map $(\psi,x)\rightarrow\psi(x)$.
The measure $\theta.\sigma$ should be regarded as the convolution
of $\theta$ with $\sigma$.

The space $(\mathrm{A}_{d,m},d_{\mathrm{A}_{d,m}})$ is equipped with
a sequence $\{\mathcal{D}_{n}^{\mathrm{A}_{d,m}}\}_{n\ge0}$ of dyadic-like
partitions (see Section \ref{subsec:Dyadic-partitions}). When there
is no risk of confusion, we write $\mathcal{D}_{n}$ in place of $\mathcal{D}_{n}^{\mathrm{A}_{d,m}}$.
We can now state our entropy increase result.
\begin{thm}
\label{thm:ent increase result}Suppose that (\ref{eq:no com fix asump})
and (\ref{eq:m-ired and m-prox assump}) hold and that $\dim\mu<d$.
Let $1\le m\le d$ be such that $\Sigma_{1}^{m-1}=m-1$ and $\Delta_{m}<1$,
let $Q\subset\mathrm{A}_{d,m}$ be compact, and let $\epsilon>0$.
Then there exists $\delta=\delta(Q,\epsilon)>0$ so that for all $n\ge N(Q,\epsilon,\delta)$
and $\theta\in\mathcal{M}(\mathrm{A}_{d,m})$ with $\mathrm{supp}(\theta)\subset Q$
and $\frac{1}{n}H(\theta,\mathcal{D}_{n})\ge\epsilon$, we have $\frac{1}{n}H(\theta\ldotp\mu,\mathcal{D}_{n})>\Sigma_{1}^{m}+\delta$.
\end{thm}

Let us explain the ideas which go into the proof of Theorem \ref{thm:ent increase result}.
For this we need the following definition.
\begin{defn*}
Let $X$ be a normed space, and let $\epsilon>0$, $\theta\in\mathcal{M}(X)$
and a linear subspace $V\subset X$ be given. We say that $\theta$
is $(V,\epsilon)$-concentrated if there exists $x\in X$ so that
$\theta(x+V^{(\epsilon)})\ge1-\epsilon$, where $V^{(\epsilon)}$
denotes the closed $\epsilon$-neighbourhood of $V$.
\end{defn*}
Now assume by contradiction that the theorem fails for some compact
$Q\subset\mathrm{A}_{d,m}$ and $\epsilon>0$. Then it can be shown that
there exist $\epsilon'=\epsilon'(Q,\epsilon)>0$, compact sets $B_{1}\subset\mathrm{A}_{d,m}^{\mathrm{vec}}$
and $B_{2}\subset\mathrm{A}_{d,m}$, measures $\xi_{1},\xi_{2},...\in\mathcal{M}(B_{1})$,
and linear maps $\pi_{1},\pi_{2},...\in B_{2}$, so that for each $n\ge1$
\begin{enumerate}
\item $\mu\left\{ x\::\:\xi_{n}\ldotp x\mbox{ is }(\pi_{n}L_{m-1}(x),1/n)\text{-concentrated}\right\} >1-1/n$,
where $\xi_{n}\ldotp x$ is the pushforward of $\xi_{n}$ via the
map $\psi\rightarrow\psi(x)$;
\item $\xi_{k}$ is not $(\{0\},\epsilon')$-concentrated in $(\mathrm{A}_{d,m}^{\mathrm{vec}},\Vert\cdot\Vert)$,
where $\Vert\cdot\Vert$ is some fixed norm on $\mathrm{A}_{d,m}^{\mathrm{vec}}$.
\end{enumerate}
This part of the argument is carried out in Proposition \ref{prop:key prep for ent inc}.
It uses the aforementioned entropy lower bounds obtained in Propositions \ref{prop:L =000026 ent of comp}
and \ref{prop:inductive main proj prop}, an entropy increase result
of Hochman \cite{Ho} for linear convolutions in $\mathbb{R}^{m}$,
and ideas previously appearing in \cite{BHR,Ho,HR}.

By a compactness argument, it is now possible to show that there exist
$\xi\in\mathcal{M}(\mathrm{A}_{d,m}^{\mathrm{vec}})$, a linear map $\pi\in\mathrm{A}_{d,m}$,
and a map $T:\mathbb{R}^{d}\rightarrow\mathbb{R}^{m}$, so that $\xi$
is not a mass point and
\[
\xi\ldotp x\left(T(x)+\pi L_{m-1}(x)\right)=1\text{ for }\mu\text{-a.e. }x.
\]
This in turn will enable us to deduce that there exist $0\ne\psi\in\mathrm{A}_{d,m}^{\mathrm{vec}}$,
$x\in\mathbb{R}^{d}$ and independent $x_{1},...,x_{m-1}\in L_{m-1}(x)$,
so that

\begin{equation}
\psi(\varphi_{u}(x))\wedge\pi A_{\varphi_{u}}x_{1}\wedge...\wedge\pi A_{\varphi_{u}}x_{m-1}=0\text{ for all }u\in\Lambda^{*},\label{eq:psi(phi(Pi(om))) in piALom =0000231}
\end{equation}
where $A_{\varphi_{u}}$ is the linear part of $\varphi_{u}$.

Write $\mathbf{Z}_{\Phi}$ for the Zariski closure in $\mathrm{A}_{d,d}$
of the semigroup generated by the maps in $\Phi$ (see Section \ref{subsec:The-non-affinity-of L_m-1}
for details). The relation appearing in (\ref{eq:psi(phi(Pi(om))) in piALom =0000231})
is easily seen to be polynomial in $\varphi_{u}$. Hence, by the definition
of the Zariski topology,
\begin{equation}
\psi(\varphi(x))\wedge\pi A_{\varphi}x_{1}\wedge...\wedge\pi A_{\varphi}x_{m-1}=0\text{ for all }\varphi\in\mathbf{Z}_{\Phi}.\label{eq:psi(phi(Pi(om))) in piALom =0000232}
\end{equation}

It is well known that $\mathbf{Z}_{\Phi}$ is a group (see \cite[Lemma 6.15]{BQ}).
Moreover, by a classical result of Whitney \cite{Wh}, it follows
that $\mathbf{Z}_{\Phi}$ has finitely many connected components with
respect to the standard metric topology. Thus, since $\mathbf{Z}_{\Phi}$
contains $\Phi$ (which consists of strict contractions), it is not
hard to verify that the Lie algebra of $\mathbf{Z}_{\Phi}$ cannot
be semisimple. From this together with assumptions (\ref{eq:no com fix asump})
and (\ref{eq:m-ired and m-prox assump}) we are able to deduce (see
Lemma \ref{lem:T_a in Z_Phi all a}) that $T_{a}\in\mathbf{Z}_{\Phi}$
for all $a\in\mathbb{R}^{d}$, where $T_{a}x=x+a$ for $x\in\mathbb{R}^{d}$.
It is now easy to see that (\ref{eq:psi(phi(Pi(om))) in piALom =0000232})
is not possible, which shows that Theorem \ref{thm:ent increase result}
must be true.

\subsubsection{\label{subsec:Asymptotic-entropies-of}Asymptotic entropies of convolutions}

By applying Theorem \ref{thm:ent increase result}, we will be able
to establish the following statement. It resembles \cite[Theorem 1.4]{Ho1}
from which Theorem \ref{thm:Ho1}, dealing with the case $d=1$, follows
almost directly.

Recall from Section \ref{subsec:Full-dimensionality} that $p_{\Phi}^{*n}\in\mathcal{M}(\mathrm{A}_{d,d})$
denotes the convolution of $p_{\Phi}:=\sum_{i\in\Lambda}p_{i}\delta_{\varphi_{i}}$
with itself $n$ times. For each $V\in\mathrm{Gr}_{m}(d)$ fix $\pi_{V}\in\mathrm{A}_{d,m}$
for which there exists a linear isometry $F:V\rightarrow\mathbb{R}^{m}$
so that $F\circ P_{V}=\pi_{V}$. In what follows, we write $\pi_{V}p_{\Phi}^{*n}\in\mathcal{M}(\mathrm{A}_{d,m})$
for the pushforward of $p_{\Phi}^{*n}$ via the map $\psi\rightarrow\pi_{V}\circ\psi$.
\begin{thm}
\label{thm:asympt of cond ent}Suppose that (\ref{eq:no com fix asump})
and (\ref{eq:m-ired and m-prox assump}) hold and that $\dim\mu<d$.
Let $1\le m\le d$ be with $\Sigma_{1}^{m-1}=m-1$ and $\Delta_{m}<1$,
and let $V\in\mathrm{Gr}_{m}(d)$ be such that $\pi_{V}\mu$ is exact
dimensional with $\dim\pi_{V}\mu=\Sigma_{1}^{m}$. Then for all $M\ge1$,
\[
\underset{n\rightarrow\infty}{\lim}\:\frac{1}{n}H\left(\pi_{V}p_{\Phi}^{*n},\mathcal{D}_{Mn}\mid\mathcal{D}_{0}\right)=0.
\]
\end{thm}

Note that by Theorem \ref{thm:cons of led-you intro} it holds that
$\pi_{V}\mu$ is exact dimensional with $\dim\pi_{V}\mu=\Sigma_{1}^{m}$
for $\nu_{m}^{*}$-a.e. $V\in\mathrm{Gr}_{m}(d)$. Also note that
$\mathcal{D}_{Mn}$ and $\mathcal{D}_{0}$ are the dyadic-like partitions
of $\mathrm{A}_{d,m}$, which were referred to before the statement
of Theorem \ref{thm:ent increase result}.

The main ingredient in the proof of Theorem \ref{thm:asympt of cond ent}
is Theorem \ref{thm:ent increase result}. The proof also uses certain
non-conformal partitions of $\mathbb{R}^{m}$, which were previously
introduced in \cite[Section 7]{HR} in the case $d=m=2$.

From Theorem \ref{thm:asympt of cond ent} we are able to deduce the
following statement. Recall from Definition \ref{def:of rw ent}
that $h(p_{\Phi})$ denotes the random walk entropy of $p_{\Phi}$.
\begin{thm}
\label{thm:=00003DRW ent}Suppose that (\ref{eq:no com fix asump})
and (\ref{eq:m-ired and m-prox assump}) hold, $\Phi$ is Diophantine,
and $\dim\mu<d$. Let $1\le m\le d$ be with $\Sigma_{1}^{m-1}=m-1$
and $\Delta_{m}<1$. Then for $\nu_{m}^{*}$-a.e. $V\in\mathrm{Gr}_{m}(d)$,
\[
\underset{n\rightarrow\infty}{\lim}\:\frac{1}{n}H\left(\pi_{V}p_{\Phi}^{*n},\mathcal{D}_{0}\right)=h(p_{\Phi}).
\]
\end{thm}

In order to deduce Theorem \ref{thm:=00003DRW ent} from Theorem \ref{thm:asympt of cond ent},
we show that projections of Diophantine systems remain Diophantine
almost surely. More precisely, in Proposition \ref{prop:diophantine --> diophantine for projs}
we prove that when $\Phi$ is Diophantine there exists $\epsilon>0$
so that for $\nu_{m}^{*}$-a.e. $V$ and all $n\ge N(V)\ge1$,
\[
d_{\mathrm{A}_{d,m}}(\pi_{V}\varphi_{u},\pi_{V}\varphi_{w})\ge\epsilon^{n}\text{ for all }u,w\in\Lambda^{n}\text{ with }\varphi_{u}\ne\varphi_{w}.
\]
When $d=2$, $m=1$ and $\Phi$ satisfies the ESC, this was deduced
in \cite[Section 5]{BHR} from the positive dimensionality of $\nu_{1}^{*}$.
In our higher dimensional situation we need to use a stronger regularity
property of $\nu_{1}^{*}$, which is due to Guivarc\textquoteright h
\cite{Gu}. For $x\in\mathbb{R}^{d}$ with $|x|=1$ and small $r>0$,
this property will enable us to control the $\nu_{1}^{*}$-measure
of the set of lines $V\in\mathrm{Gr}_{1}(d)$ so that $|\pi_{V}x|\le r$
(see Lemma \ref{lem:There-exist-alpha =000026 C}).

\subsubsection{\label{subsec:Decompositions-of-sub-linear}Decompositions of sub-linear
diameter}

Recall that for $V\in\mathrm{Gr}_{m}(d)$ we denote by $\{\beta_{\omega}^{V}\}_{\omega\in\Lambda^{\mathbb{N}}}$
the disintegration of $\beta$ with respect to $\Pi^{-1}(P_{V}^{-1}\mathcal{B}_{\mathbb{R}^{d}})$
(see Section \ref{subsec:Disintegrations}). In order to deduce our
main results from Theorem \ref{thm:=00003DRW ent}, for $\nu_{m}^{*}$-a.e.
$V$ and $\beta$-a.e. $\omega$ we shall need to construct decompositions of 
$\beta_{\omega}^{V}$ into a controllable number of pieces, so that most pieces project onto a subset of $\mathrm{A}_{d,m}$ of sub-linear diameter.

Let us formulate this more precisely. Write $\mathcal{K}$ for the
set of integers $1\le k<m$ for which (\ref{eq:factor cond in gen criteria thm})
is satisfied. That is, $\mathcal{K}$ is the set of all $1\le k<m$
so that $P_{V}L_{k}\beta_{\omega}^{V}=\delta_{P_{V}L_{k}(\omega)}$
for $\nu_{m}^{*}$-a.e. $V$ and $\beta$-a.e. $\omega$. Let $1\le k_{1}<...<k_{b}<m$
be an enumeration of $\{1,...,m-1\}\setminus\mathcal{K}$, and set
$q=0$ if $b=0$ and $q=\sum_{l=1}^{b}k_{l}(k_{l+1}-k_{l})$ with
$k_{b+1}:=k_{b}+1$ if $b>0$. Additionally, for $n\ge1$ let $\Pi_{n}:\Lambda^{\mathbb{N}}\rightarrow\mathrm{A}_{d,d}$
be such that $\Pi_{n}(\omega)=\varphi_{\omega_{0}}\circ...\circ\varphi_{\omega_{n-1}}$
for $\omega\in\Lambda^{\mathbb{N}}$.

In Proposition \ref{prop:sub lin decomp}
we show that for $\nu_{m}^{*}$-a.e. $V$, $\beta$-a.e. $\omega$,
each $\epsilon>0$, and all $n\ge N(V,\omega,\epsilon)\ge1$, there
exist a Borel set $S\subset\Lambda^{\mathbb{N}}$ and a Borel partition
$\mathcal{E}$ of $S$ so that $\beta_{\omega}^{V}(S)>1-\epsilon$,
$|\mathcal{E}|\le2^{n(\epsilon+q(\chi_{1}-\chi_{m}))}$, and
\[
d_{\mathrm{A}_{d,m}}(\pi_{V}\Pi_{n}\eta,\pi_{V}\Pi_{n}\zeta)\le\epsilon n\text{ for all }E\in\mathcal{E}\text{ and }\eta,\zeta\in E.
\]
The aforementioned decompositions are obtained by conditioning $\beta_{\omega}^{V}$
on the atoms of the various partitions $\mathcal{E}\cup\{\Lambda^{\mathbb{N}}\setminus S\}$.

The appearance of the constant $q$ in the statement of Proposition \ref{prop:sub lin decomp} can be explained as follows. When $b>0$ and for each $W\in\mathrm{Gr}_{k_{b+1}}(d)$, the flag space
\[
\mathrm{F}:=\left\{ (V_{l})_{l=1}^{b+1}\in\mathrm{Gr}_{k_{1}}(W)\times...\times\mathrm{Gr}_{k_{b+1}}(W)\::\:V_{1}\subset...\subset V_{b+1}\right\}
\]
is a compact smooth manifold of dimension $q$. Thus, there exist a sequence $\left\{ \mathcal{D}_{n}^{\mathrm{F}}\right\} _{n\ge0}$ of dyadic-like partitions of $\mathrm{F}$ so that $|\mathcal{D}_{n}^{\mathrm{F}}|= O(2^{qn})$ for each $n\ge0$.  Such dyadic-like partitions are used in the proof of the proposition.

The proof of Proposition \ref{prop:sub lin decomp} also relies on certain
distance estimates in the space $(\mathrm{A}_{d,m},d_{\mathrm{A}_{d,m}})$.
These estimates in turn rely on a lemma from Ruelle\textquoteright s proof of the multiplicative ergodic theorem \cite[Lemma I.4]{Ru}. This Lemma will enable us to control the rate of convergence in the
limit appearing in the definition of the Furstenberg boundary maps.

When $d=m=2$ and $\mathcal{K}=\{1\}$, similar decompositions and distance
estimates were previously obtained in \cite[Section 8]{HR}. In that case a stronger version of Proposition \ref{prop:sub lin decomp}
is established (see the proof of \cite[Proposition 8.5]{HR}), in
which the images of the sets $E\in\mathcal{E}$ into $\mathrm{A}_{d,m}$
are of diameter $O(1)$ instead of $\le\epsilon n$.
The fact that our version is weaker in this sense is not very significant,
since sets of diameter $\le\epsilon n$ in $\mathrm{A}_{d,m}$
can intersect at most $2^{O(\epsilon n)}$ atoms of $\mathcal{D}_{0}^{\mathrm{A}_{d,m}}$
(see Lemma \ref{lem:ball int at most exp many 0-atoms}). This follows
from the fact that the volume of balls of radius $R$ in $(\mathrm{A}_{d,m},g)$
grows at most exponentially in $R$, where $g$ is the Riemannian
metric on $\mathrm{A}_{d,m}$ from which $d_{\mathrm{A}_{d,m}}$ is
induced. This in turn follows from $g$ being $\mathrm{A}_{m,m}\times\mathrm{SO}(d)$-invariant (see the proof of Lemma \ref{lem: def of inv met and prop}),
and hence of bounded curvature.

\subsubsection{\label{subsec:Proofs-of-the}Proofs of the main results}

All of our main results will follow from the following statement.
Given a discrete probability measure $\theta$ we denote its Shannon
entropy by $H(\theta)$.
\begin{thm}
\label{thm:gen dim result}Suppose that (\ref{eq:no com fix asump})
and (\ref{eq:m-ired and m-prox assump}) hold, $\Phi$ is Diophantine,
and $\dim\mu<d$. Let $1\le m\le d$ be with $\Sigma_{1}^{m-1}=m-1$
and $\Delta_{m}<1$, let $0\le b<m$ and $1\le k_{1}<...<k_{b}<m$
be integers, and set
\[
\mathcal{K}:=\{1,...,m-1\}\setminus\{k_{1},...,k_{b}\}.
\]
Assume that for $\nu_{m}^{*}$-a.e. $V$,
\[
P_{V}L_{k}\beta_{\omega}^{V}=\delta_{P_{V}L_{k}(\omega)}\text{ for }\beta\text{-a.e. }\omega\text{ and each }k\in\mathcal{K}.
\]
Then for $\nu_{m}^{*}$-a.e. $V$
\[
\underset{n\rightarrow\infty}{\limsup}\int\frac{1}{n}H\left(\Pi_{n}\beta_{\omega}^{V}\right)\:d\beta(\omega)\le q(\chi_{1}-\chi_{m}),
\]
where $q=0$ if $b=0$ and $q=\sum_{l=1}^{b}k_{l}(k_{l+1}-k_{l})$
with $k_{b+1}:=k_{b}+1$ if $b>0$.

\end{thm}

Theorem \ref{thm:gen dim result} will be deduced from Theorem \ref{thm:=00003DRW ent}
by using the decompositions obtained in Proposition \ref{prop:sub lin decomp}.

Regarding the proofs of the main results, our general criteria Theorem
\ref{thm:gen criteria} will follow by applying Theorem \ref{thm:gen dim result}
with $\mathcal{K}=\{1,...,m-1\}$ and from the Ledrappier-Young formula.
Note that when $m\ge2$ we always have $m-1\in\mathcal{K}$ by Theorems
\ref{thm:cons of led-you intro} and \ref{thm:factor of proj of L_m-1}.
This together with Theorem \ref{thm:gen criteria} will imply the
result for low dimensional projections (Theorem \ref{thm:proj onto 1 =000026 2 dim subspaces}),
and then the result for systems with the SOSC in $\mathbb{R}^{3}$
(Theorem \ref{thm:main d=00003D3}). Our result regarding Full dimensionality
of $\mu$ and its projections (Theorem \ref{thm:main all d}) will
be obtained by applying Theorem \ref{thm:gen dim result} with $\mathcal{K}=\{m-1\}$.

\subsubsection*{\textbf{\emph{Structure of the paper}}}

In Section \ref{sec:Preliminaries} we develop necessary notations
and background. In Section \ref{sec:Entropy-estimates} we establish
the entropy lower bounds described in Section \ref{subsec:Entropy-estimates}.
In Section \ref{sec:Factorization-of} we prove Theorems \ref{thm:factor of proj of L_m-1}
and \ref{thm:L_m-1 factors via Pi} regarding the factorization of
$P_{V}L_{m-1}$ and $L_{m-1}$, and obtain the refined lower bounds described after the statement of Theorem \ref{thm:L_m-1 factors via Pi}. In Section
\ref{sec:Entropy-growth-under conv} we prove our entropy increase
result. In Section \ref{sec:Asymptotic-entropies-of} we prove Theorems
\ref{thm:asympt of cond ent} and \ref{thm:=00003DRW ent} concerning
asymptotic entropies of convolutions. In Section \ref{sec:Decompositions-of-sub-linear}
we construct the decompositions described in Section \ref{subsec:Decompositions-of-sub-linear}.
In Section \ref{sec:Proofs-of-the main} we prove Theorem \ref{thm:gen dim result}
and deduce from it our main results.$\newline$

We include a summary of our main notation:

\begin{longtable}[c]{|c|>{\centering}p{7.5cm}|}
\hline 
$\Phi=\{\varphi_{i}(x)=A_{i}x+a_{i}\}_{i\in\Lambda}$ & Affine IFS on $\mathbb{R}^{d}$, Sec. \ref{subsec:The-setup}\tabularnewline
\hline 
$p=(p_{i})_{i\in\Lambda}$ & Probability vector, Sec. \ref{subsec:The-setup}\tabularnewline
\hline 
$K_{\Phi}$ & Attractor of $\Phi$, Sec. \ref{subsec:The-setup}\tabularnewline
\hline 
$\mu$ & Self-affine measure, Sec. \ref{subsec:The-setup}\tabularnewline
\hline 
$\mathbf{S}_{\Phi}^{\mathrm{L}}$ & Semigroup generated by $\{A_{i}\}_{i\in\Lambda}$, Sec. \ref{subsec:The-setup}\tabularnewline
\hline 
$p_{\Phi}$, $p_{\Phi}^{*n}$ & Probabilities on $\mathrm{A}_{d,d}$, Sec. \ref{subsec:Full-dimensionality}\tabularnewline
\hline 
$h(p_{\Phi})$ & Random walk entropy, Sec. \ref{subsec:Full-dimensionality}\tabularnewline
\hline 
$\varphi_{u}$,$A_{u}$,$p_{u}$ with $u\in\Lambda^{*}$ & Compositions and products, Sec. \ref{subsec:General-notations}\tabularnewline
\hline 
$\mathcal{M}(X)$ & Compactly supported probabilities, Sec. \ref{subsec:General-notations}\tabularnewline
\hline 
$\Vert\cdot\Vert_{TV}$, $d_{TV}$ & Total variation norm and distance, Sec. \ref{subsec:General-notations}\tabularnewline
\hline 
$E^{(r)}$ & Closed $r$-neighbourhood of $E$, Sec. \ref{subsec:General-notations}\tabularnewline
\hline 
$\{e_{m,i}\}_{i=1}^{m}$ & Standard basis of $\mathbb{R}^{m}$, Sec. \ref{subsec:Algebraic-notations}\tabularnewline
\hline 
$P_{V}$ & Orthogonal projection, Sec. \ref{subsec:Algebraic-notations}\tabularnewline
\hline 
$\mathrm{Gr}_{k}(m)$, $\mathrm{Gr}_{l}(V)$ & Grassmannians, Sec. \ref{subsec:Grassmannians-and-the flag space}\tabularnewline
\hline 
$d_{\mathrm{Gr}_{k}}$ & Metric on Grassmannian, Sec. \ref{subsec:Grassmannians-and-the flag space}\tabularnewline
\hline 
$\kappa(V_{1},V_{2})$ & Distance between $\mathrm{Gr}_{1}(V_{1})$, $\mathrm{Gr}_{1}(V_{2})$,
Sec. \ref{subsec:Grassmannians-and-the flag space}\tabularnewline
\hline 
$\mathrm{F}(m)$ & Manifold of complete flags, Sec. \ref{subsec:Grassmannians-and-the flag space}\tabularnewline
\hline 
$\wedge^{k}\mathbb{R}^{m}$  & Space of alternating $k$-forms, Sec. \ref{subsec:spaces of alt forms}\tabularnewline
\hline 
$\iota_{k}$ & Plücker embedding, Sec. \ref{subsec:spaces of alt forms}\tabularnewline
\hline 
$\wedge^{k}T$ & Linear map on $\wedge^{k}\mathbb{R}^{m}$, Sec. \ref{subsec:spaces of alt forms}\tabularnewline
\hline 
$\mathrm{A}_{m,k}^{\mathrm{vec}}$ & Affine maps from $\mathbb{R}^{m}$ into $\mathbb{R}^{k}$, Sec. \ref{subsec:Spaces-of-affine maps}\tabularnewline
\hline 
$A_{\psi}$, $a_{\psi}$  & Linear and translation parts of $\psi$, Sec. \ref{subsec:Spaces-of-affine maps}\tabularnewline
\hline 
$\mathrm{A}_{m,k}$ & Surjective elements in $\mathrm{A}_{m,k}^{\mathrm{vec}}$, Sec. \ref{subsec:Spaces-of-affine maps} \tabularnewline
\hline 
$\theta\ldotp\sigma$, $\theta\ldotp x$ & Pushforwards via action map, Sec. \ref{subsec:Spaces-of-affine maps}\tabularnewline
\hline 
$S_{c}$, $T_{a}$ & Scaling and translation maps, Sec. \ref{subsec:Spaces-of-affine maps}\tabularnewline
\hline 
$\pi_{m,k}$, $\pi_{V}$ & Elements in $\mathrm{A}_{m,k}$, Sec. \ref{subsec:Spaces-of-affine maps}\tabularnewline
\hline 
$\mathrm{diag}_{k,m}(t_{1},...,t_{k})$ & Diagonal $k\times m$ matrix, Sec. \ref{subsec:Singular-values-and SVD}\tabularnewline
\hline 
$\alpha_{1}(A)\ge...\ge\alpha_{k}(A)$ & Singular values of $A$, Sec. \ref{subsec:Singular-values-and SVD}\tabularnewline
\hline 
$L_{l}(A)$ & $l$-dimensional subspace, Sec. \ref{subsec:Singular-values-and SVD}\tabularnewline
\hline 
$d_{\mathrm{A}_{d,m}}$ & Invariant metric on $\mathrm{A}_{d,m}$, Sec. \ref{subsec:An-invariant-metric}\tabularnewline
\hline 
$\mathcal{D}_{n}^{m}$, $\mathcal{D}_{n}^{\mathrm{A}_{d,m}}$, $\mathcal{D}_{n}$ & Level-$n$ dyadic partitions, Sec. \ref{subsec:Dyadic-partitions}\tabularnewline
\hline 
$\theta_{x,n}$, $\theta^{x,n}$ & Component \& re-scaled component, Sec. \ref{subsec:Component-measures}\tabularnewline
\hline 
$\mathcal{N}_{k,n}$, $\mathcal{N}_{n}$ & Sets of integers, Sec. \ref{subsec:Component-measures}\tabularnewline
\hline 
$\lambda_{k,n}$, $\lambda_{n}$ & Uniform probabilities on $\mathcal{N}_{k,n}$, $\mathcal{N}_{n}$,
Sec. \ref{subsec:Component-measures}\tabularnewline
\hline 
$H(\xi)$, $H(\theta,\mathcal{D})$, $H(\theta,\mathcal{D}\mid\mathcal{E})$ & Entropy \& conditional entropy, Sec. \ref{subsec:Entropy}\tabularnewline
\hline 
$\dim_{e}$, $\overline{\dim}_{e}\theta$, $\underline{\dim}_{e}\theta$ & Entropy dimensions, Sec. \ref{subsec:Entropy-in-Rd}\tabularnewline
\hline 
$\omega|_{n}$ & $n$-prefix of $\omega$, Sec. \ref{subsec:Symbolic-notations}\tabularnewline
\hline 
$[u]$  & Cylinder set in $\Lambda^{\mathbb{N}}$, Sec. \ref{subsec:Symbolic-notations}\tabularnewline
\hline 
$\mathcal{P}_{n}$ & Partition of $\Lambda^{\mathbb{N}}$ into $n$-cylinders, Sec. \ref{subsec:Symbolic-notations}\tabularnewline
\hline 
$\beta$  & Bernoulli measure associated to $p$, Sec. \ref{subsec:Symbolic-notations}\tabularnewline
\hline 
$\Psi_{1}(n),...,\Psi_{d}(n)$ & Minimal cut-sets for $\Lambda^*$, Sec. \ref{subsec:Symbolic-notations}\tabularnewline
\hline 
$\Psi_{m}(n;\omega)$ & unique $u\in\Psi_{m}(n)$ with $\omega\in[u]$, Sec. \ref{subsec:Symbolic-notations}\tabularnewline
\hline 
$\mathbf{I}_{1}(n),...,\mathbf{I}_{d}(n),\mathbf{U}(n)$ & Random words, Sec. \ref{subsec:Symbolic-notations}\tabularnewline
\hline 
$\chi_{1}>...>\chi_{d}$ & Lyapunov exponents, Sec. \ref{subsec:Lyapunov-exponents-and Fur meas}\tabularnewline
\hline 
$\nu,\nu_{m},\nu^{*},\nu_{m}^{*}$ & Furstenberg measures, Sec. \ref{subsec:Lyapunov-exponents-and Fur meas}\tabularnewline
\hline 
$\Pi:\Lambda^{\mathbb{N}}\rightarrow K_{\Phi}$ & Coding map, Sec. \ref{subsec:Coding and Furstenberg maps}\tabularnewline
\hline 
$\Pi_{n}:\Lambda^{\mathbb{N}}\rightarrow\mathrm{A}_{d,d}$ & $\Pi_{n}(\omega):=\varphi_{\omega|_{n}}$ for $\omega\in\Lambda^{\mathbb{N}}$,
Sec. \ref{subsec:Coding and Furstenberg maps}\tabularnewline
\hline 
$L_{m}:\Lambda^{\mathbb{N}}\rightarrow\mathrm{Gr}_{m}(d)$ & Furstenberg boundary maps, Sec. \ref{subsec:Coding and Furstenberg maps}\tabularnewline
\hline 
$\{\theta_{x}^{V}\}_{x\in\mathbb{R}^{d}}$, $\{\xi_{\omega}^{V}\}_{\omega\in\Lambda^{\mathbb{N}}}$ & Disintegrations, Sec. \ref{subsec:Disintegrations}\tabularnewline
\hline 
$\mathrm{H}_{m}$, $\Delta_{m}$, $\Sigma_{l}^{m}$ & Constants appearing in LY formula, Sec. \ref{subsec:Ledrappier-Young-formula}\tabularnewline
\hline 
\end{longtable}

\subsubsection*{\textbf{\emph{Acknowledgment}}}

I would like to thank Mike Hochman for many inspiring discussions around the themes of this paper.

\section{\label{sec:Preliminaries}Preliminaries}

\subsection{\label{subsec:The-setup}The setup}

Throughout the rest of this paper, fix an integer $d\ge1$ and an
affine IFS $\Phi=\{\varphi_{i}(x)=A_{i}x+a_{i}\}_{i\in\Lambda}$ on
$\mathbb{R}^{d}$. Recall that $a_{i}\in\mathbb{R}^{d}$, $A_{i}\in\mathrm{GL}(d,\mathbb{R})$
and $\Vert A_{i}\Vert_{op}<1$ for each $i\in\Lambda$, where $\Vert\cdot\Vert_{op}$
is the operator norm. Also fix a strictly positive probability vector
$p=(p_{i})_{i\in\Lambda}$.

Let $K_{\Phi}$ be the attractor of $\Phi$, and let $\mu$ be the
self-affine measure corresponding to $\Phi$ and $p$. That is, $K_{\Phi}$
is the unique nonempty compact subset of $\mathbb{R}^{d}$ which satisfies
$K_{\Phi}=\cup_{i\in\Lambda}\varphi_{i}(K_{\Phi})$, and $\mu$ is
the unique Borel probability measure on $\mathbb{R}^{d}$ with $\mu=\sum_{i\in\Lambda}p_{i}\cdot\varphi_{i}\mu$.
Recall that $\varphi_{i}\mu:=\mu\circ\varphi_{i}^{-1}$ is the pushforward
of $\mu$ via $\varphi_{i}$.

Let $\mathbf{S}_{\Phi}^{\mathrm{L}}\subset\mathrm{GL}(d,\mathbb{R})$
denote the semigroup generated by the linear parts $\{A_{i}\}_{i\in\Lambda}$.
As noted above, we shall always assume that (\ref{eq:no com fix asump})
and (\ref{eq:m-ired and m-prox assump}) hold.
\begin{rem}
\label{rem:semi of trans also SI =000026 prox}By (\ref{eq:m-ired and m-prox assump})
and \cite[Lemma III.3.3]{BL}, it follows that the semigroup generated
by $\{A_{i}^{*}\}_{i\in\Lambda}$ is also $m$-strongly irreducible
and $m$-proximal for each $1\le m<d$.
\end{rem}

\subsection{\label{subsec:General-notations}General notations}

Throughout this paper the base of the $\log$ function is always $2$.

We shall use repeatedly the standard big-$O$ notation. Given parameters
$s_{1},...,s_{k}$ and a positive real number $M$, we write $O_{s_{1},...,s_{k}}(M)$
in place of an unspecified $Q\in\mathbb{R}$ with $|Q|\le CM$, where
$C$ is a positive constant depending only on $s_{1},...,s_{k}$.
We write $O(M)$ when no parameters are involved. We do not indicate
dependence on the objects fixed in Section \ref{subsec:The-setup}.
For instance, we write $O(M)$ in place of $O_{d}(M$).

Let $\Lambda^{*}$ be the set of finite words over $\Lambda$. Given
a group $G$, indexed elements $\{g_{i}\}_{i\in\Lambda}\subset G$,
and a word $i_{1}...i_{n}=u\in\Lambda^{*}$, we often write $g_{u}$
in place of $g_{i_{1}}\cdot...\cdot g_{i_{n}}$. For the empty word
$\emptyset$ we write $g_{\emptyset}$ in place of $1_{G}$, where
$1_{G}$ is the identity of $G$.

For a topological space $X$ denote the collection of compactly supported
Borel probability measures on $X$ by $\mathcal{M}(X)$. For $\theta\in\mathcal{M}(X)$
and a Borel set $E\subset X$ we write for $\theta|_{E}$ the restriction
of $\theta$ to $E$. That is,
\[
\theta|_{E}(F):=\theta(E\cap F)\text{ for all }F\subset X\text{ Borel.}
\]
When $\theta(E)>0$ we set $\theta_{E}:=\frac{1}{\theta(E)}\theta|_{E}$.
If $Y$ is another topological space and $f:X\rightarrow Y$ is Borel
measurable, we write $f\theta:=\theta\circ f^{-1}$ for the pushforward
of $\theta$ via $f$.

Given a signed Borel measure $\sigma$ on $X$, we denote its total
variation norm by $\Vert\sigma\Vert_{TV}$. That is,
\[
\Vert\sigma\Vert_{TV}:=\sup\{\sigma(E)-\sigma(X\setminus E)\::\:E\subset X\text{ is Borel}\}.
\]
For positive and finite Borel measures $\sigma_{1},\sigma_{2}$ on
$X$ we write $d_{TV}(\sigma_{1},\sigma_{2})$ in place of $\Vert\sigma_{1}-\sigma_{2}\Vert_{TV}$.

If $(X,\rho)$ is a metric space, given $\emptyset\ne E\subset X$ and $r>0$ set
\[
\mathrm{diam}(E):=\sup\{\rho(x_{1},x_{2})\::\:x_{1},x_{2}\in E\}\text{ and }E^{(r)}:=\{x\in X\::\:\rho(x,E)\le r\}.
\]
For $x\in X$ we denote by $B(x,r)$ the closed ball in $X$ with
centre $x$ and radius $r$.

Given $C\ge1$, two partitions $\mathcal{D},\mathcal{E}$ of the same set are said to be $C$-commensurable
if for each $D\in\mathcal{D}$ and $E\in\mathcal{E}$
\[
\#\{E'\in\mathcal{E}\::\:E'\cap D\ne\emptyset\}\le C\text{ and }\#\{D'\in\mathcal{D}\::\:D'\cap E\ne\emptyset\}\le C.
\]

\subsection{\label{subsec:Algebraic-notations}Algebraic notations}

For $m\ge1$ denote by $\left\langle \cdot,\cdot\right\rangle $ and
$|\cdot|$ the Euclidean inner product and norm of $\mathbb{R}^{m}$.
Let $\{e_{m,i}\}_{i=1}^{m}$ be the standard basis of $\mathbb{R}^{m}$.
Given a linear subspace $V$ of $\mathbb{R}^{m}$ write $P_{V}$ for
its orthogonal projection.

\subsubsection{\label{subsec:Grassmannians-and-the flag space}Grassmannians and
the flag space}

For $0\le k\le m$ denote by $\mathrm{Gr}_{k}(m)$ the Grassmannian
manifold of $k$-dimensional linear subspaces of $\mathbb{R}^{m}$.
For $W_{1},W_{2}\in\mathrm{Gr}_{k}(m)$ set
\[
d_{\mathrm{Gr}_{k}}(W_{1},W_{2}):=\Vert P_{W_{1}}-P_{W_{2}}\Vert_{op},
\]
which defines a metric on $\mathrm{Gr}_{k}(m)$.

For $V\in\mathrm{Gr}_{k}(m)$ and $0\le l\le k$ denote by $\mathrm{Gr}_{l}(V)$
the Grassmannian of $l$-dimensional linear subspaces of $V$. For
linear subspaces $V_{1},V_{2}$ of $\mathbb{R}^{m}$ we write $\kappa(V_{1},V_{2})$
in place of $d_{\mathrm{Gr}_{1}}(\mathrm{Gr}_{1}(V_{1}),\mathrm{Gr}_{1}(V_{2}))$,
that is
\[
\kappa(V_{1},V_{2}):=\inf\left\{ d_{\mathrm{Gr}_{1}}(W_{1},W_{2})\::\:W_{1}\in\mathrm{Gr}_{1}(V_{1})\text{ and }W_{2}\in\mathrm{Gr}_{1}(V_{2})\right\} .
\]
We define $\kappa(V_{1},V_{2})$ to be $\infty$ if $\dim V_{1}=0$ or $\dim V_{2}=0$.

Let $\mathrm{F}(m)$ denote the manifold of complete flags in $\mathbb{R}^{m}$,
that is
\[
\mathrm{F}(m):=\left\{ (V_{j})_{j=0}^{m}\in\mathrm{Gr}_{0}(m)\times...\times\mathrm{Gr}_{m}(m)\::\:V_{0}\subset...\subset V_{m}\right\}.
\]

\subsubsection{\label{subsec:spaces of alt forms}The spaces of alternating forms}

For $1\le k\le m$ denote by $\wedge^{k}\mathbb{R}^{m}$ the vector
space of alternating $k$-forms on the dual of $\mathbb{R}^{m}$ (see
e.g. \cite[Section 14]{Le}). Given $x_{1},...,x_{k}\in\mathbb{R}^{m}$,
let $x_{1}\wedge...\wedge x_{k}\in\wedge^{k}\mathbb{R}^{m}$ be with
\[
x_{1}\wedge...\wedge x_{k}(f_{1},...,f_{k})=\det\left((f_{i}(x_{j}))_{i,j=1}^{k}\right)\text{ for }f_{1},...,f_{k}\in\left(\mathbb{R}^{m}\right)^{*}.
\]
We have,
\begin{equation}
x_{1}\wedge...\wedge x_{k}\ne0\text{ if and only if }x_{1},...,x_{k}\text{ are linearly independent.}\label{eq:equiv cond for lin indep}
\end{equation}
Moreover, for linearly independent $y_{1},...,y_{k}\in\mathbb{R}^{m}$
\begin{equation}
x_{1}\wedge...\wedge x_{k}\mathbb{R}=y_{1}\wedge...\wedge y_{k}\mathbb{R}\text{ if and only if }\mathrm{span}\{x_{1},...,x_{k}\}=\mathrm{span}\{y_{1},...,y_{k}\},\label{eq:equiv cond for lin span}
\end{equation}
where $x_{1}\wedge...\wedge x_{k}\mathbb{R}$ denotes the line spanned by
$x_{1}\wedge...\wedge x_{k}$.

Let $\left\langle \cdot,\cdot\right\rangle $ be the inner product
on $\wedge^{k}\mathbb{R}^{m}$ which satisfies
\[
\left\langle x_{1}\wedge...\wedge x_{k},y_{1}\wedge...\wedge y_{k}\right\rangle =\det\left(\left\{ \left\langle x_{i},y_{j}\right\rangle \right\} _{i,j=1}^{k}\right)\text{ for }x_{1},...,x_{k},y_{1},...,y_{k}\in\mathbb{R}^{m}.
\]
We denote the norm induced by this inner product by $\Vert\cdot\Vert$.
Given $x_{1},...,x_{k}\in\mathbb{R}^{m}$, note that $\Vert x_{1}\wedge...\wedge x_{k}\Vert$
is equal to the $k$-dimensional volume of the parallelotope formed
by $x_{1},...,x_{k}$. Thus,
\begin{equation}
\Vert x_{1}\wedge...\wedge x_{k}\Vert\le\prod_{j=1}^{k}|x_{j}|\text{ for }x_{1},...,x_{k}\in\mathbb{R}^{m}.\label{eq:ub on norm of wedge}
\end{equation}
Moreover, given unit vectors $x,y\in\mathbb{R}^{m}$, the modulus
of the sine of the angle between $x\mathbb{R}$ and $y\mathbb{R}$
is equal to $\Vert x\wedge y\Vert$. Hence, there exists $C>1$ so
that for all nonzero $x,y\in\mathbb{R}^{m}$
\begin{equation}
C^{-1}d_{\mathrm{Gr}_{1}}(x\mathbb{R},y\mathbb{R})\le|x|^{-1}|y|^{-1}\Vert x\wedge y\Vert\le Cd_{\mathrm{Gr}_{1}}(x\mathbb{R},y\mathbb{R}).\label{eq:alt def for dist of lines}
\end{equation}

Write $\mathrm{P}(\wedge^{k}\mathbb{R}^{m})$ for the projective space
of $\wedge^{k}\mathbb{R}^{m}$, and let $\iota_{k}$ denote the Plücker
embedding of $\mathrm{Gr}_{k}(m)$ into $\mathrm{P}(\wedge^{k}\mathbb{R}^{m})$.
That is,
\[
\iota_{k}(V)=x_{1}\wedge...\wedge x_{k}\mathbb{R}\text{ for all }V\in\mathrm{Gr}_{k}(m)\text{ and independent }x_{1},...,x_{k}\in V.
\]

Given $l\ge k$ and a linear map $T:\mathbb{R}^{m}\rightarrow\mathbb{R}^{l}$,
let $\wedge^{k}T:\wedge^{k}\mathbb{R}^{m}\rightarrow\wedge^{k}\mathbb{R}^{l}$
be the linear map with
\[
\wedge^{k}T(x_{1}\wedge...\wedge x_{k})=(Tx_{1})\wedge...\wedge(Tx_{k})\text{ for }x_{1},...,x_{k}\in\mathbb{R}^{m}.
\]

\subsubsection{\label{subsec:Spaces-of-affine maps}Spaces of affine maps}

For $1\le k\le m$ denote by $\mathrm{A}_{m,k}^{\mathrm{vec}}$ the
vector space of affine maps from $\mathbb{R}^{m}$ into $\mathbb{R}^{k}$.
For $\psi\in\mathrm{A}_{m,k}^{\mathrm{vec}}$ set
\[
\Vert\psi\Vert:=\sup\{|\psi(x)|\::\:x\in\mathbb{R}^{m}\text{ and }|x|\le1\},
\]
which defines a norm on $\mathrm{A}_{m,k}^{\mathrm{vec}}$. Let $A_{\psi}$
and $a_{\psi}$ be the linear and translation parts of $\psi$ respectively,
that is $\psi(x)=A_{\psi}x+a_{\psi}$ for $x\in\mathbb{R}^{m}$.

Write $\mathrm{A}_{m,k}$ for the set of surjective elements in $\mathrm{A}_{m,k}^{\mathrm{vec}}$.
That is, $\mathrm{A}_{m,k}$ is the set of all $\psi\in\mathrm{A}_{m,k}^{\mathrm{vec}}$
so that $\mathrm{rank}(A_{\psi})=k$.

Unless stated otherwise explicitly, the spaces $\mathrm{A}_{m,k}^{\mathrm{vec}}$
and $\mathrm{A}_{m,k}$ are always assumed to be equipped with the
topology induced by $\Vert\cdot\Vert$.

Given $\theta\in\mathcal{M}(\mathrm{A}_{m,k}^{\mathrm{vec}})$ and
$\sigma\in\mathcal{M}(\mathbb{R}^{m})$, let $\theta\ldotp\sigma\in\mathcal{M}(\mathbb{R}^{k})$
be with
\[
\theta\ldotp\sigma(f)=\int\int f(\psi(x))\:d\sigma(x)\:d\theta(\psi)\text{ for all continuous }f:\mathbb{R}^{k}\rightarrow\mathbb{R}.
\]
That is, $\theta.\sigma$ is the pushforward of $\theta\times\sigma$
via the map $(\psi,x)\rightarrow\psi(x)$. For $x\in\mathbb{R}^{m}$
we write $\theta\ldotp x$ in place of $\theta.\delta_{x}$, where
$\delta_{x}$ is the Dirac mass at $x$. Given $1\le l\le k$ and
$\pi\in\mathrm{A}_{k,l}^{\mathrm{vec}}$, denote by $\pi\theta\in\mathcal{M}(\mathrm{A}_{m,l}^{\mathrm{vec}})$
the pushforward of $\theta$ via the map $\psi\rightarrow\pi\circ\psi$.
For $E\subset\mathrm{A}_{m,l}^{\mathrm{vec}}$ set
\[
\pi^{-1}E:=\{\psi\in\mathrm{A}_{m,k}^{\mathrm{vec}}\::\:\pi\circ\psi\in E\}.
\]
Thus, $\pi\theta(E)=\theta(\pi^{-1}E)$ for every Borel set $E\subset\mathrm{A}_{m,l}^{\mathrm{vec}}$.

Given $c>0$ and $a\in\mathbb{R}^{m}$, let $S_{c},T_{a}\in\mathrm{A}_{m,m}$
be with $S_{c}x=cx$ and $T_{a}x=x+a$ for $x\in\mathbb{R}^{m}$.

Let $\pi_{m,k}\in\mathrm{A}_{m,k}$ be with $\pi_{m,k}(x)=(x_{1},...,x_{k})$
for $(x_{1},...,x_{m})=x\in\mathbb{R}^{m}$. For each $V\in\mathrm{Gr}_{k}(m)$
fix some orthonormal basis $\{u_{V,1},...,u_{V,k}\}$ of $V$, and
let $\pi_{V}\in\mathrm{A}_{m,k}$ be with
\[
\pi_{V}(x)=\sum_{1\le i\le k}\left\langle x,u_{V,i}\right\rangle e_{k,i}\text{ for }x\in\mathbb{R}^{m},
\]
where recall that $\{e_{k,i}\}_{i=1}^{k}$ is the standard basis of
$\mathbb{R}^{k}$. Note that there exists a linear isometry $F:V\rightarrow\mathbb{R}^{k}$
so that $F\circ P_{V}=\pi_{V}$. Also note that for each $\psi\in\mathrm{A}_{m,k}$,
\begin{equation}
\psi=\varphi\pi_{(\ker A_{\psi})^{\perp}}\text{ for some }\varphi\in\mathrm{A}_{k,k}.\label{eq:rep of psi via pi_V}
\end{equation}

\subsubsection{\label{subsec:Singular-values-and SVD}Singular values and SVD}

Given $1\le k\le m$ and $t_{1},...,t_{k}\in\mathbb{R}$, we denote
by $\mathrm{diag}_{k,m}(t_{1},...,t_{k})$ the $k\times m$ real matrix
$(d_{i,j})$ with $d_{i,i}=t_{i}$ for $1\le i\le k$ and $d_{i,j}=0$
for $1\le i\le k$ and $1\le j\le m$ with $i\ne j$.

If $A$ is a $k\times m$ real matrix or a linear operator from $\mathbb{R}^{m}$
into $\mathbb{R}^{k}$, denote by $\alpha_{1}(A)\ge...\ge\alpha_{k}(A)\ge0$
its singular values. That is, $\alpha_{1}(A),...,\alpha_{k}(A)$ are
the square roots of the eigenvalues of $AA^{*}$, where $A^{*}$ is
the transpose of $A$. It is easy to verify that for each $1\le l\le k$,
\begin{equation}
\alpha_{l}(A)=\min\left\{ \max\left\{ |Ax|\::\:x\in V\text{ and }|x|=1\right\} \::\:V\in\mathrm{Gr}_{m+1-l}(m)\right\} .\label{eq:sing val as min max}
\end{equation}
Moreover, as shown in \cite[Lemma III.5.3]{BL},
\begin{equation}
\Vert\wedge^{l}A\Vert_{op}=\alpha_{1}(A)\cdot...\cdot\alpha_{l}(A),\label{eq:norm =00003D prod of sing vals}
\end{equation}
where $\Vert\cdot\Vert_{op}$ is the operator norm corresponding to
the norms defined in Section \ref{subsec:spaces of alt forms}.

Let $A=U_{1}DU_{2}$ be a singular value decomposition (SVD) of $A$. That is, $U_{1}\in\mathrm{O}(k)$, $U_{2}\in\mathrm{O}(m)$ and
\[
D=\mathrm{diag}_{k,m}(\alpha_{1}(A),...,\alpha_{k}(A)),
\]
where $\mathrm{O}(l)$ denotes the orthogonal group of $\mathbb{R}^{l}$.
For $1\le l<k$ with $\alpha_{l}(A)>\alpha_{l+1}(A)$ we write
\[
L_{l}(A):=\mathrm{span}\{U_{1}e_{k,1},...,U_{1}e_{k,l}\}.
\]
Since $\alpha_{l}(A)>\alpha_{l+1}(A)$, the subspace $L_{l}(A)\in\mathrm{Gr}_{l}(k)$ is easily seen to be independent
of the specific choice of the SVD. We also write $L_{0}(A)$ for the
zero subspace $\{0\}\subset\mathbb{R}^{k}$ and $L_{k}(A)$ for $\mathbb{R}^{k}$.

The following simple linear algebraic lemma will be needed later on.
\begin{lem}
\label{lem:bounds on sing vals of prod}Let $m\ge l\ge1$ and $A,B\in\mathrm{GL}(m,\mathbb{R})$
be given. Then
\[
\alpha_{l}(A)\alpha_{m}(B)\le\alpha_{l}(AB)\le\alpha_{l}(A)\alpha_{1}(B).
\]
\end{lem}

\begin{proof}
By (\ref{eq:sing val as min max}) there exists $V\in\mathrm{Gr}_{m+1-l}(m)$
so that $\alpha_{l}(A)=\Vert A|_{V}\Vert_{op}$, where $A|_{V}$ denotes
the restriction of $A$ to $V$. By (\ref{eq:sing val as min max})
it also follows that,
\[
\alpha_{l}(AB)\le\Vert(AB)|_{B^{-1}(V)}\Vert_{op}\le\Vert A|_{V}\Vert_{op}\cdot\Vert B|_{B^{-1}(V)}\Vert_{op}\le\alpha_{l}(A)\alpha_{1}(B).
\]
By applying this inequality with the matrices $AB$ and $B^{-1}$, we
obtain
\[
\alpha_{l}(A)=\alpha_{l}(ABB^{-1})\le\alpha_{l}(AB)\alpha_{1}(B^{-1}).
\]
Since $\alpha_{1}(B^{-1})^{-1}=\alpha_{m}(B)$, this completes the
proof of the lemma.
\end{proof}

\subsection{\label{subsec:An-invariant-metric}An invariant metric on $\mathrm{A}_{d,m}$}

We next equip the spaces $\mathrm{A}_{d,m}$ with an $\mathrm{A}_{m,m}$-invariant
metric.
\begin{lem}
\label{lem: def of inv met and prop}For every $1\le m\le d$ there
exists a metric $d_{\mathrm{A}_{d,m}}$ on $\mathrm{A}_{d,m}$ so
that,
\begin{enumerate}
\item $d_{\mathrm{A}_{d,m}}(\varphi\circ\psi_{1},\varphi\circ\psi_{2})=d_{\mathrm{A}_{d,m}}(\psi_{1},\psi_{2})$
for all $\psi_{1},\psi_{2}\in\mathrm{A}_{d,m}$ and $\varphi\in\mathrm{A}_{m,m}$;
\item for every compact $Q\subset\mathrm{A}_{d,m}$ there exists $C=C(Q)>1$
with
\[
C^{-1}\Vert\psi_{1}-\psi_{2}\Vert\le d_{\mathrm{A}_{d,m}}(\psi_{1},\psi_{2})\le C\Vert\psi_{1}-\psi_{2}\Vert\text{ for all }\psi_{1},\psi_{2}\in Q\:;
\]
\item $(\mathrm{A}_{d,m},d_{\mathrm{A}_{d,m}})$ is a complete metric space;
\item $(\mathrm{A}_{d,m},d_{\mathrm{A}_{d,m}})$ satisfies the Heine--Borel
property (i.e. closed and bounded subsets of $\mathrm{A}_{d,m}$ are
compact).
\end{enumerate}
\end{lem}

\begin{proof}
Suppose first that $m=d$, and note that $\mathrm{A}_{d,d}$ can be
identified with a closed subgroup of the connected group $G:=\mathrm{SL}(d+2,\mathbb{R})$
in a natural way. Let $d_{G}$ be the Riemannian distance function
induced by a left-invariant Riemannian metric on $G$, and let $d_{\mathrm{A}_{d,d}}$
be the restriction of $d_{G}$ to $\mathrm{A}_{d,d}$. It is easy
to verify that the first three properties in the statement of the
lemma are satisfied. The fourth property follows by the Hopf--Rinow
theorem (see e.g. \cite[Chapter 7]{dC}).

Suppose next that $m<d$, and write
\[
\mathrm{A}_{m,m}^{+}:=\{\varphi\in\mathrm{A}_{m,m}\::\:\det(A_{\varphi})>0\}\text{ and }\mathrm{SO}(d):=\{U\in\mathrm{O}(d)\::\:\det(U)=1\}.
\]
The connected group $G:=\mathrm{A}_{m,m}^{+}\times\mathrm{SO}(d)$
acts on $\mathrm{A}_{d,m}$ from the left as follows,
\[
(\varphi,U).\psi:=\varphi\circ\psi\circ U^{-1}\text{ for }\psi\in\mathrm{A}_{d,m}\text{ and }(\varphi,U)\in G.
\]
It is easy to verify that this action is transitive with compact stabilizers.
Thus, by \cite[Chapter 3,  Proposition 3.16]{CE}, there exists a
$G$-invariant Riemannian metric $\tilde{g}$ on $\mathrm{A}_{d,m}$.

Fix some $\varphi\in\mathrm{A}_{m,m}$ with $\det(A_{\varphi})<0$,
and denote by $\varphi^{*}\tilde{g}$ the pullback of $\tilde{g}$
via the map $\psi\rightarrow\varphi\circ\psi$ from $\mathrm{A}_{d,m}$
onto itself. Set $g:=\tilde{g}+\varphi^{*}\tilde{g}$ and let $d_{\mathrm{A}_{d,m}}$
be the Riemannian distance function induced by $g$. It is now easy
to verify that $d_{\mathrm{A}_{d,m}}$ satisfies all of the required
properties, where the last one follows again from the Hopf--Rinow
theorem.
\end{proof}
From now on, unless stated otherwise explicitly, all metric concepts
in $\mathrm{A}_{d,m}$ are with respect to $d_{\mathrm{A}_{d,m}}$.
In particular, for $E\subset\mathrm{A}_{d,m}$ we write $\mathrm{diam}(E)$
for the diameter of $E$ with respect to $d_{\mathrm{A}_{d,m}}$,
and for $\psi\in\mathrm{A}_{d,m}$ and $r>0$ we write $B(\psi,r)$
for the closed ball in $(\mathrm{A}_{d,m},d_{\mathrm{A}_{d,m}})$
with centre $\psi$ and radius $r$. Note that by Lemma \ref{lem: def of inv met and prop}
it follows that $B(\psi,r)$ is compact.

\subsection{\label{subsec:Dyadic-partitions}Dyadic partitions}

For $m\ge1$ and $n\in\mathbb{Z}$ denote by $\mathcal{D}_{n}^{m}$
the level-$n$ dyadic partition of $\mathbb{R}^{m}$. That is,
\[
\mathcal{D}_{n}^{m}:=\left\{ [\frac{k_{1}}{2^{n}},\frac{k_{1}+1}{2^{n}})\times...\times[\frac{k_{m}}{2^{n}},\frac{k_{m}+1}{2^{n}})\::\:k_{1},...,k_{m}\in\mathbb{Z}\right\} .
\]
We usually omit the superscript $m$ when it is clear from the context.
For $t\in\mathbb{R}$ we write $\mathcal{D}_{t}$ in place of $\mathcal{D}_{\left\lfloor t\right\rfloor }$.

Given $\psi\in\mathrm{A}_{d,m}$, the following lemma controls the $\psi\mu$-measure of points which are
very close to the boundary of some $D\in\mathcal{D}_{n}^{m}$.
\begin{lem}
\label{lem:small mass of pts close to bd}Let $1\le m\le d$ be given. Then for every $\epsilon,R>0$ there exists $\delta>0$ so that,
\[
\psi\mu\left(\cup_{D\in\mathcal{D}_{n}^m}(\partial D)^{(2^{-n}\delta)}\right)<\epsilon\text{ for all }\psi\in B(\pi_{d,m},R)\text{ and }n\ge0.
\]
\end{lem}

\begin{proof}
From  assumptions (\ref{eq:no com fix asump}) and (\ref{eq:m-ired and m-prox assump}) it follows easily that $K_{\Phi}$ is not contained in a proper
affine subspace of $\mathbb{R}^{d}$ (see Lemma \ref{lem:K not cont in V} below). From this and by an argument
identical to the one given in \cite[Lemma 2.5]{Ra_Rajchman}, we get that
$\pi\mu\{t\}=0$ for all $\pi\in\mathrm{A}_{d,1}$ and $t\in\mathbb{R}$.
Thus, by compactness, it follows that for each $\epsilon>0$ there
exists $\delta>0$ so that $\pi_V\mu(B(t,\delta))<\epsilon$ for all $V\in\mathrm{Gr}_1(d)$ and $t\in\mathbb{R}$. Now by an argument identical to the one given in \cite[Lemma 3.13]{BHR}, it follows that for every $\epsilon>0$ there exists $\delta>0$ so that
\begin{equation}\label{from_BHR_proj_lem}
\pi_V\mu(B(t,r\delta))\le\epsilon\cdot \pi_V\mu(B(t,r))\text{ for all }V\in\mathrm{Gr}_1(d),\:t\in \mathbb{R}\text{ and }0<r\le1.
\end{equation}

For each $1\le i\le m$ let $f_i\in\mathrm{A}_{m,1}$ be with $f_i(x)=x_i$ for $(x_1,...,x_m)=x\in\mathbb{R}^m$. From (\ref{from_BHR_proj_lem}) and (\ref{eq:rep of psi via pi_V}) it follows easily that for every $\epsilon,R>0$ there exists $\delta>0$ so that for all $\psi\in B(\pi_{d,m},R)$,
\[
f_i\psi\mu(B(t,r\delta))\le\epsilon\cdot f_i\psi\mu(B(t,r))\text{ for all }1\le i\le m,\:t\in \mathbb{R}\text{ and }0<r\le1.
\]
Since for each $\delta>0$ and $n\ge0$
\[
\cup_{D\in\mathcal{D}_{n}^m}(\partial D)^{(2^{-n}\delta)}=\cup_{i=1}^m\cup_{k\in\mathbb{Z}}f_i^{-1}\left(B(k2^{-n},2^{-n}\delta)\right),
\]
this completes the proof of the lemma.
\end{proof}

Let $1\le m\le d$ be given. We also need to introduce dyadic-like
partitions for $\mathrm{A}_{d,m}$. By \cite[Remark 2.2]{KRS} there
exists a sequence $\{\mathcal{D}_{n}^{\mathrm{A}_{d,m}}\}_{n\ge0}$
of Borel partitions of $\mathrm{A}_{d,m}$ so that,
\begin{enumerate}
\item $\mathcal{D}_{n+1}^{\mathrm{A}_{d,m}}$ refines $\mathcal{D}_{n}^{\mathrm{A}_{d,m}}$
for each $n\ge0$. That is, for each $D\in\mathcal{D}_{n+1}^{\mathrm{A}_{d,m}}$
there exists $D'\in\mathcal{D}_{n}^{\mathrm{A}_{d,m}}$ with $D\subset D'$;
\item There exists a constant $C>1$ so that for each $n\ge0$ and $D\in\mathcal{D}_{n}^{\mathrm{A}_{d,m}}$
there exists $\psi_{D}\in D$ with
\[
B(\psi_{D},C^{-1}2^{-n})\subset D\subset B(\psi_{D},C2^{-n}).
\]
\end{enumerate}
When there is no risk of confusion we write $\mathcal{D}_{n}$
in place of $\mathcal{D}_{n}^{\mathrm{A}_{d,m}}$. We next derive
some useful properties of these partitions.
\begin{lem}
\label{lem:commens partiti of A_d,m}Let $1\le m\le d$ be given.
Then there exists a constant $C>1$ so that,
\begin{enumerate}
\item $\#\{D'\in\mathcal{D}_{n+1}^{\mathrm{A}_{d,m}}\::\:D'\subset D\}\le C$
for all $n\ge0$ and $D\in\mathcal{D}_{n}^{\mathrm{A}_{d,m}}$;
\item the partitions $\mathcal{D}_{n}^{\mathrm{A}_{d,m}}$ and $\varphi^{-1}\mathcal{D}_{n}^{\mathrm{A}_{d,m}}:=\{\varphi^{-1}(D)\::\:D\in\mathcal{D}_{n}^{\mathrm{A}_{d,m}}\}$
are $C$-commensurable for all $n\ge0$ and $\varphi\in\mathrm{A}_{m,m}$.
\end{enumerate}
\end{lem}

\begin{proof}
By volume considerations in the normed space $(\mathrm{A}_{d,m}^{\mathrm{vec}},\Vert\cdot\Vert)$, for each $\epsilon>0$ there exists $C(\epsilon)>1$ so that every ball of radius $r>0$ in $\mathrm{A}_{d,m}^{\mathrm{vec}}$ contains at most $C(\epsilon)$ disjoint
balls of radius $\epsilon r$ in $\mathrm{A}_{d,m}^{\mathrm{vec}}$. From this and Lemma \ref{lem: def of inv met and prop}, it follows that for each $R>1$ and $\epsilon>0$ there exists $C(R,\epsilon)>1$ so that every ball of radius $r\in(0,R]$ in $\mathrm{A}_{d,m}$ contains at most $C(R,\epsilon)$ disjoint balls of radius $\epsilon r$ in $\mathrm{A}_{d,m}$.
The lemma now follows easily from this fact together with the second defining property of the partitions $\{\mathcal{D}_{n}^{\mathrm{A}_{d,m}}\}_{n\ge0}$.
\end{proof}
\begin{lem}
\label{lem:ball int at most exp many 0-atoms}Let $1\le m\le d$ be
given. Then there exists a constant $C>1$ so that for all $\psi\in\mathrm{A}_{d,m}$
and $R>0$
\[
\#\left\{ D\in\mathcal{D}_{0}^{\mathrm{A}_{d,m}}\::\:D\cap B(\psi,R)\ne\emptyset\right\} \le C2^{CR}.
\]
\end{lem}

\begin{proof}
We carry out the proof in the case $m<d$. The case $m=d$ follows by
similar considerations. Set $G:=\mathrm{A}_{m,m}\times\mathrm{SO}(d)$,
and let $g$ be the $G$-invariant Riemannian metric on $\mathrm{A}_{d,m}$
from which $d_{\mathrm{A}_{d,m}}$ is induced (see the proof of Lemma
\ref{lem: def of inv met and prop}). Denote respectively by $\mathrm{vol}$
and $\mathrm{Ric}$ the Riemannian volume measure and Ricci curvature
tensor (see \cite[Chapter 2]{Pe}) corresponding to $g$.

Since $g$ is $G$-invariant it is clear that there exists $0>k>-\infty$
so that
\[
\mathrm{Ric}(v,v)\ge(n-1)k\cdot g(v,v)\text{ for }\psi\in A_{d,m}\text{ and }v\in T_{\psi}A_{d,m},
\]
where $n:=\dim A_{d,m}$ and $T_{\psi}A_{d,m}$ is the tangent space
of $A_{d,m}$ at $\psi$. From this and by \cite[Lemma 35]{Pe} it
follows that
\[
\mathrm{vol}(B(\psi,R))\le\upsilon(n,k,R)\text{ for }\psi\in A_{d,m}\text{ and }R>0,
\]
where $\upsilon(n,k,R)$ is the volume of the ball of radius $R$
in the space form $S_{k}^{n}$ of dimension $n$ and constant sectional
curvature $k$. By the definition of the metric on $S_{k}^{n}$ (see
\cite[Chapter 3]{Pe}), there exists $C_{1}=C_{1}(n,k)>1$ so that
$\upsilon(n,k,R)\le2^{C_{1}R}$ for $R>0$. Hence, $\mathrm{vol}(B(\psi,R))\le2^{C_{1}R}$
for $\psi\in A_{d,m}$ and $R>0$.

By the second defining property of the partitions $\{\mathcal{D}_{n}^{\mathrm{A}_{d,m}}\}_{n\ge0}$,
there exist $\epsilon>0$ and $C_{2}>1$ so that $\mathrm{diam}(D)\le C_{2}$
and $\mathrm{vol}(D)\ge\epsilon$ for all $D\in\mathcal{D}_{0}^{\mathrm{A}_{d,m}}$.
Let $\psi\in A_{d,m}$ and $R>0$, set
\[
\mathcal{Q}:=\left\{ D\in\mathcal{D}_{0}^{\mathrm{A}_{d,m}}\::\:D\cap B(\psi,R)\ne\emptyset\right\} ,
\]
and note that $D\subset B(\psi,R+C_{2})$ for $D\in\mathcal{Q}$.
Thus,
\[
\epsilon|\mathcal{Q}|\le\sum_{D\in\mathcal{Q}}\mathrm{vol}(D)\le\mathrm{vol}(B(\psi,R+C_{2}))\le2^{C_{1}(R+C_{2})},
\]
which completes the proof of the lemma. 
\end{proof}

\subsection{\label{subsec:Component-measures}Component measures}

In this subsection we introduce the necessary notations for component measures.
For more details and examples we refer the reader to \cite[Section 2.2]{Ho1}
or \cite[Section 2.3]{Ho}.

Suppose that $X$ is either $\mathbb{R}^{m}$ or $\mathrm{A}_{d,m}$
for some $1\le m\le d$. Given $x\in X$ and a Borel partition $\mathcal{D}$
of $X$ we denote by $\mathcal{D}(x)$ the unique $D\in\mathcal{D}$
with $x\in D$. As indicated in Section \ref{subsec:Dyadic-partitions},
the space $X$ is equipped with a sequence of refining partitions
$\{\mathcal{D}_{n}\}_{n\ge0}$. For $\theta\in\mathcal{M}(X)$ and
$n\ge0$ we define a measure valued random element $\theta_{x,n}$
such that $\theta_{x,n}=\theta_{\mathcal{D}_{n}(x)}$ with probability
$\theta(\mathcal{D}_{n}(x))$ for each $x\in X$. Thus, for an event
$\mathcal{U}\subset\mathcal{M}(X)$
\[
\mathbb{P}\left(\theta_{x,n}\in\mathcal{U}\right):=\theta\left\{ x\in X\::\:\theta_{\mathcal{D}_{n}(x)}\in\mathcal{U}\right\}.
\]
We call $\theta_{x,n}$ an $n$th level component of $\theta$. For
a given $x\in X$ with $\theta(\mathcal{D}_{n}(x))>0$, we often write
$\theta_{x,n}$ in place of $\theta_{\mathcal{D}_{n}(x)}$ even when
no randomization is involved.

Sometimes $n$ is chosen randomly as well, usually uniformly in some
range. For example, for integers $n_{2}\geq n_{1}$
\[
\mathbb{P}_{n_{1}\leq i\leq n_{2}}\left(\theta_{x,i}\in\mathcal{U}\right):=\frac{1}{n_{2}-n_{1}+1}\sum_{i=n_{1}}^{n_{2}}\mathbb{P}(\theta_{x,i}\in\mathcal{U}).
\]
This notation implicitly defines $x$ and $i$ as random variables. For
instance, for a Borel set $E\subset X$
\[
\mathbb{P}_{n_{1}\leq i\leq n_{2}}\left(\theta_{x,i}\in\mathcal{U}\text{ and }x\in E\right):=\frac{1}{n_{2}-n_{1}+1}\sum_{i=n_{1}}^{n_{2}}\theta\left\{ x:\theta_{x,i}\in\mathcal{U}\text{ and }x\in E\right\} .
\]

If $\sigma$ is another probability measure on $X$ and $\theta_{x,n}$
and $\sigma_{y,n}$ appear in the same probabilistic expression, we
assume $x$ and $y$ are chosen independently unless stated otherwise.

We denote by $\mathbb{E}$ and $\mathbb{E}_{n_{1}\leq i\leq n_{2}}$
the expected value with respect to the probabilities $\mathbb{P}$
and $\mathbb{P}_{n_{1}\leq i\leq n_{2}}$.

We also introduce notation for randomly chosen integers in interval
ranges. Given integers $n\ge k\ge1$ let $\mathcal{N}_{k,n}:=\{k,k+1,..,n\}$
and denote the normalized counting measure on $\mathcal{N}_{k,n}$
by $\lambda_{k,n}$, i.e. $\lambda_{k,n}\{i\}=\frac{1}{n-k+1}$ for
each $k\le i\le n$. We write $\mathcal{N}_{n}$ and $\lambda_{n}$
in place of $\mathcal{N}_{1,n}$ and $\lambda_{1,n}$.

Component distributions have the convenient property that they are
almost invariant under repeated sampling, i.e. choosing components
of components. More precisely, for $\theta\in\mathcal{M}(X)$ and
$k,n\ge1$, let $\mathbb{P}_{n}^{\theta}$ denote the distribution
of components $\theta_{x,i}$, $0\leq i\leq n$, as defined above;
and let $\mathbb{Q}_{n,k}^{\theta}$ denote the distribution on components
obtained by first choosing a random component $\theta_{x,i}$, $0\leq i\leq n$,
and then, conditionally on $\sigma=\theta_{x,i}$, choosing a component
$\sigma_{y,j}$, $i\leq j\leq i+k$ with the usual distribution (note
that $\sigma_{y,j}=\theta_{y,j}$ is indeed a component of $\theta$).
For a proof of the following lemma see \cite[Lemma 2.7]{Ho}.
\begin{lem}
\label{lem:distribution-of-components-of-components}Given $\theta\in\mathcal{M}(X)$
and $k,n\ge1$, the total variation distance between $\mathbb{P}_{n}^{\theta}$
and $\mathbb{Q}_{n,k}^{\theta}$ satisfies 
\[
d_{TV}(\mathbb{P}_{n}^{\theta},\mathbb{Q}_{n,k}^{\theta})=O(k/n).
\]
\end{lem}

In Euclidean space we also introduce re-scaled components. For $m\ge1$,
$\theta\in\mathcal{M}(\mathbb{R}^{m})$ and $n\ge0$, denote by $\theta^{x,n}$
the push-forward of $\theta_{x,n}$ by the unique homothety which
maps $\mathcal{D}_{n}(x)$ onto $[0,1)^{m}$. We view these as random
variables using the same conventions as above.

\subsection{\label{subsec:Entropy}Entropy}

In this subsection we provide some useful properties of entropy. These
properties will be used repeatedly throughout the rest of the paper,
often without further reference. For more details on entropy we refer
the reader to \cite[Section 2.2]{Pa}.

Let $(X,\mathcal{F})$ be a measurable space. Given a probability
measure $\theta$ on $X$ and a countable partition $\mathcal{D}\subset\mathcal{F}$
of $X$, the entropy of $\theta$ with respect to $\mathcal{D}$ is
defined by
\[
H(\theta,\mathcal{D}):=-\sum_{D\in\mathcal{D}}\theta(D)\log\theta(D).
\]
If $\mathcal{E}\subset\mathcal{F}$ is another countable partition
of $X$, the conditional entropy given $\mathcal{E}$ is defined as
follows
\[
H(\theta,\mathcal{D}\mid\mathcal{E}):=\sum_{E\in\mathcal{E}}\theta(E)\cdot H(\theta_{E},\mathcal{D}).
\]
For a sub-$\sigma$-algebra $\mathcal{A}$ of $\mathcal{F}$, the
conditional entropy given $\mathcal{A}$ is defined by
\[
H(\theta,\mathcal{D}\mid\mathcal{A}):=\int-\sum_{D\in\mathcal{D}}\theta(D\mid\mathcal{A})\log\theta(D\mid\mathcal{A})\:d\theta,
\]
where $\theta(D\mid\mathcal{A})$ is the conditional probability of
$D$ given $\mathcal{A}$. Note that
\[
H(\theta,\mathcal{D}\mid\mathcal{E})=H(\theta,\mathcal{D}\mid\sigma(\mathcal{E})),
\]
where $\sigma(\mathcal{E})$ is the sigma algebra generated by $\mathcal{E}$.
If $\xi$ is a discrete probability measure on $X$, its Shannon entropy
is given by
\[
H(\xi):=-\sum_{x\in X}\xi\{x\}\log\xi\{x\}.
\]

In what follows all entropies are assumed to be finite. Note that
we always have the following upper bound,
\begin{equation}
H(\theta,\mathcal{D})\le\log\left(\#\{D\in\mathcal{D}\::\:\theta(D)>0\}\right).\label{eq:card ub for ent}
\end{equation}

Given another countable partition $\mathcal{Q}\subset\mathcal{F}$
of $X$, conditional entropy satisfies the formula
\begin{equation}
H(\theta,\mathcal{D}\vee\mathcal{E}\mid\mathcal{Q})=H(\theta,\mathcal{E}\mid\mathcal{Q})+H(\theta,\mathcal{D}\mid\mathcal{Q}\vee\mathcal{E}),\label{eq:extended cond ent form}
\end{equation}
where
\[
\mathcal{D}\vee\mathcal{E}:=\{D\cap E\::\:D\in\mathcal{D}\text{ and }E\in\mathcal{E}\}.
\]
By taking $\mathcal{Q}$ to be the trivial partition, this gives
\begin{equation}
H(\theta,\mathcal{D}\mid\mathcal{E})=H(\theta,\mathcal{D}\vee\mathcal{E})-H(\theta,\mathcal{E}).\label{eq:cond ent form}
\end{equation}

If $(Y,\mathcal{G})$ is another measurable space, $\mathcal{C}\subset\mathcal{G}$
is a countable partition of $Y$, and $f:X\rightarrow Y$ is measurable, we have
\[
H(\theta,f^{-1}(\mathcal{C}))=H(f\theta,\mathcal{C}),
\]
where $f^{-1}(\mathcal{C}):=\{f^{-1}(C)\::\:C\in\mathcal{C}\}$.

The conditional entropy $H(\theta,\mathcal{D}\mid\mathcal{E})$ is
increasing in the $\mathcal{D}$-argument and decreasing in the $\mathcal{E}$-argument.
More precisely, if $\mathcal{D}',\mathcal{E}'\subset\mathcal{F}$
are countable partitions refining $\mathcal{D},\mathcal{E}$ respectively,
we have
\[
H(\theta,\mathcal{D}\mid\mathcal{E}')\le H(\theta,\mathcal{D}\mid\mathcal{E})\le H(\theta,\mathcal{D}'\mid\mathcal{E}).
\]
Similarly
\begin{equation}
H(\theta,\mathcal{D}\mid\mathcal{A}_{1})\le H(\theta,\mathcal{D}\mid\mathcal{A}_{2}),\label{eq:monot of ent in cond sig-alg}
\end{equation}
for sub-$\sigma$-algebras of $\mathcal{A}_{2}\subset\mathcal{A}_{1}\subset\mathcal{F}$.

The entropy and conditional entropy functions are concave and almost
convex in the measure argument. That is, given probability measures
$\theta_{1},...,\theta_{k}$ on $X$ and a probability vector $q=(q_{i})_{i=1}^{k}$
so that $\theta=\sum_{i=1}^{k}q_{i}\theta_{i}$, we have
\begin{equation}
\sum_{i=1}^{k}q_{i}H(\theta_{i},\mathcal{D})\le H(\theta,\mathcal{D})\le\sum_{i=1}^{k}q_{i}H(\theta_{i},\mathcal{D})+H(q),\label{eq:conc =000026 almo conv of ent}
\end{equation}
where $H(q):=-\sum q_{i}\log q_{i}$. These inequalities remain true
with $H(\cdot,\mathcal{D}\mid\mathcal{E})$ in place of $H(\cdot,\mathcal{D})$.

Recall that given $C\ge1$ we say that $\mathcal{D},\mathcal{E}$ are $C$-commensurable
if for each $D\in\mathcal{D}$ and $E\in\mathcal{E}$,
\[
\#\{E'\in\mathcal{E}\::\:E'\cap D\ne\emptyset\}\le C\text{ and }\#\{D'\in\mathcal{D}\::\:D'\cap E\ne\emptyset\}\le C.
\]
From (\ref{eq:cond ent form}) and (\ref{eq:card ub for ent}), it
follows easily that 
\[
\left|H(\theta,\mathcal{D})-H(\theta,\mathcal{E})\right|\le2\log C
\]
whenever $\mathcal{D}$ and $\mathcal{E}$ are $C$-commensurable.

Entropy is continuous in the total variation distance $d_{TV}$. In
particular we have the following lemma. As explained in \cite[Lemma 3.4]{Ho},
its proof follows easily from the concavity and almost convexity of conditional
entropy.
\begin{lem}
\label{lem:ent cont wrt tot var}Let $\theta,\xi$ be probability
measures on the measurable space $(X,\mathcal{F})$, and let $\mathcal{D},\mathcal{E}\subset\mathcal{F}$
be countable partitions of $X$. Suppose that there exists $C\ge1$
so that
\[
\#\{D\in\mathcal{D}\::\:D\cap E\ne\emptyset\}\le C\text{ for each }E\in\mathcal{E}.
\]
Then,
\[
\left|H(\theta,\mathcal{D}\mid\mathcal{E})-H(\xi,\mathcal{D}\mid\mathcal{E})\right|\le1+d_{TV}(\theta,\xi)\log C.
\]
\end{lem}

\subsection{\label{subsec:Entropy-in-Rd}Entropy in $\mathbb{R}^{m}$}

Let $1\le m\le d$, $\theta\in\mathcal{M}(\mathbb{R}^{m})$ and $n\ge1$
be given. We call $H(\theta,\mathcal{D}_{n})$ the scale-$n$ entropy
of $\theta$, where recall that $\mathcal{D}_{n}$ is the level-$n$
dyadic partition of $\mathbb{R}^{m}$. We collect here some basic
properties of this quantity. As in the previous subsection, these properties
will be used repeatedly, often without further reference.

We often normalize by $n$, in which case
\[
\frac{1}{n}H(\theta,\mathcal{D}_{n})\leq m+O\left(\frac{\log(2+\mathrm{diam}(\mathrm{supp}(\theta)))}{n}\right).
\]
Note that $m$ does not appear in the subscript of the big-$O$. This
is so since $m\le d$ and $d$ is considered as global.

By the definition of the distribution on components, for $k\ge1$
\begin{equation}
H(\theta,\mathcal{D}_{n+k}|\mathcal{D}_{n})=\mathbb{E}_{i=n}(H(\theta_{x,i},\mathcal{D}_{i+k})).\label{eq:cond ent as avg of ent of comp}
\end{equation}
Hence we have the bound,
\[
\frac{1}{k}H(\theta,\mathcal{D}_{n+k}|\mathcal{D}_{n})\leq m.
\]

Scale-$n$ entropy transforms nicely under similarity maps. Indeed,
for any similarity $f:\mathbb{R}^{m}\rightarrow\mathbb{R}^{m}$ with
Lipschitz constant $\mathrm{Lip}(f)$,
\begin{equation}
H(f\theta,\mathcal{D}_{n})=H(\theta,\mathcal{D}_{n+\log\mathrm{Lip}(f)})+O(1).\label{eq:ent under sim}
\end{equation}
Recalling the notation $T_{a},S_{c}$ for translation and scaling
from Section \ref{subsec:Spaces-of-affine maps}, it follows from
(\ref{eq:ent under sim}) that
\begin{align}
H(T_{a}\theta,\mathcal{D}_{n}) & =H(\theta,\mathcal{D}_{n})+O(1)\text{ for }a\in\mathbb{R}^{m},\label{eq:entropy-under-translation}\\
H(S_{c}\theta,\mathcal{D}_{n}) & =H(\theta,\mathcal{D}_{n+\log c})+O(1)\text{ for }c>0.\nonumber 
\end{align}
By (\ref{eq:entropy-under-translation}) and the concavity of entropy,
it follows that for any $\sigma\in\mathcal{M}(\mathbb{R}^{m})$
\begin{equation}
H(\sigma*\theta,\mathcal{D}_{n})\geq H(\theta,\mathcal{D}_{n})-O(1),\label{eq:conv does not dec ent}
\end{equation}
where $\sigma*\theta$ denotes the convolution of $\sigma$ and $\theta$.

Lipschitz maps, with bounded Lipschitz constant, do not increase entropy
too much. Indeed, if $1\le l\le d$ and $f:\mathrm{supp}(\theta)\rightarrow\mathbb{R}^{l}$
is Lipschitz,
\begin{equation}
H(f\theta,\mathcal{D}_{n})\le H(\theta,\mathcal{D}_{n})+O_{\mathrm{Lip}(f)}(1).\label{eq:ent under lip map}
\end{equation}
Consequently, if $f$ is bi-Lipschitz onto its image,
\begin{equation}
H(f\theta,\mathcal{D}_{n})=H(\theta,\mathcal{D}_{n})+O_{\mathrm{Lip}(f)}(1).\label{eq:ent under bi-lip map}
\end{equation}
From (\ref{eq:ent under lip map}) it also follows that for any $c>0$,
\begin{equation}
H(A\theta,\mathcal{D}_{n})\ge H(\theta,\mathcal{D}_{n})-O_{c}(1)\text{ for all }A\in\mathrm{GL}(m,\mathbb{R})\text{ with }\alpha_{m}(A)\ge c.\label{eq:ent under expand lin map}
\end{equation}

Given $1\le l\le m$ and a compact subset $K$ of $\mathrm{A}_{m,l}$,
it follows from (\ref{eq:rep of psi via pi_V}) and (\ref{eq:ent under bi-lip map}) that
\begin{equation}
H(\psi\theta,\mathcal{D}_{n})=H(\pi_{(\ker A_{\psi})^{\perp}}\theta,\mathcal{D}_{n})+O_{K}(1)\text{ for all }\psi\in K.\label{eq:ent of psi comp to ent of proj}
\end{equation}

The entropy of images is nearly continuous in the map. Indeed, if
$1\le l\le d$ and $f,g:\mathrm{supp}(\theta)\rightarrow\mathbb{R}^{l}$
satisfy $|f(x)-g(x)|\le2^{-n}$, then
\begin{equation}
|H(f\theta,\mathcal{D}_{n})-H(g\theta,\mathcal{D}_{n})|=O(1).\label{eq:ent is cont in map}
\end{equation}

Scale-$n$ entropy is insensitive to coordinate changes. Given $1\le l<m$
and $V\in\mathrm{Gr}_{l}(m)$, it is easy to see that the partitions
$\mathcal{D}_{n}^{m}$ and $\pi_{V}^{-1}\mathcal{D}_{n}^{l}\vee\pi_{V^{\perp}}^{-1}\mathcal{D}_{n}^{m-l}$
are $O(1)$-commensurable. Consequently,
\begin{eqnarray}
H(\theta,\mathcal{D}_{n}) & = & H(\theta,\pi_{V}^{-1}\mathcal{D}_{n}\vee\pi_{V^{\perp}}^{-1}\mathcal{D}_{n})+O(1)\label{eq:ent =000026 coordinate change}\\
 & = & H(\pi_{V}\theta,\mathcal{D}_{n})+H(\theta,\mathcal{D}_{n}\mid\pi_{V}^{-1}\mathcal{D}_{n})+O(1).\nonumber 
\end{eqnarray}
Also note that $\pi_{V}^{-1}\mathcal{D}_{n}^{l}$ and $P_{V}^{-1}\mathcal{D}_{n}^{m}$
are $O(1)$-commensurable, and so
\[
H(P_{V}\theta,\mathcal{D}_{n})=H(\pi_{V}\theta,\mathcal{D}_{n})+O(1).
\]

The following lemma express entropy in terms of the contribution of
different ``scales''. Its proof is identical to the proof of \cite[Lemma 3.4]{Ho1}
and is therefore omitted.
\begin{lem}
\label{lem:multiscale-entropy-formula}Let $1\le m\le d$, $\theta\in\mathcal{M}(\mathbb{R}^{m})$,
$n\ge k\ge1$, $l\ge0$ and $C\ge1$ be given. Suppose that $\mathrm{diam}(\mathrm{supp}(\theta))\le C2^{-l}$,
then
\[
\frac{1}{n}H(\theta,\mathcal{D}_{l+n})=\mathbb{E}_{l\le i\le l+n}\left(\frac{1}{k}H\left(\theta_{x,i},\mathcal{D}_{i+k}\right)\right)+O_{C}\left(\frac{k}{n}\right).
\]
\end{lem}

Given $\theta\in\mathcal{M}(\mathbb{R}^{m})$, the upper and lower
entropy dimensions of $\theta$ are defined as
\[
\overline{\dim}_{e}\theta:=\underset{n\rightarrow\infty}{\limsup}\:\frac{1}{n}H(\theta,\mathcal{D}_{n})\text{ and }\underline{\dim}_{e}\theta:=\underset{n\rightarrow\infty}{\liminf}\:\frac{1}{n}H(\theta,\mathcal{D}_{n}).
\]
If $\overline{\dim}_{e}\theta=\underline{\dim}_{e}\theta$ we denote
the common value by $\dim_{e}\theta$, which is called the entropy
dimension of $\theta$. The proof of the following lemma can be found
e.g. in \cite{FLR}.
\begin{lem}
\label{lem:dim_e=00003Ddim}If $\theta\in\mathcal{M}(\mathbb{R}^{m})$
is exact dimensional then $\dim_{e}\theta$ exists and is equal to
$\dim\theta$.
\end{lem}

\subsection{\label{subsec:Symbolic-notations}Symbolic notations}

Recall that $\Lambda$ is the finite index set of the IFS $\Phi$ (see Section \ref{subsec:The-setup}). Write $\Lambda^{\mathbb{N}}$ for the set of sequences $\omega=(\omega_{k})_{k\ge0}$
with $\omega_{k}\in\Lambda$ for each $k\ge0$. We equip $\Lambda^{\mathbb{N}}$
with the product topology, where each copy of $\Lambda$ is equipped
with the discrete topology. 

For $n\ge0$ and $\omega\in\Lambda^{\mathbb{N}}$ write $\omega|_{n}$
for the prefix of $\omega$ of length $n$. That is, $\omega|_{n}:=\omega_{0}...\omega_{n-1}$
and $\omega|_{0}$ is the empty word $\emptyset$. Given a word $u\in\Lambda^{n}$
denote by $[u]$ the cylinder set in $\Lambda^{\mathbb{N}}$ corresponding
to $u$, i.e.
\[
[u]:=\left\{ \omega\in\Lambda^{\mathbb{N}}\::\:\omega|_{n}=u\right\} .
\]
Let $\mathcal{P}_{n}$ be the partition of $\Lambda^{\mathbb{N}}$
into $n$-cylinders, that is $\mathcal{P}_{n}:=\{[u]\}_{u\in\Lambda^{n}}$.
Given $\xi\in\mathcal{M}(\Lambda^{\mathbb{N}})$ and $\mathcal{U}\subset\Lambda^{*}$
we often write $\xi(\mathcal{U})$ in place of $\xi\left(\cup_{u\in\mathcal{U}}[u]\right)$.

Denote by $\beta$ the Bernoulli measure on $\Lambda^{\mathbb{N}}$
corresponding to $p$. That is, $\beta$ is the unique element in
$\mathcal{M}(\Lambda^{\mathbb{N}})$ so that $\beta([u])=p_{u}$ for
each $u\in\Lambda^{*}$.

We say that a finite set of words $\mathcal{U}\subset\Lambda^{*}$
is a minimal cut-set for $\Lambda^{*}$ if for every $\omega\in\Lambda^{\mathbb{N}}$
there exists a unique $u\in\mathcal{U}$ with $\omega\in[u]$. Note
that in this case, by iterating the relation $\mu=\sum_{i\in\Lambda}p_{i}\cdot\varphi_{i}\mu$, we get
\begin{equation}
\mu=\sum_{u\in\mathcal{U}}p_{u}\cdot\varphi_{u}\mu.\label{eq:self-sim rel for cut set}
\end{equation}
We sometimes refer to measures of the form $\varphi_{u}\mu$ as cylinder
measures of $\mu$.

We shall need to consider minimal cut-sets for $\Lambda^{*}$,
which are defined in terms of the singular values of elements in $\mathbf{S}_{\Phi}^{\mathrm{L}}$.
For $1\le m\le d$ and $n\ge1$ we define
\[
\Psi_{m}(n):=\left\{ i_{0}...i_{k}\in\Lambda^{*}\::\:\alpha_{m}(A_{i_{0}...i_{k}})\le2^{-n}<\alpha_{m}(A_{i_{0}...i_{k-1}})\right\} .
\]
By Lemma \ref{lem:bounds on sing vals of prod} and since $\Vert A_{i}\Vert_{op}<1$
for $i\in\Lambda$, it follows that $\Psi_{m}(n)$ is indeed a minimal
cut-set for $\Lambda^{*}$. By Lemma \ref{lem:bounds on sing vals of prod}
it also follows that,
\begin{equation}
\min_{i\in\Lambda}\alpha_{d}(A_{i})\le\alpha_{m}(A_{u})2^{n}\le1\text{ for all }u\in\Psi_{m}(n).\label{eq:words in Phi_m are Theta(2^n)}
\end{equation}
Given $\omega\in\Lambda^{\mathbb{N}}$ we denote by $\Psi_{m}(n;\omega)$
the unique $u\in\Psi_{m}(n)$ with $\omega\in[u]$.

It will often be useful to choose words from $\Psi_{m}(n)$ at random.
For this, let $\mathbf{I}_{m}(n)$ be the random word with
\[
\mathbb{P}\{\mathbf{I}_{m}(n)=u\}=\begin{cases}
p_{u} & \text{ if }u\in\Psi_{m}(n)\\
0 & \text{ otherwise}
\end{cases}.
\]
Note that,
\begin{equation}\label{eq:conn bet I_m(n) & beta}
\mathbb{P}\{\mathbf{I}_{m}(n)\in\mathcal{U}\}=\beta(\mathcal{U})\text{ for all }\mathcal{U}\subset\Psi_{m}(n).
\end{equation}
Additionally, let $\mathbf{U}(n)$ be the random word with
\[
\mathbb{P}\{\mathbf{U}(n)=u\}=\begin{cases}
p_{u} & \text{ if }u\in\Lambda^{n}\\
0 & \text{ otherwise}
\end{cases}.
\]
From (\ref{eq:self-sim rel for cut set}) it follows that
\begin{equation}
\mu=\mathbb{E}(\varphi_{\mathbf{I}_{m}(n)}\mu)=\mathbb{E}(\varphi_{\mathbf{U}(n)}\mu).\label{eq:decomp of mu according to cut sets}
\end{equation}

We shall often randomize $n$ in the same way as we do in the case
of components. For instance, given $1\le m\le d$, $n_{2}\ge n_{1}\ge1$
and $F:\mathcal{M}(\mathbb{R}^{d})\rightarrow\mathbb{R}$, we write
\[
\mathbb{E}_{n_{1}\leq i\leq n_{2}}(F(\varphi_{\mathbf{I}_{m}(i)}\mu)):=\frac{1}{n_{2}-n_{1}+1}\sum_{i=n_{1}}^{n_{2}}\mathbb{E}(F(\varphi_{\mathbf{I}_{m}(i)}\mu)).
\]

The following lemma shows that events holding with high probability
over many scales with respect to $\{\mathbf{U}(n)\}_{n\ge1}$, also
hold with high probability with respect to $\{\mathbf{I}_{m}(n)\}_{n\ge1}$.
\begin{lem}
\label{lem:ac of stopping times}There exist constants $C_{1},C_{2},N\ge1$,
which depend only on $\{A_{i}\}_{i\in\Lambda}$, so that for every
$1\le m\le d$ and $n\ge N$
\[
\mathbb{E}_{1\le i\le n}\left(\delta_{\mathbf{I}_{m}(i)}\right)\ll\mathbb{E}_{1\le i\le C_{1}n}\left(\delta_{\mathbf{U}(i)}\right),
\]
with Radon-Nykodym derivative bounded by $C_{2}$.
\end{lem}

\begin{proof}
By Lemma \ref{lem:bounds on sing vals of prod} it follows that for
all $1\le m\le d$ and $i_{1}...i_{n}=u\in\Lambda^{*}$,
\[
\alpha_{m}(A_{i_{1}...i_{n-1}})\alpha_{d}(A_{i_{n}})\le\alpha_{m}(A_{u})\le\prod_{j=1}^{n}\Vert A_{i_{j}}\Vert_{op}.
\]
The lemma now follows from these inequalities by an argument identical
to the one given in \cite[Proposition 2.8]{HR}.
\end{proof}

\subsection{\label{subsec:Lyapunov-exponents-and Fur meas}Lyapunov exponents
and Furstenberg measures}

Let $0>\chi_{1}\ge...\ge\chi_{d}$ be the Lyapunov exponents corresponding
to $\{A_{i}\}_{i\in\Lambda}$ and $p$. That is,
\[
\chi_{m}=\underset{n\rightarrow\infty}{\lim}\:\frac{1}{n}\log\alpha_{m}(A_{\omega|_{n}})\text{ for }1\le m\le d\text{ and }\beta\text{-a.e. }\omega.
\]
By (\ref{eq:m-ired and m-prox assump}) and \cite[Theorem IV.1.2]{BL},
it follows that $\chi_{m}>\chi_{m+1}$ for all $1\le m<d$.

Assumption (\ref{eq:m-ired and m-prox assump}) also implies the uniqueness
of various stationary measures. Recall from Section \ref{subsec:spaces of alt forms}
that for $1\le m\le d$ the space $\mathrm{Gr}_{m}(d)$ can be identified
with a subset of $\mathrm{P}(\wedge^{m}\mathbb{R}^{d})$ via the Plücker
embedding $\iota_{m}$. Moreover, note that $\iota_{m}\circ A(V)=\wedge^{m}A\circ\iota_{m}(V)$
for each $A\in\mathrm{GL}(d,\mathbb{R})$ and $V\in\mathrm{Gr}_{m}(d)$.
Thus, by \cite[Theorem IV.1.2]{BL} and (\ref{eq:m-ired and m-prox assump}),
it follows that for each $0\le m\le d$ there exist a unique $\nu_{m}\in\mathcal{M}(\mathrm{Gr}_{m}(d))$
with $\nu_{m}=\sum_{i\in\Lambda}p_{i}\cdot A_{i}\nu_{m}$, where $A_{i}\nu_{m}$
is the pushforward of $\nu_{m}$ via the map $V\rightarrow A_{i}(V)$.
Similarly, by Remark \ref{rem:semi of trans also SI =000026 prox},
for each $0\le m\le d$ there exist a unique $\nu_{m}^{*}\in\mathcal{M}(\mathrm{Gr}_{m}(d))$
with $\nu_{m}^{*}=\sum_{i\in\Lambda}p_{i}\cdot A_{i}^{*}\nu_{m}^{*}$.

Additionally, by \cite[Proposition IV.4.1]{BL} and (\ref{eq:m-ired and m-prox assump}),
there exist a unique $\nu\in\mathcal{M}(\mathrm{F}(d))$ with $\nu=\sum_{i\in\Lambda}p_{i}\cdot A_{i}\nu$,
where $A_{i}\nu$ is the pushforward of $\nu$ via the map taking
$(V_{j})_{j=0}^{d}\in\mathrm{F}(d)$ to $(A_{i}(V_{j}))_{j=0}^{d}$.
Similarly, by Remark \ref{rem:semi of trans also SI =000026 prox},
there exists a unique $\nu^{*}\in\mathcal{M}(\mathrm{F}(d))$ with
$\nu^{*}=\sum_{i\in\Lambda}p_{i}\cdot A_{i}^{*}\nu^{*}$. The measures
$\nu,\nu_{0},...,\nu_{d}$ and $\nu^{*},\nu_{0}^{*},...,\nu_{d}^{*}$
are called the Furstenberg measures corresponding to $\sum_{i\in\Lambda}p_{i}\delta_{A_{i}}$
and $\sum_{i\in\Lambda}p_{i}\delta_{A_{i}^{*}}$ respectively.

Given $0\le m\le d$, note that by the uniqueness of $\nu_{m}^{*}$
\begin{equation}
\nu_{m}^{*}(E)=\nu^{*}\left\{ (V_{i})_{i=0}^{d}\::\:V_{m}\in E\right\} \text{ for every Borel set }E\in\mathrm{Gr}_{m}(d),\label{eq:fur =00003D proj of fur}
\end{equation}
and similarly for the measures $\nu_{m}$ and $\nu$.

The following useful lemma will be used several times below. Recall the definition of $\kappa(V,W)$ from Section \ref{subsec:Grassmannians-and-the flag space}.
\begin{lem}
\label{lem:nu=00007Bkappa<del=00007D<eps}Let $0\le m\le d$ and $\epsilon>0$
be given. Then there exists $\delta>0$ so that,
\[
\nu_{m}\left\{ V\::\:\kappa(V,W)\le\delta\right\} <\epsilon\text{ for all }W\in\mathrm{Gr}_{d-m}(d).
\]
In particular,
\[
\nu_{m}\{V\::\:V\cap W\ne\{0\}\}=0\text{ for all }W\in\mathrm{Gr}_{d-m}(d).
\]
\end{lem}

\begin{proof}
The lemma holds trivially when $m=0$ or $m=d$. Hence assume that
$1\le m<d$.

Let $W\in\mathrm{Gr}_{d-m}(d)$ and $0\ne\gamma\in\wedge^{d-m}\mathbb{R}^{d}$
be with $\iota_{d-m}(W)=\gamma\mathbb{R}$. Let $T:\wedge^{m}\mathbb{R}^{d}\rightarrow\wedge^{d}\mathbb{R}^{d}$
be with $T(\zeta)=\zeta\wedge\gamma$ for $\zeta\in\wedge^{m}\mathbb{R}^{d}$,
where $\zeta\wedge\gamma$ is the wedge product of $\zeta$ and $\gamma$.
Since $\gamma\ne0$ the linear map $T$ is nonzero. Additionally,
from $\nu_{m}=\sum_{i\in\Lambda}p_{i}\cdot A_{i}\nu_{m}$ it follows
easily that
\[
\iota_{m}\nu_{m}=\sum_{i\in\Lambda}p_{i}\cdot\wedge^{m}A_{i}\iota_{m}\nu_{m}.
\]
From these facts, since $\mathbf{S}_{\Phi}^{\mathrm{L}}$ is $m$-strongly
irreducible, and by \cite[Proposition III.2.3]{BL}, it follows that
\[
\iota_{m}\nu_{m}\{\zeta\mathbb{R}\::\:\zeta\mathbb{R}\subset\ker T\}=0.
\]
Thus, since for $V\in\mathrm{Gr}_{m}(d)$ we have $\iota_{m}(V)\subset\ker T$
if and only if $V\cap W\ne\{0\}$, we obtain
\begin{equation}
\nu_{m}\{V\::\:V\cap W\ne\{0\}\}=0.\label{eq:nu_m=00007Bint not triv=00007D=00003D0}
\end{equation}

Next assume by contradiction that there does not exist $\delta>0$
as in the statement of the lemma. Then for each $n\ge1$ there exists
$W_{n}\in\mathrm{Gr}_{d-m}(d)$ so that
\[
\nu_{m}\left\{ V\::\:\kappa(V,W_{n})\le1/n\right\} \ge\epsilon.
\]
Since $\mathrm{Gr}_{d-m}(d)$ is compact, and by moving to a subsequence
without changing the notation, we may assume that there exists $W\in\mathrm{Gr}_{d-m}(d)$
so that $d_{\mathrm{Gr}_{d-m}}(W,W_{n})\overset{n}{\rightarrow}0$.
By the definition $\kappa$ it now follows easily that
\[
\nu_{m}\{V\::\:V\cap W\ne\{0\}\}\ge\epsilon,
\]
which contradicts (\ref{eq:nu_m=00007Bint not triv=00007D=00003D0})
and completes the proof of the lemma.
\end{proof}

\subsection{\label{subsec:Coding and Furstenberg maps}The coding map and the
Furstenberg boundary maps}

Let $\Pi:\Lambda^{\mathbb{N}}\rightarrow K_{\Phi}$ be the coding
map corresponding to $\Phi$, that is
\[
\Pi\omega:=\underset{n\rightarrow\infty}{\lim}\:\varphi_{\omega|_{n}}(0)\text{ for }\omega\in\Lambda^{\mathbb{N}}.
\]
Given $u\in\Lambda^{*}$
\begin{equation}
\Pi(u\omega)=\varphi_{u}(\Pi\omega)\text{ for }\omega\in\Lambda^{\mathbb{N}},\label{eq:equivar of coding map}
\end{equation}
where $u\omega$ denotes the concatenation of $u$ with $\omega$.
It is easy to verify that $\Pi\beta=\mu$.

Given $n\ge1$ let $\Pi_{n}:\Lambda^{\mathbb{N}}\rightarrow\mathrm{A}_{d,d}$
be with $\Pi_{n}(\omega):=\varphi_{\omega|_{n}}$ for $\omega\in\Lambda^{\mathbb{N}}$.
If we identify $\Pi$ with the map from $\Lambda^{\mathbb{N}}$ to
$\mathrm{A}_{d,d}^{\mathrm{vec}}$ taking $\omega\in\Lambda^{\mathbb{N}}$
to the constant function $\Pi\omega$, then $\{\Pi_{n}\}_{n\ge1}$
converges pointwise to $\Pi$ with respect to the norm metric on $\mathrm{A}_{d,d}^{\mathrm{vec}}$.

Given $A\in\mathrm{GL}(d,\mathbb{R})$ and $0\le m\le d$, recall
the notation $L_{m}(A)$ from Section \ref{subsec:Singular-values-and SVD}.
Also recall that for $1\le m\le d$ the space $\mathrm{Gr}_{m}(d)$
can be identified with a subset of $\mathrm{P}(\wedge^{m}\mathbb{R}^{d})$
via $\iota_{m}$. Thus, by (\ref{eq:m-ired and m-prox assump}) and
\cite[Lemma 2.17 and Proposition 4.7]{BQ} (see also \cite[Proposition III.3.2]{BL}),
it follows that for each $0\le m\le d$ there exists a Borel map $L_{m}:\Lambda^{\mathbb{N}}\rightarrow\mathrm{Gr}_{m}(d)$
so that,
\begin{enumerate}
\item $\{A_{\omega_{0}}...A_{\omega_{n}}\nu_{m}\}_{n\ge1}$ converges weakly
to $\delta_{L_{m}(\omega)}$ for $\beta$-a.e. $\omega$;
\item \label{enu:def of L_m via L_m of matrices}$L_{m}(\omega)=\underset{n\rightarrow\infty}{\lim}\:L_{m}(A_{\omega_{0}}...A_{\omega_{n}})$
for $\beta$-a.e. $\omega$;
\item $L_{m}\beta=\nu_{m}$.
\end{enumerate}
The maps $L_{0},...,L_{d}$ are called the Furstenberg boundary maps.
By property (\ref{enu:def of L_m via L_m of matrices}) and by altering the definition of the boundary maps on a set of $\beta$-measure
zero, we may clearly assume that
\[
L_{0}(\omega)\subset...\subset L_{d}(\omega)\text{ for all }\omega\in\Lambda^{\mathbb{N}}.
\]
It is easy to verify that for all $0\le m\le d$ and $u\in\Lambda^{*}$,
\begin{equation}
L_{m}(u\omega)=A_{u}(L_{m}(\omega))\text{ for }\beta\text{-a.e. }\omega.\label{eq:equivar of bd maps}
\end{equation}

With the aid of the boundary maps we can prove the following lemma.
\begin{lem}
\label{lem:P=00007Bkappa<delta=00007D<epsilon}Let $0\le m\le d$
and $1\le k\le d$ be given. Then for every $\epsilon>0$ there exists
$\delta>0$ and $N\ge1$ so that,
\begin{equation}
\mathbb{P}_{i=n}\left\{ \kappa\left(L_{m}(A_{\mathbf{I}_{k}(i)}),V\right)\le\delta\right\} <\epsilon\text{ for all }V\in\mathrm{Gr}_{d-m}(d)\text{ and }n\ge N.\label{eq:P=00007Bkappa<delta=00007D<epsilon all V all n}
\end{equation}
\end{lem}

\begin{rem*}
When $1\le m<d$ we implicitly assume that $\alpha_{m}(A_{\mathbf{I}_{k}(i)})>\alpha_{m+1}(A_{\mathbf{I}_{k}(i)})$
in the the definition of the event $\left\{ \kappa\left(L_{m}(A_{\mathbf{I}_{k}(i)}),V\right)\le\delta\right\} $.
This is necessary in order for $L_{m}(A_{\mathbf{I}_{k}(i)})$ to
be well defined. We make similar implicit assumptions below. Note
that since $\chi_{m}>\chi_{m+1}$,
\[
\underset{n\rightarrow\infty}{\lim}\mathbb{P}_{i=n}\left\{ \alpha_{m}(A_{\mathbf{I}_{k}(i)})>\alpha_{m+1}(A_{\mathbf{I}_{k}(i)})\right\} =1.
\]
\end{rem*}

\begin{proof}
We have
\[
L_{m}(\omega)=\underset{n\rightarrow\infty}{\lim}\:L_{m}(A_{\Psi_{k}(n;\omega)})\text{ for }\beta\text{-a.e. }\omega.
\]
 This together with $L_{m}\beta=\nu_{m}$ implies that $\{L_{m}(A_{\mathbf{I}_{k}(n)})\}_{n\ge1}$
converges in distribution to $\nu_{m}$. In particular,
\begin{equation}
\underset{n\rightarrow\infty}{\limsup}\:\mathbb{P}_{i=n}\left\{ L_{m}(A_{\mathbf{I}_{k}(i)})\in F\right\} \le\nu_{m}(F)\text{ for every closed }F\subset\mathrm{Gr}_{m}(d).\label{eq:by conv in dist for cl sets}
\end{equation}

Let $\epsilon>0$ be given. By Lemma \ref{lem:nu=00007Bkappa<del=00007D<eps},
there exists $\delta>0$ so that
\[
\nu_{m}\{W\::\:\kappa(W,V)\le2\delta\}<\epsilon/2\text{ for all }V\in\mathrm{Gr}_{d-m}(d).
\]
Hence by (\ref{eq:by conv in dist for cl sets}), for each $V\in\mathrm{Gr}_{d-m}(d)$
there exists $N_{V}\ge1$ so that for all $n\ge N_{V}$
\[
\mathbb{P}_{i=n}\left\{ \kappa(L_{m}(A_{\mathbf{I}_{k}(i)}),V)\le2\delta\right\} <\epsilon.
\]
By the compactness of $\mathrm{Gr}_{d-m}(d)$ it now follows easily
that there exists $N\ge1$ so that (\ref{eq:P=00007Bkappa<delta=00007D<epsilon all V all n})
holds, which completes the proof of the lemma.
\end{proof}

\subsection{\label{subsec:Disintegrations}Disintegration of measures}

In this subsection we give the necessary facts and notations regarding disintegration
of measures.
Let $X$ be a complete separable metric space, and denote its Borel $\sigma$-algebra
by $\mathcal{B}_{X}$. Let $\mathcal{A}$ be a sub-$\sigma$-algebra
of $\mathcal{B}_{X}$, and suppose that $\mathcal{A}$ is countably
generated. Given $x\in X$ write $[x]_{\mathcal{A}}$ for the intersection
of all $A\in\mathcal{A}$ with $x\in A$. Note that $[x]_{\mathcal{A}}\in\mathcal{A}$
since $\mathcal{A}$ is countably generated.

Let $\theta\in\mathcal{M}(X)$ be given. By \cite[Theorem 5.14]{EiWa}
it follows that there exists a family $\{\theta_{x}^{\mathcal{A}}\}_{x\in X}\subset\mathcal{M}(X)$
with the following properties:
\begin{enumerate}
\item The map taking $x\in X$ to $\int f\:d\theta_{x}^{\mathcal{A}}$ is
$\mathcal{A}$-measurable for every Borel function $f:X\rightarrow[0,\infty]$;
\item For every $f\in L^{1}(\theta)$ and $A\in\mathcal{A}$,
\[
\int_{A}f\:d\theta=\int_{A}\int f\:d\theta_{x}^{\mathcal{A}}\:d\theta(x).
\]
In particular
$\theta=\int\theta_{x}^{\mathcal{A}}\:d\theta(x)$;
\item There exists $A\in\mathcal{A}$ so that $\theta(A)=1$ and $\theta_{x}^{\mathcal{A}}([x]_{\mathcal{A}})=1$
for all $x\in A$.
\end{enumerate}
If $\{\tilde{\theta}_{x}^{\mathcal{A}}\}_{x\in X}\subset\mathcal{M}(X)$
is another family with these properties, then $\theta_{x}^{\mathcal{A}}=\tilde{\theta}_{x}^{\mathcal{A}}$
for $\theta$-a.e. $x$. The family $\{\theta_{x}^{\mathcal{A}}\}_{x\in X}$
is called the disintegration of $\theta$ with respect to $\mathcal{A}$.

Note that by the first and third defining properties,
\begin{equation}
\theta_{y}^{\mathcal{A}}=\theta_{x}^{\mathcal{A}}\text{ for }\theta\text{-a.e. }x\text{ and }\theta_{x}^{\mathcal{A}}\text{-a.e. }y.\label{eq:theta_y=00003Dtheta_x theta-a.e. x}
\end{equation}

It follows directly from the definitions that for a countable partition
$\mathcal{D}\subset\mathcal{B}_{X}$
\begin{equation}
H(\theta,\mathcal{D}\mid\mathcal{A})=\int H(\theta_{x}^{\mathcal{A}},\mathcal{D})\:d\theta(x),\label{eq:cond ent via disinteg}
\end{equation}
where recall from Section \ref{subsec:Entropy} that $H(\theta,\mathcal{D}\mid\mathcal{A})$
is the conditional entropy of $\theta$ with respect to $\mathcal{D}$
given $\mathcal{A}$. Also note that
\begin{equation}
H(\theta_{x}^{\mathcal{A}},\mathcal{D}\mid\mathcal{A})=H(\theta_{x}^{\mathcal{A}},\mathcal{D})\text{ for }\theta\text{-a.e. }x.\label{eq:cond ent =00003D ent for slices}
\end{equation}

Let $\mathcal{A}'$ be another countably generated sub-$\sigma$-algebra
of $\mathcal{B}_{X}$, and suppose that $\mathcal{A}'\subset\mathcal{A}$.
By \cite[Proposition 5.20]{EiWa},
\begin{equation}
(\theta_{x}^{\mathcal{A}'})_{y}^{\mathcal{A}}=\theta_{y}^{\mathcal{A}}\text{ for }\theta\text{-a.e. }x\text{ and }\theta_{x}^{\mathcal{A}'}\text{-a.e. }y.\label{eq:cond of cond =00003D cond of orig}
\end{equation}

Let $Y$ be another complete separable metric space and let $f:Y\rightarrow X$
be Borel measurable. Suppose that there exists $\xi\in\mathcal{M}(Y)$
so that $f\xi=\theta$. Then by \cite[Corollary 5.24]{EiWa},
\begin{equation}
f(\xi_{y}^{f^{-1}\mathcal{A}})=\theta_{f(y)}^{\mathcal{A}}\text{ for }\xi\text{-a.e. }y\in Y.\label{eq:push oh slice =00003D slice of push}
\end{equation}

Denote the Borel $\sigma$-algebra of $\mathbb{R}^{d}$ by $\mathcal{B}_{\mathbb{R}^{d}}$,
and let $V$ be a linear subspace of $\mathbb{R}^{d}$. Given $\theta\in\mathcal{M}(\mathbb{R}^{d})$
we write $\{\theta_{x}^{V}\}_{x\in\mathbb{R}^{d}}$ for the disintegration
of $\theta$ with respect to $P_{V}^{-1}\mathcal{B}_{\mathbb{R}^{d}}$.
Note that,
\[
\mathrm{supp}(\theta_{x}^{V})\subset x+V^{\perp}\text{ for }\theta\text{-a.e. }x.
\]
For $\varphi\in\mathrm{A}_{d,d}$ we have $\varphi^{-1}P_{V}^{-1}\mathcal{B}_{\mathbb{R}^{d}}=P_{A_{\varphi}^{*}(V)}^{-1}\mathcal{B}_{\mathbb{R}^{d}}$,
where recall that $A_{\varphi}$ is the linear part of $\varphi$.
Thus by (\ref{eq:push oh slice =00003D slice of push}),
\begin{equation}
\varphi(\theta_{x}^{A_{\varphi}^{*}(V)})=(\varphi\theta)_{\varphi(x)}^{V}\text{ for }\theta\text{-a.e. }x.\label{eq:push of slive via affine =00003D slice of push via affine}
\end{equation}

Given $\xi\in\mathcal{M}(\Lambda^{\mathbb{N}})$ we
write $\{\xi_{\omega}^{V}\}_{\omega\in\Lambda^{\mathbb{N}}}$ for
the disintegration of $\xi$ with respect to $\Pi^{-1}(P_{V}^{-1}\mathcal{B}_{\mathbb{R}^{d}})$.
By (\ref{eq:push oh slice =00003D slice of push}),
\[
\Pi(\xi_{\omega}^{V})=(\Pi\xi)_{\Pi\omega}^{V}\text{ for }\xi\text{-a.e. }\omega.
\]
Additionally, from (\ref{eq:cond of cond =00003D cond of orig}) it follows that
for a linear subspace $W\subset V$,
\begin{equation}
\xi_{\omega}^{W}=\int\xi_{\eta}^{V}\:d\xi_{\omega}^{W}(\eta)\text{ for }\xi\text{-a.e. }\omega.\label{eq:cond of cond =00003D cond of orig on sym space}
\end{equation}

\subsection{\label{subsec:Ledrappier-Young-formula}Ledrappier-Young formula
for self-affine measures}

Recall that $\mathcal{P}_{1}=\{[i]\}_{i\in\Lambda}$, and for $0\le m\le d$
set
\[
\mathrm{H}_{m}:=\int H(\beta,\mathcal{P}_{1}\mid\Pi^{-1}P_{V}^{-1}\mathcal{B}_{\mathbb{R}^{d}})\:d\nu_{m}^{*}(V).
\]
By (\ref{eq:cond ent via disinteg}),
\[
\mathrm{H}_{m}=\int\int H(\beta_{\omega}^{V},\mathcal{P}_{1})\:d\beta(\omega)\:d\nu_{m}^{*}(V).
\]
For $1\le m\le d$ write
\[
\Delta_{m}:=\frac{\mathrm{H}_{m}-\mathrm{H}_{m-1}}{\chi_{m}},
\]
and for $1\le l\le m\le d$ set $\Sigma_{l}^{m}:=\sum_{i=l}^{m}\Delta_{i}$.
For later use, it will also be convenient to write $\Sigma_{1}^{0}$ in
place of $0$.

The following important result, which is crucial for the present work,
was obtained by Feng \cite[Theorem 1.4]{Fe}. Under our assumption
(\ref{eq:m-ired and m-prox assump}), similar versions of it were
first established by Bárány and Käenmäki \cite{BK}.
\begin{thm}
\label{thm:LY formula for SA}For $\nu^{*}$-a.e. $(V_{i})_{i=0}^{d}$,
$\mu$-a.e. $x$, and each $0\le l<m\le d$, the measure $P_{V_{m}}\mu_{x}^{V_{l}}$
is exact dimensional with $\dim P_{V_{m}}\mu_{x}^{V_{l}}=\Sigma_{l+1}^{m}$.
\end{thm}

\begin{rem}
\label{rem:0<=00003DDelta_m<=00003D1}Note that for $(V_{i})_{i=0}^{d}\in\mathrm{F}(d)$,
$1\le m\le d$, and $\mu$-a.e. $x$, the measure $P_{V_{m}}\mu_{x}^{V_{m-1}}$
is supported on an affine $1$-dimensional subspace. Thus, by Theorem
\ref{thm:LY formula for SA}, it follows that $0\le\Delta_{m}\le1$
for each $1\le m\le d$. Moreover, we have $\dim\mu=d$ if and only
if $\Delta_{m}=1$ for each $1\le m\le d$, and if $\dim\mu<d$ then
there exists a unique $1\le m\le d$ so that $\Sigma_{1}^{m-1}=m-1$
and $\Delta_{m}<1$.
\end{rem}

\section{\label{sec:Entropy-estimates}Entropy estimates}

In this section we establish lower bounds on the entropy of projections
of components of cylinders of $\mu$. By using these, we also obtain
lower bounds on the entropy of projections of $\mu$ and its components.
We start with the following definition.
\begin{defn}
\label{def:def of mathcal(Z)}For $\Sigma,\epsilon>0$, $k\ge1$,
$i\ge0$ and a linear subspace $V$ of $\mathbb{R}^{d}$, denote by
$\mathcal{Z}(\Sigma,\epsilon,k,i,V)$ the set of all words $u\in\Lambda^{*}$
so that for $\theta:=\varphi_{u}\mu$ we have
\[
\mathbb{P}_{j=i}\left\{ \frac{1}{k}H\left(\pi_{V}(\theta_{x,j}),\mathcal{D}_{j+k}\right)>\Sigma-\epsilon\right\} >1-\epsilon.
\]
\end{defn}

For $1\le m\le d$ and $n\ge1$ recall the minimal cut-set $\Psi_{m}(n)$
and random word $\mathbf{I}_{m}(n)$ defined in Section \ref{subsec:Symbolic-notations}.
Only for this section, it will be convenient to define these objects
also for $m=0$. Thus we define $\Psi_{0}(n):=\Psi_{1}(n)$ and $\mathbf{I}_{0}(n):=\mathbf{I}_{1}(n)$
for all $n\ge1$. The following proposition is the main result of
this section.
\begin{prop}
\label{prop:inductive main proj prop}Let $0\le m\le d$ be given.
Then for every $\epsilon>0$, $k\ge K(\epsilon)\ge1$, $n\ge N(\epsilon,k)\ge1$
and $V\in\mathrm{Gr}_{m}(d)$,
\[
\mathbb{P}_{1\le i\le n}\left\{ \mathbf{I}_{m}(i)\in\mathcal{Z}(\Sigma_{1}^{m},\epsilon,k,i,V)\right\} >1-\epsilon.
\]
\end{prop}

Since $\Sigma_{1}^{0}=0$, the case $m=0$ is of course trivial. It
is useful to consider this case since the proposition is proven by
induction on $m$.

\subsection{Entropy of projections of concentrated measures}

The purpose of the present subsection is to prove the following simple
but useful lemma. It will be used during the proof of Proposition
\ref{prop:inductive main proj prop} and in other places below. The
lemma bounds from below the entropy of projections of measures which
are highly concentrated near a proper linear subspace. Given $\emptyset\ne E\subset\mathbb{R}^{d}$
and $r>0$, recall from Section \ref{subsec:General-notations} that
$E^{(r)}$ denotes the closed $r$-neighbourhood of $E$.
\begin{lem}
\label{lem:gen proj lb}Let $1\le q\le l\le m\le d$ and $\epsilon>0$
be given. Then there exists $C=C(\epsilon)>1$ such that the following
holds. Let $V\in\mathrm{Gr}_{l}(m)$, $W\in\mathrm{Gr}_{q}(m)$, $\theta\in\mathcal{M}(\mathbb{R}^{m})$,
and $n,k\ge1$ be such that $\kappa(W,V^{\perp})\ge\epsilon$, $\mathrm{diam}(\mathrm{supp}(\theta))=O(2^{-n})$,
and $\mathrm{supp}(\theta)\subset x+V^{(2^{-n-k})}$ for some $x\in\mathbb{R}^{m}$.
Then,
\[
\frac{1}{k}H(\pi_{W}\theta,\mathcal{D}_{n+k})\ge\frac{1}{k}H(\theta,\mathcal{D}_{n+k})-l+q-C/k.
\]
\end{lem}

The proof requires some preparations.
\begin{lem}
\label{lem:gen low bd by comp}Let $1\le q\le l\le m\le d$ and $\epsilon>0$
be given. Then there exists $\delta=\delta(\epsilon)>0$ such that
the following holds. Let $V\in\mathrm{Gr}_{l}(m)$ and $W\in\mathrm{Gr}_{q}(m)$
be with $\kappa(W,V^{\perp})\ge\epsilon$. Then $\alpha_{q}(\pi_{W}\circ F)\ge\delta$
for any linear isometry $F:\mathbb{R}^{l}\rightarrow V$.
\end{lem}

\begin{proof}
Assume by contradiction that the lemma is false. Then by compactness
there exist $V\in\mathrm{Gr}_{l}(m)$, $W\in\mathrm{Gr}_{q}(m)$ and
a linear isometry $F:\mathbb{R}^{l}\rightarrow V$, so that $\kappa(W,V^{\perp})\ge\epsilon$
and $\alpha_{q}(\pi_{W}\circ F)=0$.

Since $\kappa(W,V^{\perp})\ge\epsilon$
we have $W\cap V^{\perp}=\{0\}$. Hence $W^{\perp}+V=\mathbb{R}^{m}$,
which implies that $\pi_{W}\circ F(\mathbb{R}^{l})=\mathbb{R}^{q}$.
Thus, by (\ref{eq:equiv cond for lin indep}) there exists $\zeta\in\wedge^{q}\mathbb{R}^{l}$
so that $\wedge^{q}(\pi_{W}\circ F)(\zeta)\ne0$. On the other hand,
from $\alpha_{q}(\pi_{W}\circ F)=0$ and (\ref{eq:norm =00003D prod of sing vals}),
it follows that $\wedge^{q}(\pi_{W}\circ F)=0$. This contradiction
completes the proof of the lemma.
\end{proof}
\begin{lem}
\label{lem:gen low bd step I}Let $1\le q\le l\le d$, $\delta>0$,
$n,k\ge1$, $\theta\in\mathcal{M}(\mathbb{R}^{l})$ and a linear $A:\mathbb{R}^{l}\rightarrow\mathbb{R}^{q}$
be given. Suppose that $\mathrm{diam}(\mathrm{supp}(\theta))=O(2^{-n})$
and $\alpha_{q}(A)\ge\delta$. Then,
\[
\frac{1}{k}H(A\theta,\mathcal{D}_{n+k})\ge\frac{1}{k}H(\theta,\mathcal{D}_{n+k})-l+q-O_{\delta}\left(\frac{1}{k}\right).
\]
\end{lem}

\begin{proof}
When $l=q$ the lemma follows directly from (\ref{eq:ent under expand lin map}),
hence assume that $l>q$. Set $W:=\ker A$, then by (\ref{eq:ent =000026 coordinate change})
and (\ref{eq:cond ent form})
\[
H\left(\theta,\mathcal{D}_{n+k}\right)\le H\left(\pi_{W}\theta,\mathcal{D}_{n+k}\right)+H\left(\pi_{W^{\perp}}\theta,\mathcal{D}_{n+k}\right)+O(1).
\]
Since $\mathrm{diam}(\mathrm{supp}(\theta))=O(2^{-n})$ and $\dim W=l-q$,
\[
\frac{1}{k}H\left(\pi_{W}\theta,\mathcal{D}_{n+k}\right)\le l-q+O\left(\frac{1}{k}\right).
\]
From $\alpha_{q}(A)\ge\delta$ it follows that there exists $B\in\mathrm{GL}(q,\mathbb{R})$
so that $\alpha_{q}(B)\ge\delta$ and $A=B\circ\pi_{W^{\perp}}$.
Thus by (\ref{eq:ent under expand lin map}),
\[
H\left(\pi_{W^{\perp}}\theta,\mathcal{D}_{n+k}\right)\le H(A\theta,\mathcal{D}_{n+k})+O_{\delta}(1).
\]
The lemma now follows by combining all of these inequalities.
\end{proof}
\begin{proof}[Proof of Lemma \ref{lem:gen proj lb}]
Let $\delta>0$ be small with respect to $\epsilon$, let $V$, $W$,
$\theta$, $n$ and $k$ be as in the statement of the lemma, and
let $F:\mathbb{R}^{l}\rightarrow V$ be a linear isometry. Since $\kappa(W,V^{\perp})\ge\epsilon$
and by Lemma \ref{lem:gen low bd by comp}, we have $\alpha_{q}(\pi_{W}\circ F)\ge\delta$.
From this, since
\[
\mathrm{diam}(\mathrm{supp}(F^{-1}P_{V}\theta))=O(2^{-n}),
\]
and by applying Lemma \ref{lem:gen low bd step I} with $A:=\pi_W\circ F$,
\[
\frac{1}{k}H(\pi_{W}P_{V}\theta,\mathcal{D}_{n+k})\ge\frac{1}{k}H(F^{-1}P_{V}\theta,\mathcal{D}_{n+k})-l+q-O_{\delta}\left(\frac{1}{k}\right).
\]
Additionally, since $\mathrm{supp}(\theta)\subset x+V^{(2^{-n-k})}$
and by (\ref{eq:ent under bi-lip map})
and (\ref{eq:ent is cont in map}),
\begin{eqnarray*}
H(F^{-1}P_{V}\theta,\mathcal{D}_{n+k}) & = & H(\theta,\mathcal{D}_{n+k})+O(1),\\
H(\pi_{W}P_{V}\theta,\mathcal{D}_{n+k}) & = & H(\pi_{W}\theta,\mathcal{D}_{n+k})+O(1),
\end{eqnarray*}
which completes the proof of the lemma.
\end{proof}

\subsection{Additional preparations for the proof of Proposition \ref{prop:inductive main proj prop}}

For $n\ge1$ recall the random word $\mathbf{U}(n)$ defined in Section
\ref{subsec:Symbolic-notations}. In what follows, given a word $i_1...i_n=u\in\Lambda^*$ we write $A_u^*$ in place of $(A_u)^*=A_{i_n}^*...A_{i_1}^*$.
\begin{lem}
\label{lem:A(V) close to L(A)}Let $1\le m\le d$ be given. Then for
every $\epsilon>0$ and $n\ge N(\epsilon)$,
\begin{equation}
\int\mathbb{P}_{i=n}\left\{ d_{\mathrm{Gr}_{m}}(A_{\mathbf{U}(i)}^{*}V,L_{m}(A_{\mathbf{U}(i)}^{*}))<\epsilon\right\} \:d\nu_{m}^{*}(V)>1-\epsilon.\label{eq:event with implicit assumption}
\end{equation}
\end{lem}

\begin{proof}
As explained in Section \ref{subsec:Coding and Furstenberg maps}, by Remark \ref{rem:semi of trans also SI =000026 prox} and \cite[Lemma 2.17 and Proposition 4.7]{BQ}, there exists a Borel map $L_{m}^{*}:\Lambda^{\mathbb{N}}\rightarrow\mathrm{Gr}_{m}(d)$
so that for $\beta$-a.e. $\omega$

\[
L_{m}^{*}(\omega)=\underset{n\rightarrow\infty}{\lim}\:L_{m}(A_{\omega_{0}}^{*}\ldots A_{\omega_{n}}^{*})\text{ and }A_{\omega_{0}}^{*}...A_{\omega_{n}}^{*}\nu_{m}^{*}\overset{n}{\rightarrow}\delta_{L_{m}^{*}(\omega)}\text{ weakly.}
\]
This clearly implies,
\[
\underset{n\rightarrow\infty}{\lim}\int\int d_{\mathrm{Gr}_{m}}(A_{\omega_{0}}^{*}...A_{\omega_{n-1}}^{*}V,L_{m}(A_{\omega_{0}}^{*}\ldots A_{\omega_{n-1}}^{*}))\:d\nu_{m}^{*}(V)\:d\beta(\omega)=0.
\]
Thus, since for each $n\ge1$ the distribution of $A_{\omega_{0}}^{*}...A_{\omega_{n-1}}^{*}$
with respect to $\beta$ is equal to the distribution of $A_{\mathbf{U}(n)}^{*}$,
we have
\[
\underset{n\rightarrow\infty}{\lim}\mathbb{E}_{i=n}\left(\int d_{\mathrm{Gr}_{m}}(A_{\mathbf{U}(i)}^{*}V,L_{m}(A_{\mathbf{U}(i)}^{*}))\:d\nu_{m}^{*}(V)\right)=0,
\]
which completes the proof of the lemma.
\end{proof}
In the next lemma we use the Ledrappier-Young formula in order to
bound from below the entropy of slices of cylinders of $\mu$.
\begin{lem}
\label{lem:lb on ent of slices}Let $1\le m\le d$ be given. Then
for every $\epsilon>0$, $k\ge K(\epsilon)\ge1$ and $n\ge N(\epsilon,k)\ge1$,
\[
\int\mathbb{E}_{1\le i\le n}\left(\varphi_{\mathbf{I}_{m}(i)}\mu\left\{ x\::\:\frac{1}{k}H\left((\varphi_{\mathbf{I}_{m}(i)}\mu)_{x}^{V},\mathcal{D}_{i+k}\right)>\Delta_{m}-\epsilon\right\} \right)\:d\nu_{m-1}^{*}(V)>1-\epsilon.
\]
\end{lem}

\begin{proof}
Let $\epsilon>0$, let $k\ge1$ be large with respect to $\epsilon$,
and let $n\ge1$ be large with respect to $k$. Let $C_{1}\ge1$ be
as in Lemma \ref{lem:ac of stopping times}. By Theorem \ref{thm:LY formula for SA} and Lemma \ref{lem:dim_e=00003Ddim} we may assume,
\[
\int\mu\left\{ x\::\:\frac{1}{k}H\left(P_{V_{m}}\mu_{x}^{V_{m-1}},\mathcal{D}_{k}\right)>\Delta_{m}-\epsilon\right\} \:d\nu^{*}(V_{j})_{j=0}^{d}>1-\epsilon.
\]
From this and the decomposition
\[
\nu^{*}=\mathbb{E}_{1\le i\le C_{1}n}\left(A_{\mathbf{U}(i)}^{*}\nu^{*}\right),
\]
we get
\begin{multline*}
\int\mathbb{E}_{1\le i\le C_{1}n}\left(\mu\left\{ x:\frac{1}{k}H\left(P_{A_{\mathbf{U}(i)}^{*}V_{m}}\mu_{x}^{A_{\mathbf{U}(i)}^{*}V_{m-1}},\mathcal{D}_{k}\right)>\Delta_{m}-\epsilon\right\}
\right)d\nu^{*}(V_{j})_{j=0}^{d}\\
>1-\epsilon.
\end{multline*}
Thus, by Lemma \ref{lem:A(V) close to L(A)}, since the entropy of
images is nearly continuous in the map (see (\ref{eq:ent is cont in map})),
and by assuming that $k$ and $n$ are large enough with respect to
the specified parameters,
\begin{multline*}
\int\mathbb{E}_{1\le i\le C_{1}n}\left(\mu\left\{ x:\frac{1}{k}H\left(P_{L_{m}(A_{\mathbf{U}(i)}^{*})}\mu_{x}^{A_{\mathbf{U}(i)}^{*}V},\mathcal{D}_{k}\right)>\Delta_{m}-2\epsilon\right\} \right)d\nu_{m-1}^{*}(V)\\
>1-2\epsilon.
\end{multline*}
By Lemma \ref{lem:ac of stopping times} we now obtain,
\begin{multline}
\int\mathbb{E}_{1\le i\le n}\left(\mu\left\{ x:\frac{1}{k}H\left(P_{L_{m}(A_{\mathbf{I}_{m}(i)}^{*})}\mu_{x}^{A_{\mathbf{I}_{m}(i)}^{*}V},\mathcal{D}_{k}\right)>\Delta_{m}-2\epsilon\right\} \right)d\nu_{m-1}^{*}(V)\\
>1-O(\epsilon).\label{eq:proj of slices greater}
\end{multline}

Next let $V\in\mathrm{Gr}_{m-1}(d)$, $1\le i\le n$ and $u\in\Psi_{m}(i)$
be given, where $\Psi_{m}(i)$ is defined in Section \ref{subsec:Symbolic-notations}.
If $m<d$ assume additionally that $\alpha_{m}(A_{u})>\alpha_{m+1}(A_{u})$.
Let $UDU'$ be a singular value decomposition for $A_{u}^{*}$. By
(\ref{eq:push of slive via affine =00003D slice of push via affine})
it follows that for $\mu$-a.e. $x$,
\begin{eqnarray}
H\left((\varphi_{u}\mu)_{\varphi_{u}(x)}^{V},\mathcal{D}_{i+k}\right) & = & H\left(\varphi_{u}(\mu_{x}^{A_{u}^{*}V}),\mathcal{D}_{i+k}\right)\nonumber \\
 & = & H\left(DU^{*}(\mu_{x}^{A_{u}^{*}V}),\mathcal{D}_{i+k}\right)+O(1)\nonumber \\
 & \ge & H\left(S_{2^{i}}\pi_{d,m}DU^{*}(\mu_{x}^{A_{u}^{*}V}),\mathcal{D}_{k}\right)+O(1),\label{eq:LB on ent of slice of cyl}
\end{eqnarray}
where $\pi_{d,m}\in\mathrm{A}_{d,m}$ is defined in Section \ref{subsec:Spaces-of-affine maps}.
Note that,
\[
S_{2^{i}}\pi_{d,m}D=\mathrm{diag}_{m,m}(2^{i}\alpha_{1}(A_{u}),...,2^{i}\alpha_{m}(A_{u}))\pi_{d,m}.
\]
Thus, from $u\in\Psi_{m}(i)$, (\ref{eq:words in Phi_m are Theta(2^n)}),
(\ref{eq:ent under expand lin map}) and (\ref{eq:LB on ent of slice of cyl}),
\[
H\left((\varphi_{u}\mu)_{\varphi_{u}(x)}^{V},\mathcal{D}_{i+k}\right)\ge H\left(\pi_{d,m}U^{*}(\mu_{x}^{A_{u}^{*}V}),\mathcal{D}_{k}\right)+O(1)\text{ for }\mu\text{-a.e. }x.
\]
Additionally, note that $\tilde{U}\pi_{d,m}U^{*}=P_{L_{m}(A_{u}^{*})}$
for some linear isometry $\tilde{U}:\mathbb{R}^{m}\rightarrow L_{m}(A_{u}^{*})$.
Hence,
\[
H\left((\varphi_{u}\mu)_{\varphi_{u}(x)}^{V},\mathcal{D}_{i+k}\right)\ge H\left(P_{L_{m}(A_{u}^{*})}(\mu_{x}^{A_{u}^{*}V}),\mathcal{D}_{k}\right)+O(1)\text{ for }\mu\text{-a.e. }x.
\]
This together with (\ref{eq:proj of slices greater}) completes the
proof of the lemma.
\end{proof}
In the following lemma we bound from above, with large probability, the
diameter of the supports of the sliced measures appearing in Lemma
\ref{lem:lb on ent of slices}.
\begin{lem}
\label{lem:small diam of slices}Let $1\le m\le d$ be given. Then
for every $\epsilon>0$ there exist $C,N\ge1$ such that for every
$n\ge N$ and $V\in\mathrm{Gr}_{d-m+1}(d)$,
\[
\mathbb{P}_{i=n}\left\{ \mathrm{diam}((x+V)\cap\varphi_{\mathbf{I}_{m}(i)}(K_{\Phi}))>C2^{-n}\text{ for some }x\in\mathbb{R}^{d}\right\} <\epsilon.
\]
\end{lem}

\begin{proof}
Since $\mathrm{diam}(\varphi_{u}(K_{\Phi}))=O(2^{-n})$ for every
$n\ge1$ and $u\in\Psi_{1}(n)$, the lemma clearly holds for $m=1$.
Thus we can assume that $m\ge2$.

Let $\epsilon>0$ be given. By Lemma \ref{lem:P=00007Bkappa<delta=00007D<epsilon},
there exists $\delta>0$ and $N\ge1$ so that
\begin{equation}
\mathbb{P}_{i=n}\left\{ \kappa(L_{m-1}(A_{\mathbf{I}_{m}(i)}),V)\le\delta\right\} <\epsilon\text{ for all }V\in\mathrm{Gr}_{d-m+1}(d)\text{ and }n\ge N.\label{eq:small prob for small kappa}
\end{equation}

Fix $n\ge N$, $V\in\mathrm{Gr}_{d-m+1}(d)$ and $u\in\Psi_{m}(n)$
with $\alpha_{m-1}(A_{u})>\alpha_{m}(A_{u})$ and $\kappa(L_{m-1}(A_{u}),V)\ge\delta$.
Let $x\in\mathbb{R}^{d}$, let $y,z\in(x+V)\cap\varphi_{u}(K_{\Phi})$,
and set $v:=y-z$. By (\ref{eq:small prob for small kappa}), in order
to complete the proof of the lemma it suffices to show that $|v|=O_{\delta}(2^{-n})$.
Since $y,z\in x+V$, we have $v\in V$. Let $C>1$ be a large constant
depending only on $K_{\Phi}$. From $u\in\Psi_{m}(n)$ and $y,z\in\varphi_{u}(K_{\Phi})$,
it follows that there exist $v_{1}\in L_{m-1}(A_{u})$ and $v_{2}\in L_{m-1}(A_{u})^{\perp}$
so that $v=v_{1}+v_{2}$ and $|v_{2}|\le C2^{-n}$. We may clearly
assume that $v_{1}\ne0$. Hence, from $\kappa(L_{m-1}(A_{u}),V)\ge\delta$,
$v\in V$,$v_{1}\in L_{m-1}(A_{u})$ and (\ref{eq:alt def for dist of lines}),
\begin{eqnarray*}
\delta\le d_{\mathrm{Gr}_{1}}(v\mathbb{R},v_{1}\mathbb{R}) & \le & C|v|^{-1}|v_{1}|^{-1}\Vert v\wedge v_{1}\Vert\\
 & = & C|v|^{-1}|v_{1}|^{-1}\Vert v_{2}\wedge v_{1}\Vert=C|v|^{-1}|v_{2}|\le C^{2}2^{-n}|v|^{-1}.
\end{eqnarray*}
This gives $|v|\le\delta^{-1}C^{-2}2^{-n}$, which completes the proof
of the lemma.
\end{proof}
We shall also need the following simple lemma. Its proof follows easily from the
compactness of the Grassmannians and is therefore omitted.
\begin{lem}
\label{lem:lb on kappa of perps}Let $1\le m\le d$ and $\delta>0$
be given. Then there exists $\rho=\rho(\delta)>0$ so that $\kappa(V,L^{\perp})\ge\rho$
for all $V,L\in\mathrm{Gr}_{m}(d)$ with $\kappa(V^{\perp},L)\ge\delta$.
\end{lem}

\subsection{Proof of the proposition}
\begin{proof}[Proof of Proposition \ref{prop:inductive main proj prop}]
As mentioned above, the proof is carried out by induction on $m$.
Since $\Sigma_{1}^{0}=0$, the proposition holds trivially for $m=0$.
Let $1\le m\le d$ and suppose that the proposition has already been
proven for $m-1$.

Let $\epsilon>0$ be small, let $\delta>0$ be small with respect
to $\epsilon$, let $l\ge1$ be large with respect to $\delta$, let
$k\ge1$ be large with respect to $l$, and let $n\ge1$ be large
with respect to $k$.

By Lemma \ref{lem:lb on ent of slices}, and since $k$ is large with
respect to $l$, there exists $W\in\mathrm{Gr}_{m-1}(d)$ such that
\begin{equation}
\mathbb{E}_{1\le i\le n}\left(\varphi_{\mathbf{I}_{m}(i+l)}\mu\left\{ x\::\:\frac{1}{k}H\left((\varphi_{\mathbf{I}_{m}(i+l)}\mu)_{x}^{W},\mathcal{D}_{i+k}\right)>\Delta_{m}-\epsilon\right\} \right)>1-\epsilon.\label{eq:ineq1 in ind main proj prop}
\end{equation}
By the induction hypothesis,
\[
\mathbb{P}_{1\le i\le n}\left(\mathbf{I}_{m-1}(i+l)\in\mathcal{Z}(\Sigma_{1}^{m-1},\epsilon,k,i+l,W)\right)>1-\epsilon.
\]
From this, since $k$ is large with respect to $l$, by the concavity
of entropy, and by replacing $\epsilon$ with a larger quantity (which
is still small) without changing the notation, we may assume that
\begin{equation}
\mathbb{P}_{1\le i\le n}\left(\mathbf{I}_{m-1}(i+l)\in\mathcal{Z}(\Sigma_{1}^{m-1},\epsilon,k,i,W)\right)>1-\epsilon.\label{eq:ineq2 in ind main proj prop}
\end{equation}

Recall the measure $\lambda_{n}\in\mathcal{M}(\mathbb{N})$ defined
in Section \ref{subsec:Component-measures}. From (\ref{eq:ineq1 in ind main proj prop})
and (\ref{eq:ineq2 in ind main proj prop}), and by replacing $\epsilon$
with a larger quantity (which is still small) without changing the
notation, we may assume that $\lambda_{n}(\mathcal{Q})>1-\epsilon$,
where $\mathcal{Q}$ is the set of all $1\le i\le n$ with
\begin{equation}
\mathbb{E}_{j=i+l}\left(\varphi_{\mathbf{I}_{m}(j)}\mu\left\{ x\::\:\frac{1}{k}H\left((\varphi_{\mathbf{I}_{m}(j)}\mu)_{x}^{W},\mathcal{D}_{i+k}\right)>\Delta_{m}-\epsilon\right\} \right)>1-\epsilon\label{eq:large ent of slices on avg}
\end{equation}
and
\begin{equation}
\mathbb{P}_{j=i+l}\left(\mathbf{I}_{m-1}(j)\in\mathcal{Z}(\Sigma_{1}^{m-1},\epsilon,k,i,W)\right)>1-\epsilon.\label{eq:cyl meas in Z with h-prob}
\end{equation}
Moreover, since $\chi_{1}>...>\chi_{d}$ and by (\ref{eq:words in Phi_m are Theta(2^n)}), if $m<d$ we may also assume
that for $i\in\mathcal{Q}$
\begin{equation}
\mathbb{P}_{j=i+l}\left(\alpha_{m+1}(A_{\mathbf{I}_{m}(j)})<2^{-i-2k}\right)>1-\epsilon,\label{eq:alpha_m+1 is small}
\end{equation}
while still having $\lambda_{n}(\mathcal{Q})>1-\epsilon$.

Fix $i\in\mathcal{Q}$ until the end of the proof. By Lemma \ref{lem:small mass of pts close to bd},
\begin{equation}
\mu\left(\cup_{D\in\mathcal{D}_{i}}(\partial D)^{(2^{-i}\delta)}\right)<\epsilon.\label{eq:small mass of bnd}
\end{equation}
Additionally, from (\ref{eq:decomp of mu according to cut sets})
and by basic properties of disintegrations (see Section \ref{subsec:Disintegrations}),
\[
\mu=\mathbb{E}_{j=i+l}\left(\int(\varphi_{\mathbf{I}_{m}(j)}\mu)_{x}^{W}\:d\varphi_{\mathbf{I}_{m}(j)}\mu(x)\right).
\]
From this decomposition, by (\ref{eq:small mass of bnd}), by Lemma
\ref{lem:small diam of slices}, and since $l$ is large with respect
to $\delta$,
\begin{equation}
\mathbb{E}_{j=i+l}\left(\varphi_{\mathbf{I}_{m}(j)}\mu\left\{ x\::\:\begin{array}{c}
\text{there exists }D\in\mathcal{D}_{i}\text{ with}\\
\mathrm{supp}(\varphi_{\mathbf{I}_{m}(j)}\mu)_{x}^{W}\subset D
\end{array}\right\} \right)>1-\epsilon.\label{eq:slice contained in cell with high prob}
\end{equation}

Let $\mathcal{U}_{1}$ be the set of all words $u\in\Psi_{m}(i+l)$
with
\[
\varphi_{u}\mu\left\{ x\::\:\frac{1}{k}H\left((\varphi_{u}\mu)_{x}^{W},\mathcal{D}_{i+k}\right)>\Delta_{m}-\epsilon\right\} >1-\epsilon,
\]
and
\[
\varphi_{u}\mu\left\{ x\::\:\begin{array}{c}
\text{there exists }D\in\mathcal{D}_{i}\text{ with}\\
\mathrm{supp}(\varphi_{u}\mu)_{x}^{W}\subset D
\end{array}\right\} >1-\epsilon.
\]
By (\ref{eq:large ent of slices on avg}) and (\ref{eq:slice contained in cell with high prob}),
and by replacing $\epsilon$ with a larger quantity (which is still
small) without changing the notation, we may assume that
\begin{equation}
\mathbb{P}_{j=i+l}\{\mathbf{I}_{m}(j)\in\mathcal{U}_{1}\}>1-\epsilon.\label{eq:ind pf lb on P(U_1)}
\end{equation}
Moreover, by the concavity of conditional entropy, by (\ref{eq:cond ent =00003D ent for slices}),
and by replacing $\epsilon$ with a larger quantity once more, we
may assume that for $u\in\mathcal{U}_{1}$
\begin{equation}
\mathbb{P}_{j=i}\left\{ \frac{1}{k}H\left((\varphi_{u}\mu)_{x,j},\mathcal{D}_{j+k}\mid P_{W}^{-1}(\mathcal{B}_{\mathbb{R}^{d}})\right)>\Delta_{m}-\epsilon\right\} >1-\epsilon.\label{eq:lb on prob of large slices}
\end{equation}

Let $\mathcal{U}_{2}$ be set of all $u\in\Psi_{m}(i+l)$ with,
\[
\beta_{[u]}\left(\bigcup\left\{ [v]\::\:\begin{array}{c}
v\in\Psi_{m-1}(i+l),\:[v]\subset[u]\text{ and }\\
v\in\mathcal{Z}(\Sigma_{1}^{m-1},\epsilon,k,i,W)
\end{array}\right\} \right)>1-\epsilon.
\]
Note that the partition $\{[v]\::\:v\in\Psi_{m-1}(i+l)\}$ of $\Lambda^\mathbb{N}$ refines
the partition $\{[u]\::\:u\in\Psi_{m}(i+l)\}$. From this, from (\ref{eq:cyl meas in Z with h-prob}),
and by replacing $\epsilon$ with a larger quantity, we may assume
that
\begin{equation}
\mathbb{P}_{j=i+l}\{\mathbf{I}_{m}(j)\in\mathcal{U}_{2}\}>1-\epsilon.\label{eq:ind pf lb on P(U_2)}
\end{equation}
Moreover, for every $u\in\Psi_{m}(i+l)$ and $D\in\mathcal{D}_{i}$
with $\varphi_{u}\mu(D)>0$,
\[
(\varphi_{u}\mu)_{D}=\sum_{v\in\Psi_{m-1}(i+l),\:[v]\subset[u]}\frac{p_{v}}{p_{u}}\cdot\frac{\varphi_{v}\mu(D)}{\varphi_{u}\mu(D)}\cdot(\varphi_{v}\mu)_{D}.
\]
From these decompositions, by the concavity of entropy, by the definition
of $\mathcal{Z}(\Sigma_{1}^{m-1},\epsilon,k,i,W)$, and by replacing
$\epsilon$ with a larger quantity, we may assume that for $u\in\mathcal{U}_{2}$
\begin{equation}
\mathbb{P}_{j=i}\left(\frac{1}{k}H\left(\pi_{W}(\varphi_{u}\mu)_{x,j},\mathcal{D}_{j+k}\right)>\Sigma_{1}^{m-1}-\epsilon\right)>1-\epsilon.\label{eq:lb on prob of large proj}
\end{equation}

By (\ref{eq:ent =000026 coordinate change}) and (\ref{eq:monot of ent in cond sig-alg}),
it follows that for every $\theta\in\mathcal{M}(\mathbb{R}^{d})$
\[
\frac{1}{k}H\left(\theta,\mathcal{D}_{i+k}\right)\ge\frac{1}{k}H\left(\pi_{W}\theta,\mathcal{D}_{i+k}\right)+\frac{1}{k}H\left(\theta,\mathcal{D}_{i+k}\mid P_{W}^{-1}(\mathcal{B}_{\mathbb{R}^{d}})\right)-O(1/k).
\]
Thus, by (\ref{eq:lb on prob of large slices}) and (\ref{eq:lb on prob of large proj}),
\begin{equation}
\mathbb{P}_{j=i}\left(\frac{1}{k}\left((\varphi_{u}\mu)_{x,j},\mathcal{D}_{j+k}\right)>\Sigma_{1}^{m}-3\epsilon\right)\ge1-2\epsilon\text{ for }u\in\mathcal{U}_{1}\cap\mathcal{U}_{2}.\label{eq:lb for u in U_1 cap U_2}
\end{equation}

Given $V\in\mathrm{Gr}_{m}(d)$, let $\mathcal{U}_{V}$ be the set
of all $u\in\Psi_{m}(i+l)$ so that $\kappa(V^{\perp},L_{m}(A_{u}))>\delta$
and $\varphi_{u}\mu$ is supported on $x+L_{m}(A_{u})^{(2^{-i-k})}$
for some $x\in\mathbb{R}^{d}$. By Lemma \ref{lem:P=00007Bkappa<delta=00007D<epsilon}
and (\ref{eq:alpha_m+1 is small}), we may assume that
\begin{equation}
\mathbb{P}_{j=i+l}\{\mathbf{I}_{m}(j)\in\mathcal{U}_{V}\}>1-2\epsilon\text{ for all }V\in\mathrm{Gr}_{m}(d).\label{eq:ind pf lb on P(U_V)}
\end{equation}
Moreover, from (\ref{eq:lb for u in U_1 cap U_2}) and by Lemmata
\ref{lem:lb on kappa of perps} and \ref{lem:gen proj lb}, it follows
that for $V\in\mathrm{Gr}_{m}(d)$ and $u\in\mathcal{U}_{1}\cap\mathcal{U}_{2}\cap\mathcal{U}_{V}$
\[
\mathbb{P}_{j=i}\left(\frac{1}{k}\left(\pi_{V}((\varphi_{u}\mu)_{x,j}),\mathcal{D}_{j+k}\right)>\Sigma_{1}^{m}-4\epsilon\right)\ge1-2\epsilon.
\]
Thus, from (\ref{eq:ind pf lb on P(U_1)}), (\ref{eq:ind pf lb on P(U_2)})
and (\ref{eq:ind pf lb on P(U_V)}), we obtain
\begin{equation}
\mathbb{P}\left(\mathbf{I}_{m}(i+l)\in\mathcal{Z}(\Sigma_{1}^{m},4\epsilon,k,i,V)\right)>1-4\epsilon\text{ for all }V\in\mathrm{Gr}_{m}(d).\label{eq:fin stage with I(i+l)}
\end{equation}

Note that for every $u\in\Psi_{m}(i)$ and $D\in\mathcal{D}_{i}$
with $\varphi_{u}\mu(D)>0$,
\[
(\varphi_{u}\mu)_{D}=\sum_{v\in\Psi_{m}(i+l),\:[v]\subset[u]}\frac{p_{v}}{p_{u}}\cdot\frac{\varphi_{v}\mu(D)}{\varphi_{u}\mu(D)}\cdot(\varphi_{v}\mu)_{D}.
\]
From these decompositions, by the concavity of entropy, and by replacing
$\epsilon$ with a larger quantity, we may assume that (\ref{eq:fin stage with I(i+l)})
holds with $\mathbf{I}_{m}(i)$ in place of $\mathbf{I}_{m}(i+l)$.
Hence, since $i$ is an arbitrary member of $\mathcal{Q}$ and $\lambda_{n}(\mathcal{Q})>1-\epsilon$,
we have thus shown that
\[
\mathbb{P}_{1\le i\le n}\left(\mathbf{I}_{m}(i)\in\mathcal{Z}(\Sigma_{1}^{m},4\epsilon,k,i,V)\right)>1-5\epsilon\text{ for all }V\in\mathrm{Gr}_{m}(d).
\]
This completes the induction and the proof of the proposition.
\end{proof}

\subsection{Additional entropy estimates}

We now draw some consequences of Proposition \ref{prop:inductive main proj prop}.
The next lemma provides a lower bound on the entropy of projections
of components of $\mu$.

Given $1\le m\le d$, recall that metric concepts in $\mathrm{A}_{d,m}$
are considered with respect to $d_{\mathrm{A}_{d,m}}$ unless stated
otherwise (see Section \ref{subsec:An-invariant-metric}). In particular,
$B(\pi_{d,m},R)$ denotes the closed ball in $(\mathrm{A}_{d,m},d_{\mathrm{A}_{d,m}})$
with centre $\pi_{d,m}$ and radius $R$.
\begin{lem}
\label{lem:lb on ent of proj of comp of mu}Let $1\le m\le d$ be
given. Then for every $\epsilon,R>0$, $k\ge K(\epsilon,R)\ge1$,
$n\ge N(\epsilon,R,k)\ge1$ and $\psi\in B(\pi_{d,m},R)$,
\[
\mathbb{P}_{1\le i\le n}\left\{ \frac{1}{k}H\left(\psi\mu_{x,i},\mathcal{D}_{i+k}\right)>\Sigma_{1}^{m}-\epsilon\right\} >1-\epsilon.
\]
\end{lem}

\begin{proof}
For every $i\ge0$ and $D\in\mathcal{D}_{i}$ with $\mu(D)>0$,
\[
\mu_{D}=\sum_{u\in\Psi_{m}(i)}p_{u}\frac{\varphi_{u}\mu(D)}{\mu(D)}(\varphi_{u}\mu)_{D}.
\]
Thus, by Proposition \ref{prop:inductive main proj prop} and the
concavity of entropy, the lemma holds when $\psi=\pi_{V}$ for some
$V\in\mathrm{Gr}_{m}(d)$. The general case of the lemma follows from
this, from (\ref{eq:ent of psi comp to ent of proj}), and since $B(\pi_{d,m},R)$
is a compact subset of $\mathrm{A}_{d,m}$ (see Lemma \ref{lem: def of inv met and prop}).
\end{proof}
\begin{rem*}
Since $\dim\mu=\Sigma_{1}^{d}$ and by Lemmata \ref{lem:lb on ent of proj of comp of mu}
and \ref{lem:multiscale-entropy-formula}, it follows easily that
most components of $\mu$ have normalized entropy close to $\dim\mu$.
That is, in the terminology of \cite{Ho1,HR}, the measure $\mu$ has uniform entropy dimension $\dim\mu$.
\end{rem*}

For the proof of our entropy increase result (Theorem \ref{thm:ent increase result}),
we shall need the following extension of Lemma \ref{lem:lb on ent of proj of comp of mu}. It bounds from below the entropy of components of projections of components of $\mu$.
\begin{lem}
\label{lem:lb on ent of comp of proj of comp of mu}Let $1\le m\le d$
be given. Then for every $\epsilon,R>0$, $l\ge L(\epsilon,R)\ge1$, $k\ge K(\epsilon,R)\ge1$, $n\ge N(\epsilon,R,k,l)\ge1$ and $\psi\in B(\pi_{d,m},R)$,
\[
\int\mathbb{P}_{i\le j\le i+k}\left\{ \frac{1}{l}H\left((\psi\mu_{x,i})_{y,j},\mathcal{D}_{j+l}\right)>\Sigma_{1}^{m}-\epsilon\right\} \:d\lambda_{n}\times\mu(i,x)>1-\epsilon.
\]
\end{lem}

\begin{proof}
Let $\epsilon>0$ be small, let $R>0$, let $\delta>0$ be small with respect to $\epsilon,R$, let $q\ge1$ be large with respect to $\delta$, let $l,k\ge1$ be large with respect to $q$, let $n\ge1$ be large with respect to $l,k$, and fix $\psi\in B(\pi_{d,m},R)$.

By Lemma \ref{lem:small mass of pts close to bd},
\begin{multline*}
\int\int\psi\mu_{x,i}\left(\cup_{D\in\mathcal{D}_{j}^m}(\partial D)^{(2^{-j}\delta)} \right) d\lambda_{i,i+k}(j)\: d\lambda_n\times\mu(i,x)\\
=\int\int\psi\mu\left(\cup_{D\in\mathcal{D}_{j}^m}(\partial D)^{(2^{-j}\delta)} \right) d\lambda_{i,i+k}(j)\: d\lambda_n(i)
<\epsilon,
\end{multline*}
where $\lambda_{i,i+k}$ is defined in Section \ref{subsec:Component-measures}. By Lemmata \ref{lem:lb on ent of proj of comp of mu}
and \ref{lem:distribution-of-components-of-components},
\begin{multline*}
\int\mathbb{P}_{i\le j\le i+k}\left\{ \frac{1}{l}H\left(\psi(\mu_{x,i})_{y,j},\mathcal{D}_{j+l}\right)>\Sigma_{1}^{m}-\epsilon\right\} \:d\lambda_{n}\times\mu(i,x)\\
>1-\epsilon/2-O\left(k/n\right)>1-\epsilon.
\end{multline*}
Recall from Section \ref{subsec:Component-measures} that $\mathcal{N}_n:=\{1,...,n\}$. From the last two inequalities, since $l,k$ are large with respect to $q$, and by replacing $\epsilon$ with a larger quantity (which is still small) without changing the notation, we may assume that $\lambda_n\times\mu(Z)>1-\epsilon$, where $Z$ is the set of all $(i,x)\in\mathcal{N}_n\times\mathbb{R}^d$ so that
\[
\lambda_{i,i+k}\left\{i\le j\le i+k\::\: \psi\mu_{x,i}\left(\cup_{D\in\mathcal{D}_{j}^m}(\partial D)^{(2^{-j}\delta)} \right)<\epsilon \right\}>1-\epsilon
\]
and
\[
\mathbb{P}_{i+q\le j\le i+q+k}\left\{ \frac{1}{l}H\left(\psi(\mu_{x,i})_{y,j},\mathcal{D}_{j-q+l}\right)>\Sigma_{1}^{m}-\epsilon\right\}>1-\epsilon.
\]

Fix $(i,x)\in Z$, and let $\mathcal{Q}_{i,x}$ be the set of all $i\le j\le i+k$ so that
\begin{equation}\label{psi mu(bd dyad) <eps}
\psi\mu_{x,i}\left(\cup_{D\in\mathcal{D}_{j}^m}(\partial D)^{(2^{-j}\delta)} \right)<\epsilon
\end{equation}
and
\begin{equation}\label{P_j+q>1-epsi}
\mathbb{P}_{s=j+q}\left\{\frac{1}{l}H\left(\psi(\mu_{x,i})_{y,s},\mathcal{D}_{s-q+l}\right)>\Sigma_{1}^{m}-\epsilon\right\}>1-\epsilon.
\end{equation}
Since $(i,x)\in Z$ and by replacing $\epsilon$ with a larger quantity without changing the notation, we may assume that $\lambda_{i,i+k}(\mathcal{Q}_{i,x})>1-\epsilon$.

Fix $j\in\mathcal{Q}_{i,x}$. Since $\psi$ belongs to the compact set $B(\pi_{d,m},R)$,
\[
\mathrm{diam}(\mathrm{supp}(\psi(\mu_{x,i})_{D}))=O_R(2^{-j-q}) \text{ for all }D\in\mathcal{D}_{j+q}^d\text{ with }\mu_{x,i}(D)>0.
\]
Hence, since $q$ is large with respect to $\delta$ and by (\ref{psi mu(bd dyad) <eps}),
\[
\mathbb{P}_{s=j+q}\left\{\mathrm{supp}(\psi(\mu_{x,i})_{y,s})\subset D\text{ for some }D\in\mathcal{D}_j^m\right\}>1-\epsilon.
\]
From this, from (\ref{P_j+q>1-epsi}), since
\[
\mathbb{E}_{s=j}\left((\psi\mu_{x,i})_{y,s}\right)=\psi\mu_{x,i}=\mathbb{E}_{s=j+q}\left(\psi(\mu_{x,i})_{y,s}\right),
\]
by the concavity of entropy, and by replacing $\epsilon$ with a larger quantity, we may assume that
\[
\mathbb{P}_{s=j}\left\{ \frac{1}{l}H\left((\psi\mu_{x,i})_{y,s},\mathcal{D}_{s+l}\right)>\Sigma_{1}^{m}-\epsilon\right\}>1-\epsilon.
\]

Note that the last inequality holds for all $(i,x)\in Z$ and $j\in\mathcal{Q}_{i,x}$.
This, together with $\lambda_n\times\mu(Z)>1-\epsilon$ and $\lambda_{i,i+k}(\mathcal{Q}_{i,x})>1-\epsilon$, completes the proof of the lemma.
\end{proof}
From Lemma \ref{lem:lb on ent of proj of comp of mu} we also get
the following useful statement. It shows in particular that $\underline{\dim}_{e}P_{V}\mu\ge\Sigma_{1}^{m}$
for all $1\le m\le d$ and $V\in\mathrm{Gr}_{m}(d)$, where $\underline{\dim}_{e}P_{V}\mu$
is the lower entropy dimension of $P_{V}\mu$ (see Section \ref{subsec:Entropy-in-Rd}).
\begin{lem}
\label{lem:lb on ent of psi mu}Let $1\le m\le d$ be given. Then
for every $\epsilon,R>0$, $n\ge N(\epsilon,R)\ge1$ and $\psi\in B(\pi_{d,m},R)$,
\[
\frac{1}{n}H(\psi\mu,\mathcal{D}_{n})>\Sigma_{1}^{m}-\epsilon.
\]
\end{lem}

\begin{proof}
Let $\epsilon,R>0$, let $k\ge1$ be large with respect to $\epsilon,R$,
let $n\ge1$ be large with respect to $k$, and let $\psi\in B(\pi_{d,m},R)$.
By Lemma \ref{lem:multiscale-entropy-formula} and (\ref{eq:cond ent as avg of ent of comp}),
\[
\frac{1}{n}H(\psi\mu,\mathcal{D}_{n})=\mathbb{E}_{1\le i\le n}\left(\frac{1}{k}H\left(\psi\mu,\mathcal{D}_{i+k}\mid\mathcal{D}_{i}\right)\right)+O_{R}\left(\frac{k}{n}\right).
\]
Thus, by the concavity of conditional entropy and since $\mathrm{diam}(\psi(D))=O_{R}(2^{-i})$
for all $i\ge0$ and $D\in\mathcal{D}_{i}^{d}$,
\begin{eqnarray*}
\frac{1}{n}H(\psi\mu,\mathcal{D}_{n}) & \ge & \mathbb{E}_{1\le i\le n}\left(\frac{1}{k}H\left(\psi\mu_{x,i},\mathcal{D}_{i+k}\mid\mathcal{D}_{i}\right)\right)+O_{R}\left(\frac{k}{n}\right)\\
 & = & \mathbb{E}_{1\le i\le n}\left(\frac{1}{k}H\left(\psi\mu_{x,i},\mathcal{D}_{i+k}\right)\right)+O_{R}\left(\frac{k}{n}+\frac{1}{k}\right).
\end{eqnarray*}
Now the lemma follows directly from Lemma \ref{lem:lb on ent of proj of comp of mu}.
\end{proof}

\section{\label{sec:Factorization-of}Factorization of $L_{m-1}$}

The main result of this section is Theorem \ref{thm:factor of proj of L_m-1},
which provides conditions under which $\pi_{V}L_{m-1}$ factors via
$\pi_{V}\Pi$ (see Remark \ref{rem:equiv form of cond in gen criteria}).
The proof is carried out in Section \ref{subsec:Factorization-of-projections}.
In Section \ref{subsec:Factorization-of L_m-1 itself} we use Theorem
\ref{thm:factor of proj of L_m-1} in order to deduce Theorem \ref{thm:L_m-1 factors via Pi},
which guarantees the factorization of $L_{m-1}$ via $\Pi$ whenever
$\Sigma_{1}^{m-1}=m-1$ and $\Delta_{m}<1$. In Section \ref{subsec:Projections of components revisited}
we utilize this in order to get additional lower bounds on the entropy
of projections of components of $\mu$. These lower bounds will be
used in Section \ref{sec:Entropy-growth-under conv} when we prove
our entropy increase result.

\subsection{\label{subsec:Factorization-of-projections}Proof of Theorem \ref{thm:factor of proj of L_m-1}}

Throughout this subsection fix $2\le m\le d$ and $V\in\mathrm{Gr}_{m}(d)$
such that $\Sigma_{1}^{m-1}=m-1$ and $\pi_{V}\mu$ is exact dimensional
with $\dim\pi_{V}\mu<m$.

For $F\subset\mathbb{R}^{m}$ and $B\subset\mathrm{Gr}_{m-1}(m)$
write
\[
\tilde{F}:=\Pi^{-1}\pi_{V}^{-1}F\text{ and }\tilde{B}:=L_{m-1}^{-1}\pi_{V}^{-1}B,
\]
where
\[
\pi_{V}^{-1}(B):=\{W\in\mathrm{Gr}_{m-1}(d)\::\:\pi_{V}(W)\in B\}
\]
and $L_{m-1}:\Lambda^{\mathbb{N}}\rightarrow\mathrm{Gr}_{m-1}(d)$ is the $(m-1)$-th Furstenberg boundary map.

The following proposition is the main ingredient in the proof of theorem
\ref{thm:factor of proj of L_m-1}. Recall that for a probability
measure $\theta$ and an event $E$ with $\theta(E)>0$ we write $\theta_{E}:=\frac{1}{\theta(E)}\theta|_{E}$.
\begin{prop}
\label{prop:sing proj of meas}Let $B_{1}$ and $B_{2}$ be Borel
subsets of $\mathrm{Gr}_{m-1}(m)$ with $\beta(\tilde{B}_{1}),\beta(\tilde{B}_{2})>0$
and $d_{\mathrm{Gr}_{m-1}}(B_{1},B_{2})>0$. Then the measures $\pi_{V}\Pi\beta_{\tilde{B}_{1}}$
and $\pi_{V}\Pi\beta_{\tilde{B}_{2}}$ are singular.
\end{prop}

\subsubsection{Preparations for the proof of Proposition \ref{prop:sing proj of meas}}

For $\epsilon>0$ and $k,n\ge1$ set
\[
\mathcal{Z}(\epsilon,k,n):=\Psi_{m-1}(n)\cap\mathcal{Z}(m-1,\epsilon,k,n,V),
\]
where the notation $\mathcal{Z}(m-1,\epsilon,k,n,V)$ was introduced
in Definition \ref{def:def of mathcal(Z)}. That is, $\mathcal{Z}(\epsilon,k,n)$
is the set of all words $u\in\Psi_{m-1}(n)$ so that for $\theta:=\varphi_{u}\mu$
we have
\begin{equation}
\mathbb{P}_{j=n}\left\{ \frac{1}{k}H\left(\pi_{V}(\theta_{x,j}),\mathcal{D}_{j+k}\right)>m-1-\epsilon\right\} >1-\epsilon.\label{eq:def prop of Z(epsi,k,n)}
\end{equation}

The proof of Proposition \ref{prop:sing proj of meas} relies on the
following two lemmata. Recall that for $\xi\in\mathcal{M}(\Lambda^{\mathbb{N}})$
and $\mathcal{U}\subset\Lambda^{*}$ we write $\xi(\mathcal{U})$
in place of $\xi\left(\cup_{u\in\mathcal{U}}[u]\right)$.
\begin{lem}
\label{lem:ent > xi(Z)}For every $\epsilon,\rho>0$ and $k\ge K(\epsilon,\rho)\ge1$
the following holds. Let $B$ be a Borel subset of $\mathrm{Gr}_{m-1}(m)$
with $\beta(\tilde{B})>0$, and let $\xi\in\mathcal{M}(\Lambda^{\mathbb{N}})$
be with $\xi\ll\beta_{\tilde{B}}$. Then for each $W\in\mathrm{Gr}_{m-1}(m)$
with $d_{\mathrm{Gr}_{m-1}}(W,B)\ge\rho$ and $n\ge N(\epsilon,\rho,k,B,\xi,W)\ge1$,
\[
\frac{1}{k}H\left(\pi_{V}\Pi\xi,\pi_{W^{\perp}}^{-1}\mathcal{D}_{n+k}\mid\mathcal{D}_{n}\right)>\xi(\mathcal{Z}(\epsilon,k,n))-O(\epsilon).
\]
\end{lem}

\begin{lem}
\label{lem:cond ent > (m-1)xi(Z)}For every $\epsilon>0$ and $k\ge K(\epsilon)\ge1$
there exists $\delta=\delta(\epsilon,k)>0$ so that the following
holds. Let $B$ be a Borel subset of $\mathrm{Gr}_{m-1}(m)$ satisfying
$\beta(\tilde{B})>0$ and $\mathrm{diam}(B)\le\delta$, and let $\xi\in\mathcal{M}(\Lambda^{\mathbb{N}})$
be with $\xi\ll\beta_{\tilde{B}}$. Then for each $W\in B$ and $n\ge N(\epsilon,k,\delta,B,\xi,W)\ge1$,
\[
\frac{1}{k}H\left(\pi_{V}\Pi\xi,\mathcal{D}_{n+k}\mid\mathcal{D}_{n}\vee\pi_{W^{\perp}}^{-1}\mathcal{D}_{n+k}\right)>(m-1)\xi(\mathcal{Z}(\epsilon,k,n))-O(\epsilon).
\]
\end{lem}

For the proofs we need the following simple lemma.
\begin{lem}
\label{lem:by the SW thm}Let $\xi\in\mathcal{M}(\Lambda^{\mathbb{N}})$
and $\epsilon>0$ be given, and suppose that $\xi\ll\beta$. Then
for every $n\ge N(\xi,\epsilon)\ge1$ there exist nonnegative numbers
$\{c_{u}\}_{u\in\Psi_{m-1}(n)}$ so that for $h:=\sum_{u\in\Psi_{m-1}(n)}c_{u}1_{[u]}$
and $\lambda:=h\:d\beta$ we have $\lambda\in\mathcal{M}(\Lambda^{\mathbb{N}})$
and $d_{TV}(\xi,\lambda)<\epsilon$, where $1_{[u]}$ denotes the
indicator function of $[u]$.
\end{lem}

\begin{proof}
Since $\xi\ll\beta$, there exists a Borel function $f:\Lambda^{\mathbb{N}}\rightarrow[0,\infty)$
with $\xi=f\:d\beta$. Since $C(\Lambda^{\mathbb{N}})$ is dense in
$L^{1}(\beta)$, there exists a continuous $g:\Lambda^{\mathbb{N}}\rightarrow[0,\infty)$
so that $\Vert g-f\Vert_{L^{1}(\beta)}<\epsilon$.

Let $Q$ be the set of all $q:\Lambda^{\mathbb{N}}\rightarrow\mathbb{R}$
of the form $q=\sum_{j=1}^{J}b_{j}1_{[u_{j}]}$ for some $J\ge1$,
$u_{1},...,u_{J}\in\Lambda^{*}$ and $b_{1},...,b_{J}\in\mathbb{R}$.
It is easy to see that $Q$ is a subalgebra of $C(\Lambda^{\mathbb{N}})$
which separates points and vanishes nowhere. Thus, by the Stone--Weierstrass
theorem, there exists $q\in Q$ so that $\Vert q-g\Vert_{L^{1}(\beta)}<\epsilon$,
which implies $\Vert q-f\Vert_{L^{1}(\beta)}<2\epsilon$.

We may clearly assume that $q\ge0$ while still having $\Vert q-f\Vert_{L^{1}(\beta)}<2\epsilon$.
Setting $h:=q/\int q\:d\beta$ and $\lambda:=h\:d\beta$, we obtain
$\lambda\in\mathcal{M}(\Lambda^{\mathbb{N}})$ and $d_{TV}(\xi,\lambda)=O(\epsilon)$.
Moreover, it is easy to see that for every $n\ge1$ large enough there
exist $\{c_{u}\}_{u\in\Psi_{m-1}(n)}\subset[0,\infty)$ so that $h=\sum_{u\in\Psi_{m-1}(n)}c_{u}1_{[u]}$.
This completes the proof of the lemma.
\end{proof}
\begin{proof}[Proof of Lemma \ref{lem:ent > xi(Z)}]
Let $0<\epsilon,\rho<1$, let $k\ge1$ be large with respect to $\epsilon,\rho,$
let $B$ be a Borel subset of $\mathrm{Gr}_{m-1}(m)$ with $\beta(\tilde{B})>0$,
let $\xi\in\mathcal{M}(\Lambda^{\mathbb{N}})$ be with $\xi\ll\beta_{\tilde{B}}$,
let $W\in\mathrm{Gr}_{m-1}(m)$ be with $d_{\mathrm{Gr}_{m-1}}(W,B)\ge\rho$,
and let $n\ge1$ be large with respect to all previous parameters.

Let $\mathcal{U}$ be the set of words $u\in\Psi_{m-1}(n)$ so that,
\begin{itemize}
\item $\alpha_{m}(A_{u})<\alpha_{m-1}(A_{u})2^{-2k}$;
\item $\pi_{V}L_{m-1}(A_{u})\in\mathrm{Gr}_{m-1}(m)$;
\item $d_{\mathrm{Gr}_{m-1}}(\pi_{V}L_{m-1}(A_{u}),W)>\rho/2$.
\end{itemize}
Since $\chi_{m}<\chi_{m-1}$,
\begin{equation}
\underset{n'\rightarrow\infty}{\lim}\:\frac{\alpha_{m}(A_{\Psi_{m-1}(n';\omega)})}{\alpha_{m-1}(A_{\Psi_{m-1}(n';\omega)})}=0\text{ for }\beta\text{-a.e. }\omega,\label{eq:ration of sing vals --> 0}
\end{equation}
where the notation $\Psi_{m-1}(n';\omega)$ was introduced in Section
\ref{subsec:Symbolic-notations}. By (\ref{enu:def of L_m via L_m of matrices})
in Section \ref{subsec:Coding and Furstenberg maps},
\begin{equation}
L_{m-1}(\omega)=\underset{n'\rightarrow\infty}{\lim}\:L_{m-1}(A_{\Psi_{m-1}(n';\omega)})\text{ for }\beta\text{-a.e. }\omega.\label{eq:L_m-1 as lim with Psi}
\end{equation}
By Lemma \ref{lem:nu=00007Bkappa<del=00007D<eps} and since $L_{m-1}\beta=\nu_{m-1}$,
\begin{equation}
L_{m-1}(\omega)\cap\ker\pi_{V}=\{0\}\text{ for }\beta\text{-a.e. }\omega.\label{eq:L_m-1 cup ker =00003D triv}
\end{equation}
From $d_{\mathrm{Gr}_{m-1}}(W,B)\ge\rho$,
\[
d_{\mathrm{Gr}_{m-1}}(\pi_{V}L_{m-1}(\omega),W)\ge\rho\text{ for }\beta_{\tilde{B}}\text{-a.e. }\omega.
\]
From all of this, since $\xi\ll\beta_{\tilde{B}}$, and by assuming
that $n$ is sufficiently large, we get $\xi(\mathcal{U})>1-\epsilon$.

Write $\mathcal{Z}$ in place of $\mathcal{Z}(\epsilon,k,n)$ for
the rest of the proof. Let $u\in\mathcal{U}\cap\mathcal{Z}$ be given,
and set $\theta:=\varphi_{u}\mu$ and $L:=L_{m-1}(A_{u})$. Let $D\in\mathcal{D}_{n}^{d}$
be such that $\theta(D)>0$ and
\[
\frac{1}{k}H\left(\pi_{V}\theta_{D},\mathcal{D}_{n+k}\right)>m-1-\epsilon.
\]
From $u\in\Psi_{m-1}(n)$ and $\alpha_{m}(A_{u})<\alpha_{m-1}(A_{u})2^{-2k}$,
and by assuming that $k$ is sufficiently large, it follows that $\theta$
is supported on $x+L^{(2^{-n-k})}$ for some $x\in\mathbb{R}^{d}$.
Thus, $\pi_{V}\theta_{D}$ is supported on $\pi_{V}x+(\pi_{V}L)^{(2^{-n-k})}$.
From this, since
\[
d_{\mathrm{Gr}_{1}}((\pi_{V}L)^{\perp},W^{\perp})=d_{\mathrm{Gr}_{m-1}}(\pi_{V}L,W)>\rho/2,
\]
and by Lemma \ref{lem:gen proj lb},
\begin{eqnarray*}
\frac{1}{k}H\left(\pi_{V}\theta_{D},\pi_{W^{\perp}}^{-1}\mathcal{D}_{n+k}\mid\mathcal{D}_{n}\right) & \ge & \frac{1}{k}H\left(\pi_{V}\theta_{D},\pi_{W^{\perp}}^{-1}\mathcal{D}_{n+k}\right)-\epsilon\\
 & \ge & \frac{1}{k}H\left(\pi_{V}\theta_{D},\mathcal{D}_{n+k}\right)-m+2-2\epsilon\\
 & > & 1-3\epsilon.
\end{eqnarray*}
Thus, by the concavity of conditional entropy and since (\ref{eq:def prop of Z(epsi,k,n)})
holds for $\theta$,
\[
\frac{1}{k}H\left(\pi_{V}\theta,\pi_{W^{\perp}}^{-1}\mathcal{D}_{n+k}\mid\mathcal{D}_{n}\right)>(1-\epsilon)(1-3\epsilon).
\]
We have thus shown that,
\begin{equation}
\frac{1}{k}H\left(\pi_{V}\varphi_{u}\mu,\pi_{W^{\perp}}^{-1}\mathcal{D}_{n+k}\mid\mathcal{D}_{n}\right)>1-O(\epsilon)\text{ for all }u\in\mathcal{U}\cap\mathcal{Z}.\label{eq:lb ent all u in U cap Z}
\end{equation}

Since $\xi\ll\beta$ and by Lemma \ref{lem:by the SW thm}, there
exist $\{c_{u}\}_{u\in\Psi_{m-1}(n)}\subset[0,\infty)$ so that for
$h:=\sum_{u\in\Psi_{m-1}(n)}c_{u}1_{[u]}$ and $\lambda:=h\:d\beta$
we have $\lambda\in\mathcal{M}(\Lambda^{\mathbb{N}})$ and $d_{TV}(\xi,\lambda)<\epsilon$.
Note that
\begin{equation}
\Pi\lambda=\sum_{u\in\Psi_{m-1}(n)}c_{u}p_{u}\Pi\beta_{[u]}=\sum_{u\in\Psi_{m-1}(n)}c_{u}p_{u}\cdot\varphi_{u}\mu,\label{eq:Pi lambda =00003D}
\end{equation}
and
\begin{equation}
\lambda(\mathcal{U}\cap\mathcal{Z})=\sum_{u\in\Psi_{m-1}(n)}c_{u}p_{u}\beta_{[u]}(\mathcal{U}\cap\mathcal{Z})=\sum_{u\in\mathcal{U}\cap\mathcal{Z}}c_{u}p_{u}.\label{eq:lambda(U cap Z) =00003D}
\end{equation}
Thus, from the concavity of conditional entropy and (\ref{eq:lb ent all u in U cap Z}),
\begin{eqnarray}
\frac{1}{k}H\left(\pi_{V}\Pi\lambda,\pi_{W^{\perp}}^{-1}\mathcal{D}_{n+k}\mid\mathcal{D}_{n}\right) & \ge & \sum_{u\in\Psi_{m-1}(n)}c_{u}p_{u}\cdot\frac{1}{k}H\left(\pi_{V}\varphi_{u}\mu,\pi_{W^{\perp}}^{-1}\mathcal{D}_{n+k}\mid\mathcal{D}_{n}\right)\nonumber \\
 & \ge & (1-O(\epsilon))\sum_{u\in\mathcal{U}\cap\mathcal{Z}}c_{u}p_{u}\label{eq:lb ent of Pi lambda}\\
 & = & (1-O(\epsilon))\lambda(\mathcal{U}\cap\mathcal{Z}).\nonumber 
\end{eqnarray}
Additionally, from $d_{TV}(\xi,\lambda)<\epsilon$ and $\xi(\mathcal{U})>1-\epsilon$,
\[
\lambda(\mathcal{U}\cap\mathcal{Z})>\xi(\mathcal{U}\cap\mathcal{Z})-\epsilon\ge\xi(\mathcal{Z})-2\epsilon.
\]
From this and (\ref{eq:lb ent of Pi lambda}), since
\begin{equation}
d_{TV}(\pi_{V}\Pi\xi,\pi_{V}\Pi\lambda)\le d_{TV}(\xi,\lambda)<\epsilon,\label{eq:d_TV of proj}
\end{equation}
and by Lemma \ref{lem:ent cont wrt tot var}, we obtain
\[
\frac{1}{k}H\left(\pi_{V}\Pi\xi,\pi_{W^{\perp}}^{-1}\mathcal{D}_{n+k}\mid\mathcal{D}_{n}\right)\ge(1-O(\epsilon))(\xi(\mathcal{Z})-2\epsilon)-O(\epsilon),
\]
which completes the proof of the lemma.
\end{proof}
\begin{proof}[Proof of Lemma \ref{lem:cond ent > (m-1)xi(Z)}]
Let $\epsilon>0$, let $k\ge1$ be large with respect to $\epsilon$,
let $\delta>0$ be small with respect to $k$, let $B$ be a Borel
subset of $\mathrm{Gr}_{m-1}(m)$ with $\beta(\tilde{B})>0$ and $\mathrm{diam}(B)\le\delta$,
let $\xi\in\mathcal{M}(\Lambda^{\mathbb{N}})$ be with $\xi\ll\beta_{\tilde{B}}$,
let $W\in B$, and let $n\ge1$ be large with respect to all previous
parameters.

Let $\mathcal{U}$ be the set of words $u\in\Psi_{m-1}(n)$ so that,
\begin{itemize}
\item $\alpha_{m}(A_{u})<\alpha_{m-1}(A_{u})2^{-3k}$;
\item $\pi_{V}L_{m-1}(A_{u})\in\mathrm{Gr}_{m-1}(m)$;
\item $d_{\mathrm{Gr}_{m-1}}(\pi_{V}L_{m-1}(A_{u}),W)<2\delta$.
\end{itemize}
Since $\mathrm{diam}(B)\le\delta$ and $W\in B$,
\[
d_{\mathrm{Gr}_{m-1}}(\pi_{V}L_{m-1}(\omega),W)\le\delta\text{ for }\beta_{\tilde{B}}\text{-a.e. }\omega.
\]
From this, from (\ref{eq:ration of sing vals --> 0}), (\ref{eq:L_m-1 as lim with Psi})
and (\ref{eq:L_m-1 cup ker =00003D triv}), since $\xi\ll\beta_{\tilde{B}}$,
and by assuming that $n$ is sufficiently large, we get $\xi(\mathcal{U})>1-\epsilon$.

Write $\mathcal{Z}:=\mathcal{Z}(\epsilon,k,n)$ for the rest of the
proof. Let $u\in\mathcal{U}\cap\mathcal{Z}$ be given, and set $\theta:=\varphi_{u}\mu$
and $L:=L_{m-1}(A_{u})$. Let $D\in\mathcal{D}_{n}^{d}$ be such that
$\theta(D)>0$ and
\begin{equation}
\frac{1}{k}H\left(\pi_{V}\theta_{D},\mathcal{D}_{n+k}\right)>m-1-\epsilon.\label{eq:H(pi theta_D) > m-1-eps}
\end{equation}
From $u\in\Psi_{m-1}(n)$ and $\alpha_{m}(A_{u})<\alpha_{m-1}(A_{u})2^{-3k}$,
and by assuming that $k$ is sufficiently large, it follows that $\theta$
is supported on $x+L^{(2^{-n-2k})}$ for some $x\in\mathbb{R}^{d}$.
Thus, $\pi_{V}\theta_{D}$ is supported on $\pi_{V}x+(\pi_{V}L)^{(2^{-n-2k})}$.
From $d_{\mathrm{Gr}_{m-1}}(\pi_{V}L,W)<2\delta$ and $\mathrm{diam}(\mathrm{supp}(\pi_{V}\theta_{D}))=O(2^{-n})$, and since $\delta$ is small with respect to $k$,
we obtain that $\pi_{V}\theta_{D}$ is supported on $y+W^{(2^{-n-k})}$
for some $y\in\mathbb{R}^{m}$. This, together with $\mathrm{diam}(\mathrm{supp}(\pi_{V}\theta_{D}))=O(2^{-n})$
and (\ref{eq:H(pi theta_D) > m-1-eps}), gives
\[
\frac{1}{k}H\left(\pi_{V}\theta_{D},\mathcal{D}_{n+k}\mid\mathcal{D}_{n}\vee\pi_{W^{\perp}}^{-1}\mathcal{D}_{n+k}\right)>m-1-2\epsilon.
\]
Thus, by the concavity of conditional entropy and since (\ref{eq:def prop of Z(epsi,k,n)})
holds for $\theta$,
\[
\frac{1}{k}H\left(\pi_{V}\theta,\mathcal{D}_{n+k}\mid\mathcal{D}_{n}\vee\pi_{W^{\perp}}^{-1}\mathcal{D}_{n+k}\right)>(1-\epsilon)(m-1-2\epsilon).
\]
We have thus shown that,
\begin{equation}
\frac{1}{k}H\left(\pi_{V}\varphi_{u}\mu,\mathcal{D}_{n+k}\mid\mathcal{D}_{n}\vee\pi_{W^{\perp}}^{-1}\mathcal{D}_{n+k}\right)>m-1-O(\epsilon)\text{ for all }u\in\mathcal{U}\cap\mathcal{Z}.\label{eq:lb cond ent all u in U cap Z}
\end{equation}

The rest of the proof proceeds as in the proof of Lemma \ref{lem:ent > xi(Z)}.
Since $\xi\ll\beta$ and by Lemma \ref{lem:by the SW thm}, there
exist $\{c_{u}\}_{u\in\Psi_{m-1}(n)}\subset[0,\infty)$ so that for
$h:=\sum_{u\in\Psi_{m-1}(n)}c_{u}1_{[u]}$ and $\lambda:=h\:d\beta$
we have $\lambda\in\mathcal{M}(\Lambda^{\mathbb{N}})$ and $d_{TV}(\xi,\lambda)<\epsilon$.
By the concavity of conditional entropy, from (\ref{eq:Pi lambda =00003D}),
(\ref{eq:lambda(U cap Z) =00003D}) and (\ref{eq:lb cond ent all u in U cap Z}),
and since $d_{TV}(\xi,\lambda)<\epsilon$ and $\xi(\mathcal{U})>1-\epsilon$,
we get
\begin{multline*}
\frac{1}{k}H\left(\pi_{V}\Pi\lambda,\mathcal{D}_{n+k}\mid\mathcal{D}_{n}\vee\pi_{W^{\perp}}^{-1}\mathcal{D}_{n+k}\right)\\
\ge\sum_{u\in\Psi_{m-1}(n)}c_{u}p_{u}\cdot\frac{1}{k}H\left(\pi_{V}\varphi_{u}\mu,\mathcal{D}_{n+k}\mid\mathcal{D}_{n}\vee\pi_{W^{\perp}}^{-1}\mathcal{D}_{n+k}\right)\\
\ge\left(m-1-O(\epsilon)\right)\lambda(\mathcal{U}\cap\mathcal{Z})
\ge\left(m-1-O(\epsilon)\right)\left(\xi(\mathcal{Z})-2\epsilon\right).
\end{multline*}
Thus, from (\ref{eq:d_TV of proj}) and by Lemma \ref{lem:ent cont wrt tot var},
it follows that
\[
\frac{1}{k}H\left(\pi_{V}\Pi\xi,\mathcal{D}_{n+k}\mid\mathcal{D}_{n}\vee\pi_{W^{\perp}}^{-1}\mathcal{D}_{n+k}\right)\ge\left(m-1-O(\epsilon)\right)\left(\xi(\mathcal{Z})-2\epsilon\right)-O(\epsilon),
\]
which completes the proof of the lemma.
\end{proof}
For the proof of Proposition \ref{prop:sing proj of meas} we also
need the following simple lemmas.
\begin{lem}
\label{lem:decomp of xi}Let $X$ be a measurable space, let $\sigma,\xi$
be positive and finite measures on $X$, and let $\sigma_{1},...,\sigma_{k}$
be positive measures on $X$ with $\sigma=\sum_{i=1}^{k}\sigma_{i}$.
Then there exist positive measures $\xi_{1},...,\xi_{k+1}$ on $X$
so that $\xi=\sum_{i=1}^{k+1}\xi_{i}$ and $\sum_{i=1}^{k}d_{TV}(\sigma_{i},\xi_{i})=\frac{1}{2}d_{TV}(\sigma,\xi)$.
\end{lem}

\begin{proof}
Write $\lambda:=\sigma+\xi$. For a positive measure $\zeta$ on $X$
with $\zeta\ll\lambda$, denote by $h_{\zeta}\ge0$ the Radon-Nikodym
derivative of $\zeta$ with respect to $\lambda$. Set $h_{1}:=\min\{h_{\xi},h_{\sigma_{1}}\}$,
for $1<i\le k$ set
\[
h_{i}:=\min\{h_{\xi}-\sum_{j=1}^{i-1}h_{j},h_{\sigma_{i}}\},
\]
let $h_{k+1}:=h_{\xi}-\sum_{i=1}^{k}h_{i}$, and for $1\le i\le k+1$
write $\xi_{i}:=h_{i}\:d\lambda$.
It is clear that $\xi_{1},...,\xi_{k+1}$
are positive measures on $X$. Since $\sum_{i=1}^{k+1}h_{i}=h_{\xi}$,
we have $\sum_{i=1}^{k+1}\xi_{i}=\xi$.

Write
\[
E:=\{x\in X\::\:h_{\sigma}(x)\ge h_{\xi}(x)\},
\]
and note that $\sum_{i=1}^{k}h_{i}(x)=h_{\xi}(x)$ for $x\in E$ and
$h_{i}(x)=h_{\sigma_{i}}(x)$ for $x\in X\setminus E$ and $1\le i\le k$.
Thus,
\begin{eqnarray*}
\sum_{i=1}^{k}d_{TV}(\sigma_{i},\xi_{i}) & = & \sum_{i=1}^{k}\int|h_{\sigma_{i}}(x)-h_{i}(x)|\:d\lambda(x)\\
 & = & \int_{E}h_{\sigma}(x)-h_{\xi}(x)\:d\lambda(x)=\frac{1}{2}d_{TV}(\sigma,\xi),
\end{eqnarray*}
which completes the proof of the lemma.
\end{proof}
Recall that for $F\subset\mathbb{R}^{m}$ we write $\tilde{F}:=\Pi^{-1}\pi_{V}^{-1}F$.
\begin{lem}
\label{lem:existance of measures with small tv dist}Let $B_{1}$
and $B_{2}$ be Borel subsets of $\mathrm{Gr}_{m-1}(m)$ so that $\beta(\tilde{B}_{1}),\beta(\tilde{B}_{2})>0$
and the measures $\theta_{1}:=\pi_{V}\Pi\beta_{\tilde{B}_{1}}$ and
$\theta_{2}:=\pi_{V}\Pi\beta_{\tilde{B}_{2}}$ are not singular. Then
for each $\epsilon>0$ there exists a Borel set $F\subset\mathbb{R}^{m}$
with $\beta(\tilde{B}_{i}\cap\tilde{F})>0$ for $i=1,2$ and
\[
d_{TV}\left(\pi_{V}\Pi\beta_{\tilde{B}_{1}\cap\tilde{F}},\pi_{V}\Pi\beta_{\tilde{B}_{2}\cap\tilde{F}}\right)<\epsilon.
\]
\end{lem}

\begin{proof}
Let $0<\epsilon<1$ be given. Since $\theta_{1}$ and $\theta_{2}$
are not singular, and by the Lebesgue--Radon--Nikodym theorem, there
exist a Borel set $E\subset\mathbb{R}^{m}$ and a Borel function $f:\mathbb{R}^{m}\rightarrow[0,\infty)$
so that $f=0$ on $\mathbb{R}^{m}\setminus E$, $\theta_{i}(E)>0$
for $i=1,2$, and $\theta_{2}|_{E}=f\:d\theta_{1}$. Fix
\[
t\in(0,\infty)\cap\mathrm{supp}(f\theta_{1}),
\]
where, as always, $f\theta_{1}$ denotes the pushforward of $\theta_{1}$
via $f$. Let $0<\delta<t\epsilon/4$ and set $F:=f^{-1}(t-\delta,t+\delta)$.

Since $\delta<t/2$ we have $f(x)>t/2$ for $x\in F$, and so $\theta_{2}(F)\ge t\theta_{1}(F)/2$.
Hence, for all $x\in F$
\[
\left|\frac{1}{\theta_{1}(F)}-\frac{f(x)}{\theta_{2}(F)}\right|\le\frac{1}{\theta_{1}(F)\theta_{2}(F)}\int_{F}|f(y)-f(x)|\:d\theta_{1}(y)\le\frac{2\delta}{\theta_{2}(F)}\le\frac{4\delta}{t\theta_{1}(F)}.
\]
Thus, from $(\theta_{1})_{F}=\frac{1_{F}}{\theta_{1}(F)}\:d\theta_{1}$
and $(\theta_{2})_{F}=\frac{f1_{F}}{\theta_{2}(F)}\:d\theta_{1}$,
we obtain
\[
d_{TV}\left((\theta_{1})_{F},(\theta_{2})_{F}\right)=\int_{F}\left|\frac{1}{\theta_{1}(F)}-\frac{f(x)}{\theta_{2}(F)}\right|\:d\theta_{1}(x)\le4\delta/t<\epsilon.
\]
Since $(\theta_{i})_{F}=\pi_{V}\Pi\beta_{\tilde{B}_{i}\cap\tilde{F}}$
for $i=1,2$, this completes the proof of the lemma.
\end{proof}

\subsubsection{Proofs of Proposition \ref{prop:sing proj of meas} and Theorem \ref{thm:factor of proj of L_m-1}}
\begin{proof}[Proof of Proposition \ref{prop:sing proj of meas}]
Let $B_{1}$ and $B_{2}$ be Borel subsets of $\mathrm{Gr}_{m-1}(m)$
with $\beta(\tilde{B}_{1}),\beta(\tilde{B}_{2})>0$ and $d_{\mathrm{Gr}_{m-1}}(B_{1},B_{2})>0$.
Assume by contradiction that $\pi_{V}\Pi\beta_{\tilde{B}_{1}}$ and
$\pi_{V}\Pi\beta_{\tilde{B}_{2}}$ are not singular. Recall that $\dim\pi_{V}\mu<m$,
and let $\epsilon>0$ be small with respect to $m-\dim\pi_{V}\mu$.
By Lemma \ref{lem:existance of measures with small tv dist}, there
exists a Borel set $F\subset\mathbb{R}^{m}$ with $\beta(\tilde{B}_{i}\cap\tilde{F})>0$
for $i=1,2$ and
\begin{equation}
d_{TV}\left(\pi_{V}\Pi\beta_{\tilde{B}_{1}\cap\tilde{F}},\pi_{V}\Pi\beta_{\tilde{B}_{2}\cap\tilde{F}}\right)<\epsilon.\label{eq:d_TV(,)<epsi}
\end{equation}

Let $k$ be large with respect to $\epsilon$, $\beta(\tilde{B}_{1}\cap\tilde{F})$,
$\beta(\tilde{B}_{2}\cap\tilde{F})$ and $d_{\mathrm{Gr}_{m-1}}(B_{1},B_{2})$,
and let $\delta>0$ be small with respect to $k$. Let $\{C_{1},...,C_{s}\}$
be a Borel partition of $B_{1}$ such that for each $1\le j\le s$
we have $\mathrm{diam}(C_{j})<\delta$. For $1\le j\le s$ set
\[
\sigma_{j}:=\frac{1}{\beta(\tilde{B}_{1}\cap\tilde{F})}\pi_{V}\Pi\beta|_{\tilde{C}_{j}\cap\tilde{F}},
\]
so that $\pi_{V}\Pi\beta_{\tilde{B}_{1}\cap\tilde{F}}=\sum_{j=1}^{s}\sigma_{j}$.
Thus, by Lemma \ref{lem:decomp of xi} and (\ref{eq:d_TV(,)<epsi}),
there exist positive Borel measures $\xi_{1},...,\xi_{s+1}$ on $\mathbb{R}^{m}$
such that $\pi_{V}\Pi\beta_{\tilde{B}_{2}\cap\tilde{F}}=\sum_{j=1}^{s+1}\xi_{j}$
and $\sum_{j=1}^{s}d_{TV}(\sigma_{j},\xi_{j})<\epsilon/2$.

For $1\le j\le s+1$ denote by $h_{j}$ the Radon-Nykodym derivative
of $\xi_{j}$ with respect to $\pi_{V}\Pi\beta_{\tilde{B}_{2}\cap\tilde{F}}$,
and set
\[
\tilde{\xi}_{j}:=h_{j}\circ\pi_{V}\circ\Pi\:d\beta_{\tilde{B}_{2}\cap\tilde{F}}.
\]
It is easy to verify that $\pi_{V}\Pi\tilde{\xi}_{j}=\xi_{j}$ and
$\beta_{\tilde{B}_{2}\cap\tilde{F}}=\sum_{j=1}^{s+1}\tilde{\xi}_{j}$.

Recall from Section \ref{subsec:General-notations} that $\Vert\cdot\Vert_{TV}$
denote the total variation norm. Let $J$ be the set of all $1\le j\le s$
such that $\Vert\sigma_{j}\Vert_{TV}>0$ and $\Vert\xi_{j}\Vert_{TV}>0$.
For each $j\in J$ choose some $W_{j}\in C_{j}$, and set
\[
\sigma_{j}':=\frac{1}{\Vert\sigma_{j}\Vert_{TV}}\sigma_{j}\;\text{ and }\;\xi_{j}':=\frac{1}{\Vert\xi_{j}\Vert_{TV}}\xi_{j}.
\]
Let $N\ge1$ be large with respect to all previous objects in the
proof and let $n\ge N$.

Let $j\in J$ be given. By Lemma \ref{lem:cond ent > (m-1)xi(Z)}
and since $\mathrm{diam}(C_{j})<\delta$ and $\sigma_{j}'=\pi_{V}\Pi\beta_{\tilde{C}_{j}\cap\tilde{F}}$,
\begin{equation}
\frac{1}{k}H\left(\sigma_{j}',\mathcal{D}_{n+k}\mid\mathcal{D}_{n}\vee\pi_{W_{j}^{\perp}}^{-1}\mathcal{D}_{n+k}\right)>(m-1)\beta_{\tilde{C}_{j}\cap\tilde{F}}(\mathcal{Z}(\epsilon,k,n))-O(\epsilon).\label{eq:>m-1 all j in J}
\end{equation}
Since $W_{j}\in C_{j}\subset B_{1}$,
\[
d_{\mathrm{Gr}_{m-1}}(W_{j},B_{2})\ge d_{\mathrm{Gr}_{m-1}}(B_{1},B_{2}).
\]
From this, since $k$ is large with respect to $d_{\mathrm{Gr}_{m-1}}(B_{1},B_{2})$,
from $\xi_{j}'=\Vert\xi_{j}\Vert_{TV}^{-1}\pi_{V}\Pi\tilde{\xi}_{j}$
and $\tilde{\xi}_{j}\ll\beta_{\tilde{B}_{2}}$, and by Lemma \ref{lem:ent > xi(Z)},
\[
\frac{1}{k}H\left(\xi_{j}',\pi_{W_{j}^{\perp}}^{-1}\mathcal{D}_{n+k}\mid\mathcal{D}_{n}\right)>\Vert\xi_{j}\Vert_{TV}^{-1}\tilde{\xi}_{j}(\mathcal{Z}(\epsilon,k,n))-O(\epsilon).
\]
Combining this with Lemma \ref{lem:ent cont wrt tot var}, we obtain
\begin{equation}
\frac{1}{k}H\left(\sigma_{j}',\pi_{W_{j}^{\perp}}^{-1}\mathcal{D}_{n+k}\mid\mathcal{D}_{n}\right)>\Vert\xi_{j}\Vert_{TV}^{-1}\tilde{\xi}_{j}(\mathcal{Z}(\epsilon,k,n))-O(\epsilon+d_{TV}(\sigma_{j}',\xi_{j}')).\label{eq:>xi_j(Z) - O(d_TV) all j in J}
\end{equation}

For $n'\ge1$ set $\mathcal{Y}_{n'}:=\Psi_{m-1}(n')\setminus\mathcal{Z}(\epsilon,k,n')$.
As pointed out in Section \ref{subsec:Entropy-in-Rd}, the partitions
$\mathcal{D}_{n+k}^{m}$ and $\pi_{W_{j}}^{-1}\mathcal{D}_{n+k}^{m-1}\vee\pi_{W_{j}^{\perp}}^{-1}\mathcal{D}_{n+k}^{1}$
are $O(1)$-commensurable. From this, by the conditional entropy formula
(see (\ref{eq:extended cond ent form})), and from (\ref{eq:>m-1 all j in J})
and (\ref{eq:>xi_j(Z) - O(d_TV) all j in J}), we get 
\[
\frac{1}{k}H\left(\sigma_{j}',\mathcal{D}_{n+k}\mid\mathcal{D}_{n}\right)\ge m-O\left(\beta_{\tilde{C}_{j}\cap\tilde{F}}(\mathcal{Y}_{n})+\Vert\xi_{j}\Vert_{TV}^{-1}\tilde{\xi}_{j}(\mathcal{Y}_{n})+\epsilon+d_{TV}(\sigma_{j}',\xi_{j}')\right).
\]

Write $J^{c}:=\{1,...,s\}\setminus J$, and note that the last inequality
holds for each $j\in J$. Thus, since $\pi_{V}\Pi\beta_{\tilde{B}_{1}\cap\tilde{F}}=\sum_{j=1}^{s}\sigma_{j}$
and by the concavity of conditional entropy,
\begin{eqnarray}
\frac{1}{k}H\left(\pi_{V}\Pi\beta_{\tilde{B}_{1}\cap\tilde{F}},\mathcal{D}_{n+k}\mid\mathcal{D}_{n}\right) & \ge & \sum_{j\in J}\Vert\sigma_{j}\Vert_{TV}\frac{1}{k}H\left(\sigma_{j}',\mathcal{D}_{n+k}\mid\mathcal{D}_{n}\right)\nonumber \\
 & \ge & m-O\left(\sum_{j\in J^{c}}\Vert\sigma_{j}\Vert_{TV}+\sum_{j\in J}\Vert\sigma_{j}\Vert_{TV}d_{TV}(\sigma_{j}',\xi_{j}')\right)\label{eq:>=00003D m- O()-O()}\\
 & - & O\left(\epsilon+\sum_{j\in J}\Vert\sigma_{j}\Vert_{TV}\left(\beta_{\tilde{C}_{j}\cap\tilde{F}}(\mathcal{Y}_{n})+\Vert\xi_{j}\Vert_{TV}^{-1}\tilde{\xi}_{j}(\mathcal{Y}_{n})\right)\right)\nonumber \\
 & =: & m-O(\gamma_{1})-O(\gamma_{2}).\nonumber 
\end{eqnarray}

Recall that by definition $d_{TV}(\theta_{1},\theta_{2})=\Vert\theta_{1}-\theta_{2}\Vert_{TV}$
for positive measures $\theta_{1},\theta_{2}$. From this and since
$\sum_{j=1}^{s}d_{TV}(\sigma_{j},\xi_{j})<\epsilon/2$,
\begin{eqnarray*}
\gamma_{1} & \le & \sum_{j\in J^{c}}d_{TV}(\sigma_{j},\xi_{j})\\
 & + & \sum_{j\in J}\Vert\sigma_{j}\Vert_{TV}\left(d_{TV}\left(\sigma_{j}',\frac{1}{\Vert\sigma_{j}\Vert_{TV}}\xi_{j}\right)+d_{TV}\left(\frac{1}{\Vert\sigma_{j}\Vert_{TV}}\xi_{j},\xi_{j}'\right)\right)\\
 & = & \sum_{j=1}^{s}d_{TV}(\sigma_{j},\xi_{j})+\sum_{j\in J}\Vert\sigma_{j}\Vert_{TV}\Vert\xi_{j}\Vert_{TV}\left|\frac{1}{\Vert\sigma_{j}\Vert_{TV}}-\frac{1}{\Vert\xi_{j}\Vert_{TV}}\right|\\
 & < & \epsilon/2+\sum_{j\in J}d_{TV}(\sigma_{j},\xi_{j})<\epsilon.
\end{eqnarray*}

Next, let $\mathcal{Q}$ be the set of all integers $n'\ge1$ so that
\[
\beta(\mathcal{Y}_{n'})<\epsilon\cdot\min\{\beta(\tilde{B}_{1}\cap\tilde{F}),\beta(\tilde{B}_{2}\cap\tilde{F})\}.
\]
Assuming $n\in\mathcal{Q}$, we have $\beta_{\tilde{B}_{i}\cap\tilde{F}}(\mathcal{Y}_{n})<\epsilon$
for $i=1,2$. Thus, by the definition of the measures $\sigma_{j}$
and since $\{C_{1},...,C_{s}\}$ is a partition of $B_{1}$,
\[
\sum_{j\in J}\Vert\sigma_{j}\Vert_{TV}\beta_{\tilde{C}_{j}\cap\tilde{F}}(\mathcal{Y}_{n})\le\beta_{\tilde{B}_{1}\cap\tilde{F}}(\mathcal{Y}_{n})<\epsilon.
\]
Moreover, since $\Vert\xi_{j}\Vert_{TV}^{-1}\tilde{\xi}_{j}(\mathcal{Y}_{n})\le1$
for $j\in J$ and since $\beta_{\tilde{B}_{2}\cap\tilde{F}}=\sum_{j=1}^{s+1}\tilde{\xi}_{j}$,
\begin{eqnarray*}
\sum_{j\in J}\Vert\sigma_{j}\Vert_{TV}\Vert\xi_{j}\Vert_{TV}^{-1}\tilde{\xi}_{j}(\mathcal{Y}_{n}) & \le & \sum_{j\in J}\left|\Vert\sigma_{j}\Vert_{TV}-\Vert\xi_{j}\Vert_{TV}\right|+\sum_{j\in J}\tilde{\xi}_{j}(\mathcal{Y}_{n})\\
 & \le & \sum_{j\in J}d_{TV}(\sigma_{j},\xi_{j})+\beta_{\tilde{B}_{2}\cap\tilde{F}}(\mathcal{Y}_{n})<2\epsilon.
\end{eqnarray*}

From the last two inequalities we obtain that $\gamma_{2}<4\epsilon$ if $n\in\mathcal{Q}$.
From this, $\gamma_{1}<\epsilon$ and (\ref{eq:>=00003D m- O()-O()}),
we get
\begin{equation}
\frac{1}{k}H\left(\pi_{V}\Pi\beta_{\tilde{B}_{1}\cap\tilde{F}},\mathcal{D}_{n+k}\mid\mathcal{D}_{n}\right)\ge m-O(\epsilon)\text{ for all }n\ge N\text{ with }n\in\mathcal{Q}.\label{eq:cond ent >=00003Dm-O(ep)}
\end{equation}
Moreover, since $\Sigma_{1}^{m-1}=m-1$, by Proposition \ref{prop:inductive main proj prop} and (\ref{eq:conn bet I_m(n) & beta}),
and since $k$ is large with respect to $\epsilon$, $\beta(\tilde{B}_{1}\cap\tilde{F})$
and $\beta(\tilde{B}_{2}\cap\tilde{F})$, we may assume that
\begin{equation}
\underset{M\rightarrow\infty}{\liminf}\:\lambda_{M}(\mathcal{Q}\cap\mathcal{N}_{N,M})>1-\epsilon,\label{eq:large mass for mathcal(Q)}
\end{equation}
where $\lambda_{M}$ and $\mathcal{N}_{N,M}$ are as defined in Section
\ref{subsec:Component-measures}.

Now recall that $\pi_{V}\mu=\pi_{V}\Pi\beta$ is exact dimensional
with $\dim\pi_{V}\mu<m$. Since $\pi_{V}\Pi\beta_{\tilde{B}_{1}\cap\tilde{F}}\ll\pi_{V}\Pi\beta$,
it follows that $\pi_{V}\Pi\beta_{\tilde{B}_{1}\cap\tilde{F}}$ is
also exact dimensional with dimension equal to $\dim\pi_{V}\mu$.
From this, by Lemmata \ref{lem:dim_e=00003Ddim} and \ref{lem:multiscale-entropy-formula},
and from (\ref{eq:cond ent >=00003Dm-O(ep)}) and (\ref{eq:large mass for mathcal(Q)}),
we obtain
\[
\dim\pi_{V}\mu=\dim_{e}(\pi_{V}\Pi\beta_{\tilde{B}_{1}\cap\tilde{F}})\ge m-O(\epsilon).
\]
By taking $\epsilon$ to be sufficiently small with respect to $m-\dim\pi_{V}\mu$,
this yields a contradiction and completes the proof of the proposition.
\end{proof}
We can now prove Theorem \ref{thm:factor of proj of L_m-1}, whose
statement we first recall.
\begin{thm*}
Recall that $2\le m\le d$ and $V\in\mathrm{Gr}_{m}(d)$ are such
that $\Sigma_{1}^{m-1}=m-1$ and $\pi_{V}\mu$ is exact dimensional
with $\dim\pi_{V}\mu<m$. Then $\pi_{V}L_{m-1}\beta_{\omega}^{V}=\delta_{\pi_{V}L_{m-1}(\omega)}$
for $\beta$-a.e. $\omega$.
\end{thm*}
\begin{proof}
Given that we have already established Proposition \ref{prop:sing proj of meas},
the proof is now similar to the proof of \cite[Proposition 4.4]{HR}.
Clearly it suffices to show that $\pi_{V}L_{m-1}\beta_{\omega}^{V}$
is a mass point for $\beta$-a.e. $\omega$.

It follows easily from Proposition \ref{prop:sing proj of meas} that
there exist sequences $\{B_{1,k}\}_{k\ge1}$ and $\{B_{2,k}\}_{k\ge1}$
so that,
\begin{enumerate}
\item $B_{i,k}$ is an open subset of $\mathrm{Gr}_{m-1}(\mathbb{R}^{m})$
with $\beta(\tilde{B}_{i,k})>0$ for $i=1,2$ and $k\ge1$;
\item for all distinct $W_{1},W_{2}\in\mathrm{supp}(\pi_{V}L_{m-1}\beta)$
there exists $k\ge1$ with $W_{i}\in B_{i,k}$ for $i=1,2$;
\item the measures $\pi_{V}\Pi\beta_{\tilde{B}_{1,k}}$ and $\pi_{V}\Pi\beta_{\tilde{B}_{2,k}}$
are singular for $k\ge1$.
\end{enumerate}

For each $k\ge1$ there exists a Borel set $E_{k}\subset\mathbb{R}^{m}$
with $\pi_{V}\Pi\beta_{\tilde{B}_{1,k}}(E_{k})=0$ and $\pi_{V}\Pi\beta_{\tilde{B}_{2,k}}(\mathbb{R}^{m}\setminus E_{k})=0$.
Recall that $\mathcal{B}_{\mathbb{R}^{m}}$ is the Borel $\sigma$-algebra
of $\mathbb{R}^{m}$, and note that
\[
\tilde{E}_{k}\in\Pi^{-1}\pi_{V}^{-1}(\mathcal{B}_{\mathbb{R}^{m}})=\Pi^{-1}P_{V}^{-1}(\mathcal{B}_{\mathbb{R}^{d}}).
\]
Thus, by basic properties of disintegrations (see Section \ref{subsec:Disintegrations}),
\[
0=\beta(\tilde{B}_{1,k})\cdot\pi_{V}\Pi\beta_{\tilde{B}_{1,k}}(E_{k})=\beta(\tilde{B}_{1,k}\cap\tilde{E}_{k})=\int_{\tilde{E}_{k}}\beta_{\omega}^{V}(\tilde{B}_{1,k})\:d\beta(\omega).
\]
And similarly,
\[
0=\int_{\Lambda^{\mathbb{N}}\setminus\tilde{E}_{k}}\beta_{\omega}^{V}(\tilde{B}_{2,k})\:d\beta(\omega).
\]
Hence, for $\beta$-a.e. $\omega$
\begin{equation}
\beta_{\omega}^{V}(\tilde{B}_{1,k})=0\:\text{ or }\:\beta_{\omega}^{V}(\tilde{B}_{2,k})=0\:\text{ for all }k\ge1.\label{eq:mass(B_1,k)=00003D0 or mass(B_2,k)=00003D0}
\end{equation}
Additionally, it is clear that for $\beta$-a.e. $\omega$
\begin{equation}
\mathrm{supp}(\pi_{V}L_{m-1}\beta_{\omega}^{V})\subset\mathrm{supp}(\pi_{V}L_{m-1}\beta).\label{eq:supp(slice) sub supp(beta)}
\end{equation}

Fix $\omega\in\Lambda^{\mathbb{N}}$ for which (\ref{eq:mass(B_1,k)=00003D0 or mass(B_2,k)=00003D0})
and (\ref{eq:supp(slice) sub supp(beta)}) are both satisfied. Assume
by contradiction that $\pi_{V}L_{m-1}\beta_{\omega}^{V}$ is not a
mass point. Then there exist distinct 
\[
W_{1},W_{2}\in\mathrm{supp}(\pi_{V}L_{m-1}\beta_{\omega}^{V})\subset\mathrm{supp}(\pi_{V}L_{m-1}\beta),
\]
and so there exists $k\ge1$ with $W_{i}\in B_{i,k}$ for $i=1,2$.
Since $W_{1},W_{2}\in\mathrm{supp}(\pi_{V}L_{m-1}\beta_{\omega}^{V})$,
and since $B_{1,k}$ and $B_{2,k}$ are open in $\mathrm{Gr}_{m-1}(\mathbb{R}^{m})$,
we have $\pi_{V}L_{m-1}\beta_{\omega}^{V}(B_{i,k})>0$ for $i=1,2$.
This contradicts (\ref{eq:mass(B_1,k)=00003D0 or mass(B_2,k)=00003D0}),
which shows that $\pi_{V}L_{m-1}\beta_{\omega}^{V}$ must be a mass
point and completes the proof of the theorem.
\end{proof}

\subsection{\label{subsec:Factorization-of L_m-1 itself}Factorization of $L_{m-1}$}

In this subsection we use Theorem \ref{thm:factor of proj of L_m-1}
in order to deduce Theorem \ref{thm:L_m-1 factors via Pi}, which
is the following statement.
\begin{thm*}
Let $1\le m\le d$ and suppose that $\Sigma_{1}^{m-1}=m-1$ and $\Delta_{m}<1$.
Then there exists a Borel map $\widehat{L}_{m-1}:\mathbb{R}^{d}\rightarrow\mathrm{Gr}_{m-1}(d)$
so that $L_{m-1}(\omega)=\widehat{L}_{m-1}(\Pi\omega)$ for $\beta$-a.e.
$\omega$.
\end{thm*}
\begin{rem*}
In what follows we write $L_{m-1}$ in place of $\widehat{L}_{m-1}$.
Which $L_{m-1}$ is intended will be clear from the context.
\end{rem*}
\begin{proof}
Since $\mathrm{Gr}_{0}(d)$ consists of a single element, the theorem
holds trivially when $m=1.$ Thus, we may assume that $m\ge2$.

In order to prove the theorem it suffices to show that,
\begin{equation}
L_{m-1}(\eta)=L_{m-1}(\omega)\text{ for }\beta\text{-a.e. }\omega\text{ and }\beta_{\omega}^{\mathbb{R}^{d}}\text{-a.e. }\eta.\label{eq:facto of L_m-1 suffices to show}
\end{equation}

Since $\Delta_{m}<1$ we have $\Sigma_{1}^{m}<m$. Hence, by Theorem
\ref{thm:LY formula for SA}, for $\nu_{m}^{*}$-a.e. $V\in\mathrm{Gr}_{m}(d)$
the measure $\pi_{V}\mu$ is exact dimensional with $\dim\pi_{V}\mu<m$.
From this, since $\Sigma_{1}^{m-1}=m-1$, and by Theorem \ref{thm:factor of proj of L_m-1},
\begin{equation}
\pi_{V}(L_{m-1}(\eta))=\pi_{V}(L_{m-1}(\omega))\text{ for }\nu_{m}^{*}\text{-a.e. }V,\:\beta\text{-a.e. }\omega\text{ and }\beta_{\omega}^{V}\text{-a.e. }\eta.\label{eq:proj eq a.e. V}
\end{equation}
Note that (\ref{eq:facto of L_m-1 suffices to show}) follows directly
from (\ref{eq:proj eq a.e. V}) when $m=d$. Thus, we may assume that
$m<d$.

Let $\nu_{d-m}^{-1}$ be the Furstenberg measure on $\mathrm{Gr}_{d-m}(d)$
corresponding to $\sum_{i\in\Lambda}p_{i}\delta_{A_{i}^{-1}}$. That
is, $\nu_{d-m}^{-1}$ is the unique member of $\mathcal{M}(\mathrm{Gr}_{d-m}(d))$
satisfying
\[
\nu_{d-m}^{-1}=\sum_{i\in\Lambda}p_{i}\cdot A_{i}^{-1}\nu_{d-m}^{-1}.
\]
It is easy to verify that $\nu_{d-m}^{-1}$ is equal to pushforward of $\nu_{m}^{*}$ via the map sending $V\in\mathrm{Gr}_{m}(d)$ to $V^{\perp}$. That is, for every Borel set $B\subset\mathrm{Gr}_{d-m}(d)$
\begin{equation}
\nu_{d-m}^{-1}(B)=\nu_{m}^{*}\left\{ V\in\mathrm{Gr}_{m}(d)\::\:V^{\perp}\in B\right\} .\label{eq:char of nu^-1}
\end{equation}

Recall from Section \ref{subsec:spaces of alt forms} that for $1\le k\le d$
we denote by $\iota_{k}$ the Plücker embedding of $\mathrm{Gr}_{k}(d)$
into $\mathrm{P}(\wedge^{k}\mathbb{R}^{d})$. As in the proof of Lemma
\ref{lem:nu=00007Bkappa<del=00007D<eps}, since $\mathbf{S}_{\Phi}^{\mathrm{L}}$
is $(d-m)$-strongly irreducible and by \cite[Proposition III.2.3]{BL},
it follows that for every proper linear subspace $Z$ of $\wedge^{d-m}\mathbb{R}^{d}$
\[
\nu_{d-m}^{-1}\left\{ V\in\mathrm{Gr}_{d-m}(d)\::\:\iota_{d-m}(V)\subset Z\right\} =0.
\]
Hence, by (\ref{eq:char of nu^-1})
\begin{equation}
\nu_{m}^{*}\left\{ V\in\mathrm{Gr}_{m}(d)\::\:\iota_{d-m}(V^{\perp})\subset Z\right\} =0.\label{eq:nu_m^* of prop sub =00003D 0}
\end{equation}

From (\ref{eq:nu_m^* of prop sub =00003D 0}) and (\ref{eq:proj eq a.e. V}),
it follows that there exist $V_{1},...,V_{b}\in\mathrm{Gr}_{m}(d)$
and $\phi_{1},...,\phi_{b}\in\wedge^{d-m}\mathbb{R}^{d}\setminus\{0\}$
so that $\iota_{d-m}(V_{i}^{\perp})=\phi_{i}\mathbb{R}$ for $1\le i\le b$,
$\{\phi_{1},..,\phi_{b}\}$ is a basis for $\wedge^{d-m}\mathbb{R}^{d}$,
and
\begin{equation}
\pi_{V_{i}}(L_{m-1}(\eta))=\pi_{V_{i}}(L_{m-1}(\omega))\text{ for }1\le i\le b,\:\beta\text{-a.e. }\omega\text{ and }\beta_{\omega}^{V_{i}}\text{-a.e. }\eta.\label{eq:proj eq all i =0000231}
\end{equation}

By (\ref{eq:cond of cond =00003D cond of orig on sym space}),
\[
\beta_{\omega}^{V_{i}}=\int\beta_{\eta}^{\mathbb{R}^{d}}\:d\beta_{\omega}^{V_{i}}(\eta)\text{ for }\beta\text{-a.e. }\omega\text{ and }1\le i\le b.
\]
Moreover, by Lemma \ref{lem:nu=00007Bkappa<del=00007D<eps} and since
$L_{m-1}\beta=\nu_{m-1}$,
\[
\pi_{V_{i}}(L_{m-1}(\omega))\in\mathrm{Gr}_{m-1}(m)\text{ for }\beta\text{-a.e. }\omega\text{ and }1\le i\le b.
\]
From these facts together with (\ref{eq:proj eq all i =0000231}),
it follows easily that that for $\beta$-a.e. $\omega$ and $\beta_{\omega}^{\mathbb{R}^{d}}$-a.e. $\eta$
\begin{equation}
\pi_{V_{i}}(L_{m-1}(\eta))=\pi_{V_{i}}(L_{m-1}(\omega))\in\mathrm{Gr}_{m-1}(m)\text{ for }1\le i\le b.\label{eq:proj eq all i =0000232}
\end{equation}
Fix $\omega,\eta\in\Lambda^{\mathbb{N}}$ for which (\ref{eq:proj eq all i =0000232})
is satisfied. In order to complete the proof it suffices to show that
$L_{m-1}(\eta)=L_{m-1}(\omega)$.

Assume by contradiction that $L_{m-1}(\eta)\ne L_{m-1}(\omega)$,
and let $x\in L_{m-1}(\eta)\setminus L_{m-1}(\omega)$. Let $\gamma_{\omega},\gamma_{\eta}\in\wedge^{m-1}\mathbb{R}^{d}\setminus\{0\}$
be with $\iota_{m-1}(L_{m-1}(\omega))=\gamma_{\omega}\mathbb{R}$
and $\iota_{m-1}(L_{m-1}(\eta))=\gamma_{\eta}\mathbb{R}$. Since $x\notin L_{m-1}(\omega)$
and by (\ref{eq:equiv cond for lin indep}), there exist $\zeta\in\wedge^{d-m}\mathbb{R}^{d}$
so that $\gamma_{\omega}\wedge\zeta\wedge x\ne0$. Since $\{\phi_{1},..,\phi_{b}\}$
is a basis for $\wedge^{d-m}\mathbb{R}^{d}$, there exist $c_{1},...,c_{b}\in\mathbb{R}$
with $\zeta=\sum_{i=1}^{b}c_{i}\phi_{i}$. Given $1\le i\le b$, from
(\ref{eq:proj eq all i =0000232}) we obtain
\[
L_{m-1}(\eta)+V_{i}^{\perp}=L_{m-1}(\omega)+V_{i}^{\perp}\in\mathrm{Gr}_{d-1}(d).
\]
Thus, from $\iota_{d-m}(V_{i}^{\perp})=\phi_{i}\mathbb{R}$ and by
(\ref{eq:equiv cond for lin span}), it follows that there exists
$0\ne t_{i}\in\mathbb{R}$ so that $\gamma_{\omega}\wedge\phi_{i}=t_{i}\gamma_{\eta}\wedge\phi_{i}$.
Moreover, since $x\in L_{m-1}(\eta)$ and by (\ref{eq:equiv cond for lin indep}),
we have $\gamma_{\eta}\wedge\phi_{i}\wedge x=0$ for $1\le i\le b$.
Combining all of this we get,
\[
0\ne\gamma_{\omega}\wedge\zeta\wedge x=\sum_{i=1}^{b}c_{i}\gamma_{\omega}\wedge\phi_{i}\wedge x=\sum_{i=1}^{b}c_{i}t_{i}\gamma_{\eta}\wedge\phi_{i}\wedge x=0.
\]
This contradiction shows that we must have $L_{m-1}(\eta)=L_{m-1}(\omega)$,
which completes the proof of the theorem.
\end{proof}

\subsection{\label{subsec:Projections of components revisited}Projections of
components, revisited}

For $1\le l<m\le d$, $L\in\mathrm{Gr}_{m-1}(m)$ and $\rho>0$, set
\[
T_{l,m}(L,\rho):=\left\{ W\in\mathrm{Gr}_{l}(m)\::\:\kappa(L^{\perp},W)\ge\rho\right\} .
\]
Assuming $\Sigma_{1}^{m-1}=m-1$ and $\Delta_{m}<1$, recall that
we write $L_{m-1}$ in place of the map $\widehat{L}_{m-1}:\mathbb{R}^{d}\rightarrow\mathrm{Gr}_{m-1}(d)$
obtained in Theorem \ref{thm:L_m-1 factors via Pi}. The following
proposition is the main result of the present subsection.
\begin{prop}
\label{prop:L =000026 ent of comp}Let $1\le l<m\le d$ be given and
suppose that $\Sigma_{1}^{m-1}=m-1$ and $\Delta_{m}<1$. Then for
every $\epsilon,\rho>0$, $k\ge K(\epsilon,\rho)\ge1$, $n\ge N(\epsilon,\rho,k)\ge1$
and $V\in\mathrm{Gr}_{m}(d)$,
\[
\mathbb{P}_{1\le i\le n}\left\{ \frac{1}{k}H(\pi_{W}\pi_{V}\mu_{x,i},\mathcal{D}_{i+k})>l-\epsilon\quad\forall\;W\in T_{l,m}(\pi_{V}L_{m-1}(x),\rho)\right\} >1-\epsilon.
\]
\end{prop}

\begin{rem}
\label{rem:L_m-1mu=00003Dnu_m-1}Under the conditions of the proposition
and by Theorem \ref{thm:L_m-1 factors via Pi},
\[
L_{m-1}\mu=L_{m-1}\Pi\beta=L_{m-1}\beta=\nu_{m-1}.
\]
Thus, by Lemma \ref{lem:nu=00007Bkappa<del=00007D<eps}
\[
\mu\{x\::\:\pi_{V}L_{m-1}(x)\in\mathrm{Gr}_{m-1}(m)\}=1\text{ for all }V\in\mathrm{Gr}_{m}(d).
\]
\end{rem}

For the proof of the proposition we need the following two lemmas.
The first one follows easily by compactness and its proof is therefore
omitted.

\begin{lem}
\label{lem:proj of close sub are close}Let $2\le m\le d$ and $\epsilon,\rho>0$
be given. Then there exists $\delta=\delta(\epsilon,\rho)>0$ so that
$d_{\mathrm{Gr}_{m-1}}(\pi_{V}W_{1},\pi_{V}W_{2})<\epsilon$ for all
$V\in\mathrm{Gr}_{m}(d)$ and $W_{1},W_{2}\in\mathrm{Gr}_{m-1}(d)$
with $\kappa(V^{\perp},W_{1})\ge\rho$, $\kappa(V^{\perp},W_{2})\ge\rho$
and $d_{\mathrm{Gr}_{m-1}}(W_{1},W_{2})<\delta$.
\end{lem}

\begin{lem}
\label{lem:uni cont of pi_V comp L}Let $2\le m\le d$ and $\epsilon>0$
be given and suppose that $\Sigma_{1}^{m-1}=m-1$ and $\Delta_{m}<1$.
Then there exists $\eta=\eta(\epsilon)>0$ so that for all $V\in\mathrm{Gr}_{m}(d)$
there exists a Borel set $F_{V}\subset\mathbb{R}^{d}$ with $\mu(F_{V})>1-\epsilon$
and
\[
d_{\mathrm{Gr}_{m-1}}(\pi_{V}L_{m-1}(x),\pi_{V}L_{m-1}(y))<\epsilon\text{ for all }x,y\in F_{V}\text{ with }|x-y|<\eta.
\]
\end{lem}

\begin{proof}
Since $L_{m-1}\mu=\nu_{m-1}$ and by Lemma \ref{lem:nu=00007Bkappa<del=00007D<eps},
there exists $\rho>0$ so that
\[
\mu\left\{ x\::\:\kappa(V^{\perp},L_{m-1}(x))\le\rho\right\} <\epsilon/2\text{ for all }V\in\mathrm{Gr}_{m}(d).
\]
Let $\delta=\delta(\epsilon,\rho)>0$ be as obtained in Lemma \ref{lem:proj of close sub are close}.
Since $L_{m-1}:\mathbb{R}^{d}\rightarrow\mathrm{Gr}_{m-1}(d)$ is
Borel measurable and by Lusin's theorem, there exist $\eta>0$ and
a Borel set $F\subset\mathbb{R}^{d}$ so that $\mu(F)>1-\epsilon/2$
and 
\[
d_{\mathrm{Gr}_{m-1}}(L_{m-1}(x),L_{m-1}(y))<\delta\text{ for all }x,y\in F\text{ with }|x-y|<\eta.
\]
The lemma now follows by setting
\[
F_{V}:=\{x\in F\::\:\kappa(V^{\perp},L_{m-1}(x))\ge\rho\}
\]
for each $V\in\mathrm{Gr}_{m}(d)$.
\end{proof}
\begin{proof}[Proof of Proposition \ref{prop:L =000026 ent of comp}]
Let $\epsilon>0$ be small, let $\rho>0$, and let $k\ge1$ be large with respect to
$\epsilon,\rho$. Given $V\in\mathrm{Gr}_{m}(d)$, let $\mathcal{Q}_{V}$
be the set of all integers $n\ge1$ so that
\[
\mathbb{P}_{i=n}\left\{ \mathbf{I}_{m-1}(i)\in\mathcal{Z}\left(m-1,\epsilon,k,i,V\right)\right\} >1-\epsilon,
\]
where we have used here the notation introduced in Definition \ref{def:def of mathcal(Z)}.
From $\Sigma_{1}^{m-1}=m-1$, by assuming that $k$ is sufficiently
large, and by Proposition \ref{prop:inductive main proj prop},
\begin{equation}
\underset{M\rightarrow\infty}{\liminf}\:\inf\left\{ \lambda_{M}(\mathcal{N}_{M}\cap\mathcal{Q}_{V})\::\:V\in\mathrm{Gr}_{m}(d)\right\} >1-\epsilon.\label{eq:inf of lambda_N}
\end{equation}

Let $N\ge1$ be large with respect to $k$, fix $V\in\mathrm{Gr}_{m}(d)$,
and let $n\in\mathcal{Q}_{V}$ be with $n\ge N$. We emphasise that
$n$ is not assumed to be large with respect to $V$. For $u\in\Psi_{m-1}(n)$
and $D\in\mathcal{D}_{n}^{d}$ with $\varphi_{u}\mu(D)>0$ set $\theta_{u,D}:=(\varphi_{u}\mu)_{D}$.
Let $\Theta\in\mathcal{M}(\Psi_{m-1}(n)\times\mathcal{D}_{n}^{d})$
be with
\[
\Theta:=\sum_{u\in\Psi_{m-1}(n)}p_{u}\sum_{D\in\mathcal{D}_{n}^{d}}\varphi_{u}\mu(D)\cdot\delta_{u}\times\delta_{D}.
\]
It is easy to verify that $\mu=\int\theta_{u,D}\:d\Theta(u,D)$. Moreover,
since $n\in\mathcal{Q}_{V}$
\begin{equation}
\Theta\left\{ (u,D)\::\:\frac{1}{k}H\left(\pi_{V}\theta_{u,D},\mathcal{D}_{n+k}\right)>m-1-\epsilon\right\} >1-O(\epsilon).\label{eq:large ent for theta_u,D}
\end{equation}

Set $L_{m-1}^{V}:=\pi_{V}\circ L_{m-1}$, and recall that for $\omega\in\Lambda^{\mathbb{N}}$
we write $\Psi_{m-1}(n;\omega)$ for the unique $u\in\Psi_{m-1}(n)$
with $\omega\in[u]$. In what follows, given linear subspaces $W_{1},W_{2}$
of $\mathbb{R}^{m}$ we set $d_{\mathrm{Gr}_{m-1}}(W_{1},W_{2}):=\infty$
whenever $\dim W_{1}<m-1$ or $\dim W_{2}<m-1$.

From (\ref{enu:def of L_m via L_m of matrices}) in Section \ref{subsec:Coding and Furstenberg maps},
by Lemmata \ref{lem:nu=00007Bkappa<del=00007D<eps} and \ref{lem:proj of close sub are close},
and by assuming that $N$ is large enough with respect to $\epsilon,\rho$,
\[
\beta\left\{ \omega\::\:d_{\mathrm{Gr}_{m-1}}\left(L_{m-1}^{V}(A_{\Psi_{m-1}(n;\omega)}),L_{m-1}^{V}(\omega)\right)<\rho/3\right\} >1-\epsilon.
\]
Thus, by the decomposition $\beta=\mathbb{E}_{i=n}\left(\beta_{[\mathbf{I}_{m-1}(i)]}\right)$,
\[
\mathbb{E}_{i=n}\left( \beta_{[\mathbf{I}_{m-1}(i)]}\left\{ \omega\::\:d_{\mathrm{Gr}_{m-1}}\left(L_{m-1}^{V}(A_{\mathbf{I}_{m-1}(i)}),L_{m-1}^{V}(\omega)\right)<\rho/3\right\} \right) >1-\epsilon.
\]
From this, since $L_{m-1}(\omega)=L_{m-1}(\Pi\omega)$ for $\beta$-a.e.
$\omega$, and since $\Pi\beta_{[u]}=\varphi_{u}\mu$ for $u\in\Lambda^{*}$,
\[
\mathbb{E}_{i=n}\left( \varphi_{\mathbf{I}_{m-1}(i)}\mu\left\{ y\::\:d_{\mathrm{Gr}_{m-1}}\left(L_{m-1}^{V}(A_{\mathbf{I}_{m-1}(i)}),L_{m-1}^{V}(y)\right)<\rho/3\right\} \right) >1-\epsilon.
\]
Hence, since $\varphi_{u}\mu=\sum_{D\in\mathcal{D}_{n}^{d}}\varphi_{u}\mu(D)\cdot\theta_{u,D}$
for $u\in\Psi_{m-1}(n)$,
\begin{equation}
\int\theta_{u,D}\left\{ y\::\:d_{\mathrm{Gr}_{m-1}}\left(L_{m-1}^{V}(A_{u}),L_{m-1}^{V}(y)\right)<\rho/3\right\} \:d\Theta(u,D)>1-\epsilon.\label{eq:1st int dTheta}
\end{equation}

By Lemma \ref{lem:uni cont of pi_V comp L} and by assuming that $N$
is large enough with respect to $\epsilon,\rho$, there exists a Borel
set $F\subset\mathbb{R}^{d}$ with $\mu(F)>1-\epsilon$ so that
\begin{equation}
d_{\mathrm{Gr}_{m-1}}(L_{m-1}^{V}(x),L_{m-1}^{V}(y))<\rho/3\text{ for all }D\in\mathcal{D}_{n}^{d}\text{ and }x,y\in\overline{D}\cap F.\label{eq:dist is small for x,y in D cap F}
\end{equation}
Since $\mu=\int\theta_{u,D}\:d\Theta(u,D)$,
\begin{equation}
\int\theta_{u,D}(F)\:d\Theta(u,D)>1-\epsilon.\label{eq:theta_u,D(F_V) is large on avg}
\end{equation}

Let $\mathcal{Y}$ be the set of all $(u,D)\in\Psi_{m-1}(n)\times\mathcal{D}_{n}^{d}$
so that $\varphi_{u}\mu(D)>0$ and,
\begin{enumerate}
\item \label{enu:alpha_m<alpha_m-1}$\alpha_{m}(A_{u})<2^{-2k}\alpha_{m-1}(A_{u})$;
\item \label{enu:large ent thet_u,D}$\frac{1}{k}H\left(\pi_{V}\theta_{u,D},\mathcal{D}_{n+k}\right)>m-1-\epsilon$;
\item \label{enu:>1/2=0000231}$\theta_{u,D}\left\{ y\::\:d_{\mathrm{Gr}_{m-1}}\left(L_{m-1}^{V}(A_{u}),L_{m-1}^{V}(y)\right)<\rho/3\right\} >1/2$;
\item \label{enu:>1/2=0000232}$\theta_{u,D}(F)>1/2$.
\end{enumerate}
Since $\chi_{m}<\chi_{m-1}$, from (\ref{eq:large ent for theta_u,D}),
(\ref{eq:1st int dTheta}) and (\ref{eq:theta_u,D(F_V) is large on avg}),
and by replacing $\epsilon$ with a larger quantity (which is still
small) without changing the notation, we may assume that $\Theta(\mathcal{Y})>1-\epsilon$.

For $D\in\mathcal{D}_{n}^{d}$ with $\mu(D)>0$ let $\Theta_{D}\in\mathcal{M}(\Psi_{m-1}(n))$
be with
\[
\Theta_{D}:=\frac{1}{\mu(D)}\sum_{u\in\Psi_{m-1}(n)}p_{u}\cdot\varphi_{u}\mu(D)\cdot\delta_{u}.
\]
Since $\mu(F)>1-\epsilon$ and $\Theta(\mathcal{Y})>1-\epsilon$,
from the equality
\[
\Theta=\int\Theta_{\mathcal{D}_{n}(x)}\times\delta_{\mathcal{D}_{n}(x)}\:d\mu(x),
\]
and by replacing $\epsilon$ with a larger quantity without changing the notation, we may assume that $\mu(E)>1-\epsilon$, where 
$E$ is the set of all $x\in F$ so that $\mu(\mathcal{D}_{n}(x))>0$ and
\[
\Theta_{\mathcal{D}_{n}(x)}\left\{ u\::\:(u,\mathcal{D}_{n}(x))\in\mathcal{Y}\right\} >1-\epsilon.
\]

Let $x\in E$ be given and set $D=\mathcal{D}_{n}(x)$. Let $W\in T_{l,m}(L_{m-1}^{V}(x),\rho)$,
so that $W\in\mathrm{Gr}_{l}(m)$ and $\kappa(L_{m-1}^{V}(x)^{\perp},W)\ge\rho$.
Let $u\in\Psi_{m-1}(n)$ be with $(u,D)\in\mathcal{Y}$. By properties
(\ref{enu:>1/2=0000231}) and (\ref{enu:>1/2=0000232}) in the definition
of $\mathcal{Y}$, there exists $y\in F\cap\mathrm{supp}(\theta_{u,D})$
with
\[
d_{\mathrm{Gr}_{1}}\left(L_{m-1}^{V}(A_{u})^{\perp},L_{m-1}^{V}(y)^{\perp}\right)=d_{\mathrm{Gr}_{m-1}}\left(L_{m-1}^{V}(A_{u}),L_{m-1}^{V}(y)\right)<\rho/3.
\]
From (\ref{eq:dist is small for x,y in D cap F}) and since $x,y\in\overline{D}\cap F$,
\[
d_{\mathrm{Gr}_{m-1}}(L_{m-1}^{V}(x)^{\perp},L_{m-1}^{V}(y)^{\perp})=d_{\mathrm{Gr}_{m-1}}(L_{m-1}^{V}(x),L_{m-1}^{V}(y))<\rho/3.
\]
These inequalities, together with $\kappa(L_{m-1}^{V}(x)^{\perp},W)\ge\rho$,
imply that $\kappa(L_{m-1}^{V}(A_{u})^{\perp},W)\ge\rho/3$.

By property (\ref{enu:alpha_m<alpha_m-1}) in the definition of $\mathcal{Y}$,
since $u\in\Psi_{m-1}(n)$, and by assuming that $k$ is sufficiently
large, it follows that $\varphi_{u}\mu$ is supported on $z+L_{m-1}(A_{u})^{(2^{-n-k})}$
for some $z\in\mathbb{R}^{d}$. Thus, $\pi_{V}\theta_{u,D}$ is supported
on $\pi_{V}z+L_{m-1}^{V}(A_{u})^{(2^{-n-k})}$. From this and $\kappa(L_{m-1}^{V}(A_{u})^{\perp},W)\ge\rho/3$,
by Lemma \ref{lem:gen proj lb}, and from property (\ref{enu:large ent thet_u,D})
in the definition of $\mathcal{Y}$,
\begin{equation}
\frac{1}{k}H\left(\pi_{W}\pi_{V}\theta_{u,D},\mathcal{D}_{n+k}\right)\ge l-2\epsilon\text{ for }u\in\Psi_{m-1}(n)\text{ with }(u,D)\in\mathcal{Y}.\label{eq:all u in Psi with (u,D)inY}
\end{equation}

Note that $\mu_{D}=\int\theta_{u,D}\:d\Theta_{D}(u)$. Hence, by the
concavity of entropy, from (\ref{eq:all u in Psi with (u,D)inY}),
and since $x\in E$, 
\[
\frac{1}{k}H\left(\pi_{W}\pi_{V}\mu_{\mathcal{D}_{n}(x)},\mathcal{D}_{n+k}\right)\ge l-O(\epsilon)\text{ for all }W\in T_{l,m}(L_{m-1}^{V}(x),\rho).
\]
Since this holds for all $x\in E$ and $\mu(E)>1-\epsilon$,
we have thus shown that for all $V\in\mathrm{Gr}_{m}(d)$ and $n\in\mathcal{Q}_{V}$
with $n\ge N$
\[
\mathbb{P}_{i=n}\left\{ \frac{1}{k}H(\pi_{W}\pi_{V}\mu_{x,i},\mathcal{D}_{i+k})>l-O(\epsilon)\quad\forall\;W\in T_{l,m}(L_{m-1}^{V}(x),\rho)\right\} >1-\epsilon.
\]
This together with (\ref{eq:inf of lambda_N}) completes the proof
of the proposition.
\end{proof}
For the proof of Theorem \ref{thm:ent increase result}, we shall need the following version of Proposition \ref{prop:L =000026 ent of comp}
for components of projections of components.
\begin{cor}
\label{cor:ent of proj of comp of proj of comp of mu}Let $1\le l<m\le d$
be given and suppose that $\Sigma_{1}^{m-1}=m-1$ and $\Delta_{m}<1$.
Then for every $\epsilon,\rho>0$, $q\ge Q(\epsilon,\rho)\ge1$, $k\ge1$,
$n\ge N(\epsilon,\rho,q,k)\ge1$ and $V\in\mathrm{Gr}_{m}(d)$ we
have $\lambda_{n}\times\mu(Y)>1-\epsilon$, where $Y$ is the set
of all $(i,x)\in\mathcal{N}_{n}\times\mathbb{R}^{d}$ so that
\[
\mathbb{P}_{i\le j\le i+k}\left\{ \frac{1}{q}H\left(\pi_{W}(\pi_{V}\mu_{x,i})_{y,j},\mathcal{D}_{j+q})\right)>l-\epsilon\quad\forall\;W\in T_{l,m}(\pi_{V}L_{m-1}(x),\rho)\right\} >1-\epsilon.
\]
\end{cor}

\begin{proof}
Let $\epsilon>0$ be small, let $\rho>0$, let $q\ge1$ be large with respect
to $\epsilon,\rho$, let $k\ge1$, let $n\ge1$ be large with respect
to $k$ and $q$, and let $V\in\mathrm{Gr}_{m}(d)$. By Proposition
\ref{prop:L =000026 ent of comp},
\[
\mathbb{P}_{1\le i\le n}\left\{ \frac{1}{q}H(\pi_{W}\pi_{V}\mu_{x,i},\mathcal{D}_{i+q})>l-\epsilon\quad\forall\;W\in T_{l,m}(\pi_{V}L_{m-1}(x),\rho)\right\} >1-\epsilon.
\]
Hence, from Lemma \ref{lem:distribution-of-components-of-components}
and by replacing $\epsilon$ with a larger quantity (which is still
small) without changing the notation, we get $\lambda_{n}\times\mu(Y_{1})>1-\epsilon$,
where $Y_{1}$ is the set of all $(i,x)\in\mathcal{N}_{n}\times\mathbb{R}^{d}$
so that
\[
\mathbb{P}_{i\le j\le i+k}\left\{ \frac{1}{q}H(\pi_{W}\pi_{V}(\mu_{x,i})_{y,j},\mathcal{D}_{j+q})>l-\epsilon\quad\forall\;W\in T_{l,m}(\pi_{V}L_{m-1}(y),\rho)\right\} >1-\epsilon.
\]

Let $E$ be the set of $(x,y)\in\mathbb{R}^{d}\times\mathbb{R}^{d}$
with
\[
d_{\mathrm{Gr}_{m-1}}(\pi_{V}L_{m-1}(x),\pi_{V}L_{m-1}(y))<\rho,
\]
and let $Y_{2}$ be the set of $(i,x)\in\mathcal{N}_{n}\times\mathbb{R}^{d}$
so that $\mu_{x,i}\{y:(x,y)\in E\}>1-\epsilon$. By Lemma \ref{lem:uni cont of pi_V comp L}
we may assume that $\lambda_{n}\times\mu(Y_{2})>1-\epsilon$. Set
$Y:=Y_{1}\cap Y_{2}$, then $\lambda_{n}\times\mu(Y)>1-2\epsilon$.
Moreover, given $(x,y)\in E$ and $W\in T_{l,m}(\pi_{V}L_{m-1}(x),2\rho)$
we have $W\in T_{l,m}(\pi_{V}L_{m-1}(y),\rho)$. Thus, for $(i,x)\in Y$
\[
\mathbb{P}_{i\le j\le i+k}\left\{ \frac{1}{q}H(\pi_{W}\pi_{V}(\mu_{x,i})_{y,j},\mathcal{D}_{j+q})>l-\epsilon\quad\forall\;W\in T_{l,m}(\pi_{V}L_{m-1}(x),2\rho)\right\} >1-2\epsilon.
\]

The rest of the proof basically follows from the concavity and almost
convexity of entropy. We provide full details for completeness. Fix
$(i,x)\in Y$ and set $\theta:=\mu_{x,i}$ and $T:=T_{l,m}(\pi_{V}L_{m-1}(x),2\rho)$.
Let $J$ be the set of $i\le j\le i+k$ so that
\[
\mathbb{P}\left\{ \frac{1}{q}H(\pi_{W}\pi_{V}\theta_{y,j},\mathcal{D}_{j+q})>l-\epsilon\quad\forall\;W\in T\right\} >1-\epsilon.
\]
Recall the notation $\lambda_{i,i+k}$ from Section \ref{subsec:Component-measures}.
Since $(i,x)\in Y$ and by replacing $\epsilon$ with a larger quantity
without changing the notation, we may assume that $\lambda_{i,i+k}(J)>1-\epsilon$.

Fix $j\in J$, let $y\in\mathbb{R}^{d}$ be with
\begin{equation}
\frac{1}{q}H(\pi_{W}\pi_{V}\theta_{y,j},\mathcal{D}_{j+q})>l-\epsilon\text{ for all }W\in T,\label{eq:1/qH(pi pi thet_y)>l-ep}
\end{equation}
and set $\sigma:=\pi_{V}\theta_{y,j}$. Let $D\in\mathcal{D}_{j}^{m}$
be with $\sigma(D)\ge\epsilon^{1/2}$ and let $W\in T$. Since 
\begin{equation}
\#\left\{ D'\in\mathcal{D}_{j}^{m}\::\:D'\cap\mathrm{supp}(\sigma)\ne\emptyset\right\} =O(1),\label{eq:=000023=00007BD : D=00005Ccap=00005Csigma=00005Cne=00005Cempt=00007D=00003DO(1)}
\end{equation}
and by the almost convexity of entropy (see (\ref{eq:conc =000026 almo conv of ent})),
\[
l-\epsilon<\frac{1}{q}H(\sigma,\pi_{W}^{-1}\mathcal{D}_{j+q})\le\sum_{D'\in\mathcal{D}_{j}^{m}}\sigma(D')\cdot\frac{1}{q}H(\sigma_{D'},\pi_{W}^{-1}\mathcal{D}_{j+q})+O(1/q).
\]
Since $\frac{1}{q}H(\sigma_{D'},\pi_{W}^{-1}\mathcal{D}_{j+q})\le l+O(1/q)$
for $D'\in\mathcal{D}_{j}^{m}$,
\[
l-2\epsilon<(1-\sigma(D))l+\sigma(D)\frac{1}{q}H(\sigma_{D},\pi_{W}^{-1}\mathcal{D}_{j+q}).
\]
Thus, from $\sigma(D)\ge\epsilon^{1/2}$
\[
l-2\epsilon^{1/2}<\frac{1}{q}H(\sigma_{D},\pi_{W}^{-1}\mathcal{D}_{j+q}).
\]
From this and since
\[
\sigma_{D}=(\pi_{V}\theta_{y,j})_{D}=\pi_{V}((\theta_{y,j})_{\pi_{V}^{-1}(D)})=\pi_{V}\theta_{\mathcal{D}_{j}(y)\cap\pi_{V}^{-1}(D)},
\]
we obtain
\begin{equation}
\frac{1}{q}H(\theta_{\mathcal{D}_{j}(y)\cap\pi_{V}^{-1}(D)},\pi_{V}^{-1}\pi_{W}^{-1}\mathcal{D}_{j+q})>l-2\epsilon^{1/2}.\label{eq:1/qH(theta_cup)>l-eps}
\end{equation}
Note that this holds for all $D\in\mathcal{D}_{j}^{m}$ with $\sigma(D)\ge\epsilon^{1/2}$
and $W\in T$.

From (\ref{eq:=000023=00007BD : D=00005Ccap=00005Csigma=00005Cne=00005Cempt=00007D=00003DO(1)})
it follows that,
\[
\theta_{y,j}\left(\bigcup\left\{ \pi_{V}^{-1}(D)\::\:D\in\mathcal{D}_{j}^{m}\text{ and }\sigma(D)\ge\epsilon^{1/2}\right\} \right)\ge1-O(\epsilon^{1/2}).
\]
Thus, from (\ref{eq:1/qH(theta_cup)>l-eps})
\begin{multline*}
\theta_{y,j}\left(\bigcup\left\{ \pi_{V}^{-1}(D):\begin{array}{c}
D\in\mathcal{D}_{j}^{m}\text{ and for each }W\in T\text{ we have}\\
\frac{1}{q}H(\theta_{\mathcal{D}_{j}(y)\cap\pi_{V}^{-1}(D)},\pi_{V}^{-1}\pi_{W}^{-1}\mathcal{D}_{j+q})>l-2\epsilon^{1/2}
\end{array}\right\} \right)\\
>1-O(\epsilon^{1/2}).
\end{multline*}
Since this holds for all $y\in\mathbb{R}^{d}$ with (\ref{eq:1/qH(pi pi thet_y)>l-ep}),
from $j\in J$, and by replacing $\epsilon$ with a larger quantity
without changing the notation, we may assume that
\begin{equation}
\theta\left(\cup_{D\in\mathcal{E}_{j}^{m}}\pi_{V}^{-1}(D)\right)>1-\epsilon,\label{eq:theta-mass of E_j^m}
\end{equation}
where $\mathcal{E}_{j}^{m}$ is the set of all $D\in\mathcal{D}_{j}^{m}$
so that
\begin{multline*}
\theta_{\pi_{V}^{-1}(D)}\left(\bigcup\left\{ D'\in\mathcal{D}_{j}^{d}:\frac{1}{q}H(\theta_{D'\cap\pi_{V}^{-1}(D)},\pi_{V}^{-1}\pi_{W}^{-1}\mathcal{D}_{j+q})>l-\epsilon\quad\forall\;W\in T\right\} \right)\\
>1-\epsilon.
\end{multline*}

By the concavity of entropy,
\[
\frac{1}{q}H(\theta_{\pi_{V}^{-1}(D)},\pi_{V}^{-1}\pi_{W}^{-1}\mathcal{D}_{j+q})>l-O(\epsilon)\text{ for all }W\in T\text{ and }D\in\mathcal{E}_{j}^{m}.
\]
Hence, from (\ref{eq:theta-mass of E_j^m}) and since $\pi_{V}\theta_{\pi_{V}^{-1}(D)}=(\pi_{V}\theta)_{D}$
for $D\in\mathcal{D}_{j}^{m}$,
\[
\mathbb{P}\left\{ \frac{1}{q}H(\pi_{W}(\pi_{V}\theta)_{z,j},\mathcal{D}_{j+q})>l-O(\epsilon)\quad\forall\:W\in T\right\} >1-\epsilon.
\]
Since this holds for all $j\in J$, from $\lambda_{i,i+k}(J)>1-\epsilon$,
and by recalling that $\theta:=\mu_{x,i}$,
\[
\mathbb{P}_{i\le j\le i+k}\left\{ \frac{1}{q}H(\pi_{W}(\pi_{V}\mu_{x,i})_{z,j},\mathcal{D}_{j+q})>l-O(\epsilon)\quad\forall\:W\in T\right\} >1-O(\epsilon).
\]
As this holds for all $(i,x)\in Y$ and $\lambda_{n}\times\mu(Y)>1-2\epsilon$,
we have thus completed the proof of the corollary.
\end{proof}

\section{\label{sec:Entropy-growth-under conv}Entropy growth under convolution}

The purpose of the present section is to prove Theorem \ref{thm:ent increase result}.
In Sections \ref{subsec:Entropy-growth-under conv in R^m} to \ref{subsec:The-non-affinity-of L_m-1}
we make the necessary preparations for the proof of the theorem, which
is given in Section \ref{subsec:Proof-of-ent inc thm}.

For the rest of this section fix $1\le m\le d$ such that $\Sigma_{1}^{m-1}=m-1$
and $\Delta_{m}<1$. Recall from Theorem \ref{thm:L_m-1 factors via Pi}
that there exists a Borel map $L_{m-1}:\mathbb{R}^{d}\rightarrow\mathrm{Gr}_{m-1}(d)$
so that $L_{m-1}(\omega)=L_{m-1}(\Pi\omega)$ for $\beta$-a.e. $\omega$.

\subsection{\label{subsec:Entropy-growth-under conv in R^m}Entropy growth under
linear convolution in $\mathbb{R}^{m}$}
\begin{defn}
Let $\epsilon>0$, $\theta\in\mathcal{M}(\mathbb{R}^{m})$, and a
linear subspace $V\subset\mathbb{R}^{m}$ be given. We say that $\theta$
is $(V,\epsilon)$-concentrated if there exists $x\in\mathbb{R}^{m}$
so that $\theta(x+V^{(\epsilon)})\ge1-\epsilon$.
\end{defn}

\begin{defn}
Let $\epsilon>0$, $k\ge1$, $\theta\in\mathcal{M}(\mathbb{R}^{m})$,
and a linear subspace $V\subset\mathbb{R}^{m}$ be given. We say that
$\theta$ is $(V,\epsilon,k)$-saturated if
\[
\frac{1}{k}H\left(\theta,\mathcal{D}_{k}\right)\ge\frac{1}{k}H\left(\pi_{V^{\perp}}\theta,\mathcal{D}_{k}\right)+\dim V-\epsilon.
\]
\end{defn}

The following result of Hochman plays a key role in the proof of Theorem
\ref{thm:ent increase result}. Recall from Section \ref{subsec:Component-measures}
that for $\theta\in\mathcal{M}(\mathbb{R}^{m})$ we denote its re-scaled
components by $\theta^{x,n}$.
\begin{thm}[{\cite[Theorem 2.8]{Ho}}]
\label{thm:Hoch inv thm R^m}For each $\epsilon,R>0$ and $k\ge1$
there exists $\delta=\delta(\epsilon,R,k)>0$ so that for each $n\ge N(\epsilon,R,k,\delta)\ge1$
the following holds. Let $i\ge0$ and $\theta,\sigma\in\mathcal{M}(\mathbb{R}^{m})$
be such that
\[
\mathrm{diam}(\mathrm{supp}(\theta)),\mathrm{diam}(\mathrm{supp}(\sigma))\le R2^{-i},
\]
and
\[
\frac{1}{n}H(\theta*\sigma,\mathcal{D}_{i+n})<\frac{1}{n}H(\sigma,\mathcal{D}_{i+n})+\delta.
\]
Then there exist linear subspaces $V_{i},...,V_{i+n}$ of $\mathbb{R}^{m}$
so that,
\[
\mathbb{P}_{i\le j\le i+n}\left\{ \begin{array}{c}
\sigma^{x,j}\text{ is }(V_{j},\epsilon,k)\text{-saturated and}\\
\theta^{y,j}\text{ is }(V_{j},\epsilon)\text{-concentrated}
\end{array}\right\} >1-\epsilon.
\]
\end{thm}

\subsection{Auxiliary lemmas}

Given $\theta\in\mathcal{M}(\mathrm{A}_{d,m}^{\mathrm{vec}})$ and
$\sigma\in\mathcal{M}(\mathbb{R}^{d})$, recall from Section \ref{subsec:Spaces-of-affine maps}
that we write $\theta.\sigma$ for the pushforward of $\theta\times\sigma$
via the map $(\psi,x)\rightarrow\psi(x)$.

The proof of the following lemma is similar to the proof of \cite[Lemma 6.9]{HR} and is therefore omitted.
\begin{lem}
\label{lem:ent of conv >=00003D ent of conv of comp}Let $R>0$, $\theta\in\mathcal{M}(\mathrm{A}_{d,m})$
and $\sigma\in\mathcal{M}(\mathbb{R}^{d})$ be given. Suppose that
$\mathrm{supp}(\theta)\subset B(\pi_{d,m},R)$ and $\mathrm{supp}(\sigma)\subset B(0,R)$.
Then for every $n\ge k\ge1$,
\[
\frac{1}{n}H(\theta\ldotp\sigma,\mathcal{D}_{n})\ge\mathbb{E}_{1\le i\le n}\left(\frac{1}{k}H(\theta_{\psi,i}\ldotp\sigma_{x,i},\mathcal{D}_{i+k})\right)+O_{R}(\frac{1}{k}+\frac{k}{n}).
\]
\end{lem}

The proof of the following lemma is similar to the proof of \cite[Lemma 4.2]{BHR}
and is therefore omitted.
\begin{lem}
\label{lem:linearization}Let $Z\subset\mathrm{A}_{d,m}\times\mathbb{R}^{d}$
be compact. Then for every $\epsilon>0$, $k>K(\epsilon)\ge1$,
and $0<\delta<\delta(Z,\epsilon,k)$ the following holds. Let $(\psi_{0},x_{0})\in Z$,
$\theta\in\mathcal{M}(B(\psi_{0},\delta))$, and $\sigma\in\mathcal{M}(B(x_{0},\delta))$
be given. Then,
\[
\left|\frac{1}{k}H(\theta\ldotp\sigma,\mathcal{D}_{k-\log\delta})-\frac{1}{k}H((\theta\ldotp x_{0})*(\psi_{0}\sigma),\mathcal{D}_{k-\log\delta})\right|<\epsilon.
\]
\end{lem}

\begin{lem}
\label{lem:d(V,W^perp)<eps}Assume $m\ge 2$, and let $1\le l<m$ and $\epsilon>0$ be given.
Then there exists $\delta>0$ so that $d_{\mathrm{Gr}_{l}}(V,\mathrm{Gr}_{l}(W^{\perp}))<\epsilon$
for all $W\in\mathrm{Gr}_{1}(m)$ and $V\in\mathrm{Gr}_{l}(m)$ with
$\kappa(W,V^{\perp})\le\delta$.
\end{lem}

\begin{proof}
Assume by contradiction that the lemma is false. Then by a compactness
argument, there exist $W\in\mathrm{Gr}_{1}(m)$ and $V\in\mathrm{Gr}_{l}(m)$
so that $\kappa(W,V^{\perp})=0$ and $V\nsubseteq W^{\perp}$. But
from $\kappa(W,V^{\perp})=0$ we get $W\subset V^{\perp}$, which
implies $V\subset W^{\perp}$. This contradiction completes the proof
of the lemma.
\end{proof}
The proof of the following lemma is similar to the proof of \cite[Proposition 5.5]{Ho}
and is therefore omitted. See also the discussion in \cite[Section 6.2]{HR}.
\begin{lem}
\label{lem:from conc on eu to cont on A_d,m}For every $\epsilon,R>0$
there exist $\delta=\delta(\epsilon,R)>0$ and $b=b(\epsilon,R)\ge1$
so that the following holds. Let $\theta\in\mathcal{M}(A_{d,m})$,
$x\in\mathbb{R}^{d}$, $k\ge1$, and a linear subspace $V\subset\mathbb{R}^{m}$
be given. Suppose that $|x|\le R$, $\mathrm{supp}(\theta)\subset B(\pi_{d,m},R)$,
and
\[
\mathbb{P}_{i=k}\left((\theta.x)^{y,i}\text{ is }(V,\delta)\text{-concentrated}\right)>1-\delta.
\]
Then,
\[
\mathbb{P}_{k\le i\le k+b}\left(S_{2^{i}}\left(\theta_{\psi,i}\ldotp x\right)\text{ is }(V,\epsilon)\text{-concentrated}\right)>1-\epsilon.
\]
\end{lem}

Recall from Section \ref{subsec:Spaces-of-affine maps} that the vector
space $\mathrm{A}_{d,m}^{\mathrm{vec}}$ is equipped with a norm $\Vert\cdot\Vert$.
\begin{defn}
Given $\theta\in\mathcal{M}(\mathrm{A}_{d,m}^{\mathrm{vec}})$ and
$\delta>0$, we say that $\theta$ is $(\{0\},\delta)$-concentrated
in $\mathrm{A}_{d,m}^{\mathrm{vec}}$ if $\theta(B)>1-\delta$ for
some closed ball $B$ of radius $\delta$ in $(\mathrm{A}_{d,m}^{\mathrm{vec}},\Vert\cdot\Vert)$.
\end{defn}

In the following lemma, for $\psi\in\mathrm{A}_{d,m}$ set $V_{\psi}:=(\ker A_{\psi})^{\perp}\in\mathrm{Gr}_{m}(d)$
and let $f_{\psi}\in\mathrm{A}_{m,m}$ be such that $f_{\psi}^{-1}\circ\psi=\pi_{V_{\psi}}$. Recall that given $\theta\in\mathcal{M}(\mathrm{A}_{d,m})$ and $\varphi\in\mathrm{A}_{m,m}$, we denote by $\varphi\theta\in\mathcal{M}(\mathrm{A}_{d,m})$ the pushforward of $\theta$ via the map $\psi\rightarrow\varphi\circ\psi$.
\begin{lem}
\label{lem:comp not conc}For every $\epsilon,R>0$ there exists
$\delta=\delta(\epsilon,R)>0$ such that for $k\ge K(\epsilon,R,\delta)\ge1$
and $n\ge N(\epsilon,R,\delta,k)\ge1$ the following holds. Let
$\theta\in\mathcal{M}(\mathrm{A}_{d,m})$ be such that $\mathrm{diam}(\mathrm{supp}(\theta))\leq R$
and $\frac{1}{n}H(\theta,\mathcal{D}_{n})>\epsilon$. Then $\lambda_{n}\times\theta(F)>\delta$,
where $F$ is the set of all $(i,\psi)\in\mathcal{N}_{n}\times\mathrm{A}_{d,m}$
such that
\[
\mathbb{P}_{i\le j\le i+k}\left\{ S_{2^{j}}((f_{\psi}^{-1}\theta_{\psi,i})_{\varphi,j})\text{ is not }(\{0\},\delta)\text{\mbox{-concentrated} in }\mathrm{A}_{d,m}^{\mathrm{vec}}\right\} >\delta.
\]
\end{lem}

\begin{proof}
Let $C>1$ be a large global constant, let $\epsilon,R>0$, let
$l\ge1$ be large with respect $\epsilon,R$, let $k\ge1$ be large
with respect to $l$, let $n\ge1$ be large with respect to $k$,
and let $\theta\in\mathcal{M}(\mathrm{A}_{d,m})$ be with $\mathrm{diam}(\mathrm{supp}(\theta))\leq R$
and $\frac{1}{n}H(\theta,\mathcal{D}_{n})>\epsilon$. By an argument
similar to the one given in the proof of \cite[Lemma 6.8]{HR} we
may assume that $\lambda_{n}\times\theta(F')>\epsilon/C$, where $F'$
is the set of all $(i,\psi)\in\mathcal{N}_{n}\times\mathrm{A}_{d,m}$
such that
\[
\mathbb{P}_{i\le j\le i+k}\left\{ \frac{1}{l}H\left((f_{\psi}^{-1}\theta_{\psi,i})_{\varphi,j},\mathcal{D}_{j+l}\right)>\frac{\epsilon}{C}\right\} >\frac{\epsilon}{C}.
\]

Let $(i,\psi)\in\mathcal{N}_{n}\times\mathrm{A}_{d,m}$ and $(j,\varphi)\in\mathcal{N}_{i,i+k}\times\mathrm{A}_{d,m}$
be such that $\sigma:=(f_{\psi}^{-1}\theta_{\psi,i})_{\varphi,j}$
is well defined and $\frac{1}{l}H\left(\sigma,\mathcal{D}_{j+l}\right)>\frac{\epsilon}{C}$.
In order to complete the proof of the lemma it suffices to show that
$S_{2^{j}}\sigma$ is not $(\{0\},2^{-l})$-concentrated in $\mathrm{A}_{d,m}^{\mathrm{vec}}$.
Assume by contradiction that this is not the case, then $\sigma(B')>1-2^{-l}$
for some closed ball $B'$ of radius $2^{-j-l}$ in $(\mathrm{A}_{d,m}^{\mathrm{vec}},\Vert\cdot\Vert)$.
Since $\mathrm{supp}(\sigma)$ is contained in a ball of radius $O(1)$
around $\pi_{d,m}$ and by Lemma \ref{lem: def of inv met and prop},
we have $\sigma(B)>1-2^{-l}$ for some closed ball $B$ of radius
$O(2^{-j-l})$ in $(\mathrm{A}_{d,m},d_{\mathrm{A}_{d,m}})$.

Since $\mathrm{diam}(B)=O(2^{-j-l})$, by Lemma \ref{lem: def of inv met and prop},
and by the defining properties of the partitions $\{\mathcal{D}_{q}^{\mathrm{A}_{d,m}}\}_{q\ge0}$
(see Section \ref{subsec:Dyadic-partitions}),
\[
\#\left\{ D\in\mathcal{D}_{j+l}^{\mathrm{A}_{d,m}}\::\:D\cap B\ne\emptyset\right\} =O(1),
\]
which implies $H\left(\sigma_{B},\mathcal{D}_{j+l}\right)=O(1)$.
Since $\sigma$ is supported on some $D\in\mathcal{D}_{j}^{\mathrm{A}_{d,m}}$
and by Lemma \ref{lem:commens partiti of A_d,m},
\[
\frac{1}{l}H\left(\sigma_{B^{c}},\mathcal{D}_{j+l}\right)=O(1).
\]
From all of this and by the almost convexity of entropy,
\begin{multline*}
\frac{\epsilon}{C}<\frac{1}{l}H\left(\sigma,\mathcal{D}_{j+l}\right)\\
\le\frac{1}{l}\left(\sigma(B)H\left(\sigma_{B},\mathcal{D}_{j+l}\right)+\sigma(B^{c})H\left(\sigma_{B^{c}},\mathcal{D}_{j+l}\right)+O(1)\right)=O(l^{-1}).
\end{multline*}
Since $l$ is assumed to be large with respect $\epsilon$ this implies
a contradiction, which completes the proof of the lemma.
\end{proof}

\subsection{A key proposition}

The proof of Theorem \ref{thm:ent increase result} relies on the
following proposition.
\begin{prop}
\label{prop:key prep for ent inc}For every $\epsilon,R>0$ there
exists $\epsilon'=\epsilon'(\epsilon,R)>0$, such that for all $\sigma>0$
there exists $\delta=\delta(\epsilon,R,\sigma)>0$, so that for all
$n\ge N(\epsilon,R,\sigma)\ge1$ the following holds. Let $\theta\in\mathcal{M}(\mathrm{A}_{d,m})$
be such that $\mathrm{supp}(\theta)\subset B(\pi_{d,m},R)$, $\frac{1}{n}H(\theta,\mathcal{D}_{n})\ge\epsilon$
and $\frac{1}{n}H(\theta\ldotp\mu,\mathcal{D}_{n})\le\Sigma_{1}^{m}+\delta$.
Then there exist $\xi\in\mathcal{M}(\mathrm{A}_{d,m})$ and $V\in\mathrm{Gr}_{m}(d)$
so that,
\begin{enumerate}
\item $\mathrm{diam}(\mathrm{supp}(\xi))=O(1)$ in $(\mathrm{A}_{d,m}^{\mathrm{vec}},\Vert\cdot\Vert)$;
\item $\mu\left\{ x\::\:\xi\ldotp x\mbox{ is }(\pi_{V}L_{m-1}(x),\sigma)\text{-concentrated}\right\} >1-\sigma$;
\item $\xi$ is not $(\{0\},\epsilon')$-concentrated in $\mathrm{A}_{d,m}^{\mathrm{vec}}$.
\end{enumerate}
\end{prop}

\begin{proof}
Throughout the proof we use the notations introduced just before Lemma
\ref{lem:comp not conc}. That is, for $\psi\in\mathrm{A}_{d,m}$
we set $V_{\psi}:=(\ker A_{\psi})^{\perp}\in\mathrm{Gr}_{m}(d)$ and
let $f_{\psi}\in\mathrm{A}_{m,m}$ be such that $f_{\psi}^{-1}\circ\psi=\pi_{V_{\psi}}$.

Let $\epsilon,R>0$, let $\delta>0$ be small with respect to $\epsilon$
and $R$, let $k\ge1$ be large with respect to $\delta$, let $n\ge1$
be large with respect to $k$, and let $\theta\in\mathcal{M}(\mathrm{A}_{d,m})$
be such that $\mathrm{supp}(\theta)\subset B(\pi_{d,m},R)$, $\frac{1}{n}H(\theta,\mathcal{D}_{n})\ge\epsilon$
and $\frac{1}{n}H(\theta\ldotp\mu,\mathcal{D}_{n})\le\Sigma_{1}^{m}+\delta^{2}$.
By Lemma \ref{lem:ent of conv >=00003D ent of conv of comp}, we may
assume
\[
\Sigma_{1}^{m}+2\delta^{2}\ge\mathbb{E}_{1\le i\le n}\left(\frac{1}{k}H(\theta_{\psi,i}\ldotp\mu_{x,i},\mathcal{D}_{i+k})\right).
\]
From this and Lemma \ref{lem:linearization},
\[
\Sigma_{1}^{m}+3\delta^{2}\ge\mathbb{E}_{1\le i\le n}\left(\frac{1}{k}H((\theta_{\psi,i}\ldotp x)*\psi\mu_{x,i},\mathcal{D}_{i+k})\right).
\]
Since $\mathrm{supp}(\theta)\subset B(\pi_{d,m},R)$, the map $f_{\psi}^{-1}$
is $O_{R}(1)$-Lipschitz for all $\psi\in\mathrm{supp}(\theta)$.
Thus by (\ref{eq:ent under lip map}),
\begin{equation}
\Sigma_{1}^{m}+4\delta^{2}\ge\mathbb{E}_{1\le i\le n}\left(\frac{1}{k}H((f_{\psi}^{-1}\theta_{\psi,i}\ldotp x)*\pi_{V_{\psi}}\mu_{x,i},\mathcal{D}_{i+k})\right).\label{eq:ub on avg of conv of proj}
\end{equation}

Write $\Gamma:=\lambda_{n}\times\mu\times\theta$ and let $E_{0}$
be the set of all $(i,x,\psi)\in\mathcal{N}_{n}\times\mathbb{R}^{d}\times\mathrm{A}_{d,m}$
so that,
\[
\frac{1}{k}H((f_{\psi}^{-1}\theta_{\psi,i}\ldotp x)*\pi_{V_{\psi}}\mu_{x,i},\mathcal{D}_{i+k})<\frac{1}{k}H(\pi_{V_{\psi}}\mu_{x,i},\mathcal{D}_{i+k})+\delta.
\]
By (\ref{eq:conv does not dec ent}) we may assume that for $\Gamma$-a.e.
$(i,x,\psi)$,
\begin{equation}
\frac{1}{k}H((f_{\psi}^{-1}\theta_{\psi,i}\ldotp x)*\pi_{V_{\psi}}\mu_{x,i},\mathcal{D}_{i+k})\ge\frac{1}{k}H(\pi_{V_{\psi}}\mu_{x,i},\mathcal{D}_{i+k})-\delta^{2}.\label{eq:conv dont dec ent}
\end{equation}
From this, by the definition of $E_{0}$ and from $\eqref{eq:ub on avg of conv of proj}$,
\[
\Sigma_{1}^{m}+5\delta^{2}\ge\int\mathbb{E}_{1\le i\le n}\left(\frac{1}{k}H(\pi_{V_{\psi}}\mu_{x,i},\mathcal{D}_{i+k})\right)\:d\theta(\psi)+\delta\Gamma(E_{0}^{c}).
\]
Moreover, by Lemma \ref{lem:lb on ent of proj of comp of mu}
\[
\int\mathbb{E}_{1\le i\le n}\left(\frac{1}{k}H(\pi_{V_{\psi}}\mu_{x,i},\mathcal{D}_{i+k})\right)\:d\theta(\psi)\ge\Sigma_{1}^{m}-\delta^{2}.
\]
Thus, by replacing $\delta$ with a larger quantity without changing the notation, we get $\Gamma(E_{0})\ge1-\delta$.

Let $\sigma>0$ be small with respect to $\epsilon,R$ and suppose
that $\delta$ is small with respect to $\sigma$. Let $q\ge1$ be
an integer which is large with respect to $\sigma$ and suppose that
$\delta$ is small with respect to $q$. By Theorem \ref{thm:Hoch inv thm R^m},
for every $(i,x,\psi)=u\in E_{0}$ there exist linear subspaces $V_{i}^{u},...,V_{i+k}^{u}$
of $\mathbb{R}^{m}$ so that,
\begin{equation}
\mathbb{P}_{i\le j\le i+k}\left\{ \begin{array}{c}
(\pi_{V_{\psi}}\mu_{x,i})^{y,j}\text{ is }(V_{j}^{u},\sigma,q)\text{-saturated and}\\
(f_{\psi}^{-1}\theta_{\psi,i}\ldotp x)^{z,j}\text{ is }(V_{j}^{u},\sigma)\text{-concentrated}
\end{array}\right\} >1-\sigma.\label{eq:subspaces for u in E_0}
\end{equation}
\begin{lem}
We can assume that $\Gamma(E_{1})>1-\sigma$, where $E_{1}$ is the
set of all $(i,x,\psi)\in\mathcal{N}_{n}\times\mathbb{R}^{d}\times\mathrm{A}_{d,m}$
with
\begin{equation}
\mathbb{P}_{i\le j\le i+k}\left\{ (f_{\psi}^{-1}\theta_{\psi,i}\ldotp x)^{z,j}\text{ is }(\pi_{V_{\psi}}L_{m-1}(x),\sigma)\text{-concentrated}\right\} >1-\sigma.\label{eq:im of comp is L_m-1 conc}
\end{equation}
\end{lem}

\begin{proof}
Let $Z$ be the set of all $(i,x,\psi)\in\mathcal{N}_{n}\times\mathbb{R}^{d}\times\mathrm{A}_{d,m}$
so that,
\[
\mathbb{P}_{i\le j\le i+k}\left\{ \left|\frac{1}{q}H\left((\pi_{V_{\psi}}\mu_{x,i})^{y,j},\mathcal{D}_{q}\right)-\Sigma_{1}^{m}\right|<\sigma\right\} >1-\sigma.
\]
Let $\sigma'>0$ be small with respect to $\sigma$ and suppose that
$q$ is large with respect to $\sigma'$. From (\ref{eq:ub on avg of conv of proj})
and (\ref{eq:conv dont dec ent}),
\[
\Sigma_{1}^{m}+\delta\ge\int\mathbb{E}_{1\le i\le n}\left(\frac{1}{k}H(\pi_{V_{\psi}}\mu_{x,i},\mathcal{D}_{i+k})\right)\:d\theta(\psi).
\]
Thus, by Lemma \ref{lem:multiscale-entropy-formula},
\[
\Sigma_{1}^{m}+\sigma'\ge\int\mathbb{E}_{i\le j\le i+k}\left(\frac{1}{q}H\left((\pi_{V_{\psi}}\mu_{x,i})^{y,j},\mathcal{D}_{q}\right)\right)\:d\Gamma(i,x,\psi).
\]
Moreover, by Lemma \ref{lem:lb on ent of comp of proj of comp of mu},
\[
\int\mathbb{P}_{i\le j\le i+k}\left\{ \frac{1}{q}H\left((\pi_{V_{\psi}}\mu_{x,i})^{y,j},\mathcal{D}_{q}\right)>\Sigma_{1}^{m}-\sigma'\right\} \:d\Gamma(i,x,\psi)>1-\sigma'.
\]
From the last two inequalities, and since $\sigma'$ is small with
respect to $\sigma$, we obtain $\Gamma(Z)>1-\sigma$.

As in Section \ref{subsec:Projections of components revisited}, for
$1\le l<m$, $L\in\mathrm{Gr}_{m-1}(m)$ and $\eta>0$ let
\[
T_{l,m}(L,\eta):=\left\{ W\in\mathrm{Gr}_{l}(m)\::\:\kappa(L^{\perp},W)\ge\eta\right\} .
\]
Write $Y_{l}$ for the set of all $(i,x,\psi)\in\mathcal{N}_{n}\times\mathbb{R}^{d}\times\mathrm{A}_{d,m}$
with
\begin{multline*}
\mathbb{P}_{i\le j\le i+k}\left\{ \frac{1}{q}H\left(\pi_{W}((\pi_{V_{\psi}}\mu_{x,i})^{y,j}),\mathcal{D}_{q}\right)>l-\sigma\quad\forall\;W\in T_{l,m}(\pi_{V_{\psi}}L_{m-1}(x),\sigma')\right\} \\
>1-\sigma/m,
\end{multline*}
and set $Y:=\cap_{l=1}^{m-1}Y_{l}$. By Corollary \ref{cor:ent of proj of comp of proj of comp of mu}
it follows that $\Gamma(Y)>1-\sigma$. Note that $\Gamma(E_{0}\cap Z\cap Y)>1-3\sigma$,
hence it suffices to show that (\ref{eq:im of comp is L_m-1 conc})
is satisfied for $(i,x,\psi)\in E_{0}\cap Z\cap Y$.

Fix $(i,x,\psi)=u\in E_{0}\cap Z\cap Y$ and let $F_{u}$ be the set
of all $(j,y)\in\mathcal{N}_{i,i+k}\times\mathbb{R}^{m}$ such that,
\begin{itemize}
\item $(\pi_{V_{\psi}}\mu_{x,i})^{y,j}$ is $(V_{j}^{u},\sigma,q)$-saturated;
\item $\left|\frac{1}{q}H\left((\pi_{V_{\psi}}\mu_{x,i})^{y,j},\mathcal{D}_{q}\right)-\Sigma_{1}^{m}\right|<\sigma$;
\item $\frac{1}{q}H\left(\pi_{W}((\pi_{V_{\psi}}\mu_{x,i})^{y,j}),\mathcal{D}_{q}\right)>\dim W-\sigma$
for each $1\le l<m$ and $W\in T_{l,m}(\pi_{V_{\psi}}L_{m-1}(x),\sigma')$.
\end{itemize}
Since $u\in E_{0}\cap Z\cap Y$, we have $\lambda_{i,i+k}\times\pi_{V_{\psi}}\mu_{x,i}(F_{u})>1-3\sigma$.

Let $(j,y)\in F_{u}$ and set $l_{j}^{u}:=\dim V_{j}^{u}$. Assume
by contradiction that $l_{j}^{u}=m$ or $0<l_{j}^{u}<m$ with
\[
d_{\mathrm{Gr}_{l_{j}^{u}}}(V_{j}^{u},\mathrm{Gr}_{l_{j}^{u}}(\pi_{V_{\psi}}L_{m-1}(x)))\ge\sigma.
\]
If $0<l_{j}^{u}<m$ then by Lemma \ref{lem:d(V,W^perp)<eps}
\[
\kappa((\pi_{V_{\psi}}L_{m-1}(x))^{\perp},(V_{j}^{u})^{\perp})>\sigma',
\]
which implies $(V_{j}^{u})^{\perp}\in T_{m-l_{j}^{u},m}(\pi_{V_{\psi}}L_{m-1}(x),\sigma')$.
Thus,
\begin{eqnarray*}
\Sigma_{1}^{m} & > & \frac{1}{q}H\left((\pi_{V_{\psi}}\mu_{x,i})^{y,j},\mathcal{D}_{q}\right)-\sigma\\
 & \ge & \dim V_{j}^{u}+\frac{1}{q}H\left(\pi_{(V_{j}^{u})^{\perp}}(\pi_{V_{\psi}}\mu_{x,i})^{y,j},\mathcal{D}_{q}\right)-2\sigma\\
 & > & \dim V_{j}^{u}+\dim(V_{j}^{u})^{\perp}-3\sigma=m-3\sigma,
\end{eqnarray*}
where the last inequality holds trivially when $l_{j}^{u}=m$. This
contradicts $\Delta_{m}<1$ if $\sigma$ is sufficiently small, and
so we must have
\begin{equation}
0\le l_{j}^{u}<m\text{ with }d_{\mathrm{Gr}_{l_{j}^{u}}}(V_{j}^{u},\mathrm{Gr}_{l_{j}^{u}}(\pi_{V_{\psi}}L_{m-1}(x)))<\sigma.\label{eq:l^u<n and V^u close}
\end{equation}

Write
\[
S:=\left\{ j\in\mathcal{N}_{i,i+k}\::\:\pi_{V_{\psi}}\mu_{x,i}\{y\::\:(y,j)\in F_{u}\}>0\right\} ,
\]
then $\lambda_{i,i+k}(S)>1-3\sigma$ since $\lambda_{i,i+k}\times\pi_{V_{\psi}}\mu_{x,i}(F_{u})>1-3\sigma$.
Note that (\ref{eq:l^u<n and V^u close}) holds for each $j\in S$.
Thus, $(f_{\psi}^{-1}\theta_{\psi,i}\ldotp x)^{z,j}$ is $(\pi_{V_{\psi}}L_{m-1}(x),O(\sigma))$-concentrated
for every $(j,z)\in\mathcal{N}_{i,i+k}\times\mathbb{R}^{m}$ so that
$j\in S$ and $(f_{\psi}^{-1}\theta_{\psi,i}\ldotp x)^{z,j}$ is $(V_{j}^{u},\sigma)$-concentrated.
From this, $\lambda_{i,i+k}(S)>1-3\sigma$, and (\ref{eq:subspaces for u in E_0}),
it follows that (\ref{eq:im of comp is L_m-1 conc}) is satisfied
for $u=(i,x,\psi)$ with $\sigma$ replaced by $O(\sigma)$. This
completes the proof of the lemma.
\end{proof}
\begin{lem}
We can assume that $\Gamma(E_{2})>1-\sigma$, where $E_{2}$ is the
set of all $(i,x,\psi)\in\mathcal{N}_{n}\times\mathbb{R}^{d}\times\mathrm{A}_{d,m}$
with
\begin{equation}
\mathbb{P}_{i\le j\le i+k}\left\{ \begin{array}{c}
S_{2^{j}}(f_{\psi}^{-1}\theta_{\psi,i})_{\varphi,j}\ldotp x\text{ is }\\
(\pi_{V_{\psi}}L_{m-1}(x),\sigma)\text{-concentrated}
\end{array}\right\} >1-\sigma.\label{eq:conc in A_d,m near pi L(x)}
\end{equation}
\end{lem}

\begin{proof}
Fix $(i,x,\psi)\in E_{1}$ with $x\in K_{\Phi}$, write $\tau:=f_{\psi}^{-1}\theta_{\psi,i}$,
and set
\[
S:=\left\{ j\in\mathcal{N}_{i,i+k}\::\:\mathbb{P}_{l=j}\left\{ (\tau\ldotp x)^{y,l}\text{ is }(\pi_{V_{\psi}}L_{m-1}(x),\sigma)\text{-concentrated}\right\} \ge1-\sqrt{\sigma}\right\} .
\]
By (\ref{eq:im of comp is L_m-1 conc}) it follows that $\lambda_{i,i+k}(S)\ge1-\sqrt{\sigma}$.
Let $\sigma'>0$ be small with respect to $\epsilon,R$ and suppose
that $\sigma$ is small with respect to $\sigma'$. By Lemma \ref{lem:from conc on eu to cont on A_d,m},
there exists an integer $b=b(\sigma')\ge1$ such that,
\begin{equation}
\mathbb{P}_{j\le l\le j+b}\left(S_{2^{l}}\left(\tau_{\varphi,l}\ldotp x\right)\text{ is }(\pi_{V_{\psi}}L_{m-1}(x),\sigma')\text{-concentrated}\right)\ge1-\sigma'\text{ for }j\in S.\label{eq:conc for j in S}
\end{equation}

Let $\sigma''>0$ be small with respect to $\epsilon,R$ and suppose
that $\sigma'$ is small with respect to $\sigma''$. From $\lambda_{i,i+k}(S)\ge1-\sqrt{\sigma}$
and (\ref{eq:conc for j in S}), by assuming that $\sigma,\sigma'$
are sufficiently small with respect to $\sigma''$, and by assuming
that $k$ is sufficiently large with respect to $b$, it follows by
a statement similar to Lemma \ref{lem:distribution-of-components-of-components}
that (\ref{eq:conc in A_d,m near pi L(x)}) is satisfied with $\sigma''$
in place of $\sigma$. This completes the proof of the lemma.
\end{proof}
By the previous lemma, by Fubini's theorem, and by replacing $\sigma$
with a larger quantity (which is still small with respect to $\epsilon,R$)
without changing the notation, we may assume that $\lambda_{n}\times\theta(F_{1})>1-\sigma$,
where $F_{1}$ is the set of all $(i,\psi)\in\mathcal{N}_{n}\times\mathrm{A}_{d,m}$
such that
\[
\mathbb{P}_{i\le j\le i+k}\left(\mu\left\{ x\::\:\begin{array}{c}
S_{2^{j}}(f_{\psi}^{-1}\theta_{\psi,i})_{\varphi,j}\ldotp x\mbox{ is }\\
(\pi_{V_{\psi}}L_{m-1}(x),\sigma)\text{-concentrated}
\end{array}\right\} >1-\sigma\right)>1-\sigma.
\]

Let $\epsilon'>0$ be small with respect to $\epsilon,R$, and
suppose that $\sigma$ is small with respect to $\epsilon'$. By Lemma
\ref{lem:comp not conc}, and the assumptions $\mathrm{supp}(\theta)\subset B(\pi_{d,m},R)$
and $\frac{1}{n}H(\theta,\mathcal{D}_{n})\ge\epsilon$, it follows
that $\lambda_{n}\times\theta(F_{2})>\epsilon'$, where $F_{2}$ is
the set of all $(i,\psi)\in\mathcal{N}_{n}\times\mathrm{A}_{d,m}$
such that
\[
\mathbb{P}_{i\le j\le i+k}\left\{ S_{2^{j}}(f_{\psi}^{-1}\theta_{\psi,i})_{\varphi,j}\text{ is not }(\{0\},\epsilon')\mbox{-concentrated in }\mathrm{A}_{d,m}^{\mathrm{vec}}\right\} >\epsilon'.
\]
Since $\sigma$ is small with respect to $\epsilon'$, we get $\lambda_{n}\times\theta(F_{1}\cap F_{2})>\epsilon'/2$.
Moreover, for every $(i,\psi)\in F_{1}\cap F_{2}$ there exist $i\le j\le i+k$
and $\varphi\in\mathrm{supp}(f_{\psi}^{-1}\theta_{\psi,i})$ so that
$\xi:=S_{2^{j}}(f_{\psi}^{-1}\theta_{\psi,i})_{\varphi,j}$ is not
$(\{0\},\epsilon')$-concentrated in $\mathrm{A}_{d,m}^{\mathrm{vec}}$,
and
\[
\mu\left\{ x\::\:\xi\ldotp x\mbox{ is }(\pi_{V_{\psi}}L_{m-1}(x),\sigma)\text{-concentrated}\right\} >1-\sigma.
\]
By Lemma \ref{lem: def of inv met and prop} and since
\[
\mathrm{supp}(S_{2^{-j}}\xi)\subset B(\pi_{d,m},O(1))\text{ and }\mathrm{diam}(\mathrm{supp}(S_{2^{-j}}\xi))=O(2^{-j}),
\]
it follows that $\mathrm{diam}(\mathrm{supp}(S_{2^{-j}}\xi))=O(2^{-j})$
in $(\mathrm{A}_{d,m}^{\mathrm{vec}},\Vert\cdot\Vert)$. Hence, $\mathrm{diam}(\mathrm{supp}(\xi))=O(1)$
in $(\mathrm{A}_{d,m}^{\mathrm{vec}},\Vert\cdot\Vert)$, which completes
the proof of the proposition.
\end{proof}

\subsection{\label{subsec:The-non-affinity-of L_m-1}The non-affinity of $L_{m-1}$}

The purpose of the present subsection is to prove the following proposition,
which will be needed in the proof of Theorem \ref{thm:ent increase result}.
\begin{prop}
\label{prop:no psi exist}Let $\pi\in\mathrm{A}_{d,m}$ be linear.
Then there does not exist $0\ne\psi\in\mathrm{A}_{d,m}^{\mathrm{vec}}$
so that $\psi(x)\in\pi L_{m-1}(x)$ for $\mu$-a.e. $x$.
\end{prop}

The proof of the proposition requires some preparations. In what follows,
write $I$ and $1_{\mathrm{A}}$ for the identity elements of $\mathrm{GL}(d,\mathbb{R}$)
and $\mathrm{A}_{d,d}$ respectively. Let $\mathbf{S}_{\Phi}$ be
the subsemigroup of $\mathrm{A}_{d,d}$ generated by the maps in $\Phi$
with respect to composition, and recall that $\mathbf{S}_{\Phi}^{\mathrm{L}}$
denotes the semigroup generated by the linear parts $\{A_{i}\}_{i\in\Lambda}$.

We denote by $\mathbf{Z}_{\Phi}$ the Zariski closure of $\mathbf{S}_{\Phi}$
in $\mathrm{A}_{d,d}$ (see e.g. \cite[Section 6]{BQ}). That is,
$\mathbf{Z}_{\Phi}$ is the set of all $\psi\in\mathrm{A}_{d,d}$
so that $P(\psi)=0$ for every polynomial map $P:\mathrm{A}_{d,d}^{\mathrm{vec}}\rightarrow\mathbb{R}$
with $\mathbf{S}_{\Phi}\subset P^{-1}\{0\}$. By \cite[Lemma 6.15]{BQ}
it follows that $\mathbf{Z}_{\Phi}$ is a group.
\begin{lem}
\label{lem:exist map with scal lin part}There exist $\phi\in\mathbf{Z}_{\Phi}\setminus\{1_{\mathrm{A}}\}$
and $c_{\phi}\in\mathbb{R}\setminus\{0\}$ so that $A_{\phi}=c_{\phi}I$,
where recall that $A_{\phi}$ is the linear part of $\phi$.
\end{lem}

\begin{proof}
Equip $\mathbf{Z}_{\Phi}$ with the subspace topology induced by the
norm topology on $\mathrm{A}_{d,d}^{\mathrm{vec}}$. Since $\mathbf{Z}_{\Phi}$
is a closed subgroup of $\mathrm{A}_{d,d}$ it follows that $\mathbf{Z}_{\Phi}$
is a Lie group, whose Lie algebra we denote by $\mathfrak{z}$. Write
$\mathbf{Z}_{\Phi}^{0}$ for the connected component of $\mathbf{Z}_{\Phi}$
containing the identity element $1_{\mathrm{A}}$. Then $\mathbf{Z}_{\Phi}^{0}$
is a closed normal subgroup of $\mathbf{Z}_{\Phi}$. Moreover, from
\cite[Theorem 3]{Wh} and since $\mathbf{Z}_{\Phi}$ can be naturally
identified with a Zariski closed subgroup of $\mathrm{SL}(d+2,\mathbb{R})$,
it follows that $\mathbf{Z}_{\Phi}^{0}$ has finite index in $\mathbf{Z}_{\Phi}$.

Write $\mathbb{R}_{>0}$ for the set of positive real numbers. With
respect to multiplication, it is a Lie group whose Lie algebra is
identified with $\mathbb{R}$. Let $F:\mathbf{Z}_{\Phi}\rightarrow\mathbb{R}_{>0}$
be with $F(\psi):=\left|\det(A_{\psi})\right|$ for $\psi\in\mathbf{Z}_{\Phi}$,
then $F$ is a Lie group homomorphism. From $\Phi\subset\mathbf{Z}_{\Phi}$
and since the maps in $\Phi$ are strict contractions, it follows
that $F(\psi)\ne1$ for some $\psi\in\mathbf{Z}_{\Phi}$. Since $\mathbf{Z}_{\Phi}/\mathbf{Z}_{\Phi}^{0}$
is finite, this clearly implies that $F(\psi)\ne1$ for some $\psi\in\mathbf{Z}_{\Phi}^{0}$.
Since $\mathbf{Z}_{\Phi}^{0}$ is connected, it follows that $F|_{\mathbf{Z}_{\Phi}^{0}}$
is surjective. From this and by \cite[Theorems 4.14 and 7.5]{Le}
we obtain that $dF_{1_{\mathrm{A}}}:\mathfrak{z}\rightarrow\mathbb{R}$
is surjective, where $dF_{1_{\mathrm{A}}}$ is the differential of
$F$ at $1_{\mathrm{A}}$.

Since $F$ is a Lie group homomorphism, it follows that $dF_{1_{\mathrm{A}}}$
is a Lie algebra homomorphism. Thus, since $dF_{1_{\mathrm{A}}}$
is surjective, we obtain that the lie algebras $\mathfrak{z}/\ker(dF_{1_{\mathrm{A}}})$
and $\mathbb{R}$ are isomorphic. In particular, $\mathfrak{z}/\ker(dF_{1_{\mathrm{A}}})$
is nonzero and abelian. Hence, since a quotient of a semisimple Lie
algebra remains semisimple, it follows that $\mathfrak{z}$ is not
semisimple. This implies that $\mathfrak{z}$ has a nonzero abelian
ideal $\mathfrak{i}$. Let $\mathbf{N}$ be the unique connected Lie
subgroup of $\mathbf{Z}_{\Phi}^{0}$ whose lie algebra is $\mathfrak{i}$.
Since $\mathfrak{i}$ is a nonzero ideal of $\mathfrak{z}$, it follows
that $\mathbf{N}\ne\{1_{\mathrm{A}}\}$ and that $\mathbf{N}$ is
normal in $\mathbf{Z}_{\Phi}^{0}$. Moreover, since $\mathfrak{i}$
is abelian we get that $\mathbf{N}$ is abelian.

Given a subgroup $\mathbf{H}$ of $\mathbf{Z}_{\Phi}$ we write $\mathbf{H}^{\mathrm{L}}:=\{A_{\psi}\::\:\psi\in\mathbf{H}\}$,
which is a subgroup of $\mathrm{GL}(d,\mathbb{R})$. Since $\mathbf{N}$
is abelian, $\mathbf{N}^{\mathrm{L}}$ is also abelian. Thus, there
exists $0\ne v\in\mathbb{C}^{d}$ so that for each $A\in\mathbf{N}^{\mathrm{L}}$
we have $Av=c_{A}v$ for some $c_{A}\in\mathbb{C}$. Since $\mathbf{Z}_{\Phi}^{0}$
is connected, it clearly follows that $\mathbf{Z}_{\Phi}^{0,\mathrm{L}}v:=\{Bv\::\:B\in\mathbf{Z}_{\Phi}^{0,\mathrm{L}}\}$
is a connected subset of $\mathbb{C}^{d}$.

Let $A\in\mathbf{N}^{\mathrm{L}}$ be fixed for the moment. Since
$\mathbf{N}$ is normal in $\mathbf{Z}_{\Phi}^{0}$, the group $\mathbf{N}^{\mathrm{L}}$
is normal in $\mathbf{Z}_{\Phi}^{0,\mathrm{L}}$. Hence, for each
$B\in\mathbf{Z}_{\Phi}^{0,\mathrm{L}}$ there exists $A'\in\mathbf{N}^{\mathrm{L}}$
so that $AB=BA'$, and so $ABv=BA'v=c_{A'}Bv$. Thus, for each $x\in\mathbf{Z}_{\Phi}^{0,\mathrm{L}}v$
there exist $c_{A,x}\in\mathbb{C}$ so that $Ax=c_{A,x}x$. Since
$c_{A,x}=\left\langle Ax,x\right\rangle /|x|^{2}$, the map $x\rightarrow c_{A,x}$
is continuous. From this and since $\mathbf{Z}_{\Phi}^{0,\mathrm{L}}v$
is connected, it follows that $C_{A}:=\{c_{A,x}\::\:x\in\mathbf{Z}_{\Phi}^{0,\mathrm{L}}v\}$
is a connected subset of $\mathbb{C}$. Moreover, since $C_{A}$ consists
of eigenvalues of $A$ it is clearly finite. This together with the
connectedness of $C_{A}$ implies that $C_{A}$ is a singleton. We
have thus shown that
\begin{equation}
Ax=c_{A}x\text{ for each }A\in\mathbf{N}^{\mathrm{L}}\text{ and }x\in\mathbf{Z}_{\Phi}^{0,\mathrm{L}}v.\label{eq:Ax=00003Dc_Ax}
\end{equation}

Denote by $V$ the $\mathbb{C}$-span of $\mathbf{Z}_{\Phi}^{0,\mathrm{L}}v$.
Clearly we have $B(V)=V$ for each $B\in\mathbf{Z}_{\Phi}^{0,\mathrm{L}}$.
Let us show that $V=\mathbb{C}^{d}$. Since $\mathbf{Z}_{\Phi}/\mathbf{Z}_{\Phi}^{0}$
is finite, $\mathbf{Z}_{\Phi}^{\mathrm{L}}/\mathbf{Z}_{\Phi}^{0,\mathrm{L}}$
is also finite. Since $\mathbf{S}_{\Phi}^{\mathrm{L}}$ is $k$-strongly
irreducible and $k$-proximal for each $1\le k<d$, the same clearly
holds for $\mathbf{Z}_{\Phi}^{\mathrm{L}}$. Thus, since $\mathbf{Z}_{\Phi}^{\mathrm{L}}/\mathbf{Z}_{\Phi}^{0,\mathrm{L}}$
is finite, it follows that $\mathbf{Z}_{\Phi}^{0,\mathrm{L}}$ is
$k$-strongly irreducible and $k$-proximal for each $1\le k<d$.
From this and \cite[Lemmata 6.25 and 6.27]{BQ}, we obtain that there
exists $D\in\mathbf{Z}_{\Phi}^{0,\mathrm{L}}$ with $d$ distinct
real eigenvalues.

From $D(V)=V$, we get that $V$ contains an eigenvector of $D$.
Moreover, since $D\in\mathrm{GL}(d,\mathbb{R})$ and since the eigenvalues
of $D$ are real and distinct, it follows that the eigenspaces of
$D$ are $1$-dimensional and are spanned by vectors in $\mathbb{R}^{d}$.
Thus, there exists $0\ne w\in\mathbb{R}^{d}\cap V$. Since $B(V)=V$
for each $B\in\mathbf{Z}_{\Phi}^{0,\mathrm{L}}$, it follows that
$\mathbf{Z}_{\Phi}^{0,\mathrm{L}}w:=\{Bw\::\:B\in\mathbf{Z}_{\Phi}^{0,\mathrm{L}}\}$
is contained in $V$. Additionally, from the irreducibility of $\mathbf{Z}_{\Phi}^{0,\mathrm{L}}$
it follows that the $\mathbb{R}$-span of $\mathbf{Z}_{\Phi}^{0,\mathrm{L}}w$
is equal to $\mathbb{R}^{d}$, which clearly implies that $V=\mathbb{C}^{d}$.
From this, by (\ref{eq:Ax=00003Dc_Ax}), and since $\mathbf{N}\ne\{1_{\mathrm{A}}\}$,
it follows that there exist $\phi\in\mathbf{Z}_{\Phi}\setminus\{1_{\mathrm{A}}\}$
and $c_{\phi}\in\mathbb{R}\setminus\{0\}$ with $A_{\phi}=c_{\phi}I$,
which completes the proof of the lemma.
\end{proof}
\begin{lem}
\label{lem:K not cont in V}There does not exist a proper affine subspace
$V$ of $\mathbb{R}^{d}$ so that $K_{\Phi}\subset V$.
\end{lem}

\begin{proof}
Let $V$ be the affine span of $K_{\Phi}$. There exists $x\in\mathbb{R}^{d}$
and a linear subspace $W$ of $\mathbb{R}^{d}$ so that $V=x+W$.
By assumption (\ref{eq:no com fix asump}) it follows that $K_{\Phi}$
is not a singleton, which implies that $W\ne\{0\}$.

Given $i\in\Lambda$ the affine span of $\varphi_{i}(K_{\Phi})$ is
equal to $\varphi_{i}(x)+A_{i}(W)$. From this and $\varphi_{i}(K_{\Phi})\subset K_{\Phi}$,
it follows that $\varphi_{i}(x)+A_{i}(W)\subset x+W$, which gives
$A_{i}(W)=W$. Thus, from $W\ne\{0\}$ and since $\mathbf{S}_{\Phi}^{\mathrm{L}}$
is irreducible, it follows that $W=\mathbb{R}^{d}$. Hence $V=\mathbb{R}^{d}$,
which completes the proof of the lemma.
\end{proof}
Recall that for $a\in\mathbb{R}^{d}$ the map $T_{a}\in\mathrm{A}_{d,d}$
is given by $T_{a}x=x+a$ for $x\in\mathbb{R}^{d}$.
\begin{lem}
\label{lem:T_a in Z_Phi all a}For each $a\in\mathbb{R}^{d}$ we have
$T_{a}\in\mathbf{Z}_{\Phi}$.
\end{lem}

\begin{proof}
Set $W:=\{x\in\mathbb{R}^{d}\::\:T_{x}\in\mathbf{Z}_{\Phi}\}$, so
that we need to show that $W=\mathbb{R}^{d}$. Let us first show that
$W\ne\{0\}$. By Lemma \ref{lem:exist map with scal lin part}, there
exist $\phi\in\mathbf{Z}_{\Phi}\setminus\{1_{\mathrm{A}}\}$ and $c_{\phi}\in\mathbb{R}\setminus\{0\}$
so that $A_{\phi}=c_{\phi}I$. From $\phi\ne1_{\mathrm{A}}$ and by
Lemma \ref{lem:K not cont in V}, it follows that there exists $z\in K_{\Phi}$
so that $\phi(z)\ne z$.

Let $\omega\in\Lambda^{\mathbb{N}}$ be with $\Pi\omega=z$, and for
$n\ge1$ set $\tau_{n}:=\phi\circ\varphi_{\omega|_{n}}\circ\phi^{-1}\circ\varphi_{\omega|_{n}}^{-1}$.
Direct computation gives
\[
\tau_{n}(0)=c_{\phi}\varphi_{\omega|_{n}}(0)+a_{\phi}-\varphi_{\omega|_{n}}(0)-A_{\omega|_{n}}a_{\phi},
\]
where recall that $a_{\phi}$ denotes the translation part of $\phi$.
Thus, since $z=\Pi\omega=\underset{n\rightarrow\infty}{\lim}\:\varphi_{\omega|_{n}}(0)$
and $\underset{n\rightarrow\infty}{\lim}\:A_{\omega|_{n}}a_{\phi}=0$,
\[
\underset{n\rightarrow\infty}{\lim}\:\tau_{n}(0)=c_{\phi}z+a_{\phi}-z=\phi(z)-z\ne0.
\]
From this and since $\tau_{n}\in\mathbf{Z}_{\Phi}$ and $A_{\tau_{n}}=I$
for each $n\ge1$, it follows that $W\ne\{0\}$.

Let us next show that $W$ is a linear subspace of $\mathbb{R}^{d}$.
For $x,y\in W$ we have $T_{x+y}=T_{x}\circ T_{y}\in\mathbf{Z}_{\Phi}$, and so $x+y\in W$.
Let $x\in W$ and $s\in\mathbb{R}$ be given. We next show that $T_{sx}\in\mathbf{Z}_{\Phi}$,
which implies $sx\in W$. Since $\mathbf{Z}_{\Phi}$ is a
Zariski closed in $\mathrm{A}_{d,d}$, it suffices to show that $P(T_{sx})=0$
for every polynomial map $P:\mathrm{A}_{d,d}^{\mathrm{vec}}\rightarrow\mathbb{R}$
with $\mathbf{Z}_{\Phi}\subset P^{-1}\{0\}$.

Let $P$ be such a polynomial map, and for $t\in\mathbb{R}$ set $Q(t):=P(T_{tx})$.
Since $Q$ is a polynomial of a single real variable, it is either
the zero polynomial or it has at most finitely many roots. Since $x\in W$
we have $T_{x}\in\mathbf{Z}_{\Phi}$, and so $T_{x}^{n}\in\mathbf{Z}_{\Phi}$
for each $n\ge1$. Thus, from $\mathbf{Z}_{\Phi}\subset P^{-1}\{0\}$,
\[
Q(n)=P(T_{nx})=P(T_{x}^{n})=0\text{ for all }n\ge1.
\]
This shows that $Q$ is the zero polynomial, and in particular that
$P(T_{sx})=Q(s)=0$. As noted above, this implies that $sx\in W$,
which shows that $W$ is a linear subspace of $\mathbb{R}^{d}$.

We can now complete the proof. Given $\psi\in\mathbf{S}_{\Phi}$ and
$x\in W$ we have $\psi\circ T_{x}\circ\psi^{-1}\in\mathbf{Z}_{\Phi}$.
Moreover,
\[
\psi\circ T_{x}\circ\psi^{-1}(y)=y+A_{\psi}x\text{ for every }y\in\mathbb{R}^{d},
\]
which implies that $A_{\psi}x\in W$.
We have thus shown that $A(W)=W$ for every $A\in\mathbf{S}_{\Phi}^{\mathrm{L}}$.
From this, since $\mathbf{S}_{\Phi}^{\mathrm{L}}$ is irreducible,
and since $W$ is a nonzero linear subspace of $\mathbb{R}^{d}$,
it follows that $W=\mathbb{R}^{d}$, which completes the proof of
the lemma.
\end{proof}
\begin{proof}[Proof of Proposition \ref{prop:no psi exist}]
Let $\pi\in\mathrm{A}_{d,m}$ be linear, and assume by contradiction
that there exists $0\ne\psi\in\mathrm{A}_{d,m}^{\mathrm{vec}}$ so
that $\psi(x)\in\pi L_{m-1}(x)$ for $\mu$-a.e. $x$. Then for $\beta$-a.e.
$\omega$,
\[
\psi(\Pi\omega)\in\pi L_{m-1}(\Pi\omega)=\pi L_{m-1}(\omega).
\]
Thus, by (\ref{eq:equivar of coding map}) and (\ref{eq:equivar of bd maps}) and since $\beta_{[u]}\ll\beta$ for $u\in\Lambda^*$,
it follows that for all $u\in\Lambda^{*}$ and $\beta$-a.e. $\omega$
\begin{equation}
\psi(\varphi_{u}(\Pi\omega))=\psi(\Pi(u\omega))\in\pi L_{m-1}(u\omega)=\pi A_{u}L_{m-1}(\omega).\label{eq:equivariant for all u}
\end{equation}
Fix an $\omega\in\Lambda^{\mathbb{N}}$ for which \ref{eq:equivariant for all u}
holds for all $u\in\Lambda^{*}$, set $x=\Pi\omega$, and let $x_{1},...,x_{m-1}\in L_{m-1}(\omega)$
be independent vectors.

Recall from Section \ref{subsec:spaces of alt forms} that $\wedge^{m}\mathbb{R}^{m}$
is equipped with a norm $\Vert\cdot\Vert$. Let $P:\mathrm{A}_{d,d}^{\mathrm{vec}}\rightarrow\mathbb{R}$
be with
\[
P(\varphi):=\Vert\psi(\varphi(x))\wedge\pi A_{\varphi}x_{1}\wedge...\wedge\pi A_{\varphi}x_{m-1}\Vert^{2}\text{ for }\varphi\in\mathrm{A}_{d,d}^{\mathrm{vec}},
\]
and note that $P$ is a polynomial map. By the choice of $\omega$
\[
\psi(\varphi_{u}(x))\in\mathrm{span}\{\pi A_{\varphi_{u}}x_{1},...,\pi A_{\varphi_{u}}x_{m-1}\}\text{ for all }u\in\Lambda^{*}.
\]
This together with (\ref{eq:equiv cond for lin indep}) implies that
$P(\varphi)=0$ for all $\varphi\in\mathbf{S}_{\Phi}$. Thus, since
$\mathbf{Z}_{\Phi}$ the Zariski closure of $\mathbf{S}_{\Phi}$,
\begin{equation}
\psi(\varphi(x))\wedge\pi A_{\varphi}x_{1}\wedge...\wedge\pi A_{\varphi}x_{m-1}=0\:\text{ for all }\varphi\in\mathbf{Z}_{\Phi}.\label{eq:wedge =00003D0 for all varphi}
\end{equation}

Since $\psi\ne0$, there exists $y\in\mathbb{R}^{d}$ so that $\psi(y)\ne0$.
By Lemma \ref{lem:T_a in Z_Phi all a}, we have $T_{a}\in\mathbf{Z}_{\Phi}$
for all $a\in\mathbb{R}^{d}$. In particular there exists $\varphi\in\mathbf{Z}_{\Phi}$
so that $\varphi(x)=y$, and so (\ref{eq:wedge =00003D0 for all varphi})
is not possible when $m=1$. Thus, we may assume that $m>1$.

Let $F:\wedge^{m-1}\mathbb{R}^{d}\rightarrow\wedge^{m}\mathbb{R}^{m}$
be the linear map with
\[
F(z_{1}\wedge...\wedge z_{m-1})=\psi(y)\wedge\pi z_{1}\wedge...\wedge\pi z_{m-1}\text{ for }z_{1},...,z_{m-1}\in\mathbb{R}^{d}.
\]
Since $\pi$ is surjective and $\psi(y)\ne0$, we have $\ker(F)\ne\wedge^{m-1}\mathbb{R}^{d}$.
From this, since $\mathbf{S}_{\Phi}^{\mathrm{L}}$ is $(m-1)$-irreducible,
and since $x_{1}\wedge...\wedge x_{m-1}\ne0$, it follows that there
exists $\phi\in\mathbf{Z}_{\Phi}$ so that
\[
F(\wedge^{m-1}A_{\phi}(x_{1}\wedge...\wedge x_{m-1}))\ne0.
\]
That is,
\[
\psi(y)\wedge\pi A_{\phi}x_{1}\wedge...\wedge\pi A_{\phi}x_{m-1}\ne0.
\]

Now set $\varphi:=T_{y-\phi(x)}\circ\phi$, then by Lemma \ref{lem:T_a in Z_Phi all a}
we have $\varphi\in\mathbf{Z}_{\Phi}$. Moreover, since $A_{\varphi}=A_{\phi}$
and $\varphi(x)=y$,
\[
\psi(\varphi(x))\wedge\pi A_{\varphi}x_{1}\wedge...\wedge\pi A_{\varphi}x_{m-1}=\psi(y)\wedge\pi A_{\phi}x_{1}\wedge...\wedge\pi A_{\phi}x_{m-1}\ne0.
\]
But this contradicts (\ref{eq:wedge =00003D0 for all varphi}), which
completes the proof of the proposition.
\end{proof}

\subsection{\label{subsec:Proof-of-ent inc thm}Proof of Theorem \ref{thm:ent increase result}}

For the proof of the theorem we also need the following two lemmas.
\begin{lem}
\label{lem:limit step in ent inc pf}Let $B\subset\mathrm{A}_{d,m}^{\mathrm{vec}}$
be compact, let $\xi,\xi_{1},\xi_{2},...\in\mathcal{M}(B)$ be such
that $\xi_{k}\overset{k}{\rightarrow}\xi$ weakly, let $\pi,\pi_{1},\pi_{2},...\in\mathrm{A}_{d,m}$
be linear and with $\Vert\pi-\pi_{k}\Vert_{op}\overset{k}{\rightarrow}0$,
and let $\delta_{1},\delta_{2},...\in(0,\infty)$ be with $\delta_{k}\overset{k}{\rightarrow}0$.
Suppose that for each $k\ge1$,
\begin{equation}
\mu\left\{ x\::\:\xi_{k}\ldotp x\mbox{ is }(\pi_{k}L_{m-1}(x),\delta_{k})\text{-concentrated}\right\} >1-\delta_{k}.\label{eq:conc assumption}
\end{equation}
Then there exists a map $T:\mathbb{R}^{d}\rightarrow\mathbb{R}^{m}$
with
\[
\xi\ldotp x\left(T(x)+\pi L_{m-1}(x)\right)=1\text{ for }\mu\text{-a.e. }x.
\]
\end{lem}

\begin{proof}
By moving to a subsequence without changing the notation, we may assume
that $\sum_{k\ge1}\delta_{k}<\infty$. Thus, by (\ref{eq:conc assumption})
and by the Borel-Cantelli lemma, there exists a Borel subset $E$
of $\mathbb{R}^{d}$ with $\mu(E)=1$, so that for each $x\in E$
there exists $K_{x}\ge1$ such that
\begin{equation}
\xi_{k}\ldotp x\text{ is }(\pi_{k}L_{m-1}(x),\delta_{k})\text{-concentrated for each }k\ge K_{x}.\label{eq:main prop of E}
\end{equation}
Since $L_{m-1}\mu=\nu_{m-1}$ (see Remark \ref{rem:L_m-1mu=00003Dnu_m-1})
and by Lemma \ref{lem:nu=00007Bkappa<del=00007D<eps}, we may also
assume that $\dim(\pi L_{m-1}(x))=m-1$ for all $x\in E$ while still
having $\mu(E)=1$.

Fix $x\in E$, and set $B.x:=\{\psi(x)\::\:\psi\in B\}$. For
$y\in\mathbb{R}^{m}$ and $\pi'\in\mathrm{A}_{d,m}$ let $f_{y,\pi'}:\mathrm{A}_{d,m}^{\mathrm{vec}}\rightarrow\mathbb{R}$
be such that
\[
f_{y,\pi'}(\psi)=\inf\left\{ \left|\psi(x)-z\right|\::\:z\in y+\pi'L_{m-1}(x)\right\} \text{ for }\psi\in\mathrm{A}_{d,m}^{\mathrm{vec}}.
\]
By (\ref{eq:main prop of E}), it follows that for each $k\ge K_{x}$
there exists $T_{k}(x)\in B.x$ so that
\begin{equation}
\int f_{T_{k}(x),\pi_{k}}\:d\xi_{k}=O_{B,x}(\delta_{k}).\label{eq:=00003DO(delta_k)}
\end{equation}

Since $B.x$ is compact, there exist an increasing sequence $\{k_{j}\}_{j\ge1}\subset\mathbb{N}$
and $T(x)\in\mathbb{R}^{m}$ so that $T_{k_{j}}(x)\overset{j}{\rightarrow}T(x)$.
Since $f_{T(x),\pi}$ is continuous and $\xi_{k}\overset{k}{\rightarrow}\xi$
weakly,
\begin{equation}
\int f_{T(x),\pi}\:d\xi=\underset{k\rightarrow\infty}{\lim}\int f_{T(x),\pi}\:d\xi_{k}.\label{eq:from conv in dest}
\end{equation}
From $\dim(\pi L_{m-1}(x))=m-1$, $\Vert\pi-\pi_{k}\Vert_{op}\overset{k}{\rightarrow}0$
and $T_{k_{j}}(x)\overset{j}{\rightarrow}T(x)$, it follows that $f_{T_{k_{j}}(x),\pi_{k_{j}}}\overset{j}{\rightarrow}f_{T(x),\pi}$
uniformly on $B$. Thus,
\[
\underset{j\rightarrow\infty}{\lim}\left|\int f_{T(x),\pi}\:d\xi_{k_{j}}-\int f_{T_{k_{j}}(x),\pi_{k_{j}}}\:d\xi_{k_{j}}\right|=0.
\]
From this, (\ref{eq:=00003DO(delta_k)}), (\ref{eq:from conv in dest})
and $\delta_{k}\overset{k}{\rightarrow}0$, it follows that $\int f_{T(x),\pi}\:d\xi=0$.
By the definition of $f_{T(x),\pi}$, this implies
\[
\xi\ldotp x\left(T(x)+\pi L_{m-1}(x)\right)=1.
\]
Since this holds for all $x\in E$ and $\mu(E)=1$, the proof
is now complete.
\end{proof}
\begin{lem}
\label{lem:xi eval to lines}Let $\xi\in\mathcal{M}(\mathrm{A}_{d,m}^{\mathrm{vec}})$
and a linear $\pi\in\mathrm{A}_{d,m}$ be given. Suppose that $\xi$
is not a mass point, and that there exists a map $T:\mathbb{R}^{d}\rightarrow\mathbb{R}^{m}$
with
\begin{equation}
\xi\ldotp x\left(T(x)+\pi L_{m-1}(x)\right)=1\text{ for }\mu\text{-a.e. }x.\label{eq:xi.x supp on proj line}
\end{equation}
Then there exists $0\ne\psi\in\mathrm{A}_{d,m}^{\mathrm{vec}}$ so
that $\psi(x)\in\pi L_{m-1}(x)$ for $\mu$-a.e. $x$.
\end{lem}

\begin{proof}
Since $\xi$ is not a mass point, there exist distinct $\psi_{1},\psi_{2}\in\mathrm{supp}(\xi)$.
For every $x\in\mathbb{R}^{d}$, $y\in\mathbb{R}^{m}$ and $L\in\mathrm{Gr}_{m-1}(m)$,
the set
\[
\left\{ \psi\in\mathrm{A}_{d,m}^{\mathrm{vec}}\::\:\psi(x)\in y+L\right\} 
\]
is closed in $\mathrm{A}_{d,m}^{\mathrm{vec}}$. From this and (\ref{eq:xi.x supp on proj line}),
it follows that for $\mu$-a.e. $x$
\[
\psi_{1}(x),\psi_{2}(x)\in\left\{ \psi(x)\::\:\psi\in\mathrm{supp}(\xi)\right\} \subset T(x)+\pi L_{m-1}(x).
\]
Thus
\[
\psi_{1}(x)-\psi_{2}(x)\in\pi L_{m-1}(x)\text{ for }\mu\text{-a.e. }x,
\]
which completes the proof of the lemma by taking $\psi:=\psi_{1}-\psi_{2}$.
\end{proof}
We are now ready to prove Theorem \ref{thm:ent increase result},
which is the following statement.
\begin{thm*}
Let $1\le m\le d$ be such that $\Sigma_{1}^{m-1}=m-1$ and $\Delta_{m}<1$,
and let $\epsilon,R>0$ be given. Then there exists $\delta=\delta(\epsilon,R)>0$
so that for all $n\ge N(\epsilon,R,\delta)$ and $\theta\in\mathcal{M}(\mathrm{A}_{d,m})$
with $\mathrm{supp}(\theta)\subset B(\pi_{d,m},R)$ and $\frac{1}{n}H(\theta,\mathcal{D}_{n})\ge\epsilon$,
we have $\frac{1}{n}H(\theta\ldotp\mu,\mathcal{D}_{n})>\Sigma_{1}^{m}+\delta$.
\end{thm*}
\begin{proof}
Let $\epsilon,R>0$ be given, and let $\epsilon':=\epsilon'(\epsilon,R)>0$
be as obtained in Proposition \ref{prop:key prep for ent inc}. Additionally,
for each $k\ge1$ let $\delta_{k}:=\delta(\epsilon,R,1/k)>0$ and
$N_{k}:=N(\epsilon,R,1/k)\ge1$ be as obtained in Proposition \ref{prop:key prep for ent inc}.

Assume by contradiction that the theorem is false. Then for each $k\ge1$
there exist $n_{k}\ge N_{k}$ and $\theta_{k}\in\mathcal{M}(\mathrm{A}_{d,m})$
so that $\mathrm{supp}(\theta_{k})\subset B(\pi_{d,m},R)$, $\frac{1}{n_{k}}H(\theta_{k},\mathcal{D}_{n_{k}})\ge\epsilon$
and $\frac{1}{n_{k}}H(\theta_{k}\ldotp\mu,\mathcal{D}_{n_{k}})\le\Sigma_{1}^{m}+\delta_{k}$.
From this and by Proposition \ref{prop:key prep for ent inc}, it
follows that for each $k\ge1$ there exist $\xi_{k}\in\mathcal{M}(\mathrm{A}_{d,m})$
and $V_{k}\in\mathrm{Gr}_{m}(d)$ such that,
\begin{enumerate}
\item \label{enu:xi_k p1}$\mathrm{diam}(\mathrm{supp}(\xi_{k}))=O(1)$
in $(\mathrm{A}_{d,m}^{\mathrm{vec}},\Vert\cdot\Vert)$;
\item \label{enu:xi_k p2}$\mu\left\{ x\::\:\xi_{k}\ldotp x\mbox{ is }(\pi_{V_{k}}L_{m-1}(x),1/k)\text{-concentrated}\right\} >1-1/k$;
\item \label{enu:xi_k p3}$\xi_{k}$ is not $(\{0\},\epsilon')$-concentrated
in $\mathrm{A}_{d,m}^{\mathrm{vec}}$.
\end{enumerate}
By property (\ref{enu:xi_k p1}) and by translating each of the measures
$\xi_{k}$ appropriately without changing the notation, we may assume
that they are all supported on a single compact subset $B$ of $\mathrm{A}_{d,m}^{\mathrm{vec}}$.
From this, by compactness, and by moving to a subsequence without
changing the notation, we may assume that there exist a linear $\pi\in\mathrm{A}_{d,m}$
and $\xi\in\mathcal{M}(\mathrm{A}_{d,m}^{\mathrm{vec}})$ so that
$\Vert\pi-\pi_{V_{k}}\Vert_{op}\overset{k}{\rightarrow}0$ and $\xi_{k}\overset{k}{\rightarrow}\xi$
weakly.

From property (\ref{enu:xi_k p2}) and by Lemma \ref{lem:limit step in ent inc pf},
there exists a map $T:\mathbb{R}^{d}\rightarrow\mathbb{R}^{m}$ with
\[
\xi\ldotp x\left(T(x)+\pi L_{m-1}(x)\right)=1\text{ for }\mu\text{-a.e. }x.
\]
From property (\ref{enu:xi_k p3}) it clearly follows that $\xi$
is not a mass point. Thus, by Lemma \ref{lem:xi eval to lines}, there
exists $0\ne\psi\in\mathrm{A}_{d,m}^{\mathrm{vec}}$ so that $\psi(x)\in\pi L_{m-1}(x)$
for $\mu$-a.e. $x$. But this contradicts Proposition \ref{prop:no psi exist},
which completes the proof of the theorem.
\end{proof}

\section{\label{sec:Asymptotic-entropies-of}Asymptotic entropies of convolutions}

The purpose of this section is to prove Theorems \ref{thm:asympt of cond ent}
and \ref{thm:=00003DRW ent}. For the proof of Theorem \ref{thm:asympt of cond ent},
in Section \ref{subsec:Non-conformal-partitions} we consider the
entropy of measures of the form $\theta\ldotp\mu$, where $\theta\in\mathcal{M}(\mathrm{A}_{d,m})$,
with respect to certain non-conformal partitions of $\mathbb{R}^{m}$.
In order to deduce Theorem \ref{thm:=00003DRW ent} from Theorem \ref{thm:asympt of cond ent},
in Section \ref{subsec:Projections-of-Diophantine} we show that in
a certain sense projections of Diophantine systems remain Diophantine
almost surely. In Section \ref{subsec:Proofs-of-Theorems asym ent of conv}
we carry out the proofs of the two Theorems.

\subsection{\label{subsec:Non-conformal-partitions}Non-conformal partitions}

Let $1\le m\le d$ and $\psi\in\mathrm{A}_{d,m}$ be with $\alpha_{1}(A_{\psi})>...>\alpha_{m}(A_{\psi})$.
Let $UD'U'$ be a singular value decomposition of $A_{\psi}$ (see
Section \ref{subsec:Singular-values-and SVD}), and set
\[
D:=\mathrm{diag}_{m,m}(\alpha_{1}(A_{\psi}),...,\alpha_{m}(A_{\psi}))\in\mathrm{GL}(m,\mathbb{R}).
\]
For $n\ge1$ we write $\mathcal{D}_{n}^{\psi}$ in place of
\[
UD(\mathcal{D}_{n}^{m}):=\{UD(E)\::\:E\in\mathcal{D}_{n}^{m}\}.
\]
The following proposition is the main result of this subsection.
\begin{prop}
\label{prop:non conf par}Let $1\le m\le d$ be with $\Sigma_{1}^{m-1}=m-1$.
Then for every $\epsilon,R>0$ and $n\ge N(\epsilon,R)\ge1$ the following
holds. Let $\theta\in\mathcal{M}(\mathrm{A}_{d,m})$ be such that
$\mathrm{diam}(\mathrm{supp}(\theta))\le R$, let $\psi\in\mathrm{supp}(\theta)$
be with $\alpha_{1}(A_{\psi})>...>\alpha_{m}(A_{\psi})$, and for
$1\le i\le m$ set $c_{i}:=-\frac{1}{n}\log\alpha_{i}(A_{\psi})$.
If $m>1$ suppose also that $\epsilon\le c_{m}-c_{1}\le R$. Then
for every $1\le M\le R$,
\[
\left|\frac{1}{Mn}H(\theta\ldotp\mu,\mathcal{D}_{(M+c_{m})n}\mid\mathcal{D}_{c_{m}n})-\frac{1}{Mn}H(\theta\ldotp\mu,\mathcal{D}_{Mn}^{\psi})\right|<\epsilon.
\]
\end{prop}

For the proof of the proposition we need the following lemma.
\begin{lem}
\label{lem:proj prep for nonconf part}Let $2\le m\le d$ be given.
Then for every $\epsilon,R>0$, $n\ge N(\epsilon,R)\ge1$, $0\le M\le R$,
$c\ge\epsilon$, and $\psi\in B(\pi_{d,m},R)$,
\[
\frac{1}{cn}H\left(\psi\mu,\pi_{m,m-1}^{-1}\mathcal{D}_{(M+c)n}\mid\mathcal{D}_{Mn}\right)\ge\Sigma_{1}^{m-1}-\epsilon.
\]
\end{lem}

\begin{proof}
Let $\epsilon>0$ be small, let $R>0$, let $k\ge1$ be large with respect to $\epsilon,R$,
and let $n\ge1$ be large with respect to $k$. Additionally, let
$0\le M\le R$, $c\ge\epsilon$, and $\psi\in B(\pi_{d,m},R)$ be
given. Setting $\varphi:=\pi_{m,m-1}\circ\psi\in\mathrm{A}_{d,m-1}$,
it is clear that 
\begin{equation}
d_{\mathrm{A}_{d,m-1}}(\pi_{d,m-1},\varphi)=O_{R}(1).\label{eq:d_A_d,m-1(pi,phi)=00003DO_R(1)}
\end{equation}
Thus, by Lemma \ref{lem:lb on ent of proj of comp of mu},
\[
\mathbb{P}_{1\le i\le(M+c)n}\left\{ \frac{1}{k}H\left(\varphi\mu_{x,i},\mathcal{D}_{i+k}\right)>\Sigma_{1}^{m-1}-\epsilon\right\} >1-\frac{c\epsilon}{c+M}.
\]
And so,
\begin{equation}
\mathbb{P}_{Mn\le i\le(M+c)n}\left\{ \frac{1}{k}H\left(\varphi\mu_{x,i},\mathcal{D}_{i+k}\right)>\Sigma_{1}^{m-1}-\epsilon\right\} >1-\epsilon.\label{eq:most comps bet M to M+c}
\end{equation}

Let $\mathcal{E}$ be the set of all $D\in\mathcal{D}_{Mn}$ such
that $\mu(D)>0$ and for $\sigma:=\mu_{D}$,
\[
\mathbb{P}_{Mn\le i\le(M+c)n}\left\{ \frac{1}{k}H\left(\varphi\sigma_{x,i},\mathcal{D}_{i+k}\right)>\Sigma_{1}^{m-1}-\epsilon\right\} >1-\epsilon.
\]
By (\ref{eq:most comps bet M to M+c}), and by replacing $\epsilon$
with a larger quantity (which is still small) without changing the
notation, we may assume that $\mu(\cup\mathcal{E})>1-\epsilon$. Moreover,
given $D\in\mathcal{E}$ it follows by Lemma \ref{lem:multiscale-entropy-formula}
and (\ref{eq:cond ent as avg of ent of comp}), from (\ref{eq:d_A_d,m-1(pi,phi)=00003DO_R(1)}), and by the concavity of
conditional entropy, that for $\sigma:=\mu_{D}$
\begin{eqnarray*}
\frac{1}{cn}H\left(\varphi\sigma,\mathcal{D}_{(M+c)n}\right) & = & \mathbb{E}_{Mn\le i\le(M+c)n}\left(\frac{1}{k}H\left(\varphi\sigma,\mathcal{D}_{i+k}\mid\mathcal{D}_{i}\right)\right)-O_R\left(\frac{k}{cn}\right)\\
 & \ge & \mathbb{E}_{Mn\le i\le(M+c)n}\left(\frac{1}{k}H\left(\varphi\sigma_{x,i},\mathcal{D}_{i+k}\right)\right)-O_{R}\left(\frac{k}{cn}+\frac{1}{k}\right)\\
 & \ge & \left(\Sigma_{1}^{m-1}-\epsilon\right)(1-\epsilon)-\epsilon.
\end{eqnarray*}
From concavity and $d_{\mathrm{A}_{d,m}}(\pi_{d,m},\psi)\le R$ we
also get,
\begin{eqnarray*}
\frac{1}{cn}H\left(\psi\mu,\pi_{m,m-1}^{-1}\mathcal{D}_{(M+c)n}\mid\mathcal{D}_{Mn}\right) & \ge & \mathbb{E}_{i=Mn}\left(\frac{1}{cn}H\left(\psi\mu_{x,i},\pi_{m,m-1}^{-1}\mathcal{D}_{(M+c)n}\mid\mathcal{D}_{Mn}\right)\right)\\
 & = & \mathbb{E}_{i=Mn}\left(\frac{1}{cn}H\left(\varphi\mu_{x,i},\mathcal{D}_{(M+c)n}\right)\right)-O_{R,\epsilon}(\frac{1}{n}).
\end{eqnarray*}
The lemma now follows from the last two inequalities, from $\mu(\cup\mathcal{E})>1-\epsilon$,
and by replacing $\epsilon$ with a larger quantity without changing
the notation.
\end{proof}
\begin{proof}[Proof of Proposition \ref{prop:non conf par}]
Let $\epsilon,R>0$, let $n\ge1$ be large with respect to $\epsilon,R$,
and let $\theta$, $\psi$, $c_{1},...,c_{m}$, and $M$ be as in
the statement of the proposition. Let $UD'U'$ be a singular value
decomposition of $A_{\psi}$, set
\[
D:=\mathrm{diag}_{m,m}(\alpha_{1}(A_{\psi}),...,\alpha_{m}(A_{\psi})),
\]
and let $\varphi\in\mathrm{A}_{m,m}$ be with $\varphi(x)=UDx+a_{\psi}$
for $x\in\mathbb{R}^{m}$.

Since $d_{\mathrm{A}_{d,m}}$ is $\mathrm{A}_{m,m}$-invariant (see
Lemma \ref{lem: def of inv met and prop}) and from $\mathrm{diam}(\mathrm{supp}(\theta))\le R$,
it follows that for $f\in\mathrm{supp}(\varphi^{-1}\theta)$
\begin{eqnarray}
d_{\mathrm{A}_{d,m}}(\pi_{d,m},f) & \le & d_{\mathrm{A}_{d,m}}(\pi_{d,m},\varphi^{-1}\circ\psi)+d_{\mathrm{A}_{d,m}}(\varphi^{-1}\circ\psi,f)\nonumber \\
 & = & O(1)+d_{\mathrm{A}_{d,m}}(\psi,\varphi\circ f)=O_{R}(1).\label{eq:=00003DO_R(1) all f}
\end{eqnarray}
This implies that $\mathrm{diam}(\mathrm{supp}(\varphi^{-1}\theta\ldotp\mu))=O_{R}(1)$,
which gives
\begin{equation}
H(\theta\ldotp\mu,\mathcal{D}_{0}^{\psi})=H(\varphi^{-1}\theta\ldotp\mu,\mathcal{D}_{0})+O(1)=O_{R}(1).\label{eq:nonconf theta.mu ent D_0}
\end{equation}

Assuming $m=1$, we have for $Q=0,M$
\[
H(\theta\ldotp\mu,\mathcal{D}_{Qn}^{\psi})=H(\theta\ldotp\mu,\mathcal{D}_{(Q+c_{m})n})+O(1).
\]
This together with (\ref{eq:nonconf theta.mu ent D_0}) yields the
proposition in the case $m=1$. For the rest of the proof suppose
that $m\ge2$.

Recall from Section \ref{subsec:Singular-values-and SVD} the definition
of $L_{m-1}(A_{\psi})\in\mathrm{Gr}_{m-1}(m)$. It is easy to verify
that for $Q=0,M$ the partitions $\mathcal{D}_{(Q+c_{m})n}^{m}$ and
$\mathcal{D}_{Qn}^{\psi}\vee\pi_{L_{m-1}(A_{\psi})}^{-1}\mathcal{D}_{(Q+c_{m})n}^{m-1}$
are $O(1)$-commensurable. Thus,
\begin{eqnarray}
H(\theta\ldotp\mu,\mathcal{D}_{(M+c_{m})n}\mid\mathcal{D}_{c_{m}n}) & = & H(\theta\ldotp\mu,\mathcal{D}_{(M+c_{m})n})-H(\theta\ldotp\mu,\mathcal{D}_{c_{m}n})\nonumber \\
 & = & H(\theta\ldotp\mu,\mathcal{D}_{Mn}^{\psi}\vee\pi_{L_{m-1}(A_{\psi})}^{-1}\mathcal{D}_{(M+c_{m})n})\nonumber \\
 & - & H(\theta\ldotp\mu,\mathcal{D}_{0}^{\psi}\vee\pi_{L_{m-1}(A_{\psi})}^{-1}\mathcal{D}_{c_{m}n})+O(1)\nonumber \\
 & = & H(\theta\ldotp\mu,\mathcal{D}_{Mn}^{\psi})+H(\theta\ldotp\mu,\pi_{L_{m-1}(A_{\psi})}^{-1}\mathcal{D}_{(M+c_{m})n}\mid\mathcal{D}_{Mn}^{\psi})\label{eq:nonconf part first step}\\
 & - & H(\theta\ldotp\mu,\mathcal{D}_{0}^{\psi})-H(\theta\ldotp\mu,\pi_{L_{m-1}(A_{\psi})}^{-1}\mathcal{D}_{c_{m}n}\mid\mathcal{D}_{0}^{\psi})+O(1).\nonumber 
\end{eqnarray}

Let $Q$ be either $M$ or $0$, and set
\begin{equation}
\gamma_{Q}:=H(\theta\ldotp\mu,\pi_{L_{m-1}(A_{\psi})}^{-1}\mathcal{D}_{(Q+c_{m})n}\mid\mathcal{D}_{Qn}^{\psi}).\label{eq:def of gamma_Q}
\end{equation}
Note that $\pi_{L_{m-1}(A_{\psi})}=F\pi_{m,m-1}U^{-1}$ for some linear
isometry $F:\mathbb{R}^{m-1}\rightarrow\mathbb{R}^{m-1}$. Moreover,
the partitions
\[
D^{-1}\pi_{m,m-1}^{-1}\mathcal{D}_{(Q+c_{m})n}\:\text{ and }\:\bigvee_{i=1}^{m-1}\pi_{m,i}^{-1}\mathcal{D}_{(Q+c_{m}-c_{i})n}
\]
are $O(1)$-commensurable. Thus,
\begin{eqnarray}
\gamma_{Q} & = & H\left(\theta\ldotp\mu,U\pi_{m,m-1}^{-1}\mathcal{D}_{(Q+c_{m})n}\mid UD\mathcal{D}_{Qn}\right)+O(1)\nonumber \\
 & = & H\left(\varphi^{-1}\theta\ldotp\mu,D^{-1}\pi_{m,m-1}^{-1}\mathcal{D}_{(Q+c_{m})n}\mid\mathcal{D}_{Qn}\right)+O(1)\nonumber \\
 & = & H\left(\varphi^{-1}\theta\ldotp\mu,\bigvee_{i=1}^{m-1}\pi_{m,i}^{-1}\mathcal{D}_{(Q+c_{m}-c_{i})n}\mid\mathcal{D}_{Qn}\right)+O(1).\label{eq:gamma_Q =00003D}
\end{eqnarray}

For every $E\in\mathcal{D}_{Qn}^{m}$,
\[
\log\left|\left\{ E'\in\bigvee_{i=1}^{m-1}\pi_{m,i}^{-1}\mathcal{D}_{(Q+c_{m}-c_{i})n}\::\:E\cap E'\ne\emptyset\right\} \right|=\sum_{i=1}^{m-1}(c_{m}-c_{i})n+O(1).
\]
Hence, from (\ref{eq:gamma_Q =00003D})
\begin{equation}
\gamma_{Q}\le\sum_{i=1}^{m-1}(c_{m}-c_{i})n+O(1).\label{eq:ub for gamma_Q}
\end{equation}

Let $f\in\mathrm{supp}(\varphi^{-1}\theta)$ be given. By (\ref{eq:=00003DO_R(1) all f}),
Lemma \ref{lem:proj prep for nonconf part}, and $\Sigma_{1}^{m-1}=m-1$,
\begin{eqnarray*}
(m-1)(c_{m}-c_{1})-\epsilon & \le & \frac{1}{n}H\left(f\mu,\pi_{m,m-1}^{-1}\mathcal{D}_{(Q+c_{m}-c_{1})n}\mid\mathcal{D}_{Qn}\right)\\
 & = & \frac{1}{n}H\left(f\mu,\bigvee_{i=1}^{m-1}\pi_{m,i}^{-1}\mathcal{D}_{(Q+c_{m}-c_{i})n}\mid\mathcal{D}_{Qn}\right)\\
 & + & \frac{1}{n}H\left(f\mu,\pi_{m,m-1}^{-1}\mathcal{D}_{(Q+c_{m}-c_{1})n}\mid\bigvee_{i=1}^{m}\pi_{m,i}^{-1}\mathcal{D}_{(Q+c_{m}-c_{i})n}\right).
\end{eqnarray*}
Moreover, for every $E\in\bigvee_{i=1}^{m}\pi_{m,i}^{-1}\mathcal{D}_{(Q+c_{m}-c_{i})n}$
\[
\log\left|\left\{ E'\in\pi_{m,m-1}^{-1}\mathcal{D}_{(Q+c_{m}-c_{1})n}\::\:E\cap E'\ne\emptyset\right\} \right|=\sum_{i=1}^{m-1}(c_{i}-c_{1})n+O(1).
\]
Thus,
\[
\frac{1}{n}H\left(f\mu,\bigvee_{i=1}^{m-1}\pi_{m,i}^{-1}\mathcal{D}_{(Q+c_{m}-c_{i})n}\mid\mathcal{D}_{Qn}\right)\ge\sum_{i=1}^{m-1}(c_{m}-c_{i})-2\epsilon.
\]

From the last inequality, from (\ref{eq:gamma_Q =00003D}), and by
the concavity of conditional entropy,
\[
\frac{1}{n}\gamma_{Q}\ge\sum_{i=1}^{m-1}(c_{m}-c_{i})-3\epsilon.
\]
Hence, from (\ref{eq:ub for gamma_Q}),
\[
\left|\frac{1}{n}\gamma_{Q}-\sum_{i=1}^{m-1}(c_{m}-c_{i})\right|\le3\epsilon.
\]
Since this holds for $Q=M,0$, the proposition now follows from (\ref{eq:nonconf theta.mu ent D_0}),
(\ref{eq:nonconf part first step}) and (\ref{eq:def of gamma_Q}). 
\end{proof}

\subsection{\label{subsec:Projections-of-Diophantine}Projections of Diophantine
systems}

The purpose of this subsection is to prove the following proposition.
\begin{prop}
\label{prop:diophantine --> diophantine for projs}Suppose that $\Phi$
is Diophantine (see Definition \ref{def:Diophantine =000026 ESC}), and let $1\le m\le d$ be given. Then there exists
$\epsilon>0$ so that for $\nu_{m}^{*}$-a.e. $V$ and all $n\ge N(V)\ge1$,
\[
d_{\mathrm{A}_{d,m}}(\pi_{V}\varphi_{u},\pi_{V}\varphi_{w})\ge\epsilon^{n}\text{ for all }u,w\in\Lambda^{n}\text{ with }\varphi_{u}\ne\varphi_{w}.
\]
\end{prop}

The proof of the proposition requires some preparations.
\begin{lem}
\label{lem:There-exist-alpha =000026 C}There exist $\alpha,C>1$
so that,
\[
\nu_{1}^{*}\left\{ V\in\mathrm{Gr}_{1}(d)\::\:|\pi_{V}(x)|\le r^{\alpha}\right\} <Cr\text{ for all }r>0\text{ and }x\in\mathbb{R}^{d}\text{ with }|x|=1.
\]
\end{lem}

\begin{proof}
By a theorem of Guivarc\textquoteright h \cite[Section 2.8]{Gu} (see
also \cite[Theorem 14.1]{BQ}), there exist $0<\gamma<1$ and $C>1$
so that
\[
\int|\pi_{V}(x)|^{-\gamma}\:d\nu_{1}^{*}(V)<C\text{ for all }x\in\mathbb{R}^{d}\text{ with }|x|=1.
\]
Thus, by Markov's inequality, for every $r>0$ and $x\in\mathbb{R}^{d}$
with $|x|=1$
\[
r^{-1}\cdot\nu_{1}^{*}\left\{ V\::\:|\pi_{V}(x)|^{-\gamma}\ge r^{-1}\right\} <C.
\]
By taking $\alpha:=1/\gamma$, this completes the proof of the lemma.
\end{proof}
Recall from Section \ref{subsec:Spaces-of-affine maps} that for $1\le m\le d$
and $\psi\in\mathrm{A}_{d,m}^{\mathrm{vec}}$,
\[
\Vert\psi\Vert:=\sup\left\{ |\psi(x)|\::\:x\in\mathbb{R}^{d}\text{ and }|x|\le1\right\} .
\]

\begin{lem}
\label{lem:diophantine for proj to lines}Suppose that $\Phi$ is
Diophantine. Then there exists $\delta>0$ so that for $\nu_{1}^{*}$-a.e.
$V$ and all $n\ge N(V)\ge1$,
\[
\Vert\pi_{V}\varphi_{u}-\pi_{V}\varphi_{w}\Vert\ge\delta^{n}\text{ for all }u,w\in\Lambda^{n}\text{ with }\varphi_{u}\ne\varphi_{w}.
\]
\end{lem}

\begin{proof}
For $n\ge1$ set
\[
\mathcal{Y}_{n}:=\left\{ (u,w)\in\Lambda^{n}\times\Lambda^{n}\::\:\varphi_{u}\ne\varphi_{w}\right\} .
\]
Given $u,w\in\Lambda^{*}$, let $x_{u,w}\in\mathbb{R}^{d}$ be with
$|x_{u,w}|\le1$ and
\[
\Vert\varphi_{u}-\varphi_{w}\Vert=|\varphi_{u}(x_{u,w})-\varphi_{w}(x_{u,w})|.
\]
Since $\Phi$ is Diophantine, there exists $\epsilon>0$ so that
\begin{equation}
|\varphi_{u}(x_{u,w})-\varphi_{w}(x_{u,w})|\ge\epsilon^{n}\text{ for all }n\ge1\text{ and }(u,w)\in\mathcal{Y}_{n}.\label{eq:val at x_u,w >=00003D}
\end{equation}

Let $\alpha,C>1$ be as obtained in Lemma \ref{lem:There-exist-alpha =000026 C}.
For $n\ge1$ and $(u,w)\in\mathcal{Y}_{n}$ let
\[
E_{n}^{u,w}:=\left\{ V\in\mathrm{Gr}_{1}(d)\::\:\frac{\left|\pi_{V}\circ(\varphi_{u}-\varphi_{w})(x_{u,w})\right|}{\left|(\varphi_{u}-\varphi_{w})(x_{u,w})\right|}\le(2|\Lambda|^{2})^{-n\alpha}\right\} ,
\]
and set
\[
E_{n}:=\cup_{(u,w)\in\mathcal{Y}_{n}}E_{n}^{u,w}\text{ and }E:=\cap_{N\ge1}\cup_{n\ge N}E_{n}.
\]
By Lemma \ref{lem:There-exist-alpha =000026 C},
\begin{eqnarray*}
\sum_{n\ge1}\nu_{1}^{*}(E_{n}) & \le & \sum_{n\ge1}\sum_{(u,w)\in\mathcal{Y}_{n}}\nu_{1}^{*}(E_{n}^{u,w})\\
 & < & \sum_{n\ge1}\sum_{(u,w)\in\mathcal{Y}_{n}}C2^{-n}|\Lambda|^{-2n}<\infty.
\end{eqnarray*}
Thus, by the Borel-Cantelli lemma, we have $\nu_{1}^{*}(E)=0$.

Given $V\in\mathrm{Gr}_{1}(d)\setminus E$, there exists $N_{V}\ge1$
so that $V\notin\cup_{n\ge N_{V}}E_{n}$. Thus, by the definition
of the sets $E_{n}$ and from (\ref{eq:val at x_u,w >=00003D}), it
follows that for all $n\ge N_{V}$ and $(u,w)\in\mathcal{Y}_{n}$
\begin{eqnarray*}
\Vert\pi_{V}\varphi_{u}-\pi_{V}\varphi_{w}\Vert & \ge & \left|\pi_{V}\circ(\varphi_{u}-\varphi_{w})(x_{u,w})\right|\\
 & \ge & (2|\Lambda|^{2})^{-n\alpha}\left|(\varphi_{u}-\varphi_{w})(x_{u,w})\right|\\
 & \ge & (2^{-\alpha}|\Lambda|^{-2\alpha}\epsilon)^{n}.
\end{eqnarray*}
By taking $\delta:=2^{-\alpha}|\Lambda|^{-2\alpha}\epsilon$ this
completes the proof of the lemma.
\end{proof}
\begin{lem}
\label{lem:inv met >=00003D norm}Let $1\le m\le d$ be given. Then
there exists $c>0$ so that for all $V\in\mathrm{Gr}_{m}(d)$ and
$u,w\in\Lambda^{*}$,
\[
d_{\mathrm{A}_{d,m}}(\pi_{V}\varphi_{u},\pi_{V}\varphi_{w})\ge\min\left\{ 1,c\Vert\pi_{V}\varphi_{u}-\pi_{V}\varphi_{w}\Vert\right\} .
\]
\end{lem}

\begin{proof}
Write
\[
K:=\left\{ \psi\in\mathrm{A}_{d,m}\::\:d_{\mathrm{A}_{d,m}}(\psi,\pi_{W})\le1\text{ for some }W\in\mathrm{Gr}_{m}(d)\right\} .
\]
By Lemma \ref{lem: def of inv met and prop}, there exists $c>0$
so that
\begin{equation}
d_{\mathrm{A}_{d,m}}(\psi_{2},\psi_{2})\ge c\Vert\psi_{1}-\psi_{2}\Vert\text{ for all }\psi_{1},\psi_{2}\in K.\label{eq:d >=00003D c norm}
\end{equation}

Let $V\in\mathrm{Gr}_{m}(d)$ and $u,w\in\Lambda^{*}$ be given. By
(\ref{eq:rep of psi via pi_V}), there exist $W\in\mathrm{Gr}_{m}(d)$
and $\psi\in\mathrm{A}_{m,m}$ so that $\psi^{-1}\pi_{V}\varphi_{u}=\pi_{W}$.
Note that since $\varphi_{u}$ is a contraction we must have $\Vert A_{\psi}\Vert_{op}<1$,
which implies
\begin{equation}
\Vert\psi^{-1}\pi_{V}\varphi_{u}-\psi^{-1}\pi_{V}\varphi_{w}\Vert\ge\Vert\pi_{V}\varphi_{u}-\pi_{V}\varphi_{w}\Vert.\label{eq:psi ^-1 increases norm}
\end{equation}

If $\psi^{-1}\pi_{V}\varphi_{w}\in K$, then from $\psi^{-1}\pi_{V}\varphi_{u}=\pi_{W}\in K$,
(\ref{eq:d >=00003D c norm}) and (\ref{eq:psi ^-1 increases norm}),
\[
d_{\mathrm{A}_{d,m}}(\psi^{-1}\pi_{V}\varphi_{u},\psi^{-1}\pi_{V}\varphi_{w})\ge c\Vert\pi_{V}\varphi_{u}-\pi_{V}\varphi_{w}\Vert.
\]
Otherwise, if $\psi^{-1}\pi_{V}\varphi_{w}\notin K$ then from $\psi^{-1}\pi_{V}\varphi_{u}=\pi_{W}$
and the definition of $K$,
\[
d_{\mathrm{A}_{d,m}}(\psi^{-1}\pi_{V}\varphi_{u},\psi^{-1}\pi_{V}\varphi_{w})\ge1.
\]
By the $\mathrm{A}_{m,m}$-invariance of $d_{\mathrm{A}_{d,m}}$,
this completes the proof of the lemma.
\end{proof}
\begin{proof}[Proof of Proposition \ref{prop:diophantine --> diophantine for projs}]
It is clear that for every flag $(V_{j})_{j=0}^{d}\in\mathrm{F}(d)$,
\[
\Vert\pi_{V_{m}}\psi\Vert\ge\Vert\pi_{V_{1}}\psi\Vert\text{ for all }\psi\in\mathrm{A}_{d,d}^{\mathrm{vec}}.
\]
Thus, by (\ref{eq:fur =00003D proj of fur}) and Lemma \ref{lem:diophantine for proj to lines},
there exists $\delta>0$ so that for $\nu^{*}$-a.e. $(V_{j})_{j=0}^{d}$,
all $n\ge N(V_{1})\ge1$, and all $u,w\in\Lambda^{n}$ with $\varphi_{u}\ne\varphi_{w}$,
\[
\Vert\pi_{V_{m}}\varphi_{u}-\pi_{V_{m}}\varphi_{w}\Vert\ge\Vert\pi_{V_{1}}\varphi_{u}-\pi_{V_{1}}\varphi_{w}\Vert\ge\delta^{n}.
\]
The proposition now follows directly from this, (\ref{eq:fur =00003D proj of fur})
and Lemma \ref{lem:inv met >=00003D norm}.
\end{proof}

\subsection{\label{subsec:Proofs-of-Theorems asym ent of conv}Proofs of Theorems
\ref{thm:asympt of cond ent} and \ref{thm:=00003DRW ent}}

First we need the following lemmas.
\begin{lem}
\label{lem:lb on norm of wedge_pi}Let $1\le k\le m\le d$
and $\epsilon>0$ be given. Then there exists $\delta=\delta(\epsilon)>0$ so that
$\Vert\wedge^{k}\pi_{V}(u_{1}\wedge...\wedge u_{k})\Vert\ge\delta$
for every $V\in\mathrm{Gr}_{m}(d)$ and orthonormal set $\{u_{1},...,u_{k}\}\subset\mathbb{R}^{d}$
with $\kappa\left(V^{\perp},\mathrm{span}\{u_{1},...,u_{k}\}\right)\ge\epsilon$.
\end{lem}

\begin{proof}
Assume by contradiction that the lemma is false. Then by a compactness
argument, there exist $V\in\mathrm{Gr}_{m}(d)$ and an orthonormal set $\{u_{1},...,u_{k}\}\subset\mathbb{R}^{d}$
so that $\wedge^{k}\pi_{V}(u_{1}\wedge...\wedge u_{k})=0$ and for
$W:=\mathrm{span}\{u_{1},...,u_{k}\}$ we have $V^{\perp}\cap W=\{0\}$.
From $\wedge^{k}\pi_{V}(u_{1}\wedge...\wedge u_{k})=0$ together with
(\ref{eq:equiv cond for lin indep}), it follows that $\pi_{V}x=0$
for some $0\ne x\in W$. But this contradicts $V^{\perp}\cap W=\{0\}$,
which completes the proof of the lemma.
\end{proof}
\begin{lem}
\label{lem:sing vals of proj}Let $1\le m\le d$ and $V\in\mathrm{Gr}_{m}(d)$
be given. Then for $\beta$-a.e. $\omega$,
\[
\underset{n\rightarrow\infty}{\lim}\:\frac{1}{n}\log\alpha_{k}\left(\pi_{V}A_{\omega|_{n}}\right)=\chi_{k}\text{ for each }1\le k\le m.
\]
\end{lem}

\begin{proof}
For $\beta$-a.e. $\omega$,
\begin{equation}
\underset{n\rightarrow\infty}{\lim}\:\frac{1}{n}\log\alpha_{k}(A_{\omega|_{n}})=\chi_{k}\text{ for each }1\le k\le m.\label{eq:lim=00003Dchi_k asy of sin}
\end{equation}
Moreover, by Lemma \ref{lem:nu=00007Bkappa<del=00007D<eps}, since
$L_{m}\beta=\nu_{m}$, and from (\ref{enu:def of L_m via L_m of matrices})
in Section \ref{subsec:Coding and Furstenberg maps}, it follows that
for $\beta$-a.e. $\omega$
\begin{equation}
\underset{n\rightarrow\infty}{\lim}\:\kappa(V^{\perp},L_{m}(A_{\omega|_{n}}))>0.\label{eq:lim kappa >0 asy of sin}
\end{equation}
Fix $\omega\in\Lambda^{\mathbb{N}}$ for which (\ref{eq:lim=00003Dchi_k asy of sin})
and (\ref{eq:lim kappa >0 asy of sin}) are satisfied.

Let $1\le k\le m$ be given. For $n\ge1$ let $U_{1,n}D_{n}U_{2,n}$
be a singular value decomposition of $A_{\omega|_{n}}$. From (\ref{eq:lim kappa >0 asy of sin})
and by Lemma \ref{lem:lb on norm of wedge_pi}, there exist $\delta>0$
and $N\ge1$ so that
\[
\Vert\wedge^{k}\pi_{V}(U_{1,n}e_{d,1}\wedge...\wedge U_{1,n}e_{d,k})\Vert\ge\delta\text{ for }n\ge N.
\]
Thus, from (\ref{eq:norm =00003D prod of sing vals}) and since $\Vert\wedge^{k}\pi_{V}\Vert_{op}\le1$,
it follows that for $n\ge N$
\begin{multline*}
\prod_{j=1}^{k}\alpha_{j}(A_{\omega|_{n}})=\Vert\wedge^{k}A_{\omega|_{n}}\Vert_{op}\ge\Vert\wedge^{k}\pi_{V}A_{\omega|_{n}}\Vert_{op}\\
\ge\left\Vert \wedge^{k}\pi_{V}A_{\omega|_{n}}\left(U_{2,n}^{-1}e_{d,1}\wedge...\wedge U_{2,n}^{-1}e_{d,k}\right)\right\Vert \ge\delta\prod_{j=1}^{k}\alpha_{j}(A_{\omega|_{n}}).
\end{multline*}
Hence, by (\ref{eq:lim=00003Dchi_k asy of sin})
\[
\underset{n\rightarrow\infty}{\lim}\:\frac{1}{n}\log\Vert\wedge^{k}\pi_{V}A_{\omega|_{n}}\Vert_{op}=\sum_{j=1}^{k}\chi_{j}\text{ for each }1\le k\le m.
\]
This together with (\ref{eq:norm =00003D prod of sing vals}) completes
the proof of the lemma.
\end{proof}
Recall from Section \ref{subsec:Full-dimensionality} that
\[
p_{\Phi}:=\sum_{i\in\Lambda}p_{i}\delta_{\varphi_{i}}\in\mathcal{M}(\mathrm{A}_{d,d}),
\]
and that for $n\ge1$ we write $p_{\Phi}^{*n}$ for the convolution
of $p_{\Phi}$ with itself $n$ times. We can now prove Theorem \ref{thm:asympt of cond ent},
which is the following statement.
\begin{thm*}
Let $1\le m\le d$ be with $\Sigma_{1}^{m-1}=m-1$ and $\Delta_{m}<1$,
and let $V\in\mathrm{Gr}_{m}(d)$ be such that $\pi_{V}\mu$ is exact
dimensional with $\dim\pi_{V}\mu=\Sigma_{1}^{m}$. Then for all $M\ge1$,
\[
\underset{n\rightarrow\infty}{\lim}\:\frac{1}{n}H\left(\pi_{V}p_{\Phi}^{*n},\mathcal{D}_{Mn}\mid\mathcal{D}_{0}\right)=0.
\]
\end{thm*}
\begin{proof}
Let $M\ge1$, and assume by contradiction that
\begin{equation}
\underset{n\rightarrow\infty}{\limsup}\:\frac{1}{n}H\left(\pi_{V}p_{\Phi}^{*n},\mathcal{D}_{Mn}\mid\mathcal{D}_{0}\right)>0.\label{eq:liminf 1/n cond ent > 0}
\end{equation}
By the first part of Lemma \ref{lem:commens partiti of A_d,m}, there
exists $C>1$ so that
\[
\frac{1}{n}H\left(\sigma,\mathcal{D}_{n}\right)\le C\text{ for all }n\ge1,\:D\in\mathcal{D}_{0}^{\mathrm{A}_{d,m}}\text{ and }\sigma\in\mathcal{M}(D).
\]
Thus, from (\ref{eq:liminf 1/n cond ent > 0}) it follows that there
exists $\epsilon>0$ so that $|\mathcal{Q}_{\epsilon}|=\infty$, where
$\mathcal{Q}_{\epsilon}$ is the set of all integers $n\ge1$ with
\[
\mathbb{P}_{i=0}\left\{ \frac{1}{Mn}H\left((\pi_{V}p_{\Phi}^{*n})_{\psi,i},\mathcal{D}_{Mn}\right)>\epsilon\right\} >\epsilon.
\]

Let
\[
0<\delta<\frac{1}{4}\inf\{\chi_{i}-\chi_{i+1}\::\:1\le i<m\}
\]
be small with respect to $\epsilon$, let $n\in\mathcal{Q}_{\epsilon}$
be large with respect to $\delta$ and $M$, and set $\theta:=\pi_{V}p_{\Phi}^{*n}$.
From the relation $\mu=p_{\Phi}^{*n}.\mu$, since $\pi_{V}\mu$ has
exact dimension $\Sigma_{1}^{m}$, and by Lemma \ref{lem:dim_e=00003Ddim},
we may assume that
\[
\Sigma_{1}^{m}+\delta>\frac{1}{Mn}H\left(\theta.\mu,\mathcal{D}_{(M-\chi_{m})n}\mid\mathcal{D}_{-\chi_{m}n}\right).
\]
Thus, by the concavity of conditional entropy,
\begin{equation}
\Sigma_{1}^{m}+\delta>\mathbb{E}_{i=0}\left(\frac{1}{Mn}H\left(\theta_{\psi,i}.\mu,\mathcal{D}_{(M-\chi_{m})n}\mid\mathcal{D}_{-\chi_{m}n}\right)\right).\label{eq:E_i=00003D0(cond ent)}
\end{equation}

For $1\le i\le m$ and $\psi\in\mathrm{A}_{d,m}$ write $c_{i,\psi}:=-\frac{1}{n}\log\alpha_{i}(A_{\psi})$,
and let $E$ be the set of all $\psi\in\mathrm{supp}(\theta)$ so
that $|c_{i,\psi}+\chi_{i}|<\delta$ for $1\le i\le m$. By Lemma
\ref{lem:sing vals of proj} and since
\[
\int\delta_{A_{\psi}}\:d\theta(\psi)=\int\delta_{\pi_{V}A_{\omega|_{n}}}\:d\beta(\omega),
\]
we may assume that $\theta(E)>1-\delta$. Moreover, from (\ref{eq:E_i=00003D0(cond ent)})
\[
\Sigma_{1}^{m}+O(\delta)>\int_{E}\frac{1}{Mn}H\left(\theta_{\psi,0}.\mu,\mathcal{D}_{(M+c_{m,\psi})n}\mid\mathcal{D}_{c_{m,\psi}n}\right)\:d\theta(\psi).
\]
Thus, by Proposition \ref{prop:non conf par},
\begin{equation}
\Sigma_{1}^{m}+O(\delta)>\int_{E}\frac{1}{Mn}H\left(\theta_{\psi,0}.\mu,\mathcal{D}_{Mn}^{\psi}\right)\:d\theta(\psi).\label{eq:cons of non-conf part}
\end{equation}

By the definition of the non-conformal partitions given in Section
\ref{subsec:Non-conformal-partitions}, for every $\psi\in E$ there
exists $f_{\psi}\in\mathrm{A}_{m,m}$ so that
\begin{equation}
\left|\frac{1}{Mn}H\left(\theta_{\psi,0}.\mu,\mathcal{D}_{Mn}^{\psi}\right)-\frac{1}{Mn}H\left(f_{\psi}^{-1}\theta_{\psi,0}.\mu,\mathcal{D}_{Mn}\right)\right|<\delta\label{eq:dif bet non-conf and conf}
\end{equation}
and
\begin{equation}
\mathrm{supp}(f_{\psi}^{-1}\theta_{\psi,0})\subset B(\pi_{d,m},R),\label{eq:sub of B(pi,R)}
\end{equation}
where $R>1$ is a global constant. From (\ref{eq:cons of non-conf part})
and (\ref{eq:dif bet non-conf and conf}), we get
\begin{equation}
\Sigma_{1}^{m}+O(\delta)>\int_{E}\frac{1}{Mn}H\left(f_{\psi}^{-1}\theta_{\psi,0}.\mu,\mathcal{D}_{Mn}\right)\:d\theta(\psi).\label{eq:>H(f^-1theta)}
\end{equation}

Let $\eta>0$ be small with respect to $\epsilon,R$, and suppose
that $\delta$ is small with respect to $\eta$. Let $F$ be the set
of all $\psi\in E$ so that $\frac{1}{Mn}H(\theta_{\psi,0},\mathcal{D}_{Mn})>\epsilon$.
From $n\in\mathcal{Q}_{\epsilon}$ and $\theta(E)>1-\delta$, we obtain
$\theta(F)>\epsilon/2$. Additionally, by the second part of Lemma
\ref{lem:commens partiti of A_d,m},
\[
\frac{1}{Mn}H(f_{\psi}^{-1}\theta_{\psi,0},\mathcal{D}_{Mn})>\epsilon/2\text{ for all }\psi\in F.
\]
Thus, from (\ref{eq:sub of B(pi,R)}) and by Theorem \ref{thm:ent increase result},
\[
\frac{1}{Mn}H\left(f_{\psi}^{-1}\theta_{\psi,0}.\mu,\mathcal{D}_{Mn}\right)>\Sigma_{1}^{m}+\eta\text{ for all }\psi\in F.
\]
Moreover, by the concavity of conditional entropy, from (\ref{eq:sub of B(pi,R)}),
and by Lemma \ref{lem:lb on ent of psi mu},
\[
\frac{1}{Mn}H\left(f_{\psi}^{-1}\theta_{\psi,0}.\mu,\mathcal{D}_{Mn}\right)>\Sigma_{1}^{m}-\delta\text{ for all }\psi\in E.
\]

Now, from (\ref{eq:>H(f^-1theta)}), by the last two formulas, and
from $\theta(E)>1-\delta$ and $\theta(F)>\epsilon/2$, we obtain
\begin{eqnarray*}
\Sigma_{1}^{m}+O(\delta) & > & \theta(F)(\Sigma_{1}^{m}+\eta)+\theta(E\setminus F)(\Sigma_{1}^{m}-\delta)\\
 & \ge & (1-\delta)\Sigma_{1}^{m}+\epsilon\eta/2-\delta.
\end{eqnarray*}
But since $\delta$ is small with respect to $\epsilon,\eta$ we have
thus arrived at a contradiction, which completes the proof of the
theorem.
\end{proof}
Recall from Section \ref{subsec:Full-dimensionality} that $h(p_{\Phi})$
denotes the random walk entropy of $p_{\Phi}$. We next prove Theorem \ref{thm:=00003DRW ent}, which is the following statement.
\begin{thm*}
Assume that $\Phi$ is Diophantine, and let $1\le m\le d$ be with
$\Sigma_{1}^{m-1}=m-1$ and $\Delta_{m}<1$. Then for $\nu_{m}^{*}$-a.e.
$V\in\mathrm{Gr}_{m}(d)$,
\[
\underset{n\rightarrow\infty}{\lim}\:\frac{1}{n}H\left(\pi_{V}p_{\Phi}^{*n},\mathcal{D}_{0}\right)=h(p_{\Phi}).
\]
\end{thm*}
\begin{proof}
For every $n\ge0$ and $D\in\mathcal{D}_{n}^{\mathrm{A}_{d,m}}$ we
have $\mathrm{diam}(D)=O(2^{-n})$. Hence, by Proposition \ref{prop:diophantine --> diophantine for projs},
there exist a Borel set $E_{1}\subset\mathrm{Gr}_{m}(d)$ and $M\ge1$
so that $\nu_{m}^{*}(E_{1})=1$ and for all $V\in E_{1}$ and $n\ge N(V)\ge1$,
\[
\mathcal{D}_{Mn}^{\mathrm{A}_{d,m}}(\pi_{V}\varphi_{u})\ne\mathcal{D}_{Mn}^{\mathrm{A}_{d,m}}(\pi_{V}\varphi_{w})\text{ for each }u,w\in\Lambda^{n}\text{ with }\varphi_{u}\ne\varphi_{w}.
\]
Thus,
\begin{equation}
H\left(\pi_{V}p_{\Phi}^{*n},\mathcal{D}_{Mn}\right)=H(p_{\Phi}^{*n})\text{ for all }V\in E_{1}\text{ and }n\ge N(V),\label{eq:aff ent scale Mn}
\end{equation}
where recall that $H(p_{\Phi}^{*n})$ is the Shannon entropy of $p_{\Phi}^{*n}$.

Let $E_{2}$ be the set of all $V\in\mathrm{Gr}_{m}(d)$ so that $\pi_{V}\mu$
is exact dimensional with $\dim\pi_{V}\mu=\Sigma_{1}^{m}$. By Theorem
\ref{thm:LY formula for SA} we have $\nu_{m}^{*}(E_{2})=1$. Moreover,
by Theorem \ref{thm:asympt of cond ent},
\[
\underset{n\rightarrow\infty}{\lim}\:\frac{1}{n}H\left(\pi_{V}p_{\Phi}^{*n},\mathcal{D}_{Mn}\mid\mathcal{D}_{0}\right)=0\text{ for all }V\in E_{2}.
\]
From this, from (\ref{eq:aff ent scale Mn}), and by the definition
of $h(p_{\Phi})$, it follows that for all $V\in E_{1}\cap E_{2}$,
\begin{eqnarray*}
h(p_{\Phi}) & = & \underset{n\rightarrow\infty}{\lim}\:\frac{1}{n}\left(H\left(\pi_{V}p_{\Phi}^{*n},\mathcal{D}_{Mn}\right)-H\left(\pi_{V}p_{\Phi}^{*n},\mathcal{D}_{Mn}\mid\mathcal{D}_{0}\right)\right)\\
 & = & \underset{n\rightarrow\infty}{\lim}\:\frac{1}{n}H\left(\pi_{V}p_{\Phi}^{*n},\mathcal{D}_{0}\right).
\end{eqnarray*}
Since $\nu_{m}^{*}(E_{1}\cap E_{2})=1$, this completes the proof
of the Theorem.
\end{proof}

\section{\label{sec:Decompositions-of-sub-linear}Decompositions of sub-linear
diameter}

The following proposition is the main result of this section. For
$\beta$-a.e. $\omega$ it yields decompositions of $\beta_{\omega}^{V}$
into a controllable number of pieces, so that most pieces project
onto a subset of $\mathrm{A}_{d,m}$ of sub-linear diameter. Recall
from Section \ref{subsec:Coding and Furstenberg maps} that for $n\ge1$
the map $\Pi_{n}:\Lambda^{\mathbb{N}}\rightarrow\mathrm{A}_{d,d}$
is defined by $\Pi_{n}(\omega)=\varphi_{\omega|_{n}}$ for $\omega\in\Lambda^{\mathbb{N}}$.
\begin{prop}
\label{prop:sub lin decomp}Let $1\le m\le d$, $V\in\mathrm{Gr}_{m}(d)$,
$0\le b<m$ and $1\le k_{1}<...<k_{b}<m$ be given. Set
\[
\mathcal{K}:=\{1,...,m-1\}\setminus\{k_{1},...,k_{b}\},
\]
and suppose that for $\beta$-a.e. $\omega$
\[
\pi_{V}L_{k}\beta_{\omega}^{V}=\delta_{\pi_{V}L_{k}(\omega)}\text{ for each }k\in\mathcal{K}.
\]
Then for $\beta$-a.e. $\omega$, each $\epsilon>0$, and all $n\ge N(V,\omega,\epsilon)\ge1$,
there exist a Borel set $S\subset\Lambda^{\mathbb{N}}$ and a Borel
partition $\mathcal{E}$ of $S$ so that,
\begin{enumerate}
\item $\beta_{\omega}^{V}(S)>1-\epsilon$;
\item $|\mathcal{E}|\le2^{n(\epsilon+q(\chi_{1}-\chi_{m}))}$, where $q=0$
if $b=0$ and $q=\sum_{l=1}^{b}k_{l}(k_{l+1}-k_{l})$ with $k_{b+1}:=k_{b}+1$
if $b>0$;
\item $d_{\mathrm{A}_{d,m}}(\pi_{V}\Pi_{n}\eta,\pi_{V}\Pi_{n}\zeta)\le\epsilon n$
for all $E\in\mathcal{E}$ and $\eta,\zeta\in E$.
\end{enumerate}
\end{prop}

The proof of the proposition is carried out in Section \ref{subsec:Proof of prop sub lin decomp}.
It relies on certain distance estimates in the space $\mathrm{A}_{d,m}$,
which are obtained in the next subsection.

\subsection{\label{subsec:Distance-estimates-in A_d,m}Distance estimates}

The main results of the present subsection are Lemmata \ref{lem:d(piA_omega,piUE_n)=00003Do(n)},
\ref{lem:dist of proj of trans part} and \ref{lem:ub on d(TpiU_omE,TpiU_etE)}.
The proofs of these lemmas require some preparations.
\begin{lem}
\label{lem:lb on alpha_m(pi T)}Let $1\le m\le d$ and $T\in\mathrm{GL}(d,\mathbb{R})$
be given. Then $\alpha_{m}(\pi_{d,m}T)\ge\alpha_{1}(T)^{1-m}\alpha_{d}(T)^{m}$.
\end{lem}

\begin{proof}
We have,
\[
\left\Vert \wedge^{m}\pi_{d,m}T\left((T^{-1}e_{d,1})\wedge...\wedge(T^{-1}e_{d,m})\right)\right\Vert =\left\Vert e_{m,1}\wedge...\wedge e_{m,m}\right\Vert =1.
\]
Moreover, by (\ref{eq:ub on norm of wedge}),
\[
\left\Vert (T^{-1}e_{d,1})\wedge...\wedge(T^{-1}e_{d,m})\right\Vert \le\prod_{i=1}^{m}\left|T^{-1}e_{d,i}\right|\le\Vert T^{-1}\Vert_{op}^{m}=\alpha_{d}(T)^{-m}.
\]
Thus,
\[
\left\Vert \wedge^{m}\pi_{d,m}T\right\Vert _{op}\ge\frac{\left\Vert \wedge^{m}\pi_{d,m}T\left((T^{-1}e_{d,1})\wedge...\wedge(T^{-1}e_{d,m})\right)\right\Vert }{\left\Vert (T^{-1}e_{d,1})\wedge...\wedge(T^{-1}e_{d,m})\right\Vert }\ge\alpha_{d}(T)^{m}.
\]
From this, since $\alpha_1(T)\ge\alpha_1(\pi_{d,m}T)$ and by (\ref{eq:norm =00003D prod of sing vals}), we get
\[
\alpha_{1}(T)^{m-1}\alpha_{m}(\pi_{d,m}T)\ge\prod_{i=1}^{m}\alpha_{i}(\pi_{d,m}T)=\left\Vert \wedge^{m}\pi_{d,m}T\right\Vert _{op}\ge\alpha_{d}(T)^{m},
\]
which proves the lemma.
\end{proof}
\begin{lem}
\label{lem:d(pi_d,m,T)=00003DO()}Let $1\le m\le d$ and let $T\in\mathrm{A}_{d,m}$
be linear. Then,
\[
d_{\mathrm{A}_{d,m}}\left(\pi_{d,m},T\right)=O\left(1+\max\left\{ \log\alpha_{1}(T),\log\alpha_{m}(T)^{-1}\right\} \right).
\]
\end{lem}

\begin{proof}
Let $UDU'$ be a singular value decomposition of $T$, and set
\[
k:=1+\left\lfloor \max\left\{ \log\alpha_{1}(T),\log\alpha_{m}(T)^{-1}\right\} \right\rfloor
\]
and
\[
E:=\mathrm{diag}_{m,m}(\alpha_{1}(T)^{1/k},...,\alpha_{m}(T)^{1/k}).
\]
Since $\frac{1}{2}\le\alpha_{i}(T)^{1/k}\le2$ for $1\le i\le m$,
we have $d_{\mathrm{A}_{d,m}}\left(\pi_{d,m}U',E\pi_{d,m}U'\right)=O(1)$.
From this, since $D=E^{k}\pi_{d,m}$, and by the $\mathrm{A}_{m,m}$-invariance
of $d_{\mathrm{A}_{d,m}}$,
\begin{eqnarray*}
d_{\mathrm{A}_{d,m}}\left(\pi_{d,m},T\right) & \le & d_{\mathrm{A}_{d,m}}\left(\pi_{d,m},U\pi_{d,m}U'\right)+d_{\mathrm{A}_{d,m}}\left(U\pi_{d,m}U',UE^{k}\pi_{d,m}U'\right)\\
 & \le & O(1)+\sum_{i=0}^{k-1}d_{\mathrm{A}_{d,m}}\left(E^{i}\pi_{d,m}U',E^{i+1}\pi_{d,m}U'\right)\\
 & = & O(1)+k\cdot d_{\mathrm{A}_{d,m}}\left(\pi_{d,m}U',E\pi_{d,m}U'\right)=O(k),
\end{eqnarray*}
which completes the proof of the lemma.
\end{proof}
The following lemma follows directly from \cite[Lemma I.4]{Ru}, which
is part of Ruelle's proof of the multiplicative ergodic theorem.
\begin{lem}
\label{lem:from Ruelle's proof}Let $\omega\in\Lambda^{\mathbb{N}}$
be such that,
\[
\underset{n\rightarrow\infty}{\lim}\:\frac{1}{n}\log\alpha_{k}(A_{\omega|_{n}})=\chi_{k}\text{ for }1\le k\le d.
\]
For $n\ge1$ let $U_{n}D_{n}U_{n}'$ be a singular value decomposition
of $A_{\omega|_{n}}$, and for $1\le k\le d$ set $u_{n,k}:=U_{n}e_{d,k}$.
Then for each $\epsilon>0$ there exists $N=N(\omega,\epsilon)\ge1$
so that for all $n_{2}\ge n_{1}\ge N$,
\[
\left|\left\langle u_{n_{1},k},u_{n_{2},l}\right\rangle \right|\le2^{-n_{1}(|\chi_{k}-\chi_{l}|-\epsilon)}\text{ for all }1\le k,l\le d.
\]
\end{lem}

We now proceed to the main results of this subsection, which require
the following notations. Recall from Section \ref{subsec:Coding and Furstenberg maps}
that for $\omega\in\Lambda^{\mathbb{N}}$ we have $L_{0}(\omega)\subset...\subset L_{d}(\omega)$.
For $1\le k\le d$ let $u_{\omega,k}\in L_{k}(\omega)\cap(L_{k-1}(\omega)^{\perp})$
be a unit vector, and let $U_{\omega}\in\mathrm{O}(d)$ be such that
$U_{\omega}e_{d,k}=u_{\omega,k}$ for $1\le k\le d$. Since the Furstenberg
boundary maps $L_{0},...,L_{d}$ are Borel measurable, we may clearly
assume that the map $\omega\rightarrow U_{\omega}$ is also Borel
measurable. For $n\ge1$ set,
\[
E_{n}:=\mathrm{diag}_{d,d}(2^{n\chi_{1}},...,2^{n\chi_{d}}).
\]

\begin{lem}
\label{lem:d(piA_omega,piUE_n)=00003Do(n)}For $\beta$-a.e. $\omega$,
every $1\le m\le d$ and all $V\in\mathrm{Gr}_{m}(d)$,
\[
\underset{n\rightarrow\infty}{\lim}\:\frac{1}{n}d_{\mathrm{A}_{d,m}}\left(\pi_{V}A_{\omega|_{n}},\pi_{V}U_{\omega}E_{n}\right)=0.
\]
\end{lem}

\begin{proof}
For $\beta$-a.e. $\omega$
\begin{equation}
\underset{n\rightarrow\infty}{\lim}\:\frac{1}{n}\log\alpha_{k}(A_{\omega|_{n}})=\chi_{k}\text{ for }1\le k\le d,\label{eq:conv of alpha_m all m}
\end{equation}
and
\begin{equation}
\underset{n\rightarrow\infty}{\lim}\:L_{k}(A_{\omega|_{n}})=L_{k}(\omega)\text{ for }1\le k\le d.\label{eq:conv of L_m all m}
\end{equation}
Fix such an $\omega$ for the rest of the proof. Since $\omega$ is
fixed, we write $\{u_{k}\}_{k=1}^{d}$ and $U$ in place of $\{u_{\omega,k}\}_{k=1}^{d}$
and $U_{\omega}$.

Let $\epsilon>0$ be given. By (\ref{eq:conv of alpha_m all m}),
there exists $N_{1}\ge1$ so that for all $n\ge N_{1}$
\begin{equation}
2^{n(\chi_{k}-\epsilon)}\le\alpha_{k}(A_{\omega|_{n}})\le2^{n(\chi_{k}+\epsilon)}\text{ for }1\le k\le d.\label{eq:<=00003D alpha_m <=00003D}
\end{equation}

For each $n\ge1$ let $U_{n}D_{n}U_{n}'$ be a singular value decomposition
of $A_{\omega|_{n}}$, and for $1\le k\le d$ set $u_{n,k}:=U_{n}e_{d,k}$.
By (\ref{eq:conv of L_m all m}) it follows easily that,
\[
\underset{n\rightarrow\infty}{\lim}\:d_{\mathrm{Gr}_{1}}(u_{k}\mathbb{R},u_{n,k}\mathbb{R})=0\text{ for }1\le k\le d.
\]
Moreover, by Lemma \ref{lem:from Ruelle's proof} there exists $N_{2}\ge1$
so that for all $n_{2}\ge n_{1}\ge N_{2}$,
\[
\left|\left\langle u_{n_{1},k},u_{n_{2},l}\right\rangle \right|\le2^{-n_{1}(|\chi_{k}-\chi_{l}|-\epsilon)}\text{ for all }1\le k,l\le d.
\]
Thus, for every $n\ge N_{2}$
\begin{equation}
\left|\left\langle u_{k},u_{n,l}\right\rangle \right|\le2^{-n(|\chi_{k}-\chi_{l}|-\epsilon)}\text{ for all }1\le k,l\le d.\label{eq:ub on <u_n, u_n,k>}
\end{equation}

Let $n\ge\max\{N_{1},N_{2}\}$ and $1\le k\le d$ be given. Since
$\{u_{n,1},...,u_{n,d}\}$ is an orthonormal basis of $\mathbb{R}^{d}$,
\begin{eqnarray*}
D_{n}^{-1}U_{n}^{-1}UE_{n}e_{d,k} & = & 2^{n\chi_{k}}D_{n}^{-1}U_{n}^{-1}u_{k}\\
 & = & 2^{n\chi_{k}}D_{n}^{-1}U_{n}^{-1}\left(\sum_{l=1}^{d}\left\langle u_{k},u_{n,l}\right\rangle u_{n,l}\right)\\
 & = & \sum_{l=1}^{d}\frac{2^{n\chi_{k}}}{\alpha_{l}(A_{\omega|_{n}})}\left\langle u_{k},u_{n,l}\right\rangle e_{d,l}.
\end{eqnarray*}
From this, (\ref{eq:<=00003D alpha_m <=00003D}) and (\ref{eq:ub on <u_n, u_n,k>}),
we get
\begin{eqnarray*}
\left|A_{\omega|_{n}}^{-1}UE_{n}e_{d,k}\right| & = & \left|D_{n}^{-1}U_{n}^{-1}UE_{n}e_{d,k}\right|\\
 & \le & \sum_{l=1}^{d}\frac{2^{n\chi_{k}}}{\alpha_{l}(A_{\omega|_{n}})}\left|\left\langle u_{k},u_{n,l}\right\rangle \right|\\
 & \le & \sum_{l=1}^{d}2^{n\chi_{k}}2^{-n(\chi_{l}-\epsilon)}2^{-n(|\chi_{k}-\chi_{l}|-\epsilon)}\le d2^{2n\epsilon},
\end{eqnarray*}
which gives $\Vert A_{\omega|_{n}}^{-1}UE_{n}\Vert_{op}=O(2^{2n\epsilon})$.
In a similar manner it can be shown that,
\[
\alpha_{d}(A_{\omega|_{n}}^{-1}UE_{n})^{-1}=\Vert(A_{\omega|_{n}}^{-1}UE_{n})^{-1}\Vert_{op}=O(2^{2n\epsilon}).
\]

Now let $n\ge\max\{N_{1},N_{2}\}$, $1\le m\le d$ and $V\in\mathrm{Gr}_{m}(d)$
be given. There exists $B\in\mathrm{GL}(m,\mathbb{R})$ and $\tilde{U}\in\mathrm{O}(d)$
so that $B\pi_{V}A_{\omega|_{n}}=\pi_{d,m}\tilde{U}$. Thus, by the
$\mathrm{A}_{m,m}$-invariance of $d_{\mathrm{A}_{d,m}}$,
\begin{eqnarray}
d_{\mathrm{A}_{d,m}}\left(\pi_{V}A_{\omega|_{n}},\pi_{V}UE_{n}\right) & = & d_{\mathrm{A}_{d,m}}\left(B\pi_{V}A_{\omega|_{n}},B\pi_{V}A_{\omega|_{n}}A_{\omega|_{n}}^{-1}UE_{n}\right)\nonumber \\
 & = & d_{\mathrm{A}_{d,m}}\left(\pi_{d,m}\tilde{U},\pi_{d,m}\tilde{U}A_{\omega|_{n}}^{-1}UE_{n}\right)\nonumber \\
 & \le & O(1)+d_{\mathrm{A}_{d,m}}\left(\pi_{d,m},\pi_{d,m}\tilde{U}A_{\omega|_{n}}^{-1}UE_{n}\right).\label{eq:ub 1 on d_A(piA,piUE_n)}
\end{eqnarray}

Set $T:=\tilde{U}A_{\omega|_{n}}^{-1}UE_{n}$. From the estimates
above and since $\tilde{U}\in\mathrm{O}(d)$,
\[
\alpha_{1}(T),\alpha_{d}(T)^{-1}=O(2^{2n\epsilon}).
\]
From this and by Lemma \ref{lem:lb on alpha_m(pi T)},
\[
\alpha_{m}\left(\pi_{d,m}T\right)^{-1}\le\alpha_{1}(T)^{m-1}\alpha_{d}(T)^{-m}=O(2^{4mn\epsilon}).
\]
Moreover,
\[
\alpha_{1}\left(\pi_{d,m}T\right)\le\alpha_{1}(T)=O(2^{2n\epsilon}).
\]
Thus, by Lemma \ref{lem:d(pi_d,m,T)=00003DO()},
\[
d_{\mathrm{A}_{d,m}}\left(\pi_{d,m},\pi_{d,m}T\right)=O(1+n\epsilon).
\]
This together with (\ref{eq:ub 1 on d_A(piA,piUE_n)}) gives
\[
\frac{1}{n}d_{\mathrm{A}_{d,m}}\left(\pi_{V}A_{\omega|_{n}},\pi_{V}UE_{n}\right)=O\left(\frac{1}{n}+\epsilon\right),
\]
which completes the proof of the lemma. 
\end{proof}
\begin{lem}
\label{lem:dist of proj of trans part}Let $1\le m\le d$ and $V\in\mathrm{Gr}_{m}(d)$
be given, then for $\beta$-a.e. $\omega$ the following holds. For
$0\le k<m$ set $W_{k}:=\pi_{V}(L_{k}(\omega))$, then for each $1\le k\le m$
\[
\underset{n\rightarrow\infty}{\limsup}\frac{1}{n}\log\left|P_{W_{k-1}^{\perp}}\pi_{V}\left(\Pi\omega-\varphi_{\omega|_{n}}(0)\right)\right|\le\chi_{k}.
\]
\end{lem}

\begin{proof}
For $\beta$-a.e. $\omega$
\[
\underset{n\rightarrow\infty}{\lim}\:\frac{1}{n}\log\alpha_{k}(A_{\omega|_{n}})=\chi_{k}\text{ and }\underset{n\rightarrow\infty}{\lim}\:L_{k}(A_{\omega|_{n}})=L_{k}(\omega)\text{ for }1\le k\le d.
\]
Fix such an $\omega$ for the rest of the proof, and write $\{u_{k}\}_{k=1}^{d}$
and $U$ in place of $\{u_{\omega,k}\}_{k=1}^{d}$ and $U_{\omega}$.

Let $\epsilon>0$ be given. For each $n\ge1$ let $U_{n}D_{n}U_{n}'$
be a singular value decomposition of $A_{\omega|_{n}}$, and for $1\le k\le d$
set $u_{n,k}:=U_{n}e_{d,k}$. As in the proof of Lemma \ref{lem:d(piA_omega,piUE_n)=00003Do(n)},
there exists $N\ge1$ so that for all $n\ge N$
\begin{equation}
\alpha_{k}(A_{\omega|_{n}})\le2^{n(\chi_{k}+\epsilon)}\text{ for }1\le k\le d\label{eq:ub on sing vals}
\end{equation}
and
\begin{equation}
\left|\left\langle u_{k},u_{n,l}\right\rangle \right|\le2^{-n(|\chi_{k}-\chi_{l}|-\epsilon)}\text{ for all }1\le k,l\le d.\label{eq:ub on <u_n, u_n,k> 2}
\end{equation}

Fix $n\ge N$, set $y:=\Pi\omega-\varphi_{\omega|_{n}}(0)$, and let
$B\subset\mathbb{R}^{d}$ be the closed ball with centre $0$ are
radius $R:=\max_{x\in K_{\Phi}}|x|$. Note that $\Pi\omega,\varphi_{\omega|_{n}}(0)\in\varphi_{\omega|_{n}}(B)$,
and that for each $x\in\varphi_{\omega|_{n}}(B)$ there exist $c_{1},...,c_{d}\in\mathbb{R}$
so that $x=\varphi_{\omega|_{n}}(0)+\sum_{l=1}^{d}c_{l}u_{n,l}$ and
$|c_{l}|\le R\alpha_{l}(A_{\omega|_{n}})$ for $1\le l\le d$. Hence,
by (\ref{eq:ub on sing vals}),
\[
\left|\left\langle y,u_{n,l}\right\rangle \right|=O\left(2^{n(\chi_{l}+\epsilon)}\right)\text{ for }1\le l\le d.
\]
From this and (\ref{eq:ub on <u_n, u_n,k> 2}), it follows that for
each $1\le k\le d$
\begin{equation}
\left|\left\langle y,u_{k}\right\rangle \right|\le\sum_{l=1}^{d}\left|\left\langle y,u_{n,l}\right\rangle \right|\cdot\left|\left\langle u_{k},u_{n,l}\right\rangle \right|=O\left(2^{n(\chi_{k}+2\epsilon)}\right).\label{eq:|<y,u_k>|<=00003D}
\end{equation}

Now let $1\le k\le m$ be given and set $W_{k-1}:=\pi_{V}(L_{k-1}(\omega))$.
For each $1\le l<k$ we have $\pi_{V}u_{l}\in\pi_{V}(L_{l}(\omega))\subset W_{k-1}$.
Thus,
\[
P_{W_{k-1}^{\perp}}\pi_{V}(y)=P_{W_{k-1}^{\perp}}\pi_{V}\left(\sum_{l=1}^{d}\left\langle y,u_{l}\right\rangle u_{l}\right)=\sum_{l=k}^{d}\left\langle y,u_{l}\right\rangle P_{W_{k-1}^{\perp}}\pi_{V}u_{l}.
\]
This together with (\ref{eq:|<y,u_k>|<=00003D}) gives
\[
\left|P_{W_{k-1}^{\perp}}\pi_{V}(y)\right|=O\left(2^{n(\chi_{k}+2\epsilon)}\right),
\]
which completes the proof of the lemma.
\end{proof}
Recall from Section \ref{subsec:Algebraic-notations} that $\mathrm{F}(m)$
denotes the manifold of complete flags in $\mathbb{R}^{m}$, and that
for $a,x\in\mathbb{R}^{m}$ we write $T_{a}x=a+x$.
\begin{lem}
\label{lem:ub on d(TpiU_omE,TpiU_etE)}Let $\epsilon>0$, $C>1$,
$1\le m\le d$, $V\in\mathrm{Gr}_{m}(d)$, $(W_{k})_{k=0}^{m}\in\mathrm{F}(m)$,
$\omega,\eta\in\Lambda^{\mathbb{N}}$, $n\ge1$ and $a_{\omega},a_{\eta}\in\mathbb{R}^{m}$
be given. Suppose that,
\begin{enumerate}
\item $\kappa(V^{\perp},L_{m}(\omega)),\kappa(V^{\perp},L_{m}(\eta))\ge\epsilon$;
\item $d_{\mathrm{Gr}_{k}}(\pi_{V}(L_{k}(\omega)),W_{k}),d_{\mathrm{Gr}_{k}}(\pi_{V}(L_{k}(\eta)),W_{k})\le C2^{n(\chi_{m}-\chi_{k})}$
for $1\le k<m$;
\item $\left|P_{W_{k-1}^{\perp}}(a_{\omega}-a_{\eta})\right|\le C2^{n\chi_{k}}$
for $1\le k\le m$.
\end{enumerate}
Then,
\[
d_{\mathrm{A}_{d,m}}\left(T_{a_{\omega}}\pi_{V}U_{\omega}E_{n},T_{a_{\eta}}\pi_{V}U_{\eta}E_{n}\right)=O_{\epsilon,C}(1).
\]
\end{lem}

\begin{proof}
For each $1\le k\le m$ let $w_{k}\in W_{k-1}^{\perp}\cap W_{k}$
be a unit vector, let $M\in\mathrm{GL}(m,\mathbb{R})$ be such that
$Mw_{k}=2^{-n\chi_{k}}w_{k}$ for $1\le k\le m$, and set $a:=a_{\omega}-a_{\eta}$.
By the $\mathrm{A}_{m,m}$-invariance of $d_{\mathrm{A}_{d,m}}$,
in order to prove the lemma it suffices to show that
\[
d_{\mathrm{A}_{d,m}}\left(\pi_{d,m},MT_{a}\pi_{V}U_{\omega}E_{n}\right),d_{\mathrm{A}_{d,m}}\left(\pi_{d,m},M\pi_{V}U_{\eta}E_{n}\right)=O_{\epsilon,C}(1).
\]
We only estimate the first term, as the estimation of the second term
is similar.

For each $1\le k\le m$, 
\[
\left|\left\langle a,w_{k}\right\rangle \right|\le\left|P_{W_{k-1}^{\perp}}(a)\right|\le C2^{n\chi_{k}}.
\]
Hence,
\[
|Ma|=\left|M\left(\sum_{k=1}^{m}\left\langle a,w_{k}\right\rangle w_{k}\right)\right|\le\sum_{k=1}^{m}\left|\left\langle a,w_{k}\right\rangle \right|\cdot|Mw_{k}|=O_{C}(1),
\]
and so
\[
d_{\mathrm{A}_{d,m}}\left(\pi_{d,m},T_{Ma}\pi_{d,m}\right)=O_{C}(1).
\]
Setting $S:=M\pi_{V}U_{\omega}E_{n}$, we thus have
\begin{eqnarray*}
d_{\mathrm{A}_{d,m}}\left(\pi_{d,m},MT_{a}\pi_{V}U_{\omega}E_{n}\right) & = & d_{\mathrm{A}_{d,m}}\left(\pi_{d,m},T_{Ma}S\right)\\
 & \le & d_{\mathrm{A}_{d,m}}\left(\pi_{d,m},T_{Ma}\pi_{d,m}\right)+d_{\mathrm{A}_{d,m}}\left(T_{Ma}\pi_{d,m},T_{Ma}S\right)\\
 & = & O_{C}(1)+d_{\mathrm{A}_{d,m}}\left(\pi_{d,m},S\right).
\end{eqnarray*}

We turn to show that $d_{\mathrm{A}_{d,m}}\left(\pi_{d,m},S\right)=O_{\epsilon,C}(1)$,
which will complete the proof. For $1\le k\le m$ we have $\pi_{V}U_{\omega}e_{d,k}\in\pi_{V}L_{k}(\omega)$.
From this and
\[
d_{\mathrm{Gr}_{k}}(\pi_{V}(L_{k}(\omega)),W_{k})\le C2^{n(\chi_{m}-\chi_{k})},
\]
it follows that there exists $w\in W_{k}$ with $|w|\le1$ and
\[
|\pi_{V}U_{\omega}e_{d,k}-w|\le C2^{n(\chi_{m}-\chi_{k})}.
\]
Since $w\in W_{k}$ and $|w|\le1$, we obtain that $|Mw|\le2^{-n\chi_{k}}$.
Thus,
\begin{eqnarray*}
\left|Se_{d,k}\right| & = & 2^{n\chi_{k}}\left|M\pi_{V}U_{\omega}e_{d,k}\right|\\
 & \le & 2^{n\chi_{k}}\left|M\left(\pi_{V}U_{\omega}e_{d,k}-w\right)\right|+2^{n\chi_{k}}\left|Mw\right|\\
 & \le & 2^{n\chi_{k}}\Vert M\Vert_{op}C2^{n(\chi_{m}-\chi_{k})}+1\le2C.
\end{eqnarray*}
Moreover, for $m<k\le d$
\[
\left|Se_{d,k}\right|=2^{n\chi_{k}}\left|M\pi_{V}U_{\omega}e_{d,k}\right|\le2^{n\chi_{k}}\Vert M\Vert_{op}\le2^{n(\chi_{k}-\chi_{m})}\le1.
\]
From the last two inequalities it follows that $\alpha_{1}(S)=O_{C}(1)$.

We next estimate $\alpha_{m}(S)^{-1}$ from above. Since $U_{\omega}e_{d,k}=u_{\omega,k}$
for $1\le k\le d$,
\[
\wedge^{m}U_{\omega}E_{n}(e_{d,1}\wedge...\wedge e_{d,m})=u_{\omega,1}\wedge...\wedge u_{\omega,m}\prod_{k=1}^{m}e^{n\chi_{k}}.
\]
Note that $\{w_{1},...,w_{m}\}$ is an orthonormal basis for $\mathbb{R}^{m}$,
and that the vector space $\wedge^{m}\mathbb{R}^{m}$ is $1$-dimensional.
Hence $\Vert w_{1}\wedge...\wedge w_{m}\Vert=1$ and $\wedge^{m}\mathbb{R}^{m}=w_{1}\wedge...\wedge w_{m}\mathbb{R}$,
which implies
\[
\wedge^{m}\pi_{V}\left(u_{\omega,1}\wedge...\wedge u_{\omega,m}\right)=\pm\left\Vert \wedge^{m}\pi_{V}\left(u_{\omega,1}\wedge...\wedge u_{\omega,m}\right)\right\Vert w_{1}\wedge...\wedge w_{m}.
\]
Since $\kappa(V^{\perp},L_{m}(\omega))\ge\epsilon$ and by Lemma \ref{lem:lb on norm of wedge_pi}, there exists $\delta=\delta(\epsilon)>0$
so that
\[
\left\Vert \wedge^{m}\pi_{V}\left(u_{\omega,1}\wedge...\wedge u_{\omega,m}\right)\right\Vert \ge\delta.
\]
By the definition of $M$,
\[
\wedge^{m}M\left(w_{1}\wedge...\wedge w_{m}\right)=w_{1}\wedge...\wedge w_{m}\prod_{k=1}^{m}e^{-n\chi_{k}}.
\]
By combining all of this,
\begin{eqnarray*}
\left\Vert \wedge^{m}S(e_{d,1}\wedge...\wedge e_{d,m})\right\Vert  & = & \left\Vert \wedge^{m}M\pi_{V}(u_{\omega,1}\wedge...\wedge u_{\omega,m})\right\Vert \prod_{k=1}^{m}e^{n\chi_{k}}\\
 & \ge & \left\Vert \wedge^{m}M(w_{1}\wedge...\wedge w_{m})\right\Vert \delta\prod_{k=1}^{m}e^{n\chi_{k}}\\
 & = & \left\Vert w_{1}\wedge...\wedge w_{m}\right\Vert \delta=\delta.
\end{eqnarray*}

From (\ref{eq:norm =00003D prod of sing vals}) and by the last inequality,
\[
\alpha_{1}(S)^{m-1}\alpha_{m}(S)\ge\prod_{k=1}^{m}\alpha_{k}(S)=\Vert\wedge^{m}S\Vert_{op}\ge\delta.
\]
Thus, since $\delta$ depends only on $\epsilon$ and from $\alpha_{1}(S)=O_{C}(1)$,
it follows that $\alpha_{m}(S)^{-1}=O_{\epsilon,C}(1)$. From Lemma
\ref{lem:d(pi_d,m,T)=00003DO()} we now get $d_{\mathrm{A}_{d,m}}(\pi_{d,m},S)=O_{\epsilon,C}(1)$.
As noted above, this completes the proof of the lemma.
\end{proof}

\subsection{\label{subsec:Proof of prop sub lin decomp}Proof of Proposition
\ref{prop:sub lin decomp}}

Given integers $0\le b<m$ and $1\le k_{1}<...<k_{b+1}\le m$ and a
linear subspace $W\in\mathrm{Gr}_{k_{b+1}}(m)$, we write
\[
\mathrm{F}(W;k_{b+1},...,k_{1}):=\left\{ (V_{l})_{l=1}^{b+1}\in\mathrm{Gr}_{k_{1}}(W)\times...\times\mathrm{Gr}_{k_{b+1}}(W)\::\:V_{1}\subset...\subset V_{b+1}\right\} .
\]
The flag space $\mathrm{F}(W;k_{b+1},...,k_{1})$ has a natural smooth
structure (see \cite[Example 21.22]{Le}), which makes it into a smooth
manifold of dimension $\sum_{l=1}^{b}k_{l}(k_{l+1}-k_{l})$. When
$b=0$, the last sum should of course be interpreted as $0$.
\begin{lem}
\label{lem:part of flag}Let $0\le b<m$, $1\le k_{1}<...<k_{b+1}\le m$
and $W\in\mathrm{Gr}_{k_{b+1}}(m)$ be given, and set $\mathrm{F}:=\mathrm{F}(W;k_{b+1},...,k_{1})$.
Then there exist $C=C(\mathrm{F})>1$ and a sequence $\left\{ \mathcal{D}_{n}^{\mathrm{F}}\right\} _{n\ge0}$
of Borel partitions of $\mathrm{F}$, so that for each $n\ge0$
\begin{enumerate}
\item $d_{\mathrm{Gr}_{k_{l}}}(V_{l},V_{l}')\le C2^{-n}$ for all $D\in\mathcal{D}_{n}^{\mathrm{F}}$,
$(V_{j})_{j=1}^{b+1},(V_{j}')_{j=1}^{b+1}\in D$ and $1\le l\le b$;
\item $|\mathcal{D}_{n}^{\mathrm{F}}|\le C2^{qn}$, where $q=\sum_{l=1}^{b}k_{l}(k_{l+1}-k_{l})$.
\end{enumerate}
\end{lem}

\begin{proof}
Given $F\in\mathrm{F}$, $r>0$ and a metric $\rho$ on $\mathrm{F}$,
we write
\[
B_{\rho}(F,r):=\{F'\in\mathrm{F}\::\:\rho(F,F')\le r\}.
\]
For $(V_{l})_{l=1}^{b+1}=F,(V_{l}')_{l=1}^{b+1}=F'\in\mathrm{F}$
set
\[
d_{1}(F,F'):=\sum_{l=1}^{b}d_{\mathrm{Gr}_{k_{l}}}(V_{l},V_{l}'),
\]
so that $(\mathrm{F},d_{1})$ is a compact metric space. Thus, by \cite[Remark 2.2]{KRS},
there exist $C_{1}>1$ and a sequence $\left\{ \mathcal{D}_{n}^{\mathrm{F}}\right\} _{n\ge0}$
of Borel partitions of $\mathrm{F}$ so that for each $n\ge0$ and
$D\in\mathcal{D}_{n}^{\mathrm{F}}$ there exists $F_{n,D}\in D$ so
that
\[
B_{d_{1}}(F_{n,D},C_{1}^{-1}2^{-n})\subset D\subset B_{d_{1}}(F_{n,D},C_{1}2^{-n}).
\]
It is clear that the first property in the statement of the lemma
is satisfied. It remains to establish the second property.

Let $g$ be a Riemannian metric on $\mathrm{F}$. Denote by $d_{2}$
and $\gamma$ the Riemannian distance function and Riemannian volume
corresponding to $g$. Since $\mathrm{F}$ is a compact smooth manifold
of dimension $q$, we have $\gamma(\mathrm{F})<\infty$ and there
exists $C_{2}>1$ so that
\[
C_{2}^{-1}r^{q}\le\gamma(B_{d_{2}}(F,r))\le C_{2}r^{q}\text{ for all }F\in\mathrm{F}\text{ and }0<r<1.
\]
Moreover, by the definition of $d_{1}$ and by using the compactness
of $\mathrm{F}$, it is easy to show that $d_{1}$ and $d_{2}$ are
Lipschitz equivalent. Thus there exists $C_{3}>1$ so that
\[
B_{d_{2}}(F,C_{3}^{-1}r)\subset B_{d_{1}}(F,r)\subset B_{d_{2}}(F,C_{3}r)\text{ for all }F\in\mathrm{F}\text{ and }r>0.
\]

From all of this it follows that for each $n\ge0$,
\begin{eqnarray*}
\gamma(\mathrm{F})=\sum_{D\in\mathcal{D}_{n}^{\mathrm{F}}}\gamma(D) & \ge & \sum_{D\in\mathcal{D}_{n}^{\mathrm{F}}}\gamma(B_{d_{1}}(F_{n,D},C_{1}^{-1}2^{-n}))\\
 & \ge & \sum_{D\in\mathcal{D}_{n}^{\mathrm{F}}}\gamma(B_{d_{2}}(F_{n,D},C_{3}^{-1}C_{1}^{-1}2^{-n}))\\
 & \ge & |\mathcal{D}_{n}^{\mathrm{F}}|C_{2}^{-1}\left(C_{3}^{-1}C_{1}^{-1}2^{-n}\right)^{q},
\end{eqnarray*}
which implies $|\mathcal{D}_{n}^{\mathrm{F}}|\le\gamma(\mathrm{F})C_{2}C_{3}^{q}C_{1}^{q}2^{qn}$.
Since $\gamma(\mathrm{F})<\infty$, this completes the proof of the
lemma.
\end{proof}
\begin{proof}[Proof of Proposition \ref{prop:sub lin decomp}]
As in the statement of the proposition, let $1\le m\le d$, $V\in\mathrm{Gr}_{m}(d)$,
$0\le b<m$ and $1\le k_{1}<...<k_{b}<m$ be given, set
\[
\mathcal{K}:=\{1,...,m-1\}\setminus\{k_{1},...,k_{b}\},
\]
and suppose that
\begin{equation}
\pi_{V}L_{k}\beta_{\omega}^{V}=\delta_{\pi_{V}L_{k}(\omega)}\text{ for }\beta\text{-a.e. }\omega\text{ and each }k\in\mathcal{K}.\label{eq:main assump decomp prop}
\end{equation}

For $\omega\in\Lambda^{\mathbb{N}}$ and $0\le k\le m$ set $W_{\omega,k}:=\pi_{V}(L_{k}(\omega))$.
Since $L_{m}\beta=\nu_{m}$ and by Lemma \ref{lem:nu=00007Bkappa<del=00007D<eps},
we have $L_{m}(\omega)\cap V^{\perp}=\{0\}$ for $\beta$-a.e. $\omega$.
From this, by (\ref{eq:main assump decomp prop}), by Lemmata \ref{lem:d(piA_omega,piUE_n)=00003Do(n)}
and \ref{lem:dist of proj of trans part}, and from basic properties
of disintegrations (see Section \ref{subsec:Disintegrations}), it
follows that for $\beta$-a.e. $\omega$
\begin{enumerate}
\item $\pi_{V}(\Pi\eta)=\pi_{V}(\Pi\omega)$ for $\beta_{\omega}^{V}$-a.e.
$\eta$;
\item $\pi_{V}(L_{k}(\eta))=\pi_{V}(L_{k}(\omega))$ for $\beta_{\omega}^{V}$-a.e.
$\eta$ and each $k\in\mathcal{K}$;
\item $L_{m}(\omega)\cap V^{\perp}=\{0\}$ and $L_{m}(\eta)\cap V^{\perp}=\{0\}$
for $\beta_{\omega}^{V}$-a.e. $\eta$;
\item $\underset{n\rightarrow\infty}{\lim}\:\frac{1}{n}d_{\mathrm{A}_{d,m}}\left(\pi_{V}A_{\eta|_{n}},\pi_{V}U_{\eta}E_{n}\right)=0$
for $\beta_{\omega}^{V}$-a.e. $\eta$, where $U_{\eta}$ and $E_{n}$
are defined just before the statement of Lemma \ref{lem:d(piA_omega,piUE_n)=00003Do(n)};
\item for $\beta_{\omega}^{V}$-a.e. $\eta$ and each $1\le k\le m$,
\[
\underset{n\rightarrow\infty}{\limsup}\frac{1}{n}\log\left|P_{W_{\eta,k-1}^{\perp}}\pi_{V}\left(\Pi\eta-\varphi_{\eta|_{n}}(0)\right)\right|\le\chi_{k}.
\]
\end{enumerate}
Fix an $\omega$ with these properties for the rest of the proof.

Let $\epsilon>0$, let $\delta>0$ be small with respect to $\epsilon$
and $\omega$, and let $n\ge1$ be large with respect to all previous
parameters. Let $S$ be the Borel set of all $\eta\in\Lambda^{\mathbb{N}}$
so that,
\begin{enumerate}
\item $\pi_{V}(L_{k}(\eta))=\pi_{V}(L_{k}(\omega))$ for $k\in\mathcal{K}$;
\item $\kappa(V^{\perp},L_{m}(\eta))\ge\delta$;
\item $d_{\mathrm{A}_{d,m}}\left(\pi_{V}A_{\eta|_{n}},\pi_{V}U_{\eta}E_{n}\right)\le\epsilon n$;
\item $\left|P_{W_{\eta,k-1}^{\perp}}\pi_{V}\left(\Pi\omega-\varphi_{\eta|_{n}}(0)\right)\right|\le2^{n(\chi_{k}+\delta)}$
for $1\le k\le m$.
\end{enumerate}
By the choice of $\omega$ we may assume that $\beta_{\omega}^{V}(S)>1-\epsilon$.

Let $C>1$ be a large constant, and assume that $n$ is large also
with respect to $C$. If $b>0$ set $k_{b+1}:=k_{b}+1$. Otherwise,
if $b=0$ set $k_{b+1}:=m$. By Lemma \ref{lem:part of flag}, there
exists a Borel partition $\mathcal{D}$ of $\mathrm{F}(W_{\omega,k_{b+1}};k_{b+1},...,k_{1})$
so that,
\begin{enumerate}
\item $d_{\mathrm{Gr}_{k_{l}}}(V_{l},V_{l}')\le C2^{n(\chi_{m}-\chi_{1})}$
for $D\in\mathcal{D}$, $(V_{j})_{j=1}^{b+1},(V_{j}')_{j=1}^{b+1}\in D$
and $1\le l\le b$;
\item $|\mathcal{D}|\le C2^{qn(\chi_{1}-\chi_{m})}$, where $q=\sum_{l=1}^{b}k_{l}(k_{l+1}-k_{l})$.
\end{enumerate}
Let $g:S\rightarrow\mathrm{F}(W_{\omega,k_{b+1}};k_{b+1},...,k_{1})$
be with $g(\eta)=\left(\pi_{V}(L_{k_{l}}(\eta))\right)_{l=1}^{b+1}$
for $\eta\in S$. Then
\[
g^{-1}\mathcal{D}:=\{g^{-1}(D)\::\:D\in\mathcal{D}\}
\]
is a Borel partition of $S$. In order to complete the proof, it suffices
to show that for each $D\in\mathcal{D}$ there exists a Borel partition
$\mathcal{Q}_{D}$ of $g^{-1}(D)$ so that $|\mathcal{Q}_{D}|\le2^{n\epsilon}$
and
\[
d_{\mathrm{A}_{d,m}}(\pi_{V}\Pi_{n}\eta,\pi_{V}\Pi_{n}\zeta)=O(\epsilon n)\text{ for all }Q\in\mathcal{Q}_{D}\text{ and }\eta,\zeta\in Q.
\]

Fix $D\in\mathcal{D}$ and $\eta\in g^{-1}(D)$ for the rest of the
proof. Note that,
\begin{equation}
d_{\mathrm{Gr}_{k}}(W_{\eta,k},W_{\zeta,k})\le C2^{n(\chi_{m}-\chi_{1})}\text{ for }\zeta\in g^{-1}(D)\text{ and }0\le k\le m.\label{eq:dist of Ws}
\end{equation}
For $\zeta\in g^{-1}(D)$ set $a_{\zeta}:=\pi_{V}\varphi_{\zeta|_{n}}(0)$.
Let us show that,
\begin{equation}
\left|P_{W_{\eta,k-1}^{\perp}}\left(a_{\zeta}-a_{\tau}\right)\right|=O_{C}(2^{n(\chi_{k}+\delta)})\text{ for }\zeta,\tau\in g^{-1}(D)\text{ and }1\le k\le m.\label{eq:dist of b's}
\end{equation}
Given $\zeta\in g^{-1}(D)$ and $1\le k\le m$,
\begin{eqnarray*}
\left|P_{W_{\eta,k-1}^{\perp}}\left(\pi_{V}\Pi\omega-a_{\zeta}\right)\right| & \le & \left|P_{W_{\zeta,k-1}^{\perp}}\left(\pi_{V}\Pi\omega-a_{\zeta}\right)\right|\\
 & + & d_{\mathrm{Gr}_{m-k+1}}\left(W_{\eta,k-1}^{\perp},W_{\zeta,k-1}^{\perp}\right)\cdot\left|\pi_{V}\Pi\omega-a_{\zeta}\right|.
\end{eqnarray*}
Since $\zeta\in S$,
\[
\left|P_{W_{\zeta,k-1}^{\perp}}\left(\pi_{V}\Pi\omega-a_{\zeta}\right)\right|\le2^{n(\chi_{k}+\delta)}\text{ and }\left|\pi_{V}\Pi\omega-a_{\zeta}\right|\le2^{n(\chi_{1}+\delta)}.
\]
Moreover, by (\ref{eq:dist of Ws})
\[
d_{\mathrm{Gr}_{m-k+1}}\left(W_{\eta,k-1}^{\perp},W_{\zeta,k-1}^{\perp}\right)=d_{\mathrm{Gr}_{k-1}}\left(W_{\eta,k-1},W_{\zeta,k-1}\right)\le C2^{n(\chi_{m}-\chi_{1})}.
\]
Thus
\[
\left|P_{W_{\eta,k-1}^{\perp}}\left(\pi_{V}\Pi\omega-a_{\zeta}\right)\right|\le2^{n(\chi_{k}+\delta)}+C2^{n(\chi_{m}+\delta)}=O_{C}(2^{n(\chi_{k}+\delta)}),
\]
which implies (\ref{eq:dist of b's}).

From (\ref{eq:dist of b's}) and by assuming that $\delta$ is sufficiently
small with respect to $\epsilon$, it follows that there exists a
Borel partition $\mathcal{Q}_{D}$ of $g^{-1}(D)$ with $|\mathcal{Q}_{D}|\le2^{n\epsilon}$,
so that for each $Q\in\mathcal{Q}_{D}$ and $\zeta,\tau\in Q$,
\begin{equation}
\left|P_{W_{\eta,k-1}^{\perp}}\left(a_{\zeta}-a_{\tau}\right)\right|\le2^{n\chi_{k}}\text{ for }1\le k\le m.\label{eq:better dist of b's}
\end{equation}
Let $Q\in\mathcal{Q}_{D}$ and $\zeta,\tau\in Q$ be given, and set
$\psi_{\zeta}:=\pi_{V}\Pi_{n}\zeta$ and $\psi_{\tau}:=\pi_{V}\Pi_{n}\tau$.
In order to complete the proof of the proposition it suffices to show
that $d_{\mathrm{A}_{d,m}}(\psi_{\zeta},\psi_{\tau})=O(\epsilon n)$.

Note that,
\[
\psi_{\zeta}=T_{a_{\zeta}}\pi_{V}A_{\zeta|_{n}}\text{ and }\psi_{\tau}=T_{a_{\tau}}\pi_{V}A_{\tau|_{n}}.
\]
Thus, by the $\mathrm{A}_{m,m}$-invariance of $d_{\mathrm{A}_{d,m}}$,
\begin{eqnarray}
d_{\mathrm{A}_{d,m}}(\psi_{\zeta},\psi_{\tau}) & \le & d_{\mathrm{A}_{d,m}}\left(\pi_{V}A_{\zeta|_{n}},\pi_{V}U_{\zeta}E_{n}\right)\nonumber \\
 & + & d_{\mathrm{A}_{d,m}}\left(T_{a_{\zeta}}\pi_{V}U_{\zeta}E_{n},T_{a_{\tau}}\pi_{V}U_{\tau}E_{n}\right)+d_{\mathrm{A}_{d,m}}\left(\pi_{V}U_{\tau}E_{n},\pi_{V}A_{\tau|_{n}}\right).\label{eq:first bd on d(psi_1,psi_2)}
\end{eqnarray}
Since $\zeta,\tau\in S$,
\[
d_{\mathrm{A}_{d,m}}\left(\pi_{V}A_{\zeta|_{n}},\pi_{V}U_{\zeta}E_{n}\right),d_{\mathrm{A}_{d,m}}\left(\pi_{V}U_{\tau}E_{n},\pi_{V}A_{\tau|_{n}}\right)\le\epsilon n.
\]
From $\zeta,\tau\in S$ it also follows that,
\[
\kappa(V^{\perp},L_{m}(\zeta)),\kappa(V^{\perp},L_{m}(\tau))\ge\delta.
\]
From this, from (\ref{eq:dist of Ws}), from (\ref{eq:better dist of b's}),
by Lemma \ref{lem:ub on d(TpiU_omE,TpiU_etE)}, and since $n$ is
large with respect to $\epsilon$, $\delta$ and $C$,
\[
d_{\mathrm{A}_{d,m}}\left(T_{a_{\zeta}}\pi_{V}U_{\zeta}E_{n},T_{a_{\tau}}\pi_{V}U_{\tau}E_{n}\right)=O_{\delta,C}(1)\le\epsilon n.
\]
All of this gives $d_{\mathrm{A}_{d,m}}(\psi_{\zeta},\psi_{\tau})\le3n\epsilon$,
which completes the proof of the proposition.
\end{proof}

\section{\label{sec:Proofs-of-the main}Proofs of the main results}

In this section we prove the main results stated in Section \ref{subsec:Statement-of-results}.
Firstly, in the next subsection we prove Theorem \ref{thm:gen dim result}
from which all of the results will follow.

Throughout this section we denote by $\mathcal{S}^{\mathrm{A}}$ the
partition of $\mathrm{A}_{d,d}$ into singletons. Thus, $H(\theta)=H(\theta,\mathcal{S}^{\mathrm{A}})$
for any discrete $\theta\in\mathcal{M}(\mathrm{A}_{d,d})$.

\subsection{Proof of Theorem \ref{thm:gen dim result}}

Let us restate the theorem before giving its proof.
\begin{thm*}
Assume that $\Phi$ is Diophantine, and let $1\le m\le d$ be with
$\Sigma_{1}^{m-1}=m-1$ and $\Delta_{m}<1$. Let $0\le b<m$ and $1\le k_{1}<...<k_{b}<m$
be integers, and set
\[
\mathcal{K}:=\{1,...,m-1\}\setminus\{k_{1},...,k_{b}\}.
\]
Suppose that for $\nu_{m}^{*}$-a.e. $V$,
\begin{equation}
\pi_{V}L_{k}\beta_{\omega}^{V}=\delta_{\pi_{V}L_{k}(\omega)}\text{ for }\beta\text{-a.e. }\omega\text{ and each }k\in\mathcal{K}.\label{eq:factor of pi L_k}
\end{equation}
Then for $\nu_{m}^{*}$-a.e. $V$
\[
\underset{n\rightarrow\infty}{\limsup}\int\frac{1}{n}H\left(\Pi_{n}\beta_{\omega}^{V}\right)\:d\beta(\omega)\le q(\chi_{1}-\chi_{m}),
\]
where $q=0$ if $b=0$ and $q=\sum_{l=1}^{b}k_{l}(k_{l+1}-k_{l})$
with $k_{b+1}:=k_{b}+1$ if $b>0$.
\end{thm*}
\begin{proof}
By Theorem \ref{thm:=00003DRW ent}, for $\nu_{m}^{*}$-a.e. $V$
\begin{equation}
\underset{n\rightarrow\infty}{\lim}\:\frac{1}{n}H\left(\pi_{V}p_{\Phi}^{*n},\mathcal{D}_{0}\right)=h(p_{\Phi}).\label{eq:lim =00003D h(p_Phi)}
\end{equation}
Fix $V\in\mathrm{Gr}_{m}(d)$ for which (\ref{eq:factor of pi L_k})
and (\ref{eq:lim =00003D h(p_Phi)}) are satisfied for the rest of
the proof.

Let $C>1$ be a large constant, and let $\epsilon>0$ be small with
respect to $C$. By (\ref{eq:factor of pi L_k}) and Proposition \ref{prop:sub lin decomp},
there exists a Borel set $F\subset\Lambda^{\mathbb{N}}$ with $\beta(F)=1$ such
that for all $\omega\in F$ and $n\ge N(V,\omega,\epsilon)\ge1$ there
exist a Borel set $S_{\omega}^{n}\subset\Lambda^{\mathbb{N}}$ and
a Borel partition $\mathcal{E}_{\omega}^{n}$ of $S_{\omega}^{n}$
so that $\beta_{\omega}^{V}(S_{\omega}^{n})>1-\epsilon$, $|\mathcal{E}_{\omega}^{n}|\le2^{n(\epsilon+q(\chi_{1}-\chi_{m}))}$,
and
\[
d_{\mathrm{A}_{d,m}}(\pi_{V}\Pi_{n}\eta,\pi_{V}\Pi_{n}\zeta)\le\epsilon n\text{ for all }E\in\mathcal{E}_{\omega}^{n}\text{ and }\eta,\zeta\in E.
\]
By the proof of Proposition \ref{prop:sub lin decomp}, we may clearly assume that the map taking $\omega\in F$ to $N(V,\omega,\epsilon)$ is Borel measurable.

For $\omega\in F$ and $n\ge N(V,\omega,\epsilon)\ge1$ set $J_{\omega}^{n}:=|\mathcal{E}_{\omega}^{n}|$,
let $\{E_{\omega,j}^{n}\}_{j=1}^{J_{\omega}^{n}}$ be an enumeration
of $\mathcal{E}_{\omega}^{n}$, set
\[
\alpha_{\omega,0}^{n}:=\beta_{\omega}^{V}(\Lambda^{\mathbb{N}}\setminus S_{\omega}^{n})\text{ and }\theta_{\omega,0}^{n}:=\Pi_{n}(\beta_{\omega}^{V})_{\Lambda^{\mathbb{N}}\setminus S_{\omega}^{n}},
\]
for each $1\le j\le J_{\omega}^{n}$ set
\[
\alpha_{\omega,j}^{n}:=\beta_{\omega}^{V}(E_{\omega,j}^{n})\text{ and }\theta_{\omega,j}^{n}:=\Pi_{n}(\beta_{\omega}^{V})_{E_{\omega,j}^{n}},
\]
and write $\alpha_{\omega}^{n}:=(\alpha_{\omega,j}^{n})_{j=0}^{J_{\omega}^{n}}$.
Note that $J_{\omega}^{n}\le2^{n(\epsilon+q(\chi_{1}-\chi_{m}))}$,
$\alpha_{\omega,0}^{n}<\epsilon$,
\begin{equation}
\mathrm{diam}(\mathrm{supp}(\pi_{V}\theta_{\omega,j}^{n}))\le\epsilon n\text{ for each }1\le j\le J_{\omega}^{n},\label{eq:diam <=00003D eps n all j}
\end{equation}
and $\Pi_{n}\beta_{\omega}^{V}=\sum_{j=0}^{J_{\omega}^{n}}\alpha_{\omega,j}^{n}\theta_{\omega,j}^{n}$.

Let $\omega\in F$ and $n\ge N(V,\omega,\epsilon)$ be given. By the
almost convexity of entropy (see (\ref{eq:conc =000026 almo conv of ent}))
\begin{equation}
\frac{1}{n}H(\Pi_{n}\beta_{\omega}^{V})\le\sum_{j=0}^{J_{\omega}^{n}}\alpha_{\omega,j}^{n}\frac{1}{n}H(\theta_{\omega,j}^{n})+\frac{1}{n}H(\alpha_{\omega}^{n}),\label{eq:by the ent conv bd}
\end{equation}
where $H(\alpha_{\omega}^{n})$ is the entropy of the probability vector $\alpha_{\omega}^{n}$. Since $J_{\omega}^{n}+1\le2^{1+n(\epsilon+q(\chi_{1}-\chi_{m}))}$
and by (\ref{eq:card ub for ent}),
\[
\frac{1}{n}H(\alpha_{\omega}^{n})\le\frac{1}{n}+\epsilon+q(\chi_{1}-\chi_{m}).
\]
From $\alpha_{\omega,0}^{n}\le\epsilon$, since $\mathrm{supp}(\theta_{\omega,0}^{n})\subset\Pi_{n}(\Lambda^{\mathbb{N}})$,
and since the cardinality of $\Pi_{n}(\Lambda^{\mathbb{N}})$ is at
most $|\Lambda|^{n}$,
\[
\alpha_{\omega,0}^{n}\frac{1}{n}H(\theta_{\omega,0}^{n})\le\epsilon\log|\Lambda|.
\]

By (\ref{eq:diam <=00003D eps n all j}) and Lemma \ref{lem:ball int at most exp many 0-atoms},
we may assume that for all $1\le j\le J_{\omega}^{n}$
\[
\log\#\left\{ D\in\mathcal{D}_{0}^{\mathrm{A}_{d,m}}\::\:D\cap\left(\mathrm{supp}(\pi_{V}\theta_{\omega,j}^{n})\right)\ne\emptyset\right\} \le C(\epsilon n+1),
\]
which implies
\[
\frac{1}{n}H(\pi_{V}\theta_{\omega,j}^{n},\mathcal{D}_{0})\le C(\epsilon+1/n).
\]
Hence, by (\ref{eq:by the ent conv bd}) and the inequalities following
it,
\begin{eqnarray}
\frac{1}{n}H(\Pi_{n}\beta_{\omega}^{V}) & \le & \frac{1}{n}+(1+\log|\Lambda|)\epsilon+q(\chi_{1}-\chi_{m})\nonumber \\
 & + & \frac{1}{n}\sum_{j=1}^{J_{\omega}^{n}}\alpha_{\omega,j}^{n}\left(H(\theta_{\omega,j}^{n},\mathcal{S}^{\mathrm{A}}\mid\pi_{V}^{-1}\mathcal{D}_{0}^{\mathrm{A}_{d,m}})+H(\pi_{V}\theta_{\omega,j}^{n},\mathcal{D}_{0}^{\mathrm{A}_{d,m}})\right)\nonumber \\
 & \le & 2C(\epsilon+1/n)+q(\chi_{1}-\chi_{m})+\frac{1}{n}\sum_{j=1}^{J_{\omega}^{n}}\alpha_{\omega,j}^{n}H(\theta_{\omega,j}^{n},\mathcal{S}^{\mathrm{A}}\mid\pi_{V}^{-1}\mathcal{D}_{0}^{\mathrm{A}_{d,m}}),\label{eq:ub all om in F =000026 n >=00003D N(om)}
\end{eqnarray}
where
\[
\pi_{V}^{-1}\mathcal{D}_{0}^{\mathrm{A}_{d,m}}:=\left\{ \left\{ \psi\in\mathrm{A}_{d,d}\::\:\pi_{V}\psi\in D\right\} \::\:D\in\mathcal{D}_{0}^{\mathrm{A}_{d,m}}\right\} .
\]

In what follows, given a Borel set $B\subset\Lambda^{\mathbb{N}}$ and a function $f:B\rightarrow[0,\infty)$ we write
\[
\int_B^*f\:d\beta:=\inf\left\{\int_Bg\:d\beta\::\:g:B\rightarrow[0,\infty)\text{ is Borel measurable and }f\le g \right\}.
\]
Let $n\ge1$ and set,
\[
F_{n}:=\{\omega\in F\::\:N(V,\omega,\epsilon)\le n\}.
\]
Since the map $\omega\rightarrow N(V,\omega,\epsilon)$ is Borel measurable, the set $F_n$ is also Borel measurable.
For every $\sigma\in\mathcal{M}(\Lambda^{\mathbb{N}})$ the cardinality
of the support of $\Pi_{n}\sigma$ is at most $|\Lambda|^{n}$, which
gives
\[
\frac{1}{n}H(\Pi_{n}\sigma)\le\log|\Lambda|.
\]
From this and since (\ref{eq:ub all om in F =000026 n >=00003D N(om)})
holds for all $\omega\in F_{n}$,
\begin{eqnarray}
\int\frac{1}{n}H(\Pi_{n}\beta_{\omega}^{V})\:d\beta(\omega) & \le & \beta(\Lambda^{\mathbb{N}}\setminus F_{n})\log|\Lambda|+2C(\epsilon+1/n)+q(\chi_{1}-\chi_{m})\nonumber \\
 & + & \int_{F_{n}}^*\sum_{j=1}^{J_{\omega}^{n}}\alpha_{\omega,j}^{n}\frac{1}{n}H(\theta_{\omega,j}^{n},\mathcal{S}^{\mathrm{A}}\mid\pi_{V}^{-1}\mathcal{D}_{0}^{\mathrm{A}_{d,m}})\:d\beta(\omega).\label{eq:ub all n on integral}
\end{eqnarray}
Note also that from $F=\cup_{n\ge1}F_{n}$ and $\beta(F)=1$, it follows
that
\begin{equation}
\underset{n\rightarrow\infty}{\lim}\:\beta(\Lambda^{\mathbb{N}}\setminus F_{n})=0.\label{eq:lim beta(F_n^c)=00003D0}
\end{equation}

On the other hand, by (\ref{eq:lim =00003D h(p_Phi)}) and since
\[
\underset{n\rightarrow\infty}{\lim}\:\frac{1}{n}H\left(p_{\Phi}^{*n}\right)=h(p_{\Phi}),
\]
we obtain
\begin{equation}
\underset{n\rightarrow\infty}{\lim}\:\frac{1}{n}H\left(p_{\Phi}^{*n},\mathcal{S}^{\mathrm{A}}\mid\pi_{V}^{-1}\mathcal{D}_{0}^{\mathrm{A}_{d,m}}\right)=0.\label{eq:cond ent of p^n =00003D0}
\end{equation}
Note that for all $n\ge1$,
\[
p_{\Phi}^{*n}=\Pi_{n}\beta=\int\Pi_{n}\beta_{\omega}^{V}\:d\beta(\omega)=\int_{F_{n}}\sum_{j=0}^{J_{\omega}^{n}}\alpha_{\omega,j}^{n}\theta_{\omega,j}^{n}\:d\beta(\omega)+\int_{\Lambda^{\mathbb{N}}\setminus F_{n}}\Pi_{n}\beta_{\omega}^{V}\:d\beta(\omega).
\]
Thus, by (\ref{eq:cond ent of p^n =00003D0}) and the concavity of
conditional entropy,
\[
\underset{n\rightarrow\infty}{\lim}\:\int_{F_{n}}^*\sum_{j=1}^{J_{\omega}^{n}}\alpha_{\omega,j}^{n}\frac{1}{n}H\left(\theta_{\omega,j}^{n},\mathcal{S}^{\mathrm{A}}\mid\pi_{V}^{-1}\mathcal{D}_{0}^{\mathrm{A}_{d,m}}\right)\:d\beta(\omega)=0.
\]

From the last equality and from \ref{eq:ub all n on integral} and
\ref{eq:lim beta(F_n^c)=00003D0}, we obtain
\[
\underset{n\rightarrow\infty}{\limsup}\int\frac{1}{n}H(\Pi_{n}\beta_{\omega}^{V})\:d\beta(\omega)\le2C\epsilon+q(\chi_{1}-\chi_{m}).
\]
Since $\epsilon$ is arbitrarily small with respect to $C$, this
completes the proof of the theorem.
\end{proof}

\subsection{Proof of Theorems \ref{thm:main d=00003D3}, \ref{thm:proj onto 1 =000026 2 dim subspaces}
and \ref{thm:gen criteria}}

In order to apply Theorem \ref{thm:gen dim result} we shall need
the following lemma. Recall from Section \ref{subsec:Symbolic-notations}
that for $n\ge1$ we denote by $\mathcal{P}_{n}$ the partition of
$\Lambda^{\mathbb{N}}$ into $n$-cylinders, and from Section \ref{subsec:Ledrappier-Young-formula}
the definition of the constants $\mathrm{H}_{0},...,\mathrm{H}_{d}$.
\begin{lem}
\label{lem:by =00005BFe, lem 4.6=00005D}Let $0\le m\le d$ be given.
Then for $\nu_{m}^{*}$-a.e. $V$,
\[
\underset{n\rightarrow\infty}{\lim}\frac{1}{n}\int H(\beta_{\omega}^{V},\mathcal{P}_{n})\:d\beta(\omega)=\mathrm{H}_{m}.
\]
\end{lem}

\begin{proof}
By \cite[Lemma 4.6]{Fe} and (\ref{eq:theta_y=00003Dtheta_x theta-a.e. x}),
it follows that for $\nu_{m}^{*}$-a.e. $V$ and $\beta$-a.e. $\omega$
\begin{equation}
-\underset{n\rightarrow\infty}{\lim}\:\frac{1}{n}\log\beta_{\omega}^{V}([\eta|_{n}])=\mathrm{H}_{m}\text{ for }\beta_{\omega}^{V}\text{-a.e. }\eta.\label{eq:by Feng lem 4.6}
\end{equation}
Fix such $V\in\mathrm{Gr}_{m}(d)$ and $\omega\in\Lambda^{\mathbb{N}}$,
and for $n\ge1$ and $\eta\in\Lambda^{\mathbb{N}}$ set $f_{n}(\eta):=-\frac{1}{n}\log\beta_{\omega}^{V}([\eta|_{n}])$.

It is easy to verify that the sequence $\{f_{n}\}_{n\ge1}$ is uniformly
integrable with respect to $\beta_{\omega}^{V}$. Hence, by (\ref{eq:by Feng lem 4.6})
and the Vitali convergence theorem,
\[
\underset{n\rightarrow\infty}{\lim}\frac{1}{n}H(\beta_{\omega}^{V},\mathcal{P}_{n})=\underset{n\rightarrow\infty}{\lim}\int f_{n}\:d\beta_{\omega}^{V}=\mathrm{H}_{m}.
\]
Since this holds for $\nu_{m}^{*}$-a.e. $V$ and $\beta$-a.e. $\omega$,
the lemma now follows by dominated convergence.
\end{proof}
We shall also need the following lemma. Recall from Section \ref{subsec:Background} the definition of the
Lyapunov dimension $\dim_{L}(\Phi,p)$.
\begin{lem}
\label{lem:H_m=00003D0 imp dim of proj =00003D dim_L}Suppose that
$1\le m\le d$ satisfies $\Sigma_{1}^{m-1}=m-1$ and $\mathrm{H}_{m}=0$.
Then $\Sigma_{1}^{m}=\dim_{L}(\Phi,p)$.
\end{lem}

\begin{proof}
As pointed out in Remark \ref{rem:0<=00003DDelta_m<=00003D1}, we
have $0\le\Delta_{k}\le1$ for each $1\le k\le d$. Thus, from $\Sigma_{1}^{m-1}=m-1$
and $\mathrm{H}_{m}=0$,
\[
\frac{\mathrm{H}_{k}-\mathrm{H}_{k-1}}{\chi_{k}}=\Delta_{k}=1\text{ for }1\le k<m
\]
and
\[
\frac{\mathrm{H}_{m-1}}{-\chi_{m}}=\frac{\mathrm{H}_{m}-\mathrm{H}_{m-1}}{\chi_{m}}=\Delta_{m}\le1.
\]

Moreover, by the definition of $\mathrm{H}_{0}$,
\[
H(p)=\mathrm{H}_{0}=\sum_{k=1}^{m-1}(\mathrm{H}_{k-1}-\mathrm{H}_{k})+\mathrm{H}_{m-1}.
\]
Hence $-\sum_{k=1}^{m-1}\chi_{k}\le H(p)\le-\sum_{k=1}^{m}\chi_{k}$,
which implies
\[
\dim_{L}(\Phi,p)=m-1+\frac{H(p)+\sum_{k=1}^{m-1}\chi_{k}}{-\chi_{m}}.
\]

From the formulas obtained above we also get
\[
\Sigma_{1}^{m}=\Sigma_{1}^{m-1}+\Delta_{m}=m-1+\frac{\mathrm{H}_{m-1}}{-\chi_{m}}=m-1+\frac{H(p)+\sum_{k=1}^{m-1}\chi_{k}}{-\chi_{m}},
\]
which completes the proof of the lemma.
\end{proof}
We now restate and prove Theorem \ref{thm:gen criteria}. Recall that
we always assume that (\ref{eq:no com fix asump}) and (\ref{eq:m-ired and m-prox assump})
are satisfied. Moreover, recall from Definition \ref{def:Diophantine =000026 ESC}
that $\Phi$ is said to satisfy the exponential separation condition
(ESC) if it is Diophantine and generates a free semigroup.
\begin{thm*}
Let $1\le m\le d$ be with $\Sigma_{1}^{m-1}=m-1$ and $\Delta_{m}<1$.
Suppose that $\Phi$ satisfies the ESC, and that for $\nu_{m}^{*}$-a.e.
$V$
\[
\pi_{V}L_{k}\beta_{\omega}^{V}=\delta_{\pi_{V}L_{k}(\omega)}\text{ for }\beta\text{-a.e. }\omega\text{ and each }1\le k<m.
\]
Then $\dim\mu=\Sigma_{1}^{m}=\dim_{L}(\Phi,p)$.
\end{thm*}
\begin{proof}
Recall that we denote by $\mathcal{S}^{\mathrm{A}}$ the partition
of $\mathrm{A}_{d,d}$ into singletons. Since $\Phi$ satisfies the
ESC, it is Diophantine and $\Pi_{n}^{-1}\mathcal{S}^{\mathrm{A}}=\mathcal{P}_{n}$
for all $n\ge1$. Thus, by applying Theorem \ref{thm:gen dim result}
with $\mathcal{K}=\{1,...,m-1\}$, it follows that for $\nu_{m}^{*}$-a.e.
$V$
\[
\underset{n\rightarrow\infty}{\lim}\int\frac{1}{n}H(\beta_{\omega}^{V},\mathcal{P}_{n})\:d\beta(\omega)=0.
\]
Moreover, by Lemma \ref{lem:by =00005BFe, lem 4.6=00005D}, for $\nu_{m}^{*}$-a.e.
$V$
\[
\underset{n\rightarrow\infty}{\lim}\frac{1}{n}\int H(\beta_{\omega}^{V},\mathcal{P}_{n})\:d\beta(\omega)=\mathrm{H}_{m},
\]
which gives $\mathrm{H}_{m}=0$. This together with Lemma \ref{lem:H_m=00003D0 imp dim of proj =00003D dim_L}
implies $\Sigma_{1}^{m}=\dim_{L}(\Phi,p)$.

By Theorem \ref{thm:LY formula for SA},
\[
\dim\mu=\Sigma_{1}^{d}\ge\Sigma_{1}^{m}=\dim_{L}(\Phi,p).
\]
Since $\dim_{L}(\Phi,p)$ is always an upper bound for $\dim\mu$
(see \cite{JPS}), this completes the proof of the theorem.
\end{proof}
Next we prove Theorem \ref{thm:proj onto 1 =000026 2 dim subspaces},
which is the following statement. Given $\theta\in\mathcal{M}(\mathbb{R}^{m})$,
recall from Section \ref{subsec:Entropy-in-Rd} the definition of
its entropy dimensions $\dim_{e}\theta$, $\overline{\dim}_{e}\theta$
and $\underline{\dim}_{e}\theta$.
\begin{thm*}
Suppose that $\Phi$ satisfies the ESC. Then for every linear subspace
$V\subset\mathbb{R}^{d}$ with $\dim V\le2$,
\[
\dim_{e}\pi_{V}\mu=\min\left\{ \dim V,\dim_{L}(\Phi,p)\right\} .
\]
\end{thm*}
\begin{proof}
Since $\dim\mu\le\dim_{L}(\Phi,p)$, we have $\overline{\dim}_{e}\pi_{V}\mu\le\min\left\{ \dim V,\dim_{L}(\Phi,p)\right\} $
for every linear subspace $V\subset\mathbb{R}^{d}$. Moreover, by
Lemma \ref{lem:lb on ent of psi mu} it follows that $\underline{\dim}_{e}\pi_{V}\mu\ge\Sigma_{1}^{m}$
for all $1\le m\le d$ and $V\in\mathrm{Gr}_{m}(d)$. Hence, in order
to prove the theorem it suffices to show that $\Sigma_{1}^{m}\ge\min\left\{ m,\dim_{L}(\Phi,p)\right\} $
for $1\le m\le\min\{2,d\}$.

If $\Delta_{1}<1$ then it follows directly from Theorem \ref{thm:gen criteria}
that $\Sigma_{1}^{1}=\dim_{L}(\Phi,p)$. Thus, we may assume that
$d\ge2$, $\Sigma_{1}^{1}=1$ and $\Delta_{2}<1$.

Since $\Delta_{2}<1$ and by Theorem \ref{thm:LY formula for SA},
it follows that $\pi_{V}\mu$ is exact dimensional with $\dim\pi_{V}\mu<2$
for $\nu_{2}^{*}$-a.e. $V$. Thus, by Theorem \ref{thm:factor of proj of L_m-1},
\[
\pi_{V}L_{1}\beta_{\omega}^{V}=\delta_{\pi_{V}L_{1}(\omega)}\text{ for }\nu_{2}^{*}\text{-a.e. }V\text{ and }\beta\text{-a.e. }\omega.
\]
This together with Theorem \ref{thm:gen criteria} gives $\Sigma_{1}^{2}=\dim_{L}(\Phi,p)$,
which completes the proof of the theorem.
\end{proof}
Finally, we restate and prove Theorem \ref{thm:main d=00003D3}. Recall
from Definition \ref{def:OSC,  SOSC and SSC } the definition of the
strong open set condition (SOSC).
\begin{thm*}
Suppose that $d=3$, $\mathbf{S}_{\Phi}^{\mathrm{L}}$ is strongly
irreducible and proximal, and $\Phi$ satisfies the SOSC. Then $\dim\mu=\dim_{L}(\Phi,p)$.
\end{thm*}
\begin{proof}
We may assume that $\Phi$ consists of at least two maps, otherwise
$\dim\mu=0=\dim_{L}(\Phi,p)$. Thus, since $\Phi$ satisfies the SOSC,
assumption (\ref{eq:no com fix asump}) clearly holds. As pointed
out in Remark \ref{rem:m-con equiv to (d-m)-cond}, since $d=3$ the
strong irreducibility and proximality of $\mathbf{S}_{\Phi}^{\mathrm{L}}$
are equivalent to (\ref{eq:m-ired and m-prox assump}). Since $\Phi$
satisfies the SOSC, it is easy to show that it also satisfies the
ESC (see \cite[Section 6.2]{BHR}). Hence, the conditions of Theorem
\ref{thm:proj onto 1 =000026 2 dim subspaces} are all satisfied.
From that theorem, from Theorem \ref{thm:LY formula for SA} and by
Lemma \ref{lem:dim_e=00003Ddim}, we get that $\Sigma_{1}^{2}=\min\left\{ 2,\dim_{L}(\Phi,p)\right\} $.

If $\Sigma_{1}^{2}=\dim_{L}(\Phi,p)$, then from $\Sigma_{1}^{3}\ge\Sigma_{1}^{2}$,
Theorem \ref{thm:LY formula for SA} and $\dim\mu\le\dim_{L}(\Phi,p)$,
it follows that $\dim\mu=\dim_{L}(\Phi,p)$. Thus, we may assume that
$\Sigma_{1}^{2}=2$. Moreover, by \cite[Corollary 2.8]{BK} and since
$\Phi$ satisfies the SOSC, it follows that $\mathrm{H}_{3}=0$. Now
from Lemma \ref{lem:H_m=00003D0 imp dim of proj =00003D dim_L} and
Theorem \ref{thm:LY formula for SA} we again get $\dim\mu=\dim_{L}(\Phi,p)$,
which completes the proof of the theorem.
\end{proof}

\subsection{Proof of Theorem \ref{thm:main all d}}

Let us first restate the theorem. Recall that for $1\le m\le d$ we
write,
\[
\varrho_{m}:=(\chi_{1}-\chi_{m})\frac{(m-1)(m-2)}{2}-\sum_{k=1}^{m}\chi_{k}.
\]

\begin{thm*}
Suppose that $\Phi$ is Diophantine. Then for all $1\le m\le d$ with
$h(p_{\Phi})\ge\varrho_{m}$,
\[
\dim_{e}\pi_{V}\mu=m\text{ for each }V\in\mathrm{Gr}_{m}(d).
\]
In particular, if $h(p_{\Phi})\ge\varrho_{d}$ then $\dim\mu=d$.
\end{thm*}
\begin{proof}
For every linear subspace $V\subset\mathbb{R}^{d}$ we have $\overline{\dim}_{e}\pi_{V}\mu\le\dim V$.
Moreover, by Lemma \ref{lem:lb on ent of psi mu} it follows that
$\underline{\dim}_{e}\pi_{V}\mu\ge\Sigma_{1}^{l}$ for all $1\le l\le d$
and $V\in\mathrm{Gr}_{l}(d)$. From these facts, since $\mu$ is exact
dimensional, and by Lemma \ref{lem:dim_e=00003Ddim}, it follows that
in order to prove the theorem it suffices to show that $\Sigma_{1}^{l}=l$
for every $1\le l\le d$ with $h(p_{\Phi})\ge\varrho_{l}$.

Let $1\le l\le d$ be with $h(p_{\Phi})\ge\varrho_{l}$, and assume
by contraction that $\Sigma_{1}^{l}<l$. Since $0\le\Delta_{k}\le1$
for each $1\le k\le d$, there exists $1\le m\le l$ so that $\Sigma_{1}^{m-1}=m-1$
and $\Delta_{m}<1$. By Theorem \ref{thm:LY formula for SA}, the
measure $\pi_{V}\mu$ is exact dimensional with $\dim\pi_{V}\mu=\Sigma_{1}^{m}<m$
for $\nu_{m}^{*}$-a.e. $V$. Hence, from Theorem \ref{thm:factor of proj of L_m-1}
and since $\mathrm{Gr}_{m-1}(m)$ is a singleton when $m=1$,
\begin{equation}
\pi_{V}L_{m-1}\beta_{\omega}^{V}=\delta_{\pi_{V}L_{m-1}(\omega)}\text{ for }\nu_{m}^{*}\text{-a.e. }V\text{ and }\beta\text{-a.e. }\omega.\label{eq:factor of L_m-1 in last proof}
\end{equation}

Write
\[
\alpha:=(\chi_{1}-\chi_{m})\frac{(m-1)(m-2)}{2}.
\]
From (\ref{eq:factor of L_m-1 in last proof}) and by applying Theorem
\ref{thm:gen dim result} with $\mathcal{K}=\{m-1\}$ if $m\ge2$
and $\mathcal{K}=\emptyset$ if $m=1$, we get that for $\nu_{m}^{*}$-a.e.
$V$
\[
\underset{n\rightarrow\infty}{\limsup}\int\frac{1}{n}H(\Pi_{n}\beta_{\omega}^{V})\:d\beta(\omega)\le(\chi_{1}-\chi_{m})\sum_{j=1}^{m-2}j=\alpha.
\]

Note that for all $n\ge1$ the partition $\mathcal{P}_{n}$ is finer
than the partition $\Pi_{n}^{-1}\mathcal{S}^{\mathrm{A}}$. Thus,
from (\ref{eq:cond ent form}) and by the last inequality, for $\nu_{m}^{*}$-a.e.
$V$
\begin{equation}
\alpha\ge\underset{n\rightarrow\infty}{\limsup}\frac{1}{n}\int H(\beta_{\omega}^{V},\mathcal{P}_{n})-H(\beta_{\omega}^{V},\mathcal{P}_{n}\mid\Pi_{n}^{-1}\mathcal{S}^{\mathrm{A}})\:d\beta(\omega).\label{eq:alpha >=00003D limsup}
\end{equation}
By Lemma \ref{lem:by =00005BFe, lem 4.6=00005D},
\begin{equation}
\underset{n\rightarrow\infty}{\lim}\frac{1}{n}\int H(\beta_{\omega}^{V},\mathcal{P}_{n})\:d\beta(\omega)=\mathrm{H}_{m}\text{ for }\nu_{m}^{*}\text{-a.e. }V.\label{eq:by feng's paper}
\end{equation}
By the definition of $\Delta_{1},...,\Delta_{d}$, since $\Delta_{k}\le1$
for $1\le k\le d$, and since $\Delta_{m}<1$,
\[
\mathrm{H}_{k}-\mathrm{H}_{k-1}\ge\chi_{k}\text{ for each }1\le k\le d\text{ and }\mathrm{H}_{m}-\mathrm{H}_{m-1}>\chi_{m}.
\]
From this and since $\mathrm{H}_{0}=H(p)$,
\[
\mathrm{H}_{m}=\mathrm{H}_{0}+\sum_{k=1}^{m}(\mathrm{H}_{k}-\mathrm{H}_{k-1})>H(p)+\sum_{k=1}^{m}\chi_{k}.
\]
Hence, by (\ref{eq:alpha >=00003D limsup}) and (\ref{eq:by feng's paper}),
we get that for $\nu_{m}^{*}$-a.e. $V$
\begin{equation}
\alpha>H(p)+\sum_{k=1}^{m}\chi_{k}-\underset{n\rightarrow\infty}{\liminf}\frac{1}{n}\int H(\beta_{\omega}^{V},\mathcal{P}_{n}\mid\Pi_{n}^{-1}\mathcal{S}^{\mathrm{A}})\:d\beta(\omega).\label{eq:alpha >=00003D H + expo - liminf}
\end{equation}

Let $V\in\mathrm{Gr}_{m}(d)$ be given. By the concavity of conditional
entropy and since $\beta=\int\beta_{\omega}^{V}\:d\beta$, it follows
that for $n\ge1$
\begin{multline*}
\int H(\beta_{\omega}^{V},\mathcal{P}_{n}\mid\Pi_{n}^{-1}\mathcal{S}^{\mathrm{A}})\:d\beta(\omega)\le H\left(\beta,\mathcal{P}_{n}\mid\Pi_{n}^{-1}\mathcal{S}^{\mathrm{A}}\right)\\
=H(\beta,\mathcal{P}_{n})-H(\beta,\Pi_{n}^{-1}\mathcal{S}^{\mathrm{A}})
=nH(p)-H(p_{\Phi}^{*n}).
\end{multline*}
Thus, by the definition of $h(p_{\Phi})$,
\[
\underset{n\rightarrow\infty}{\liminf}\frac{1}{n}\int H(\beta_{\omega}^{V},\mathcal{P}_{n}\mid\Pi_{n}^{-1}\mathcal{S}^{\mathrm{A}})\:d\beta(\omega)\le H(p)-h(p_{\Phi}).
\]
From (\ref{eq:alpha >=00003D H + expo - liminf}) we now get
\[
\alpha>\sum_{k=1}^{m}\chi_{k}+h(p_{\Phi}),
\]
or equivalently $\varrho_{m}>h(p_{\Phi})$. But since $m\le l$ this
contradicts $h(p_{\Phi})\ge\varrho_{l}$, which completes the proof
of the theorem.
\end{proof}
\bibliographystyle{plain}
\bibliography{bibfile}

\begin{thebibliography}{10}

\bibitem{BHR}
B.~B{\'a}r{\'a}ny, M.~Hochman, and A.~Rapaport.
\newblock Hausdorff dimension of planar self-affine sets and measures.
\newblock {\em Invent. Math.}, 216(3):601--659, 2019.

\bibitem{BK}
B.~B{\'a}r{\'a}ny and A.~K{\"a}enm{\"a}ki.
\newblock Ledrappier--{Y}oung formula and exact dimensionality of self-affine
  measures.
\newblock {\em Adv. Math.}, 318:88--129, 2017.

\bibitem{BQ}
Y.~Benoist and J.-F. Quint.
\newblock {\em Random walks on reductive groups}.
\newblock Springer International Publishing, 2016.

\bibitem{BL}
P.~Bougerol and J.~Lacroix.
\newblock {\em Products of random matrices with applications to
  {S}chr\"{o}dinger operators}, volume~8 of {\em Progress in Probability and
  Statistics}.
\newblock Birkh\"{a}user Boston, Inc., Boston, MA, 1985.

\bibitem{CE}
J.~Cheeger and D.~G. Ebin.
\newblock {\em Comparison theorems in {R}iemannian geometry}.
\newblock AMS Chelsea Publishing, Providence, RI, 2008.
\newblock Revised reprint of the 1975 original.

\bibitem{dC}
M.~P. do~Carmo.
\newblock {\em Riemannian geometry}.
\newblock Mathematics: Theory \& Applications. Birkh\"{a}user Boston, Inc.,
  Boston, MA, 1992.
\newblock Translated from the second Portuguese edition by Francis Flaherty.

\bibitem{Edg-integral}
G.~A. Edgar.
\newblock {\em Integral, probability, and fractal measures}.
\newblock Springer-Verlag, New York, 1998.

\bibitem{EiWa}
M.~Einsiedler and T.~Ward.
\newblock {\em Ergodic theory with a view towards number theory}, volume 259 of
  {\em Graduate Texts in Mathematics}.
\newblock Springer-Verlag London, Ltd., London, 2011.

\bibitem{falconer1988hausdorff}
K.~J. Falconer.
\newblock The {H}ausdorff dimension of self-affine fractals.
\newblock {\em Math. Proc. Camb. Phil. Soc.}, 103(2):339--350, 1988.

\bibitem{FFJ_proj_survey}
K.~J. Falconer, J.~Fraser, and X.~Jin.
\newblock Sixty years of fractal projections.
\newblock In {\em Fractal geometry and stochastics {V}}, volume~70 of {\em
  Progr. Probab.}, pages 3--25. Birkh\"{a}user/Springer, Cham, 2015.

\bibitem{FK}
K.~J. Falconer and T.~Kempton.
\newblock The dimension of projections of self-affine sets and measures.
\newblock {\em Ann. Acad. Sci. Fenn. Math.}, 42(1):473--486, 2017.

\bibitem{FLR}
A.-H. Fan, K.-S. Lau, and H.~Rao.
\newblock Relationships between different dimensions of a measure.
\newblock {\em Monatsh. Math.}, 135(3):191--201, 2002.

\bibitem{Fe}
D.-J. Feng.
\newblock Dimension of invariant measures for affine iterated function systems.
\newblock {\em To appear in Duke Math. J.}, 2019.
\newblock arXiv:1901.01691.

\bibitem{Gu}
Y.~Guivarc'h.
\newblock Produits de matrices al\'{e}atoires et applications aux
  propri\'{e}t\'{e}s g\'{e}om\'{e}triques des sous-groupes du groupe
  lin\'{e}aire.
\newblock {\em Ergodic Theory Dynam. Systems}, 10(3):483--512, 1990.

\bibitem{Ho1}
M.~Hochman.
\newblock On self-similar sets with overlaps and inverse theorems for entropy.
\newblock {\em Ann. of Math. (2)}, 180(2):773--822, 2014.

\bibitem{Ho}
M.~Hochman.
\newblock On self-similar sets with overlaps and inverse theorems for entropy
  in {$\mathbb{R}^d$}.
\newblock {\em To appear in Mem. Am. Math. Soc.}, 2015.
\newblock arXiv:1503.09043.

\bibitem{HR}
M.~Hochman and A.~Rapaport.
\newblock Hausdorff dimension of planar self-affine sets and measures with
  overlaps.
\newblock {\em J. Eur. Math. Soc. (JEMS)}, 24(7):2361--2441, 2022.

\bibitem{HoSh}
M.~Hochman and P.~Shmerkin.
\newblock Local entropy averages and projections of fractal measures.
\newblock {\em Ann. of Math. (2)}, 175(3):1001--1059, 2012.

\bibitem{Hut}
J.~E. Hutchinson.
\newblock Fractals and self-similarity.
\newblock {\em Indiana Univ. Math. J.}, 30(5):713--747, 1981.

\bibitem{JPS}
T.~Jordan, M.~Pollicott, and K.~Simon.
\newblock Hausdorff dimension for randomly perturbed self affine attractors.
\newblock {\em Comm. Math. Phys.}, 270(2):519--544, 2007.

\bibitem{KRS}
A.~K\"{a}enm\"{a}ki, T.~Rajala, and V.~Suomala.
\newblock Existence of doubling measures via generalised nested cubes.
\newblock {\em Proc. Amer. Math. Soc.}, 140(9):3275--3281, 2012.

\bibitem{Le}
J.~M. Lee.
\newblock {\em Introduction to smooth manifolds}, volume 218 of {\em Graduate
  Texts in Mathematics}.
\newblock Springer, New York, second edition, 2013.

\bibitem{MS}
I.~D. Morris and C.~Sert.
\newblock Personal communication, 2022.

\bibitem{MoSh}
I.~D. Morris and P.~Shmerkin.
\newblock On equality of {H}ausdorff and affinity dimensions, via self-affine
  measures on positive subsystems.
\newblock {\em Trans. Amer. Math. Soc.}, 371(3):1547--1582, 2019.

\bibitem{Pa}
W.~Parry.
\newblock {\em Topics in ergodic theory}, volume~75 of {\em Cambridge Tracts in
  Mathematics}.
\newblock Cambridge University Press, Cambridge-New York, 1981.

\bibitem{Pe}
P.~Petersen.
\newblock {\em Riemannian geometry}, volume 171 of {\em Graduate Texts in
  Mathematics}.
\newblock Springer, New York, second edition, 2006.

\bibitem{Ra}
A.~Rapaport.
\newblock On self-affine measures with equal {H}ausdorff and {L}yapunov
  dimensions.
\newblock {\em Trans. Amer. Math. Soc.}, 370(7):4759--4783, 2018.

\bibitem{Ra_Rajchman}
A.~Rapaport.
\newblock On the {R}ajchman property for self-similar measures on {$\Bbb R^d$}.
\newblock {\em Adv. Math.}, 403:Paper No. 108375, 53, 2022.

\bibitem{Ru}
D.~Ruelle.
\newblock Ergodic theory of differentiable dynamical systems.
\newblock {\em Inst. Hautes \'{E}tudes Sci. Publ. Math.}, (50):27--58, 1979.

\bibitem{Sh_proj_survey}
P.~Shmerkin.
\newblock Projections of self-similar and related fractals: a survey of recent
  developments.
\newblock In {\em Fractal geometry and stochastics {V}}, volume~70 of {\em
  Progr. Probab.}, pages 53--74. Birkh\"{a}user/Springer, Cham, 2015.

\bibitem{So}
B.~Solomyak.
\newblock Measure and dimension for some fractal families.
\newblock {\em Math. Proc. Cambridge Philos. Soc.}, 124(3):531--546, 1998.

\bibitem{Wh}
H.~Whitney.
\newblock Elementary structure of real algebraic varieties.
\newblock {\em Ann. of Math. (2)}, 66:545--556, 1957.

\end{thebibliography}

$\newline$\textsc{Department of Mathematics, Technion, Haifa, Israel}$\newline$$\newline$\textit{E-mail: }
\texttt{arapaport@technion.ac.il}
\end{document}